\documentclass{amsart}

\usepackage{graphicx}
\usepackage{amsmath}
\usepackage{amsthm}
\usepackage{amsfonts}
\usepackage{amssymb}
\usepackage{mathrsfs}

\usepackage{verbatim}

\usepackage[lofdepth,lotdepth]{subfig}
\usepackage{tikz-cd}
\tikzset{
  symbol/.style={
    draw=none,
    every to/.append style={
      edge node={node [sloped, allow upside down, auto=false]{$#1$}}}
  }
}

\renewcommand{\subsubsection}[1]{\vspace{2mm}\noindent\textbf{#1.}}

\usepackage{color}
\definecolor{newblue}{RGB}{94,89,144}
\definecolor{newblue2}{cmyk}{1,0.6,0,0.06}
\definecolor{grey}{gray}{0.5}
\usepackage[pdfborder={0 0 0},
  colorlinks=true,
  citecolor=newblue,
  linkcolor=newblue2,
  urlcolor=newblue]{hyperref}

\usepackage[longnamesfirst]{natbib}

\usepackage[nameinlink]{cleveref}

\newtheorem{theorem}{Theorem}
\newtheorem{proposition}[theorem]{Proposition}
\newtheorem{lemma}[theorem]{Lemma}
\newtheorem{corollary}[theorem]{Corollary}
\newtheorem{problem}[theorem]{Problem}

\theoremstyle{remark}
\newtheorem{definition}[theorem]{Definition}
\newtheorem{remark}[theorem]{Remark}

\newtheorem{assumption}[theorem]{Assumption}
\newtheorem{example}[theorem]{Example}

\numberwithin{figure}{section}
\numberwithin{table}{section}

\numberwithin{equation}{section}
\numberwithin{theorem}{section}

\usepackage{enumitem}
\setdescription{style=unboxed,leftmargin=0cm}
\newlist{deflist}{enumerate}{1}
\setlist[deflist]{label=(\alph{deflisti}),ref={\thetheorem(\alph{deflisti})}}
\newlist{subdeflist}{enumerate}{1}
\setlist[subdeflist]{label=(\roman{subdeflisti}),ref={\thetheorem(\alph{deflisti})(\roman{subdeflisti})}}

\newlist{exlist}{enumerate}{1}
\setlist[exlist]{label=(\alph{exlisti}),ref={\thetheorem(\alph{exlisti})}}
\newlist{remlist}{enumerate}{1}
\setlist[remlist]{label=(\alph{remlisti}),ref={\thetheorem(\alph{remlisti})}}

\crefname{section}{Sec.}{Sec.}
\crefname{table}{Tab.}{Tab.}
\crefname{figure}{Fig.}{Fig.}
\crefname{appendix}{App.}{App.}
\crefname{equation}{}{}
\crefname{theorem}{Thm.}{Thm.}
\crefname{proposition}{Prop.}{Prop.}
\crefname{lemma}{Lem.}{Lem.}
\crefname{corollary}{Cor.}{Cor.}
\crefname{problem}{Prob.}{Prob.}
\crefname{definition}{Def.}{Def.}
\crefname{deflisti}{Def.}{Def.}
\crefname{subdeflisti}{Def.}{Def.}
\crefname{remark}{Rem.}{Rem.}
\crefname{remlisti}{Rem.}{Rem.}
\crefname{note}{Note}{Note}
\crefname{assumption}{Ass.}{Ass.}
\crefname{example}{Ex.}{Ex.}
\crefname{exlisti}{Ex.}{Ex.}

\newcommand{\dd}{\,\mathrm{d}}
\newcommand{\abs}[1]{\left|#1\right|}
\newcommand{\norm}[1]{\left\|#1\right\|}

\newcommand{\B}{\mathcal{B}}
\newcommand{\C}{\mathcal{C}}

\newcommand{\D}{\mathcal{D}}

\newcommand{\F}{\mathcal{F}}

\renewcommand{\H}{\mathcal{H}}

\renewcommand{\L}{\mathcal{L}}
\newcommand{\M}{\mathcal{M}}
\newcommand{\Nbb}{\mathbb{N}}
\newcommand{\N}{\mathcal{N}}
\renewcommand{\P}{P}
\newcommand{\Pp}{\mathbb{P}}

\newcommand{\R}{\mathbb{R}}
\renewcommand{\S}{\mathcal{S}}
\newcommand{\T}{\mathscr{T}}
\newcommand{\V}{\mathcal{V}}
\newcommand{\W}{\mathcal{W}}
\newcommand{\X}{\mathcal{X}}
\newcommand{\Z}{\mathcal{Z}}

\newcommand{\reg}{\mathrm{reg}}

\newcommand{\dual}[1]{\mathcal{R}(#1)}
\newcommand{\dualcoupled}[1]{\underline{\mathcal{R}}(#1)}
\newcommand{\dualbulk}[1]{{\mathcal{R}_\Omega}(#1)}
\newcommand{\dualsurface}[1]{{\mathcal{R}_\Gamma}(#1)}

\newcommand{\A}{\mathscr{A}}
\renewcommand{\b}{\mathscr{B}}
\renewcommand{\c}{\mathscr{C}}

\DeclareMathAlphabet{\mathpzc}{OT1}{pzc}{m}{it} %
\newcommand{\la}{\mathpzc{a}\,}
\newcommand{\lb}{\mathpzc{b}\,}
\newcommand{\lc}{\mathpzc{c}\,}

\newcommand{\nortime}{\partial^\circ}
\newcommand{\md}{\partial^\bullet}
\newcommand{\mdh}{{\partial^\bullet_h}}
\newcommand{\mdell}{{\partial^\bullet_\ell}}
\newcommand{\mdK}{{\partial^\bullet_K}}

\newcommand{\eps}{\varepsilon}
\newcommand{\id}{\mathrm{Id}}

\renewcommand{\tilde}{\widetilde}
\renewcommand{\min}{\mathrm{min}}
\DeclareMathOperator{\diam}{\mathrm{diam}}
\DeclareMathOperator{\meas}{\mathrm{meas}}
\DeclareMathOperator{\rank}{\mathrm{rank}}
\DeclareMathOperator{\trace}{\mathrm{trace}}

\DeclareMathOperator{\interior}{\mathrm{int}}
\DeclareMathOperator{\adv}{\mathrm{adv}}

\newcommand{\dt}{\frac{\mathrm{d}}{\mathrm{d} t}}

\renewcommand{\vec}[1]{{\ensuremath{\boldsymbol{#1}}}}

\newcommand{\utext}{\ensuremath{\mathrm{u}}}
\newcommand{\vtext}{\ensuremath{\mathrm{v}}}

\let\Oldnu\nu
\renewcommand{\nu}{\vec{\Oldnu}}
\let\Oldmu\mu
\renewcommand{\mu}{\vec{\Oldmu}}

\usepackage{mathptmx}
\DeclareMathAlphabet{\mathcal}{OMS}{cmsy}{m}{n}
\DeclareSymbolFont{newfont}{OML}{cmm}{m}{it}%
\DeclareMathSymbol{\myvarrho}{3}{newfont}{37}%

\newcounter{rconstno}
\newcommand{\rconst}[1]{\refstepcounter{rconstno}\label{#1}\ref*{#1}}

\usepackage{natbib}

\usepackage{verbatim}

\begin{document}

\title[Finite element analysis for PDE in evolving domains]{A unified  theory for continuous in time evolving finite element space 
approximations to partial differential equations in evolving domains}

\author[C. M. Elliott]{C. M.\ Elliott\textsuperscript{1}}
\thanks{\textsuperscript{1}Mathematics Institute, Zeeman Building, University of Warwick, Coventry. CV4 7AL. UK. \emph{Email}. C.M.Elliott@warwick.ac.uk\\ 
The work of CME was supported by a Royal Society Wolfson Research Merit Award.}

\author[T. Ranner]{T. Ranner\textsuperscript{2}}
\thanks{\textsuperscript{2}School of Computing, University of Leeds, Leeds. LS2 9JT. UK. %
\emph{Email}. T.Ranner@leeds.ac.uk\\ The work of TR was supported by Engineering and Physical Sciences Research
Council (EPSRC EP/J004057/1) and a Leverhulme Trust Early Career Fellowship.}


\date{\today}

\begin{abstract}
We develop a unified theory for continuous in time  finite element discretisations of partial 
differential equations posed in evolving domains including  the consideration of equations posed on evolving surfaces 
and bulk domains as well as coupled surface bulk systems. We use  an abstract
variational setting with time dependent function spaces and abstract time dependent finite element spaces. Optimal {a priori} bounds are shown under usual 
assumptions on perturbations of bilinear forms and approximation properties of the abstract finite element spaces. The abstract theory is applied to 
evolving  finite elements in both flat and curved spaces. Evolving bulk and surface isoparametric finite element spaces defined on 
evolving triangulations are defined and developed. These spaces are used to define approximations to  parabolic equations in 
general domains for which the abstract theory is shown to apply.   Numerical experiments are described which confirm the rates of convergence.
\end{abstract}

\maketitle
\begin{center}
\textbf{Dedicated to the memory of John W. Barrett}
\end{center}

  \tableofcontents

\section{Introduction}

{

}

In this paper, we develop a unified theory for finite element discretisations of partial 
differential equations posed in evolving domains including  the consideration of equations posed on evolving surfaces 
and bulk domains as well as coupled surface bulk systems. The discretisation is based on 
evolving finite element spaces defined on evolving triangulations using isoparametric elements. Optimal order a priori error bounds are proven. 
This unification is achieved by using an abstract
variational setting with time dependent abstract  function spaces and time dependent abstract  finite element spaces. Given a time dependent Hilbert triple
\begin{equation*}
  \V(t) \subset \H(t) \subset \V^*(t),
\end{equation*}
the abstract strong formulation is: Find $\utext(t)\in \V(t)$ such that
\begin{subequations}
  \begin{align}
    \md \utext + \L(t) \utext + \omega(t) \utext & = 0 && \mbox{ in  } \V^*(t) \\
    \utext(0) & = \utext_0,
  \end{align}
\end{subequations}
where $\V(t)$ is an appropriate time dependent  Hilbert space with dual $\V^*(t)$, $\md \utext$ is an 
appropriate abstract material derivative arising from the evolution of the spaces,
$\L(t)$ is an abstract elliptic operator satisfying suitable coercivity properties and $\omega(t)$ is a lower order  term arising from evolution of the space.
Similar to the case of time independent function spaces this equation may be written in variational form as
\begin{problem}
  Given $\utext_0 \in \V_0 := \V(0)$, find $\utext \in L^2_\V$ with $\md \utext \in L^2_\H$ such that for almost every $t \in [0,T]$,
  \begin{equation}
    \label{eq:cts-scheme-intro}
    m( t; \md \utext, \varphi )
    + g( t; \utext, \varphi )
    + a( t; \utext, \varphi )
    = 0
    \qquad \mbox{ for all } \varphi \in \V(t),
  \end{equation}
  subject to the initial condition $\utext( 0 ) = \utext_0$.
\end{problem}
Here $L^2_\V$ and $ L^2_\H$ are generalisations of the Bochner spaces $L^2(0,T;V)$ and $L^2(0,T;H)$, where $V \subset H \subset V^*$ is a time-independent Hilbert triple.
The bilinear form $a(\cdot,\cdot)$ is associated with the elliptic operator $\L$ and the bilinear forms $m$ and $g$ are associated 
with the $\H(t)$-inner product and its time derivative (i.e.\ the operator $\omega$). The above problem is also well posed for $\utext_0\in \H(0)$ 
where we seek  $\md \utext \in L^2_{\V^*}$. Here we seek error bounds for sufficiently smooth solutions.

We formulate and analyse  an  abstract finite element discretisation based on a Galerkin ansatz with perturbations of the bilinear forms.
Under assumptions on the approximation of geometry and the approximation of function spaces by abstract finite element spaces optimal order error bounds
are proved. Then we construct realisations of this abstract setting in the context of partial differential equations on evolving domains.
This is achieved by means of the construction of  evolving finite element spaces on evolving triangulations of time dependent surfaces and bulk domains from first principals.
We give a concrete realisation of these spaces based on evolving Lagrange isoparametric finite elements.

This approach is applied to three model problems:  a 
linear parabolic problem in an evolving, bounded bulk domain in $\R^{n+1}$,  a linear parabolic problem on an evolving compact $n$-dimensional surface embedded in $\R^{n+1}$ and a linear parabolic problem coupling problems in an
 evolving, bounded bulk domain in $\R^{n+1}$ to a problem on its boundary.
In each case, we assume that the evolution of the problem domain is prescribed.
The abstract approach is applicable to other situations including nonlinear equations, coupled equations and problems with dynamic boundary conditions.%

\subsection{Background}
\label{sec:background}

Partial differential equations posed on complex evolving domains arise in numerous settings such as surfactant transport on fluid interfaces, 
receptor ligand dynamics on cell surfaces and phase separation on dissolving alloy surfaces \citep{DecEllKor15, EllRanVen15, BarGarNur15, AlpEllTer17, TorMilArr19, ZimTosLan19}. 
Numerical approaches to solve these problems include surface finite elements, implicit surface formulations, diffuse interface approximations, 
trace finite elements, unfitted finite elements, finite volume schemes and mesh free methods. 
See the works of \citet{Dzi88, DziEll07, DecDziEll09, DziEll10, DecEllRan14, DecSty18,  OlsReu16,  BurHanLar16,  LehOlsXu18, LehOls19, GieMul14, DecEllMiu19, SucKuh19} and the review of \citet{DziEll13}.

The problem of solving parabolic problems in evolving bulk domains has been studied for many years.
In particular, we mention the ALE (arbitrary Lagrangian-Eulerian) approach first proposed by
\citet{Hirt1974} in the context of finite difference methods and by \citet{Donea1982,Hughes1981} for finite element methods.
Analysis of a similar problem considering both spatial and temporal discretisation is given by \citet{BadCod06,BofGas04,ForNob99,ForNob04,Gas01,Nob01,GawLew15}.
The analysis by \citet{Bonito2013a,Bonito2013} provides optimal order convergence for a discrete Galerkin in time approach.

The study of finite element methods for partial differential equations posed on surfaces started with the seminal work  of \citet{Dzi88}.
Evolving surface finite elements were introduced and analysed  for  an advection diffusion  equation posed on evolving surfaces by \citet{DziEll07, DziEll12, DziEll13}. In these works optimal error bounds were proved for
piecewise linear finite elements on triangulated surface  evolved using the normal and advective velocity.
In this work we consider a more general parabolic equation on surfaces and discretisations which cover the case of 
 higher-order schemes \citep{Hei05,Dem09,Kov18} and arbitrary Lagrangian-Eulerian methods \citep{EllSty12,EllVen15}.
The discretisations  presented in this paper can be combined with different time stepping schemes \citep{DziLubMan11,DziEll12,LubManVen13,KovPow18}
 to provide a fully discrete scheme.
We also mention the analysis of \citet{KovLiLub17} who study a diffusion equation on the surface which drives the a priori unknown evolution of the surface.

Our abstract approach will be applied to equations posed on evolving bulk domains and coupled bulk surface systems.
See the work of \citet{EllRan13, GroOlsReu14-pp, BurHanLar16} for approaches to stationary surface problems.
\citet{KovLub17} extended the results for piecewise linear elements to  a coupled bulk-surface system.
The functional analytic setting will be the product of spaces over the bulk domain, $\Omega(t)$, and the surface, $\Gamma(t)$.

\subsection{Some partial differential equations on evolving domains}
Let $T > 0$. For $t \in [0,T]$ let $\Omega(t)$ denote an $(n+1)$-dimensional bounded, connected, open subset (domain)  in $\R^{n+1}$, for $n=1,2,3$. We denote by $\Gamma(t)$ the boundary $\partial \Omega(t)$  assumed to be  a  sufficiently smooth compact $n$-dimensional orientable hypersurface with unit normal $\nu(t)$ pointing outward from $\Omega(t)$.
We write $\Omega_0 = \Omega(0)$, $\Gamma_0 = \Gamma(t)$. 

We assume that there exists a sufficiently smooth mapping (called the flow map)
$\Phi_{(\cdot)}(\cdot)\colon [0,T] \times \bar \Omega_0  \rightarrow \R^{n+1}$
satisfying

\begin{enumerate}
\item $\Phi_t(\cdot)$ is a diffeomorphism of $\bar \Omega_0$ onto $\bar \Omega(t)$ for each $t \in [0,T)$;
\item $\Phi_0(\cdot)=\id |_{\bar\Omega_0}$.
\end{enumerate}
We will write $\Phi_{-t} \colon \bar\Omega(t) \to \bar\Omega_0$ for the inverse of $\Phi_t$.
 We define the material velocity $w$ by
\begin{multline*}
w:\bar \Omega_T \left( :=  \bigcup_{t \in (0,T)} \bar \Omega(t) \times \lbrace t \rbrace \right)
 \rightarrow \R^{n+1}, \\
w(x,t):= \frac{\partial \Phi_t}{\partial t}\bigl( (\Phi_t(\cdot))^{-1}(x) \bigr), \quad x \in \bar \Omega(t),
t\in (0,T)
\end{multline*}
which may be written as
\begin{equation} \label{bulk-veldef}
\frac{\partial}{\partial t}\Phi_t(p) = w(\Phi_t(p),t), \; t \in (0,T), \quad
\Phi_0(p) = p.
\end{equation}

We use the terminology that $\Omega(t)$ is an evolving  bulk domain and $\Gamma(t)$ is an evolving  surface domain. In order to define the evolution of the domains as sets we need only to specify  the normal velocity $V\nu=(w\cdot \nu)|_{\Gamma(t)}\nu$ of $\Gamma(t)$. The tangential components of $w= \frac{\partial \Phi_t}{\partial t}$ allows for an  arbitrary parametrisation of the domain.

Illustrative examples of the boundary value problems we wish to consider are:
\begin{enumerate}
\item
Find a time-dependent scalar field $\utext$ such that on an evolving Cartesian bulk domain  $\Omega(t)$%
\begin{subequations}
  \label{eq:bulk-eqn}
  \begin{align}
    \label{eq:bulk-eqn-pde}
    \utext_t+\nabla \cdot (\lb_\Omega \utext) - \nabla \cdot (\la_\Omega \nabla \utext)+\lc_\Omega \utext
    & = 0
    && \mbox{ on } \Omega(t) \\
    \label{eq:bulk-eqn-bc}
    ( \la_\Omega \nabla \utext - (\lb_\Omega -V\nu)\utext) \cdot \nu
    & = 0
    && \mbox{ on } \Gamma(t) \\
    \label{eq:bulk-eqn-initial}
    \utext(0 )
    & = \utext_0
    && \mbox{ on } \Omega_0 := \Omega(0),
  \end{align}
\end{subequations}
where $\la_\Omega$ is a smooth diffusion tensor , $\lb_\Omega$ a smooth vector field 
and $\lc_\Omega$ a smooth scalar field.

\item

Seek a time-dependent scalar surface field $\utext$ such that
\begin{subequations}
  \label{eq:surf-eqn}
  \begin{align}
    \label{eq:surf-eqn-pde}
    \nortime   \utext + \nabla_\Gamma \cdot (\lb_\Gamma \utext)  - \nabla_\Gamma \cdot (\la_\Gamma \nabla_\Gamma \utext)  +\lc_\Gamma\utext
    & = 0
    && \mbox{ on } \Gamma(t) \\
    \label{eq:surf-eqn-initial}
      \utext (0 )
    & =   \utext_0
    && \mbox{ on } \Gamma_0 := \Gamma(0),
  \end{align}
\end{subequations}
where $\la_\Gamma$ is a smooth diffusion tensor which maps the tangent space of $\Gamma$ into itself %
and $\lc_\Gamma$ is a smooth scalar field. Here $\lb_\Gamma$ is a tangential vector  field.  
 We use the notation  $\nortime  \utext$ for the normal time derivative \citep{CerFriGur05, DziEll13a}, which denotes 
 the time derivative of a function along a trajectory on $\Gamma(t)$ moving in the normal direction.  
 This equation is supposed to hold pointwise on $\Gamma(t)$ on trajectories evolving from
$\Gamma(0)$ with velocity $V \nu$.

\item
 Find a time-dependent pair $(\utext,\vtext)$ with $\utext$ a scalar volumetric field and $\vtext$ a scalar surface field such that
\begin{subequations}
  \label{eq:bulksurf-eqn}
  \begin{align}
    \label{eq:bulksurf-bulk}
     \utext_t + \nabla \cdot (\lb_\Omega \utext)  -\nabla \cdot (\la_\Omega \nabla \utext)+\lc_\Omega \utext
& = 0
    && \mbox{ on } \Omega(t) \\
    \label{eq:bulksurf-coupling}
    ( \la_\Omega \nabla \utext - (\lb_\Omega-V \nu) \utext ) \cdot \nu +  \alpha \utext - \beta \vtext
    & = 0
    && \mbox{ on } \Gamma(t) \\
    \label{eq:bulksurf-surf}
    \nortime \vtext + \nabla _\Gamma \cdot ( \lb_\Gamma  \vtext)  - \nabla_\Gamma \cdot (\la_\Gamma  \vtext ) +\lc_\Gamma \vtext + \beta v-\alpha u
    & = 0
    && \mbox{ on } \Gamma(t) \\
    \label{eq:bulksurf-eqn-bulk-initial}
    \utext( \cdot, 0 )
    & = \utext_0
    && \mbox{ on } \Omega_0 := \Omega(0) \\
    \label{eq:bulksurf-eqn-surf-initial}
    \vtext( \cdot, 0 )
    & = \vtext_0
    && \mbox{ on } \Gamma_0 := \Gamma(0),
  \end{align}
\end{subequations}
where  $\alpha$ and $\beta$ are positive constants.
Equation \cref{eq:bulksurf-coupling} couples the equations posed on the domain $\Omega(t)$ and its boundary $\Gamma(t)$.
\end{enumerate}
\begin{remark}

 \begin{remlist}
   \item
  We do not explain the models in the above equations other than to comment that examples may be derived as  conservation laws for  scalar quantities subject to  
    diffusive and advective fluxes with linear  reactions. Note that in these equations the only part of the material velocity $w$   that appears is the normal velocity of $\Gamma(t)$.
    
      \item
    Our formulation of these initial boundary value problems   allows, by means of a tangential 
    velocity field, for a reparametrisation of the evolving domains.  In our definition of material velocity the velocity component tangential to $\Gamma(t)$ is used to define a parametrisation of the domains $\Omega(t)$ and $\Gamma(t)$.
This is an  {\it Arbitrary Lagrangian Eulerian (ALE)}  approach.  A consequence is that the differential operators and the function spaces in the abstract setting depend on this  {\it ALE} velocity. In the discretisations the meshes we use are transported using this velocity.
   
  \item
 Working  with evolving 
triangulations on evolving domains has advantages 
when the domain is unknown and is to be found. For example the 
hypersurface $\Gamma(t)$ may be a free boundary or a surface evolving via  a geometric evolution equation coupled to the
parabolic equation on the bulk or surface domain, \cite{PozSti17, KovLiLub17, BarDecSty17, PozSti18}. Evolving triangulated surfaces are often computed as approximations of geometric evolution flows, \cite{DecDziEll05, BarGarNur19, KovLiLub18}. 

  \item
  The analysis applies to the use of evolving finite element spaces in the case of domains which are time independent. This may be useful when devising schemes which evolve the mesh in time in order to  adjust to the solution. 
 
  \end{remlist}
\end{remark}

\subsection{Contributions}
\label{sec:contributions}
The primary contributions of this work which generalises the evolving surface finite element continuous in time analysis of \cite{DziEll07, DziEll13}  are:-
\begin{itemize}
\item
The analysis is carried out in a generalised  abstract setting  extending  earlier work by allowing for  arbitrary parametrisations and higher order approximations together with domains with boundaries.
\item
We construct  evolving finite element spaces  and derive approximation results for quasi-uniform evolving surface and bulk  triangulations using isoparametric  elements.
\item
We show that the resulting finite element methods satisfy the necessary geometric and interpolation estimates for the abstract theory to apply.
\item
  We give a complete presentation of the  theory for the numerical approximation of parabolic equations on evolving domains using  evolving finite element spaces %
and  give three examples of  concrete  realisations to initial boundary value problems involving time dependent domains in flat and curved space.
In particular the approach to evolving bulk domains with boundary requires a more complex treatment than that of surfaces without boundary and the abstract theory is structured to allow this.

\item
The analysis provides error bounds for higher order isoparametric approximations of  the case of evolving curved hypersurfaces, 
evolving bulk domains and evolving coupled surface- bulk systems.

\end{itemize}

\subsection{Extensions}
\begin{remlist}
\item
All the theory presented in this paper will be applicable with the addition of 
    sufficiently smooth right hand side functionals under natural appropriate assumptions.
\item
    Although we will apply our abstract formulation to settings where we approximate a known continuous domain, our intention is that this framework may also be applied to situations where the evolution of the domain is a priori unknown. 
       One way to interpret this work would be to determine the minimum requirements of an evolving finite element method in order to have a sensible, well-posed method.
       
       \item
In this paper we do not consider the  example of a parabolic PDE posed on an evolving $n$-dimensional curved sub-manifold $\Gamma(t)$ in $\R^{n+1}$ with a moving boundary $\partial \Gamma(t)$ on which there is  an appropriate boundary condition. 
On the other hand, it is possible to take the perspective in our first example on an evolving bulk domain 
that  $\Omega(t)\subset \R^{n+2}$ is  an evolving flat 
  sub-manifold with moving boundary  and work with the definitions from \cref{sec:sfem} (see also \cref{ex:bulk-embed-sfe}), 
   though we choose not to follow this approach.   
\end{remlist}

\subsection{Outline}
 The article is set out in three parts. The first part comprises the abstract theory.   In \cref{sec:abstract-formulation}, we introduce the abstract
 functional analytic setting in which we pose the continuous partial differential equations.
 An abstract analysis of evolving finite element methods is provided in \cref{AbstractNA}.
The second part comprises the construction of and approximation theory for evolving finite element spaces.
We introduce our basic theory for evolving bulk and evolving surface finite elements in \cref{sec:bfem} and \cref{sec:sfem}.
 \Cref{sec:bulk-lift,sec:lift-fem} give technical details on relating the finite elements used in our computational methods to the structures in the underlying continuous problem.
 The third part comprises  applications to three PDE settings.
\Cref{BULKPDE} to \cref{sec:application3} apply these ideas to tackle three model problems.
In  \cref{sec:numerical-results}  results of numerical experiments are given   confirming  the proven error bounds.
Finally, there is an appendix (\cref{sec:faa-di-bruno})
 with a technical result concerning a  Fa\'a di Bruno formula for parametric surfaces.

\clearpage
\part{Abstract theory}\label{abstract}
\section{Abstract formulation}\label{sec:abstract-formulation}
\subsection{Evolving function spaces}
We introduce an abstract functional analytic setting formulated  by \citet{AlpEllSti15a-pp} generalising  the surface parabolic PDE setting of \citet{Vie14}.
One of the key novelties of these works is to provide the basic theory for evolving Bochner-like spaces for evolving Hilbert spaces 
such as $H^1(\Gamma(t))$ in order make a definition similar to ``$L^2(0,T; H^1(\Gamma(t)))$''. Using this formulation, we can pose partial differential equations on evolving domains in a fully rigorous setting.

\begin{remark}
The work of \citet{AlpEllSti15a-pp} uses a Lagrangian formulation where the evolving domain is parametrised over the initial domain.
This matches well with the arbitrary Lagrangian-Eulerian finite element methods we will consider. The setting may be applied to evolving parametrisations of fixed domains. 
A different functional analytic setting may be more appropriate for different discretisation approaches such as the trace finite element method \citep{OlsReuXu14,OlsReu16} or the implicit surface approach \citep{DziEll10}.
\end{remark}
\begin{definition}
  [Compatibility]
  \label{def:compatibility}

  For $t \in [0,T]$, let $\X(t)$ be a separable Hilbert space and denote by $\X_0 := \X(0)$. Let $\phi_t \colon \X_0 \to \X(t)$ be a family of invertible, linear homeomorphisms, with inverse $\phi_{-t} \colon \X(t) \to \X_0$, such that there exists $C_\X > 0$ such that for every $t \in [0,T]$
  \begin{align*}
    \norm{ \phi_t \eta }_{\X(t)} 
    & \le C_\X \norm{ \eta }_{\X_0} 
    && \mbox{ for all } \eta \in \X_0 \\
    \norm{ \phi_{-t} \eta }_{\X_0} 
    & \le C_\X^{-1} \norm{ \eta }_{\X(t)}
    && \mbox{ for all } \eta \in \X(t),
  \end{align*}
  and such that the map $t \mapsto  \norm{ \phi_t \eta }_{\X(t)}$ is continuous for all $\eta \in \X_0$.
  Under these circumstances, we call the pair $( \X(t), \phi_t )_{t \in [0,T]}$ \emph{compatible}. We call the map $\phi_t$ 
  the \emph{push-forward} operator and $\phi_{-t}$ the \emph{pull-back} operator.
\end{definition}

\begin{remark}
  \label{rem:compatible-subsace}
  If $\S(t)$ is a closed subspace in $\H(t)$ for each $t \in [0,T]$ and $\phi_t$ maps $\S_0 := \S(0) \to \S(t)$, 
  then $( \S(t), \phi_t|_{\S_0} )_{t \in [0,T]}$ form a compatible pair.
\end{remark}

\begin{definition}
  [Evolving Hilbert triple]
  \label{def:evolving-hilbert-triple}
  For each $t \in [0,T]$, let $\V(t)$ and $\H(t)$ be real, separable Hilbert spaces with $\V_0 := \V(0)$ and $\H_0 := \H(0)$ such 
  that inclusion $\V(t) \subset \H(t)$ is continuous and dense.
  We will write $\norm{\cdot}_{\V(t)}$ and $\norm{\cdot}_{\H(t)}$ for the norms on $\V(t)$ and $\H(t)$, $(\cdot,\cdot)_{\H(t)}$ for 
  the inner product on $\H(t)$ and $\langle \cdot, \cdot \rangle_{\V^*(t), \V(t)}$ for the pairing of $\V(t)$ with its dual.
  We assume there exists a family of linear homeomorphisms $\phi_t \colon \H_0 \to \H(t)$ such that $( \H(t), \phi_t )_{t \in [0,T]}$ and $( \V(t), \phi_t|_{\V_0} )_{t \in [0,T]}$ are compatible.
  We will write $\phi_t$ for $\phi_t|_{\V_0}$.
  It follows that $\H(t) \subset \V^*(t)$ continuously and densely.
  Under these assumptions, we say that $( \V(t), \H(t), \V^*(t) )_{t \in [0,T]}$ is an \emph{evolving Hilbert triple}.
\end{definition}

For a compatible pair, we can define an equivalent structure to Bochner spaces in an evolving context. For $( \X(t), \phi_t )_{t \in [0,T]}$ a compatible pair, we define $L^2_\X$ to be
\begin{equation}
  \label{eq:L2X}
  L^2_\X := 
  \left\{ 
    \eta \colon [0,T] \to \bigcup_{t \in [0,T]} \X(t) \times \{ t \},
    t \mapsto ( \bar\eta(t), t ) :
    \phi_{-(\cdot)} \bar\eta(\cdot) \in L^2( 0, T; \X_0 )
  \right\},
\end{equation}
with norm
\begin{equation*}
  \norm{ \eta }_{L^2_\X} := \left( \int_0^T \norm{ \bar\eta }_{\X(t)}^2 \dd t \right)^{\frac{1}{2}}.
\end{equation*}
One can show that the space $L^2_\X$ is a separable Hilbert space \citep[Cor.~2.12]{AlpEllSti15a-pp}, $L^2_\X$ is isomorphic to $L^2(0,T;\X_0)$ and
\begin{equation*}
  {C_\X}^{-1} \norm{ \eta }_{L^2_{\X}}
  \le \norm{ \phi_{-(\cdot)} \eta }_{L^2(0,T;\X_0)}
  \le C_\X \norm{ \eta }_{L^2_\X} \qquad \mbox{ for all } \eta \in L^2_\X.
\end{equation*}

For $k \ge 0$, we also define the space of smoothly evolving in time functions by
\begin{align}
  \label{eq:CkX}
  & C^k_\X := \left\{ \eta \in L^2_\X : t \mapsto \phi_{-t} \eta( t ) \in C^k( [0,T], \X_0 ) \right\} \\
  \label{eq:DX}
  & \D_\X( 0, T) := \left\{ \eta \in L^2_\X : t \mapsto \phi_{-t} \eta(t ) \in \D( ( 0,T); \X_0 ) \right\},
\end{align}
where $\D( (0,T), \X_0 )$ is the space of $\X_0$-valued infinitely differentiable functions compactly supported in the interval $(0,T)$.

For $\eta \in C^1_\X$, we can define a strong material derivative which we denote by $\md \eta \in C^0_\X$ by
\begin{equation}
  \label{eq:strong-md}
  \md \eta := \phi_t \left( \dt (\phi_{-t} \eta) \right).
\end{equation}
This is a temporal derivative which takes into account that fact that $\X(t)$ is changing as well as the function $\eta$.

\begin{remark}
  \label{rem:different-phi}
  If $( \X(t), \phi_t^{(1)} )_{t \in [0,T]}$ and $( \X(t), \phi_t^{(2)} )_{t \in [0,T]}$ are both compatible pairs, with $\phi_t^{(1)} \neq \phi_t^{(2)}$, then the spaces $L^2_{(\X,\phi^{(1)})}$ and $L^2_{(\X,\phi^{(2)})}$ induced using each push forward map are the same vector space with distinct but equivalent norms.
  On the other hand, in general, the spaces $C^k_{(\X,\phi^{(1)})}$ and $C^k_{(\X,\phi^{(2)})}$, when they are non-trivial, are different spaces for $k>0$.
  In such cases where ambiguity will occur we will add the push-forward map to the subscript of $C^k_{(\cdot)}$ so no confusion will occur.
\end{remark}

We can define a weak material derivative for which an integration by parts in time formula holds. This  takes into account the evolution of the 
space $\H(t)$ and is often called a transport formula when applied to the derivative of the time dependent inner product. It generalises the notion of the 
Reynold's transport formula, \cite{CerFriGur05}. In order to provide this definition, we require a further assumption on $\H(t)$.

\begin{assumption}
  \label{as:c-exists}
  We shall assume the following for all $\eta_0, \zeta_0 \in \H_0$:
  \begin{align*}
    & \theta( t, \eta_0 ) := \dt \norm{ \phi_t \eta_0 }_{\H(t)}^2 \mbox{ exists classically} \\
    & \eta_0 \mapsto \theta( t, \eta_0 ) \mbox{ is continuous } \\
    & \abs{ \theta( t, \eta_0 + \zeta_0 ) - \theta( t; \eta_0 - \zeta_0 ) }
    \le c \norm{ \eta_0 }_{\H_0} \norm{ \zeta_0 }_{\H_0},
  \end{align*}
  with the constant $c$ independent of $t \in [0,T]$.
\end{assumption}

We define $\hat{g}( t; \cdot, \cdot ) \colon \H_0 \times \H_0 \to \R$ by
\begin{equation*}
  \hat{g}( t; \eta_0, \zeta_0 ) := \frac{1}{4} 
  \big(
  \theta( t, \eta_0 + \zeta_0 ) - \theta( t, \eta_0 - \zeta_0 )
  \big).
\end{equation*}
Then we have a bilinear form $g( t; \cdot, \cdot ) \colon \H(t) \times \H(t) \to \R$ by
\begin{equation}
  \label{eq:oldg}
  g_0( t; \eta, \zeta ) := \hat{g}( t; \phi_{-t} \eta, \phi_{-t} \zeta ).
\end{equation}
It can be shown that the map $t \mapsto g_0( t; \eta, \zeta )$ is measurable for $\eta, \zeta \in L^2_\H$ and 
we have the following bound independently of $t$:
\begin{equation}
  \label{eq:c-bound}
  \abs{ g_0( t; \eta, \zeta ) } \le c \norm{ \eta }_{\H(t)} \norm{ \zeta }_{\H(t)}.
\end{equation}

We say a function $\eta \in L^2_\V$ has a weak material derivative $\md \eta \in L^2_{\V^*}$ if
\begin{equation*}
  \int_0^T \langle \md \eta, \zeta \rangle_{\V^*(t), \V(t)} \dd t
  = \int_0^T -( \eta, \md \zeta )_{\H(t)} + g_0( t; \eta, \zeta ) \dd t
\end{equation*}
for all $\zeta \in \D_\V( 0, T)$.
\begin{remark}
 The weak and strong material derivatives coincide for $\eta \in C^1_\V$.
\end{remark}
\begin{lemma}
  [Abstract transport formula]
  For all $\eta, \zeta \in L^2_\V$ with weak material derivatives $\md \eta, \md \zeta \in L^2_{\V*}$ we have
  \begin{equation}
    \label{eq:abs-transport}
    \dt ( \eta, \zeta )_{\H(t)}
    = \langle \md \eta, \zeta \rangle_{\V^*(t), \V(t)}
    + \langle \md \zeta, \eta \rangle_{\V^*(t), \V(t)}
    + g_0( t; \eta, \zeta ).
  \end{equation}
\end{lemma}

\begin{proof}
  See \citet[Thm.~2.40]{AlpEllSti15a-pp}.
\end{proof}

The natural solution space is the space $\W( \V, \V^* )$ given by
\begin{equation}
  \label{eq:space-W}
  \W( \V, \V^* ) := \{ \eta \in L^2_\V : \md \eta \in L^2_{\V^*} \},
\end{equation}
which we equip with the inner product
\[
  ( \eta, \zeta )_{\W( \V, \V* )} = \int_0^T ( \eta, \zeta )_{\V(t)} \dd t + \int_0^T ( \md \eta, \md \zeta )_{\V^*(t)} \dd t.
\]

\begin{assumption}
  \label{ass:evolving-space-equivalence}
  Let $W( \V_0, \V^*_0 )$ be the space given by
  \[
    W( \V_0, \V_0^* ) := \{ \eta \in L^2( 0, T; \V_0 ) : \eta_t \in L^2( 0, T; \V_0^* ) \},
  \]
  equipped with the inner product
  \[
    ( \eta, \zeta )_{W( \V_0, \V^*_0 )} = \int_0^T ( \eta, \zeta )_{\V} \dd t + \int_0^T ( \eta', \zeta' )_{\V^*} \dd t.
  \]
  We assume there is an \emph{evolving space equivalence} between $\W( \V, \V^* )$ and $W( \V_0, \V_0^* )$ that is
  \[
    \eta \in \W( \V, \V^* ) \mbox{ if and only if } \phi_{-(\cdot)} \eta (\cdot ) \in W( \V_0, \V_0^* )
  \]
  and there exists $c_1, c_2 > 0$ such that
  \[
    c_1 \norm{ \phi_{-(\cdot)} \eta (\cdot ) }_{W( \V_0, \V_0^*)}
    \le \norm{ \eta }_{\W( \V, \V^* )}
    \le c_2 \norm{ \phi_{-(\cdot)} \eta (\cdot ) }_{W( \V_0, \V_0^*)}.
  \]
\end{assumption}

\subsection{Abstract formulation of the partial differential equation}
\label{sec:abs-pde}

Let $T > 0$. We assume we are in the setting that we have an evolving Hilbert triple $( \V(t), \H(t), \V^*(t) )_{t \in [0,T]}$ and 
\cref{ass:evolving-space-equivalence,as:c-exists} hold so that we have a weak material derivative, which we denote, \ $\md \eta$ 
for appropriate $\eta$, and a transport formula for the $\H(t)$-inner product.

We assume that we have three time dependent bilinear forms $m, g$, and $a$
\begin{align*}
  m( t; \cdot, \cdot ) \colon & \H(t) \times \H(t) \to \R \\
  g( t; \cdot, \cdot ) \colon & \H(t) \times \H(t) \to \R \\
  a( t; \cdot, \cdot ) \colon & \V(t) \times \V(t) \to \R.
\end{align*}
We consider problems of the following form:
\begin{problem}
  \label{prob:cts-scheme}
  Given $\utext_0 \in \H_0$, find $\utext \in  \W( \V, \V^* )$ %
  such that for almost every $t \in [0,T]$,
  \begin{equation}
    \label{eq:cts-scheme}
    m( t; \md \utext, \zeta )
    + g( t; \utext, \zeta )
    + a( t; \utext, \zeta )
    = 0
    \qquad \mbox{ for all } \zeta \in \V(t),
  \end{equation}
  subject to the initial condition $\utext(0 ) = \utext_0$.
\end{problem}

\begin{remark}
  \label{rem:solution-space}
  The abstract formulation of \citet{AlpEllSti15a-pp} allows for a weaker formulation with solutions in 
  $\W( \V, \V^* )$ \cref{eq:space-W} and initial condition in $\H_0$. However in order to obtain optimal order error bounds we require more smoothness for the solution.  We leave the consideration of error analysis for solutions which are less regular in time to a future work. \end{remark}

In order to make sense of this formulation we restrict to the following assumptions on the bilinear forms holding for each $t \in [0,T]$.
\begin{description}
\item[Assumptions on $m$] First, we assume that $m( t; \cdot, \cdot )$ is symmetric:
  \begin{equation}
    \label{eq:m-symmetric} \tag{M1}
    m( t; \eta, \zeta ) = m( t; \zeta, \eta ) \qquad \mbox{ for } \eta, \zeta \in \H(t).
  \end{equation}
  We assume that there exists $c_{\rconst{M2a}}, c_{\rconst{M2b}} > 0$ such that for all $t \in [0,T]$, we have
  \begin{equation}
    \label{eq:m-bounded} \tag{M2}
    c_{\ref*{M2a}} \norm{ \eta }_{\H(t)} \le \left( m( t; \eta, \eta ) \right)^{1/2} \le c_{\ref*{M2b}} \norm{ \eta }_{\H(t)} \qquad \mbox{ for all } \eta \in \H(t).
  \end{equation}
\item[Assumptions on $g$]
  We assume the bilinear form $g( t; \cdot, \cdot )$ satisfies
  \begin{equation}
    \label{eq:c-formula} \tag{G1}
    \dt m( t; \eta, \zeta ) = m( t; \md \eta, \zeta ) + m( t; \eta, \md \zeta )+ g( t; \eta, \zeta ) \qquad \mbox{ for } \eta, \zeta \in C^1_\H,
  \end{equation}
  and such that there exists $c_{\rconst{c:G2}} > 0$ such that for all $t \in [0,T]$
  \begin{equation}
    \label{eq:c-bounded} \tag{G2}
    \abs{ g( t; \eta, \zeta ) } \le c_{\ref*{c:G2}} \norm{ \eta }_{\H(t)} \norm{ \zeta }_{\H(t)} \qquad \mbox{ for } \eta, \zeta \in \H(t).
  \end{equation}

\item[Assumptions on $a$]
  \begin{enumerate}[label=(\alph*)]
  \item We assume that the map
    \begin{equation}
      \label{eq:a-meas} \tag{A1}
      t \mapsto a( t; \eta, \zeta ) \qquad \mbox{ for } \eta, \zeta \in L^2_\V
    \end{equation}
    is measurable and can be decomposed
    \begin{equation}
      \label{eq:a-decomp} \tag{A2}
      a( t; \eta, \zeta ) = a_s( t; \eta, \zeta ) + a_n( \eta, \zeta ) \qquad \mbox{ for } \eta, \zeta \in L^2_\V,
    \end{equation}
    where $a_s$ and $a_n$ are both measurable bilinear forms:
    \[
      a_s( t; \cdot, \cdot ) \colon \V(t) \times \V(t) \to \R, \qquad
      a_n( t; \cdot, \cdot ) \colon \H(t) \times \V(t) \to \R.
    \]
    We assume that $a_s$ is symmetric and allow $a_n \equiv 0$.
  \item We assume there exists $c_{\rconst{c:A3a}}, c_{\rconst{c:A3b}}, c_{\rconst{c:A4}}, c_{\rconst{c:A5}} > 0$ such that
    \begin{align}
      \label{eq:as-coercive} \tag{A3}
      a_s( t; \eta, \eta )
      & \ge c_{\ref*{c:A3a}} \norm{ \eta }_{\V(t)}^2 - c_{\ref*{c:A3b}} \norm{ \eta }_{\H(t)}^2
      && \mbox{ for } \eta \in \V(t) \\
      \label{eq:as-bounded} \tag{A${}_s$4}
      \abs{ a_s( t; \eta, \zeta ) }
      & \le c_{\ref*{c:A4}} \norm{ \eta }_{\V(t)} \norm{ \zeta }_{\V(t)}
      && \mbox{ for } \eta, \zeta \in \V(t) \\
      \label{eq:an-bounded} \tag{A${}_n$4}
      \abs{ a_n( t; \eta, \zeta ) }
      & \le c_{\ref*{c:A5}} \norm{ \eta }_{\H(t)} \norm{ \zeta }_{\V(t)}
      && \mbox{ for } \eta \in \H(t), \zeta \in \V(t).
    \end{align}
    We note that together \cref{eq:as-bounded} and \cref{eq:an-bounded} imply
    \begin{align}
      \label{eq:a-bounded} \tag{A4}
      \abs{ a( t; \eta, \zeta ) }
      & \le (c_{\ref*{c:A4}}+c_{\ref*{c:A5}}) \norm{ \eta }_{\V(t)} \norm{ \zeta }_{\V(t)}
      && \mbox{ for } \eta, \zeta \in \V(t).
    \end{align}

  \item We assume the existence of bilinear forms $b_s( t; \cdot, \cdot ) \V(t) \times \V(t) \to \R$ and $b_n( t; \cdot, \cdot ) \colon \H(t) \times \V(t) \to \R$ such that
    \begin{align}
      \label{eq:bs-exist} \tag{B${}_s$1}
      \dt a_s( t; \eta, \zeta )
      & = a_s( t; \md \eta, \zeta ) + a_s( t; \eta, \md \zeta ) + b_s( t; \eta, \zeta )
      && \mbox{ for all } \eta, \zeta \in C^1_{\V} \\
      \label{eq:bn-exist} \tag{B${}_n$1}
      \dt a_n( t; \eta, \zeta )
      & = a_n( t; \md \eta, \zeta ) + a_n( \eta, \md \zeta ) + b_n( t; \eta, \zeta )
      && \mbox{ for all } \eta \in C^1_\H, \zeta \in C^1_{\V},
    \end{align}
    and that there exists $c_{\rconst{c:B2a}}, c_{\rconst{c:B2b}} > 0$ such that
    \begin{align}
      \label{eq:bs-bound} \tag{B${}_s$2}
      \abs{ b_s( t; \eta, \zeta ) }
      & \le c_{\ref*{c:B2a}} \norm{ \eta }_{\V(t)} \norm{ \zeta }_{\V(t)}
      && \mbox{ for all } \eta, \zeta \in \V(t) \\
      \label{eq:bn-bound} \tag{B${}_n$2}
      \abs{ b_n( t; \eta, \zeta ) }
      & \le c_{\ref*{c:B2b}} \norm{ \eta }_{\H(t)} \norm{ \zeta }_{\V(t)}
      && \mbox{ for all } \eta \in \H(t), \zeta \in \V(t).
    \end{align}
  \item We define $b( t; \cdot, \cdot ) \colon \V(t) \times \V(t) \to \R$ to be
    \begin{equation}
      \label{eq:b-exist} \tag{B1}
      b( t; \eta, \zeta ) := b_s( t ; \eta, \zeta ) + b_n( t; \eta, \zeta ) \qquad \mbox{ for } \eta, \zeta \in \V(t)
    \end{equation}
    and note that \cref{eq:bs-bound} and \cref{eq:bn-bound} together imply that
    \begin{align}
      \label{eq:b-bound} \tag{B2}
      \abs{ b( t; \eta, \zeta ) }
      & \le ( c_{\ref*{c:B2a}} + c_{\ref*{c:B2b}} ) \norm{ \eta }_{\V(t)} \norm{ \zeta }_{\V(t)}
      && \mbox{ for all } \eta, \zeta \in \V(t).
    \end{align}
  \end{enumerate}

\end{description}

\begin{remark}
  We allow for the case that $a$ is non-symmetric. This is similar to the choice of \citet{AlpEllSti15a-pp} but in contrast to much of the finite element literature (e.g.\ \citet{EllVen15} use different bilinear forms for diffusion and advection terms in a parabolic operator).
\end{remark}

\begin{assumption}
  \label{ass:u0-approx}
  We assume there exists a basis $\{ \eta_j^0 \}_{j \in \Nbb}$ of $\V_0$ and a sequence $\{ \utext_{0N} \}_{N \in \Nbb}$ with $\utext_{0N} \in \textrm{span}\{ \eta_1^0, \ldots, \eta_N^0 \}$ for each $N$ such that there exists $c_1, c_2 > 0$ (which do not depend on $N$ or $\utext_0$) with
  \begin{align*}
    \utext_{0N} & \to \utext_0 && \mbox{ in } \V_0 \\
    \norm{ \utext_{0N} }_{\H_0} & \le c_1 \norm{ \utext_0 }_{\H_0} \\
    \norm{ \utext_{0N} }_{\V_0} & \le c_2 \norm{ \utext_0 }_{\V_0}.
  \end{align*}
\end{assumption}

\begin{remark}\label{rem:ass:u0-approx}
In the case of  $\V_0$ being compactly embedded in $\H_0$, \cref{ass:u0-approx} follows by Hilbert-Schmidt theory 
\end{remark}

\begin{theorem}
  \label{thm:abs-exist}
  Let Assumptions \cref{eq:m-symmetric}, \cref{eq:m-bounded}, \cref{eq:c-formula}, \cref{eq:c-bounded}, \cref{eq:a-meas}, \cref{eq:a-decomp}, \cref{eq:as-coercive}, \cref{eq:as-bounded}, \cref{eq:bs-exist}, \cref{eq:bn-exist}, \cref{eq:bs-bound}, \cref{eq:bn-bound} and \cref{ass:u0-approx} hold.
  Then problem \cref{eq:cts-scheme} has a unique solution which satisfies the stability bound
  \begin{align}\label{Hstab}
   \sup_{t\in [0,T]} \norm{ \utext }_{\H(t)}^2 + \int_0^T  \norm{ \utext }_{\V(t)}^2 \dd t
    & \le c(T) \norm{ \utext_0 }_{\H_0}^2
    \end{align}
    and if $\utext_0\in \V_0$ then
    \begin{align}
    \label{eq:cts-stab}
  \sup_{t\in [0,T]} \norm{ \utext }_{\V(t)}^2 +\int_0^T \norm{ \md \utext }_{\H(t)}^2 \dd t
    & \le c(T) \norm{ \utext_0 }_{\V_0}^2.
  \end{align}
\end{theorem}

\begin{proof}
Employing \cref{ass:u0-approx} existence may be proved  by a Galerkin argument similar to \citet[Thm.~3.6 and 3.13]{AlpEllSti15a-pp}. The  a priori estimates can be shown for Galerkin approximations and then used in standard compactness arguments.
We do not show the details here but indicate the arguments in deriving estimates in  the continuous setting. 

  Testing \cref{eq:cts-scheme} with $\utext$:
  \[
    m( t; \md \utext, \utext ) + g( t; \utext, \utext ) + a( t; \utext, \utext ) = 0.
  \]
  \cref{eq:c-formula} gives after integrating in time 

\[
    m( t; \utext, \utext ) + \int_0^t a( t'; \utext, \utext ) \dd t' = m( 0, \utext_0, \utext_0 ) - \int_0^t g( t', \utext, \utext ) \dd t'.
  \]

Applying \cref{eq:m-bounded}, \cref{eq:as-coercive}, \cref{eq:an-bounded} and \cref{eq:c-bounded}:
  \begin{multline*}
    c_{\ref*{M2a}}^2 \norm{ \utext }_{\H(t)}^2 + \int_0^t ( c_{\ref*{c:A3a}} \norm{ \utext }_{\V(t')}^2 - c_{\ref*{c:A3b}} \norm{ \utext }_{\H(t')}^2 - c_{\ref*{c:A5}} \norm{ \utext }_{\V(t')} \norm{ \utext }_{\H(t')} ) \dd t' \\
    \le c_{\ref*{M2b}}^2 \norm{ \utext_0 }_{\H_0}^2 + c_{\ref*{c:G2}} \int_0^t \norm{ \utext }_{\H(t')}^2 \dd t'.
  \end{multline*}
We infer that
  \[
    c_{\ref*{M2a}}^2 \norm{ \utext }_{\H(t)}^2 + \int_0^t \frac{c_{\ref*{c:A3a}}}{2} \norm{ \utext }_{\V(t')}^2 \dd t
    \le c_{\ref*{M2b}}^2 \norm{ \utext_0 }_{\H_0}^2 + (  c_{\ref*{c:A3b}} + \frac{c_{\ref*{c:A5}}^2}{2 c_{\ref*{c:A3a}}} + c_{\ref*{c:G2}} ) \int_0^t \norm{ \utext }_{\H(t')}^2 \dd t'
  \]
  and applying a Gr\"{o}nwall inequality we see the first stability bound.
  
  The bound \cref{eq:cts-stab} uses the decomposition $a(t;\cdot,\cdot)=a_s(t;\cdot,\cdot)+a_n(t;\cdot,\cdot)$ \cref{eq:a-decomp}. We test \cref{eq:cts-scheme} with $\md \utext$   \[
    m( t; \md \utext, \md \utext ) + g( t; \utext, \md \utext ) + a( t; \utext, \md \utext ) = 0.
  \]
  \cref{eq:bs-exist} and \cref{eq:bn-exist} gives
  \[
    a( \utext, \md \utext )
    = \dt ( \frac{1}{2} a_s( \utext, \utext ) + a_n( \utext, \utext ) )
    - \frac{1}{2} b_s( \utext, \utext ) - b_n( \utext, \utext ) - a_n( \md \utext, \utext ),
  \]
  so we have
  \begin{multline*}
    m( t; \md \utext, \md \utext )
    + \dt ( \frac{1}{2} a_s( t; \utext, \utext ) + a_n( t; \utext, \utext ) ) \\
    = -g( t; \utext, \md \utext )
    + \frac{1}{2} b_s( t; \utext, \utext ) + b_n( t; \utext, \utext ) + a_n( t; \md \utext, \utext ).
  \end{multline*}
  Integrating forwards in time gives
  \begin{align*}
    & \int_0^t m( t'; \md \utext, \md \utext ) \dd t'
    + ( \frac{1}{2} a_s( t; \utext, \utext ) + a_n( t; \utext, \utext ) ) \\
    & = ( \frac{1}{2} a_s( 0; \utext_0, \utext_0 ) + a_n( 0; \utext_0, \utext_0 ) ) \\
    & \qquad + \int_0^t ( -g( t'; \utext, \md \utext )
    + \frac{1}{2} b_s(t';  \utext, \utext ) + b_n( t'; \utext, \utext ) + a_n( t'; \md \utext, \utext ) ) \dd t'.
  \end{align*}
  Applying \cref{eq:as-coercive}, \cref{eq:an-bounded}, and  a Young's inequality we observe:
  \begin{align*}
    & ( \frac{1}{2} a_s( t; \utext, \utext ) + a_n( t; \utext, \utext ) ) \\
    & \ge \frac{c_{\ref*{c:A3a}}}{2} \norm{ \utext }_{\V(t)}^2 - c_{\ref*{c:A3b}} \norm{ \utext }_{\H(t)}^2 - c_{\ref*{c:A5}} \norm{ \utext }_{\V(t)} \norm{ \utext }_{\H(t)} \\
    & \ge \frac{c_{\ref*{c:A3a}}}{2} \norm{ \utext }_{\V(t)}^2 - c_{\ref*{c:A3b}} \norm{ \utext }_{\H(t)}^2
      - ( \frac{c_{\ref*{c:A3a}}}{4} \norm{ \utext }_{\V(t)}^2 + \frac{c_{\ref*{c:A5}}^2}{c_{\ref*{c:A3a}}} \norm{ \utext }_{\H(t)}^2 ) \\
    & = \frac{c_{\ref*{c:A3a}}}{4} \norm{ \utext }_{\V(t)}^2 - ( c_{\ref*{c:A3b}} + \frac{c_{\ref*{c:A5}}^2}{c_{\ref*{c:A3a}}} ) \norm{ \utext }_{\H(t)}^2;
  \end{align*}
  applying \cref{eq:as-bounded} and \cref{eq:an-bounded}, we see
  \[
    ( \frac{1}{2} a_s( 0; \utext_0, \utext_0 ) + a_n( 0; \utext_0, \utext_0 ) )
    \le ( \frac{c_{\ref*{c:A3b}}}{2} + c_{\ref*{c:A5}} ) \norm{ \utext_0 }_{\V_0}^2.
  \]

   From the previous two inequalities with \cref{eq:m-bounded}, \cref{eq:c-bounded}, \cref{eq:bs-bound}, \cref{eq:bn-bound} and \cref{eq:an-bounded}, we infer
  \begin{align*}
    & c_{\ref*{M2a}} \int_0^t \norm{ \md \utext }_{\H(t')}^2 \dd t'
    +\frac{c_{\ref*{c:A3a}}}{4} \norm{ \utext }_{\V(t)}^2 - ( c_{\ref*{c:A3b}} + \frac{c_{\ref*{c:A5}}^2}{c_{\ref*{c:A3a}}} ) \norm{ \utext }_{\H(t)}^2 \\
    & \le
      ( \frac{c_{\ref*{c:A3b}}}{2} + c_{\ref*{c:A5}} ) \norm{ \utext_0 }_{\V_0}^2 \\
    & \qquad
      + \int_0^t ( c_{\ref*{c:G2}} \norm{ \utext }_{\H(t')} \norm{ \md \utext }_{\H(t')}
      + \frac{c_{\ref*{c:B2a}}}{2} \norm{ \utext }_{\V(t')}^2
      + c_{\ref*{c:B2b}} \norm{ \utext }_{\V(t')}^2
      + c_{\ref*{c:A5}} \norm{ \md \utext }_{\H(t')} \norm{ \utext }_{\V(t')} ) \dd t'\\
    &  \le
      ( \frac{c_{\ref*{c:A3b}}}{2} + c_{\ref*{c:A5}} ) \norm{ \utext_0 }_{\V_0}^2
      + \int_0^t ( ( \frac{ 4 c_{\ref*{c:G2}}^2 }{ c_{\ref*{M2a}}}
      + \frac{c_{\ref*{c:B2a}}}{2}
      + c_{\ref*{c:B2b}}
      + \frac{ 4 c_{\ref*{c:A5}}^2 }{ c_{\ref*{M2a}}} ) \norm{ \utext }_{\V(t')}^2 + \frac{c_{\ref*{M2a}}}{2} \norm{ \md \utext }_{\H(t')}^2 ) \dd t'.
  \end{align*}
  Rearranging and applying a Gr\"{o}nwall inequality gives the desired result.
\end{proof}

\begin{remark}
Note that $\utext$ also satisfies the variational form of \cref{eq:cts-scheme}
\begin{align}
  \label{eq:abs-var-form}
  \dt m( t; \utext, \zeta )
  + a( t; \utext, \zeta )
  = m( t; \utext, \md \zeta )
  \qquad \mbox{ for } \zeta \in L^2_{\V} \mbox{ with } \md \zeta \in L^2_{\H}.
\end{align}
\end{remark}

\section{Abstract discretisation analysis}\label{AbstractNA}

In this section we present an abstract discretisation and a numerical analysis. %
We assume that for each $h \in (0, h_0)$ there are Hilbert spaces $( \H_h(t), \norm{ \cdot }_{\H_h(t)} )$ and $( \V_h(t), \norm{ \cdot }_{\V_h(t)} )$ for all $t \in [0,T]$, for some value of $h_0$ fixed throughout this section. %
These spaces are used to help with the stability and error analysis. We assume that all constants are independent of $h \in (0,h_0)$ unless indicated.
Throughout this section the integer $k$ will denote a further discretisation parameter denoting the order of approximation.
Our method will be based in a finite dimensional subspace $\S_h(t) \subset \V_h(t)$.
\begin{itemize}
\item We assume that the evolving Hilbert space $\V_h(t)$ is continuously embedded in $\H_h(t)$ uniformly for $h \in (0,h_0)$:
For the two norms $\norm{ \cdot }_{\H_h(t)}$ and $\norm{ \cdot }_{\V_h(t)}$ there exists a constant $c_{\rconst{c:HhVh}} > 0$ such that for all $h \in (0,h_0)$ and all $t \in (0,T)$, we have
\begin{equation}
  \norm{ \eta_h }_{\H_h(t)} \le c_{\ref*{c:HhVh}} \norm{ \eta_h }_{\V_h(t)} \qquad \mbox{ for all } \eta_h \in \V_h(t).
\end{equation}
\item We assume we have a push forward map $\phi^h_t \colon \H_{h,0} := \H_h(0) \to \H_h(t)$ such that $( \H_h(t), \phi^h_t )_{t \in [0,T]}$ and $( \V_h(t), \phi^h_t |_{\V_{h,0}} )_{t \in [0,T]}$ ($\V_{h,0}:= \V_h(0)$) are compatible pairs (\cref{def:compatibility}) uniformly in $h$:
That is there exists $c_{\rconst{c:pf-h-stab-H}}, c_{\rconst{c:pf-h-stab-V}} > 0$ such that, for all $h \in (0,h_0)$,
\begin{equation}
  \label{eq:uniformly-compatible}
  \begin{aligned}
    c_{\ref*{c:pf-h-stab-H}}^{-1} \norm{ \eta_h }_{\H_{h,0}}
    & \le \norm{ \phi^h_t \eta_h }_{\H_h(t)}
    \le c_{\ref*{c:pf-h-stab-H}} \norm{ \eta_h }_{\H_{h,0}}
    && \mbox{ for all } \eta_h \in \H_{h,0} \\
    c_{\ref*{c:pf-h-stab-V}}^{-1} \norm{ \eta_h }_{\V_{h,0}}
    & \le \norm{ \phi^h_t \eta_h }_{\V_h(t)}
    \le c_{\ref*{c:pf-h-stab-V}} \norm{ \eta_h }_{\V_{h,0}}
    && \mbox{ for all } \eta_h \in \V_{h,0}.
  \end{aligned}
\end{equation}
\item This allows us to define the spaces $L^2_{\H_h}, L^2_{\V_h}$ and $C^1_{\H_h}, C^1_{\V_h}$ as in  \cref{eq:L2X} and  \cref{eq:CkX}.
For $\eta_h\in C^1_{\H_h}$, we  denote by  $\mdh \eta_h$  the (strong) material derivative (c.f. \cref{eq:strong-md}) with respect to the push-forward map $\phi^h_t$  defined by
\begin{equation}
  \label{eq:abs-mdh}
  \mdh \eta_h:=\phi^h_t(\frac{d}{dt}\phi^h_{-t}\eta_h).
\end{equation}

\end{itemize}

\subsection{Abstract discrete method}\label{absfem}

Let $T > 0$ and $h \in (0,h_0)$.
Let $\{ \S_h(t) \}_{t \in [0,T]}$ be an evolving, finite-dimensional space which is a subspace of $\V_h(t)$ at each $t \in [0,T]$ and satisfies $\phi^h_t( \S_{h,0} ) = \S_h(t)$ (where $\S_{h,0} = \S_h(0)$).
Since $\S_h(t)$ is a closed subspace of $\V_h(t)$ it is a Hilbert space and forms a compatible pair $( \S_h(t), \phi^h_t|_{\S_{h,0}} )_{ t \in [0,T]}$ (\cref{rem:compatible-subsace}).
In particular, we have well defined spaces $L^2_{\S_h}$ \cref{eq:L2X} and $C^1_{\S_h}$ \cref{eq:CkX} and the material derivative $\mdh \chi_h$ is well defined for $\chi_h \in C^1_{\S_h}$ (\cref{eq:strong-md,eq:abs-mdh}).

\subsubsection{Basis functions}
Let the dimension of $\S_h(t)$ be $N$ for all $t \in [0,T]$.
We write $\{ \chi_i(0 ) \}_{i=1}^N$ for a basis of $\S_{h,0}$ and push-forward to construct a time dependent  basis $\{ \chi_i(t ) \}_{i=1}^N$ of $\S_h(t)$ by
\begin{equation}
  \label{eq:abs-basis-def}
  \chi_i(t ) = \phi_t^h ( \chi_i(0 ) ).
\end{equation}
The following important transport properties of the basis functions hold.
\begin{lemma}
  \label{lem:basis-trans}
  The material derivative of a basis function is zero,
  \[
    \mdh \chi_j = 0 \qquad \mbox{ for } 1 \le j \le N.
  \]
  Furthermore, any function $\chi_h \in C^1_{\S_h}$, which can be written as $\chi_h = \sum_{j=1}^N \gamma_j(t) \chi_j(t)$, satisfies
  \begin{equation}
    \mdh \chi_h(x,t) = \sum_{i=1}^N \dot\gamma_j(t) \chi_j(\cdot, t)
    \quad \mbox{ for } x \in \Gamma_h(t).
  \end{equation}
\end{lemma}

\begin{proof}
 By definition, it follows that 
\begin{equation}
  \label{eq:abs-basis-transport}
  \mdh \chi_j =\phi_t(\frac{d}{dt}\phi_{-t}\chi_j(t))=\phi_t(\frac{d}{dt}\chi_j(0))= 0
\end{equation}
so that for a decomposition
\[
  \chi_h(t):=\sum_{j=1}^N\gamma_j(t)\chi_j(t)  \qquad \text{ for all } \chi_h\in \S_h(t),
\]
we compute that
\[
  \mdh \chi_h=\sum_{j=1}^N\dot \gamma_j(t)\chi_j(t) \qquad \mbox{ for all } \chi_h\in C^1_{\S_h}.
  \qedhere
\]
\end{proof}

\subsubsection{Discrete bilinear forms}
Let $m_h$ and $a_h$ be two time dependent bilinear forms:
\begin{align*}
  m_h( t; \cdot, \cdot ) & \colon \H_h(t) \times \H_h(t) \to \R \\
  a_h( t; \cdot, \cdot ) & \colon \V_h(t) \times \V_h(t) \to \R.
\end{align*}
We make assumptions similar to those in \cref{sec:abs-pde}.

We assume that $m_h( t; \cdot, \cdot )$ is symmetric:
\begin{equation}
  \label{eq:mh-symmetric} \tag{M$_h$1}
  m_h( t; \eta_h, \zeta_h ) = m_h( t; \zeta_h, \eta_h )
  \qquad \mbox{ for } \eta_h, \zeta_h \in \H_h(t).
\end{equation}
We assume that there exists $c_{\rconst{c:Mh2a}}, c_{\rconst{c:Mh2b}} > 0$ such that for all $t \in [0,T]$, we have
\begin{equation}
  \label{eq:mh-bounded} \tag{M$_h$2}
  c_{\ref*{c:Mh2a}} \norm{ \eta_h }_{\H_h(t)} \le \big( m_h( t; \eta_h, \eta_h ) \big)^{\frac{1}{2}}
  \le c_{\ref*{c:Mh2b}} \norm{ \eta_h }_{\H_h(t)}
  \qquad \mbox{ for } \eta_h \in \H_h(t).
\end{equation}

We assume we have a transport formula for the bilinear form $m_h$: there exists a bilinear form $g_h( t; \cdot, \cdot ) \colon \H_h(t) \times \H_h(t) \to \R$ such that
\begin{equation}
  \label{eq:ch-formula} \tag{G$_h$1}
  \dt m_h( t; \eta_h, \zeta_h )
  = m_h( t; \mdh \eta_h, \zeta_h ) + m_h( t; \eta_h, \mdh \zeta_h ) + g_h( t; \eta_h, \zeta_h )
  \qquad \mbox{ for } \eta_h, \zeta_h \in C^1_{\H_h}.
\end{equation}
We assume that there exists $c_{\rconst{c:Gh2}} > 0$ such that for all $t \in [0,T]$
\begin{equation}
  \label{eq:ch-bounded} \tag{G$_h$2}
  \abs{ g_h( t; \eta_h, \zeta_h ) } \le c_{\ref*{c:Gh2}} \norm{ \eta_h }_{\H_h(t)} \norm{ \zeta_h }_{\H_h(t)}
  \qquad \mbox{ for } \eta_h, \zeta_h \in \H_h(t).
\end{equation}

We assume that the map
\begin{equation}
  \label{eq:ah-meas} \tag{A$_h$1}
  t \mapsto a_h( t; \eta_h, \zeta_h )
  \qquad \mbox{ for } \eta_h, \zeta_h \in L^2_{\V_h},
\end{equation}
is measurable, and there exists constants $c_{\rconst{c:Ah2a}}, c_{\rconst{c:Ah2b}}, c_{\rconst{c:Ah3}} > 0$ such that for all $t \in [0,T]$, we have
\begin{align}
  \label{eq:ah-coercive} \tag{A$_h$2}
  a_h( t; \eta_h, \eta_h )
  & \ge c_{\ref*{c:Ah2a}} \norm{ \eta_h }_{\V_h(t)}^2 - c_{\ref*{c:Ah2b}} \norm{ \eta_h }_{\H_h(t)}^2
  && \mbox{ for } \eta_h \in \V_h(t) \\
  \label{eq:ah-bounded} \tag{A$_h$3}
  \abs{ a_h( t; \eta_h, \zeta_h ) }
  & \le c_{\ref*{c:Ah3}} \norm{ \eta_h }_{\V_h(t)} \norm{ \zeta_h }_{\V_h(t)}
  && \mbox{ for } \eta_h, \zeta_h \in \V_h(t).
\end{align}
We assume a transport formula for the $a_h$ bilinear form, that  there exists a bilinear form $b_h( t; \cdot, \cdot ) \colon \V_h(t) \times \V_h(t) \to \R$ such that
\begin{equation}
  \label{eq:bh-formula} \tag{B$_h$1}
  \dt a_{h}( t; \eta_h, \zeta_h )
  = a_{h}( t; \mdh \eta_h, \zeta_h ) + a_h( t; \eta_h, \mdh \zeta_h ) + b_h( t; \eta_h, \zeta_h )
  \qquad \mbox{ for } \eta_h, \zeta_h \in C^1_{\V_h}
\end{equation}
and that there exists $c_{\rconst{c:Bh2}} > 0$ such that for all $t \in [0,T]$ we have
\begin{equation}
  \label{eq:bh-bounded} \tag{B$_h$2}
  \abs{ b_h( t; \eta_h, \zeta_h ) }
  \le c_{\ref*{c:Bh2}} \norm{ \eta_h }_{\V_h(t)} \norm{ \zeta_h }_{\V_h(t)}
  \qquad \mbox{ for } \eta_h, \zeta_h \in \V_h(t).
\end{equation}

\subsubsection{Abstract discrete variational problem and well posedness}
\label{sec:abs-stab-est}
Motivated by the variational form \cref{eq:abs-var-form}, we consider semi-discrete problems of the following form:
\begin{problem}
  \label{pb:fem}
  Given $U_{h,0} \in \S_{h,0}$, find $U_h \in C^1_{\S_h}$ such that $U_h(0)=U_{h,0}$ and
  \begin{equation}
    \label{eq:fem}
    \dt m_h( t; U_h, \chi_h )
    + a_h( t; U_h, \chi_h )
    = m_h( t; U_h, \mdh \chi_h )
    \qquad \mbox{ for all } \chi_h \in C^1_{\S_h}.
  \end{equation}
\end{problem}

We note that due to Assumption \cref{eq:ch-formula}, the discrete scheme \cref{eq:fem} can be re-written as
\begin{equation}
  \label{eq:abs-weak-form}
  m_h( t; \mdh U_h, \chi_h ) + g_h( t; U_h, \chi_h ) + a_h( t; U_h, \chi_h ) = 0
  \qquad
  \mbox{ for } \chi_h \in C^1_{\S_h}.
\end{equation}

The solution $U_h$ may be written as  a decomposition into the time dependent basis functions $\{ \chi_i \}_{i=1}^N$ of $\{ \S_h(t) \}_{t \in [0,T]}$,
\begin{equation}
  \label{eq:decomp}
  U_h( x, t ) = \sum_{i=1}^N \alpha_i(t) \chi_i( x, t ) \qquad \mbox{ for } x \in \Gamma_h(t),
\end{equation}
where $\alpha(t) = ( \alpha_1(t), \ldots, \alpha_N(t) ) \in \R^N$. Using this notation \cref{pb:fem} is equivalent to finding a solution $\alpha \in C^1( [0,T]; \R^N )$ of the (finite dimensional) system of ordinary differential equations:
\begin{equation}
  \label{eq:matrix-form}
  \dt \big( \M(t) \alpha(t) \big) + \S(t) \alpha(t) = 0,
\end{equation}
where
\begin{equation*}
  \M(t)_{ij} = m_h( t; \chi_j, \chi_i )
  \qquad
  \S(t)_{ij} = a_h( t; \chi_j, \chi_i )
  \qquad
  \mbox{ for } i,j=1,\ldots,N.
\end{equation*}
Here, we have used the fact that $\mdh \chi_i = 0$ for $1 \le i \le N$.

We first wish to show that there exists a solution to our discrete scheme satisfying a stability bound similar to \cref{Hstab} for the continuous case.
Due to our abstract formulation the calculations follow in a similar way to \cref{thm:abs-exist}.

\begin{theorem}
  [Existence and stability of finite element method]
  \label{thm:abs-stab}
  Let Assumptions \cref{eq:mh-symmetric}, \cref{eq:mh-bounded}, \cref{eq:ch-formula}, \cref{eq:ch-bounded}, \cref{eq:ah-meas}, \cref{eq:ah-coercive}, \cref{eq:ah-bounded}, \cref{eq:bh-formula} and \cref{eq:bh-bounded} hold with constants independent of $h \in (0,h_0)$. Then \cref{eq:fem} has a unique solution $U_h \in C^1_{\S_h}$ with $\mdh U_h \in C^0_{\S_h}$ and there exists a constant $C(T)>0$ independent of $h \in (0,h_0)$ such that
  \begin{equation}
    \label{eq:fem-stab}
    \sup_{t \in (0,T)} \norm{ U_h }_{\H_h(t)}^2
    + \int_0^T \norm{ U_h }_{\V_h(t)}^2 \dd t \le C(T) \norm{ U_{h,0} }_{\H_h(0)}^2.
  \end{equation}
\end{theorem}

\begin{proof}
  We consider the problem in the matrix form \cref{eq:matrix-form}. Since $\M(\cdot) \in C^1( 0,T; \R^{N \times N})$ \cref{eq:ch-formula} and is invertible \cref{eq:mh-bounded}, this is equivalent to
  \begin{equation}
    \label{eq:4}
    \dot\alpha(t) + \M^{-1}(t) \big( \M'(t) + \S(t) \big) \alpha(t) = 0.
  \end{equation}
  This is a linear system of ordinary equations with $C^0$ coefficients (easily verified).
  Standard theory implies there exists a unique solution $\alpha \in C^1( 0, T; \R^N )$, which can be translated as $U_h \in C^1_{\S_h}$.

  To show the energy bound, we start by testing \cref{eq:fem} with $\chi_h = U_h$:
  \begin{equation*}
    \dt m_h( t; U_h, U_h ) + a_h( t; U_h, U_h ) - m_h( t; U_h, \mdh U_h ) = 0.
  \end{equation*}
  The transport equality \cref{eq:ch-formula} implies that
  \begin{equation*}
    \frac{1}{2} \dt m_h( t; U_h, U_h ) + a_h( t; U_h, U_h )
    = -\frac{1}{2} g_h( t; U_h, U_h ).
  \end{equation*}
 The desired stability bound follows using the same calculations as used in deriving \cref{Hstab}.
 
\end{proof}

\subsection{Abstract lifted finite element spaces}
\label{sec:error-assumptions}
The discrete space $\S_h(t)$ is not assumed to be contained in the continuous space $\V(t)$. This is an example of a ``variational crime'' \citep{StrangBook}. However it is convenient  to prove error bounds in the spaces $\H(t)$ and  $\V(t)$.
To do this we use lifted discrete spaces using an embedding map $\lambda_h(\cdot,t) \colon \S_h(t) \to \S_h^\ell(t) \subset \V(t)$.
The error analysis will relate the solution of \cref{eq:cts-scheme} $\utext$ with a so-called lift $\lambda_h(U_h,t)$
of the discrete solution.

We require two further time dependent Hilbert spaces $\{ \Z_0(t) \}_{t \in [0,T]}$ and $\{ \Z(t) \}_{t \in [0,T]}$ which satisfy $\Z(t) \subset \Z_0(t) \subset \V(t)$ for each $t \in [0,T]$ with the inclusions uniformly continuous.
Furthermore, we assume that $( \Z_0(t), \phi_t|_{\Z_0(0)} )_{t \in [0,T]}$ and $( \Z(t), \phi_t|_{\Z(0)} )_{t \in [0,T]}$ are both compatible pairs \cref{def:compatibility} so that we may define the spaces $L^2_{\Z_0}, L^2_\Z$ \cref{eq:L2X} and $C^1_{\Z_0}, C^1_\Z$ \cref{eq:CkX}.
These spaces are abstractions of spaces of smoother functions.

The link between all the function spaces used in the section is shown in \cref{fig:space-relations}.

\begin{figure}[tb]
  \centering
  \[
    \begin{tikzcd}
      \H(t) \arrow[r, hookleftarrow] \arrow[ddd, leftrightarrow]
      & \V(t) \arrow[r, hookleftarrow]
      & \Z_0(t) \arrow[r,hookleftarrow]
      & \Z(t) \\
      & \S_h^\ell(t) \arrow[u,symbol=\subset] \arrow[d, leftrightarrow]
      & & \\
      & \S_h(t)  \\
      \H_h(t) \arrow[r, hookleftarrow]
      & \V_h(t) \arrow[u, symbol=\supset] \arrow[uuu, leftrightarrow, bend right=30]
      &  &
    \end{tikzcd}
  \]
  \caption{The relationships between different spaces used in this section.
    $\subset$ denotes subspace inclusion,
    $\hookrightarrow$ denotes continuous embedding,
    $\leftrightarrow$ denotes that the lift is a bijection between these spaces.
   }
  \label{fig:space-relations}
\end{figure}

\subsubsection{Lifting operator}
  Let $t \in [0,T]$.
  We assume there is a continuous, bijective, linear function $\lambda_h( \cdot, t ) \colon \H_h(t) \to \H(t)$ with inverse denoted by $\varsigma_h( \cdot, t ) \colon \H(t) \to \H_h(t)$ such that $\lambda_h( \cdot, t )|_{\V_h(t)}$ is also a bijection onto $\V(t)$.
 We will denote $(\cdot)^{\ell} := \lambda_h( \cdot, t )$ and $( \cdot )^{-\ell} := \varsigma_h( \cdot, t )$, i.e. $\eta_h^\ell(\cdot, t) := \lambda_h( \eta_h, t )( \cdot )$ for $\eta_h \in \H_h(t)$ and $\eta^{-\ell}( \cdot, t ) := \varsigma_h( \eta, t ) ( \cdot )$ for $\eta \in \H(t)$.

  We assume that the lifting map is bounded and  that there exists $c_{\rconst{c:L12a}}, c_{\rconst{c:L12b}} > 0$ such that for all $t \in [0,T]$
\begin{align}
  \label{eq:lift-stab-H} \tag{L1}
  c_{\ref*{c:L12a}} \norm{ \eta_h^\ell }_{\H(t)} & \le \norm{ \eta_h }_{\H_h(t)} \le c_{\ref*{c:L12b}} \norm{ \eta^\ell_h }_{\H(t)}
  && \mbox{ for all } \eta_h \in \H_h(t) \\
  \label{eq:lift-stab-V} \tag{L2}
  c_{\ref*{c:L12a}} \norm{ \eta^\ell_h }_{\V(t)} & \le \norm{ \eta_h }_{\V_h(t)} \le c_{\ref*{c:L12b}} \norm{ \eta^\ell_h }_{\V(t)}
  && \mbox{ for all } \eta_h \in \V_h(t).
\end{align}

\begin{remark}
  As a convention, variables denoted by $\ell$ correspond to lifted objects and $-\ell$ to inverse lifted objects. The construction of each depends on the parameter $h$.
\end{remark}

\subsubsection{Lifted push-forward and pull back maps}
  The lift of the discrete space push-forward map induces a new push-forward map on $\H(t)$.
  Let $\phi^\ell_t \colon \H(0) \to \H(t)$ by given by
  \begin{equation}
    \label{eq:phiellt}
    \phi^\ell_t( \eta ) := \lambda_h( \phi^h_t ( \varsigma_h( \eta, 0 ) ), t ) \qquad \mbox{ for } \eta \in \H(0),
  \end{equation}
  with inverse $\phi^\ell_{-t} \colon \H(t) \to \H(0)$ given by
  \begin{equation}
    \label{eq:phiellt-}
    \phi^\ell_{-t}( \eta ) := \lambda_h( \phi^h_{-t}( \varsigma( \eta, t ) ), 0 ) \quad \mbox{ for } \eta \in \H(t),
  \end{equation}
  so that
  \begin{align*}
    {\phi}^\ell_t( \eta_h^\ell ) & = ( \phi^h_t ( \eta_h ) )^\ell && \mbox{ for } \eta_h \in \H_h(0) \\
    {\phi}_{-t}^\ell( \eta_h^\ell ) & = ( \phi^h_{-t}( \eta_h ) )^{\ell} && \mbox{ for } \eta_h \in \H_h(t).
  \end{align*}
  Note that $(\cdot)^\ell$ is a time dependent operator.
  Our assumptions imply that both pairs $\{ \H(t), \phi^\ell_t \}_{t \in [0,T]}$ and $\{ \V(t), \phi^\ell_t|_{\V_0} \}_{t \in [0,T]}$ are compatible uniformly in $h$ (\cref{def:compatibility}). For example, applying that $(\H_h(t), \phi^h_t)_{t \in [0,T]}$ is compatible and \cref{eq:lift-stab-H} gives
  \begin{align*}
    \norm{ \phi^\ell_t \eta }_{\H(t)}
    & \le c \norm{ \phi^h_t ( \eta^{-\ell} ) }_{\H_h(t)}
    \le c \norm{ \eta^{-\ell} }_{\H_{h,0}}
    \le c \norm{ \eta }_{\H_0}
    && \mbox{ for } \eta \in \H_0.
  \end{align*}

  We use the notation $C^1_{(\H,\phi^\ell)}$ and $C^1_{(\V,\phi^\ell)}$ for the spaces of smoothly evolving functions, each with respect to the pull back $\phi^\ell_{-t}$ (c.f. \cref{eq:CkX}).
We recall that the definition of $L^2_\H$ and $L^2_\V$ only does not depend on the choice of push-forward map up to norm equivalence (\cref{rem:different-phi}).
We assume the following inclusions hold:
\begin{align}
  \label{eq:abs-C1-inclusion} \tag{L3}
  C^1_{(\H,\phi)} \cap C^0_\V \subset C^1_{(\H,\phi^\ell)}, \qquad \mbox{ and } \qquad
  C^1_{(\V,\phi)} \cap C^0_{\Z_0} \subset C^1_{(\V,\phi^\ell)}.
\end{align}

\subsubsection{Lifted material derivative}
We denote by $\mdell \eta$ the material derivative for the push-forward map $\phi^\ell_t$ \cref{eq:strong-md}:
\begin{equation}
  \label{eq:lift-md}
  \mdell \eta := \phi_t^\ell\frac{d}{dt}(\phi_{-t}^\ell\eta) \qquad \mbox{ for all } \eta \in C^1_{(\H, \phi^\ell)}.
\end{equation}

This is a different material derivative to the material derivative defined with respect to the push-forward map $\phi^h_t$ \cref{eq:abs-mdh}.
However, as observed in \cite{DziEll13},  our construction implies that  first taking material derivatives and then lifting is the same as first lifting and then taking material derivatives:
\begin{lemma}[Commutation of material derivative and lifting] The following hold:
\label{lem:commutation-md-lift}
\begin{itemize}
\item For all $\eta_h \in C^1_{\H_h}$
\begin{equation}
  \label{eq:abs-md-commute}
  \mdell (\eta_h^\ell) = (\mdh \eta_h)^\ell.
\end{equation}
\item
 \label{lem:abs-lift-C1-inclusions}
  $\eta_h \in C^1_{\H_h}$ if, and only if, $\eta_h^\ell \in C^1_{(\H,\phi^\ell)}$, and $\eta_h \in C^1_{\V_h}$ if, and only if, $\eta_h^\ell \in C^1_{(\V,\phi^\ell)}$.
\end{itemize}
\end{lemma}
\begin{proof}
Indeed, applying the definitions \cref{eq:phiellt} and \cref{eq:phiellt-}, we see
\begin{align*}
  \mdell ( \eta_h^\ell )
  = \phi^\ell_t \dt \bigl( \phi^\ell_{-t} ( \eta_h^\ell ) \bigr)
  = \Biggl( \phi^h_t \Bigl( \bigl( \dt ( \phi^h_{-t} \eta_h )^\ell \bigr)^{-\ell} \Bigr) \Biggr)^\ell
  = \Biggl( \phi^h_t \Bigl( \dt ( \phi^h_{-t} \eta_h ) \Bigr) \Biggr)^\ell
  = ( \mdh \eta_h )^\ell,
\end{align*}
since the lift at time $t=0$ and time derivative commute and $(\cdot)^\ell$ and $(\cdot)^{-\ell}$ are inverses.
By similar reasoning, we have that
  \[
    \mdh (\eta^{-\ell}) = \mdell \eta \qquad \mbox{ for all } \eta \in C^1_{(\H,\phi^\ell)}.
  \]
In particular, applying \cref{eq:lift-stab-H,eq:lift-stab-V}, we have the second statement.
\end{proof}

\subsubsection{Abstract lifted transport formulae}
We assume that we have a transport formula for functions in $C^1_{(\H,\phi^\ell)}$ and $C^1_{(\V,\phi^\ell)}$ for the $m$ and $a$ bilinear forms.
We assume that there exists bilinear forms $g_\ell( t; \cdot, \cdot ) \colon \H(t) \times \H(t) \to \R$ and $b_\ell( t; \cdot,\cdot ) \colon \V(t) \times \V(t) \to \R$ such that
\begin{align}
  \label{eq:ctilde-formula} \tag{G${}_\ell$1}
  \dt m( t; \eta, \zeta )
  & = m( t; \mdell \eta, \zeta ) + m( t; \eta, \mdell \zeta ) + g_\ell( t; \eta, \zeta )
  && \mbox{ for } \eta, \zeta \in C^1_{(\H,\phi^\ell)} \\
  \label{eq:btilde-formula} \tag{B${}_\ell$1}
  \dt a( t; \eta, \zeta )
  & = a( t; \mdell \eta, \zeta ) + a( t; \eta, \mdell \zeta) + b_\ell( t; \eta, \zeta )
  && \mbox{ for } \eta, \zeta \in C^1_{(\V,\phi^\ell)}.
\end{align}
We assume that for these bilinear forms there exists $c_1, c_2 > 0$ such that for all $t \in [0,T]$ and all $h \in (0,h_0)$,
\begin{align}
  \label{eq:ctilde-bounded} \tag{G${}_\ell$2}
  \abs{ g_\ell( t; \eta, \zeta ) }
  & \le c_1 \norm{ \eta }_{\H(t)} \norm{ \zeta }_{\H(t)}
  && \mbox{ for } \eta, \zeta \in \H(t) \\
  \label{eq:btilde-bounded} \tag{B${}_\ell$2}
  \abs{ b_\ell( t; \eta, \zeta ) }
  & \le c_2 \norm{ \eta }_{\V(t)} \norm{ \zeta }_{\V(t)}
  && \mbox{ for } \eta, \zeta \in \V(t).
\end{align}

\subsubsection{Lifted finite element space}
  Let $t \in [0,T]$.
  The lifting process allows us to introduce a new space $\S_h^\ell(t)$ by
  \begin{align}
    \label{eq:abs-Vhell}
    \S_h^\ell(t) & := \{ \chi_h^\ell : \chi_h \in \S_h(t) \}.
  \end{align}
  This is a subspace of $\V(t)$ for each $t \in [0,T]$.

  The lifted discrete space $\S_h^\ell(t)$ is a closed subspace of $\V(t)$ and $\phi^\ell_t( \S_{h,0}^\ell ) = \S_h^\ell(t)$ so we may infer that $( \S_h^\ell(t), \phi^\ell_t|_{\S_{h,0}^\ell} )_{t \in [0,T]}$ is a compatible pair (\cref{rem:compatible-subsace}).
  Hence, the spaces $L^2_{\S_h^\ell}$ \cref{eq:L2X} and $C^1_{\S_h^\ell}$ \cref{eq:CkX} are well defined and we will write $\mdell \chi_h^\ell$ for the material derivative of $\chi_h^\ell \in C^1_{\S_h^\ell}$.
  From our assumptions, in general it does not hold that $\phi_t( \S_{h,0}^\ell ) = \S_h^\ell(t)$.

\subsubsection{Approximation property of $\S_h^\ell(t)$}
  For each $t \in [0,T]$, we assume that there exists a well defined interpolation operator $I_h \colon \Z_0(t) \to \S_h^\ell(t)$ such that there exists a constant $c > 0$ such that for all $t \in [0,T]$
  \begin{align}
    \label{eq:interp-Z0} \tag{I1}
    \norm{ \eta - I_h \eta }_{\H(t)}
    + h \norm{ \eta - I_h \eta }_{\V(t)}
    & \le c h^2 \norm{ \eta }_{\Z_0(t)}
    && \mbox{ for } \eta \in \Z_0(t) \\
    \label{eq:interp-Z} \tag{I2}
    \norm{ \eta - I_h \eta }_{\H(t)}
    + h \norm{ \eta - I_h \eta }_{\V(t)}
    & \le c h^{k+1} \norm{ \eta }_{\Z(t)}
    && \mbox{ for } \eta \in \Z(t).
  \end{align}

  \subsubsection{Assumptions on the geometric approximation}
  Finally, we assume we have the following relations between continuous and discrete bilinear forms.
We assume that there exists constants $c > 0$ such that for all $t \in [0,T]$ the following holds for all $\eta_h, \zeta_h \in \V_h(t)$ with lifts $\eta_h^\ell, \zeta_h^\ell \in \V(t)$ we have
  \begin{align}
    \label{eq:m-error} \tag{P1}
    \abs{ m( t; \eta_h^\ell, \zeta_h^\ell ) - m_h( t; \eta_h, \zeta_h ) }
    & \le c h^{k+1} \norm{ \eta_h^\ell }_{\V(t)} \norm{ \zeta_h^\ell }_{\V(t)} \\
    \label{eq:c-error} \tag{P2}
    \abs{ g_\ell( t; \eta_h^\ell, \zeta_h^\ell ) - g_h( t; \eta_h, \zeta_h ) }
    & \le c h^{k+1} \norm{ \eta_h^\ell }_{\V(t)} \norm{ \zeta_h^\ell }_{\V(t)} \\
    \label{eq:ct-error} \tag{P3}
    \abs{ g_\ell( t; \eta_h^\ell, \zeta_h^\ell ) - g( t; \eta_h, \zeta_h ) }
    & \le c h^{k} \norm{ \eta_h^\ell }_{\V(t)} \norm{ \zeta_h^\ell }_{\V(t)} \\
    \label{eq:a-error} \tag{P4}
    \abs{ a( t; \eta_h^\ell, \zeta_h^\ell ) - a_{h}( t; \eta_h, \zeta_h ) }
    & \le c h^{k} \norm{ \eta_h^\ell }_{\V(t)} \norm{ \zeta_h^\ell }_{\V(t)} \\
    \label{eq:b-error} \tag{P5}
    \abs{ b_\ell( t; \eta_h^\ell, \zeta_h^\ell ) - b_{h}( t; \eta_h, \zeta_h ) }
    & \le c h^{k} \norm{ \eta_h^\ell }_{\V(t)} \norm{ \zeta_h^\ell }_{\V(t)} \\
    \label{eq:bt-error} \tag{P6}
    \abs{ b_\ell( t; \eta_h^\ell, \zeta_h^\ell ) - b( t; \eta_h, \zeta_h ) }
    & \le c h^{k} \norm{ \eta_h^\ell }_{\V(t)} \norm{ \zeta_h^\ell }_{\V(t)}.
  \end{align}
  For $\eta, \zeta \in \Z_0(t)$ with inverse lifts $\eta^{-\ell}, \zeta^{-\ell}$, we have
  \begin{align}
    \label{eq:a-error2} \tag{\ref*{eq:a-error}'}
    \abs{ a( t; \eta, \zeta ) - a_{h}( t; \eta^{-\ell}, \zeta^{-\ell} ) }
    & \le c h^{k+1} \norm{ \eta }_{\Z_0(t)} \norm{ \zeta }_{\Z_0(t)} \\
    \label{eq:b-error2} \tag{\ref*{eq:b-error}'}
    \abs{ b_\ell( t; \eta, \zeta) - b_{h}( t; \eta^{-\ell}, \zeta^{-\ell} ) }
    & \le c h^{k+1} \norm{ \eta }_{\Z_0(t)} \norm{ \zeta }_{\Z_0(t)}.
  \end{align}
  For $\eta \in C^1_{\Z_0}$ and $\zeta \in \Z_0(t)$, with inverse lifts $\eta^{-\ell}$ and $\zeta^{-\ell}$, we have
  \begin{align}
    \label{eq:amd-error2} \tag{P7}
    \abs{ a( t; \mdell \eta, \zeta ) - a_{h}( t; \mdh \eta^{-\ell}, \zeta^{-\ell} ) }
    & \le c h^{k+1} ( \norm{ \eta }_{\Z_0(t)} + \norm{ \md \eta }_{\Z_0(t)} ) \norm{ \zeta}_{\Z_0(t)}.
  \end{align}
  Finally, we assume
  \begin{align}
    \label{eq:md-error} \tag{P8}
    \norm{ \mdell \zeta - \md \zeta }_{\H(t)}
    & \le c h^{k+1} \norm{ \zeta }_{\V(t)}
    && \mbox{ for } \zeta \in C^1_{\V} \\
    \label{eq:mdV-error} \tag{P9}
    \norm{ \mdell \zeta - \md \zeta}_{\V(t)}
    & \le c h^{k} \norm{ \zeta }_{\Z_0(t)}
    && \mbox{ for } \zeta \in C^1_{\Z_0}.
  \end{align}

\subsection{Ritz projection}
\label{sec:ritz}

It is convenient to introduce a Ritz projection which is  a standard approach in the  finite element analysis of evolution equations \citep{Tho06} also applied to problems on evolving surfaces \citep[e.g.][]{DziEll13}. The  Ritz projection is defined  with respect to   modified positive definite bilinear forms $a^\kappa$ and $a_h^\kappa$.

\subsubsection{A new bilinear form}
We know from Assumptions \cref{eq:m-bounded}, \cref{eq:mh-bounded}, \cref{eq:as-coercive}, \cref{eq:an-bounded} and \cref{eq:ah-coercive}, there exists $\kappa > 0$ such that there exists $c > 0$ such that for all $t \in [0,T]$ and all $h \in (0,h_0)$
\begin{align}
  \label{eq:akappa-coercive}
  a( t; \eta, \eta ) + \kappa m( t; \eta, \eta )
  & \ge c \norm{ \eta }_{\V(t)}^2
  && \mbox{ for } \eta \in \V(t) \\
  \label{eq:ahkappa-coercive}
  a_h( t; \eta_h, \eta_h ) + \kappa m_h( t; \eta_h, \eta_h )
  & \ge c \norm{ \eta_h }_{\V_h(t)}^2
  && \mbox{ for } \eta_h \in \V_h(t).
\end{align}
We now take $\kappa$ fixed in the sequel. Thus, we infer that the bilinear forms:
\begin{align*}
  a^\kappa( t; \eta, \zeta )
  & := a( t; \eta, \zeta ) + \kappa m( t; \eta, \zeta )
  && \mbox{ for } \eta, \zeta \in \V(t) \\
  a_h^\kappa( t; \eta_h, \zeta_h )
  & := a_h( t; \eta_h, \zeta_h ) + \kappa m_h( t; \eta_h, \zeta_h )
  && \mbox{ for } \eta_h, \zeta_h \in \V_h(t),
\end{align*}
are uniformly coercive for all $h \in (0,h_0)$.

\pagebreak
\subsubsection{The projection}
\begin{definition}
  The Ritz projection is an operator $\Pi_h \colon \V(t) \to \S_h(t)$.
  For $z \in \V(t)$, $\Pi_h z$ is given as the unique solution of
  \begin{equation}
    \label{eq:ritz}
    a_h^\kappa( t; \Pi_h z, \chi_h )
    = a^\kappa( t; z, \chi_h^\ell )
    \qquad
    \mbox{ for all } \chi_h \in \S_h(t) \mbox{ with lift } \chi_h^\ell \in \S_h^\ell(t).
  \end{equation}
  We denote by $\pi_h z = ( \Pi_h z )^\ell \in \S_h^\ell(t)$.
\end{definition}

We will further assume that there exists a constant $c_{\rconst{c:B3}} > 0$ such that for all $t \in [0,T]$, for $\eta = z - \pi_h z \in \V(t)$ and all $\zeta \in \Z_0(t)$ we have
\begin{equation}
  \label{eq:b-bound2} \tag{B3}
  \abs{ b( t; \eta, \zeta) }
  \le c_{\ref*{c:B3}} \left( \norm{ \eta }_{\H(t)} + h \norm{ \eta }_{\V(t)} + h^{k+1} \norm{ z }_{\Z(t)} \right)
  \norm{ \zeta }_{\Z_0(t)}.
\end{equation}

\begin{remark}\label{B3prime}
In an application it may be possible to prove  that a simpler version of   \cref{eq:b-bound2} is sufficient.  For example in the  case of
 of surfaces without boundary  for all $t \in [0,T]$, for $\eta \in \V(t)$ and all $\zeta \in \Z_0(t)$ we have
 \begin{equation}
  \label{eq:b-bound2prime} \tag{B3'}
  \abs{ b( t; \eta, \zeta) }
  \le c_{\ref*{c:B3}} \norm{ \eta }_{\H(t)} 
  \norm{ \zeta }_{\Z_0(t)}
\end{equation}
holds.
\end{remark}

\subsubsection{A dual problem}
We introduce the dual problem: Given $\xi \in \H(t)$, find $\dual{\xi} \in \V(t)$ such that
\begin{equation}
  \label{eq:dual-problem}
  a^\kappa( t; \eta, \dual{\xi} ) = m( t; \xi, \eta )
  \qquad \mbox{ for } \eta \in \V(t).
\end{equation}
Assumptions on $\kappa$ along with the previous assumptions from \cref{sec:abs-pde} imply that \cref{eq:dual-problem} has a unique solution and we assume the regularity condition that there exists $c > 0$ such that
\begin{equation}
  \label{eq:dual-reg} \tag{R2}
  \norm{ \dual{\xi} }_{\Z_0(t)} \le c \norm{ \xi }_{\H(t)},
\end{equation}
where the constant is independent of $\xi$ and time $t$.

\subsubsection{Ritz error analysis}

\begin{lemma}
  \label{lem:ritz}
  For each $z \in \V(t)$, there exists a unique solution $\Pi_h z$ of \cref{eq:ritz}. There exists a constant $c > 0$ such that for all $h \in (0,h_0)$ and all $t \in [0,T]$ we have
  \begin{equation}
    \label{eq:ritz-stab}
    \norm{ \Pi_h z }_{\V_h(t)} \le c \norm{ z }_{\V(t)}
    \qquad \mbox{ for } z \in \V(t).
  \end{equation}
  Furthermore, there exists a constant $c > 0$ such that for all $t \in [0,T]$ and $h \in (0,h_0)$,
  \begin{equation}
    \label{eq:ritz-bound}
    \norm{ z - \pi_h z }_{\H(t)} + h \norm{ z - \pi_h z }_{\V(t)}
    \le c h^{k+1} \norm{ z }_{\Z(t)}
    \qquad \mbox{ for all } z \in \Z(t).
  \end{equation}
\end{lemma}

\begin{proof}
  Since $a_h^\kappa$ is uniformly coercive \cref{eq:ahkappa-coercive} and bounded (\cref{eq:ah-bounded,eq:mh-bounded}) and
  $a^\kappa$ is bounded (\cref{eq:a-bounded,eq:m-bounded}), standard Lax-Milgram theory gives that there exists a unique solution that satisfies the stability bound \cref{eq:ritz-stab}.

  To show the error bound, we consider the functional $F_h \colon \V(t) \to \R$ given by
  \begin{equation*}
    F_h( \eta ) = a^\kappa( t; z - \pi_h z, \eta ).
  \end{equation*}

  First, note that for $\eta = \chi_h^\ell \in \S_h^\ell(t)$, we can use the definition of $\Pi_h z$ \cref{eq:ritz} to see that
  \begin{equation*}
    F_h( \chi_h^\ell )
    = a^\kappa( t; z - \pi_h z, \chi_h^\ell )
    = a_h^\kappa( t; \Pi_h z, \chi_h ) - a^\kappa( t; \pi_h z, \chi_h^\ell ).
  \end{equation*}
  Then the perturbation estimates \cref{eq:m-error,eq:a-error} and the stability bound \cref{eq:ritz-stab} imply that
  \begin{equation}
    \label{eq:ritz-Fh-1}
    \abs{ F_h( \chi_h^\ell ) }
    \le c h^k \norm{ \Pi_h z }_{\V_h(t)} \norm{ \chi_h^\ell }_{\V(t)}
    \le c h^k \norm{ z }_{\V(t)} \norm{ \chi_h^\ell }_{\V(t)}.
  \end{equation}

  Next, we consider $F_h( \eta )$ for $\eta \in \Z_0$. Then, again using \cref{eq:ritz} we have
  \begin{align*}
    F_h( \eta )
    & = a^\kappa( t; z - \pi_h z, \eta ) \\
    & = a^\kappa( t; z - \pi_h z, \eta - I_h \eta )
      + a^\kappa( t; z - \pi_h z, I_h \eta ) \\
    & = a^\kappa( t; z - \pi_h z, \eta - I_h \eta )
      + \big( a_h^\kappa( t; \Pi_h z, ( I_h \eta )^{-\ell} )
      - a^\kappa( t; \pi_h z, I_h \eta ) \big)
      =: I_1 + I_2.
  \end{align*}
  Using the boundedness of $a^\kappa$ (\cref{eq:a-bounded,eq:m-bounded}) and the interpolation bounds \cref{eq:interp-Z0}, we have
  \begin{align*}
    \abs{ I_1 }
    \le c \norm{ z - \pi_h z }_{\V(t)} \norm{ \eta - I_h \eta }_{\V(t)}
    \le c h \norm{ z - \pi_h z }_{\V(t)} \norm{ \eta }_{\Z_0(t)}.
  \end{align*}
  We split $I_2$ so that together with the perturbation estimates \cref{eq:m-error}, \cref{eq:a-error} and \cref{eq:a-error2} and the interpolation result  we have
  \begin{align*}
    \abs{ I_2 }
    & \le \abs{ a_h^\kappa( t; \Pi_h z, ( I_h \eta - \eta )^{-\ell} )
      - a^\kappa( t; \pi_h z, I_h \eta - \eta ) } \\
    & \qquad + \abs{ a_h^\kappa( t; \Pi_h z - z^{-\ell}, ( \eta )^{-\ell} )
      - a^\kappa( t; \pi_h z - z, \eta ) } \\
    & \qquad + \abs{ a_h^\kappa( t; z^{-\ell}, ( \eta )^{-\ell} )
      - a^\kappa( t; z, \eta ) } \\
    & \le c h^{k+1} \norm{ \Pi_h z }_{\V_h(t)} \norm{ \eta }_{\Z_0(t)}
      + c h^k \norm{ \pi_h z - z }_{\V(t)} \norm{ \eta }_{\V(t)} \\
    & \qquad + c h^{k+1} \norm{ z }_{\Z_0(t)} \norm{ \eta }_{\Z_0(t)}.
  \end{align*}
  Then combining the above estimates with the stability bound \cref{eq:ritz-stab}, we see that
  \begin{equation}
    \label{eq:ritz-Fh-2}
    \begin{aligned}
      \abs{ F_h( \eta ) }
      & \le c \big( h \norm{ z - \pi_h z }_{\V(t)} + h^{k+1} \norm{ z }_{\Z(t)} \big)
      \norm{ \eta }_{\Z_0(t)}.
    \end{aligned}
  \end{equation}

  To show the $\V(t)$-norm error bound, we have
  \begin{equation*}
    a^\kappa( t; z - \pi_h z, z - \pi_h z )
    = a^\kappa( t; z - \pi_h z, z - I_h z ) + F_h( I_h z - \pi_h z ).
  \end{equation*}
  Applying the boundedness (\cref{eq:a-bounded,eq:m-bounded}) and coercivity \cref{eq:akappa-coercive} of $a^\kappa$, the interpolation bound \cref{eq:interp-Z} and the first bound on $F_h$ \cref{eq:ritz-Fh-1} gives
  \begin{align*}
    \norm{ z - \pi_h z }_{\V(t)}^2
     & \le c h^{k} \norm{ z - \pi_h z }_{\V(t)} \norm{ z }_{\Z(t)}
       + c h^k \norm{ z }_{\V(t)} \norm{ I_h z - \pi_h z }_{\V(t)} \\
    & \le c h^{k} \norm{ z - \pi_h z }_{\V(t)} \norm{ z }_{\Z(t)}
       + c h^k \norm{ z }_{\V(t)} \left( \norm{ I_h z - z }_{\V(t)} + \norm{ z - \pi_h z }_{\V(t)} \right).
  \end{align*}
  Using the interpolation bound \cref{eq:interp-Z} and rearranging using a Young's inequality gives
  \begin{equation}
    \label{eq:ritz-V-error}
    \norm{ z - \pi_h z }_{\V(t)} \le c h^k \norm{ z }_{\Z(t)}.
  \end{equation}

  For the $\H(t)$-norm bound, we consider the dual problem \cref{eq:dual-problem} with $\xi = z - \pi_h z \in \H(t)$. Then there exists a unique $\dual{\xi} \in \V(t)$ such that
  \begin{equation*}
    a^\kappa( t; \eta, \dual{\xi} ) = m( t; \xi, \eta )
    \qquad \mbox{ for all } \eta \in \V(t).
  \end{equation*}
  Furthermore, $\dual{\xi} \in \Z_0(t)$ and satisfies \cref{eq:dual-reg}
  \begin{equation}
    \label{eq:ritz-dual-reg}
    \norm{ \dual{\xi} }_{\Z_0(t)} \le c \norm{ z - \pi_h z }_{\H(t)}.
  \end{equation}
  Then we have from \cref{eq:m-bounded} that
  \begin{equation*}
    \norm{ z - \pi_h z }_{\H(t)}^2
    \le c_2 m( t; z - \pi_h z, z - \pi_h z )
    = c_2 a^\kappa( t; z - \pi_h z, \dual{\xi} )
    = c_2 F_h( \dual{\xi} ).
  \end{equation*}
  Then the second bound on $F_h$ \cref{eq:ritz-Fh-2} together with the $\V(t)$-norm bound \cref{eq:ritz-V-error} and the dual regularity estimate \cref{eq:ritz-dual-reg} imply that
  \begin{align*}
    \norm{ z - \pi_h z }_{\H(t)}^2
    & \le \big( c h \norm{ z - \pi_h z }_{\V(t)} + c h^{k+1} \norm{ z }_{\Z(t)} \big)
    \norm{ \dual{\xi} }_{\Z_0(t)} \\
    & \le c h^{k+1} \norm{ z }_{\Z(t)} \norm{ z - \pi_h z }_{\H(t)}.
  \end{align*}
  Rearranging this inequality provides the $\H(t)$-norm bound.
\end{proof}

\subsubsection{Time derivative of Ritz projection}
Since in general the material derivative and Ritz projection do not commute, we must provide a further estimate for this material derivative of the error $z - \pi_h z$. First, we derive an equation for $\mdh \Pi_h z$.

\begin{lemma}
  Let $z \in C^1_{\V}$ then, for each $t \in [0,T]$, we have the equation:
  \begin{multline}
  \label{eq:ritz-md}
  a_{h}^\kappa( t; \mdh \Pi_h z, \chi_h )
  = a^\kappa( t; \mdell z, \chi_h^\ell )
  + b_\ell^\kappa( t; z, \chi_h^\ell ) - b_h^\kappa( t; \Pi_h z, \chi_h ) \\
  \mbox{ for all } \chi_h \in \S_h(t).
  \end{multline}
\end{lemma}

\begin{proof}
  First consider $\chi_h \in C^1_{\S_h}$ and take the time derivative of \cref{eq:ritz}.
  We apply the discrete transport formulae \cref{eq:ch-formula,eq:bh-formula} on the left hand side and the lifted transport formulae \cref{eq:btilde-formula,eq:ctilde-formula} on the right hand side and rearrange to see:
  \begin{multline*}
    a_h^\kappa( t; \mdh \Pi_h z, \chi_h ) \\
    = a^\kappa( t; \mdell z, \chi_h^\ell )
    + a^\kappa( t; z, \mdell \chi_h^\ell ) - a_h^\kappa( t; \Pi_h z, \mdh \chi_h )
    + b_\ell^\kappa( t; z, \chi_h^\ell ) - b_h^\kappa( t; \Pi_h z, \chi_h ).
  \end{multline*}
  Using the fact that for $\chi_h \in C^1_{\S_h}$, we have $\mdh \chi_h \in \S_h(t)$ and $( \mdh \chi_h )^\ell = \mdell \chi_h^\ell$, we can apply \cref{eq:ritz} once more to see that
  \begin{equation*}
    a_h^\kappa( t; \mdh \Pi_h z, \chi_h )
    = a^\kappa( t; \mdell z, \chi_h^\ell )
    + b_\ell^\kappa( t; z, \chi_h^\ell ) - b_h^\kappa( t; \Pi_h z, \chi_h ).
  \end{equation*}

  We can expand this result to arbitrary $\chi_h \in \S_h(t)$ by considering the function $\chi_h^*$ given by $s \mapsto \phi^h_s ( \phi^h_{-t} \chi_h )$ which satisfies $\chi_h^* \in C^1_{\S_h}$ and $\chi_h^*|_{s=t} = \chi_h$.
\end{proof}

\begin{lemma}
  \label{lem:ritz-md}
  For $z \in C^1_{\Z_0}$, we have that $\Pi_h z \in C^1_{\S_h}$ and
  there exists a constants $c > 0$ such that for all $t \in [0,T]$, $h \in (0,h_0)$
  \begin{equation}
    \label{eq:ritz-md-stab}
    \norm{ \mdh \Pi_h z }_{\V_h(t)}
    \le c \big( \norm{ z }_{\Z_0(t)} + \norm{ \md z }_{\V(t)} \big)
    \qquad \mbox{ for all } z \in \Z_0(t).
  \end{equation}
  Furthermore, if $z \in C^1_{\Z}$, then
  \begin{equation}
    \label{eq:ritz-md-error}
    \norm{ \mdell (z - \pi_h z) }_{\H(t)} + h \norm{ \mdell (z - \pi_h z) }_{\V(t)}
    \le c h^{k+1} \big( \norm{ z }_{\Z(t)} + \norm{ \md z }_{\Z(t)} \big).
  \end{equation}
\end{lemma}

\begin{proof}
  For the stability bound, we see that $\mdh \Pi_h z$ satisfies the discrete elliptic problem \cref{eq:ritz-md}.
  This tells us that $\mdh \Pi_h z \in \S_h(t)$ and, combined with the boundedness of $a^\kappa$ (\cref{eq:a-bounded,eq:m-bounded}), $b_\ell^\kappa$ (\cref{eq:btilde-bounded,eq:ctilde-bounded}), and $b_h^\kappa$ (\cref{eq:bh-bounded,eq:ch-bounded}) we see that
  \[
    \norm{ \mdh \Pi_h z }_{\V_h(t)} \le c \bigl( \norm{ \md_h z }_{\V(t)} + \norm{ z }_{\V(t)} + \norm{ \Pi_h z }_{\V_h(t)} \bigr).
  \]
  Then applying the perturbation estimate \cref{eq:mdV-error} and the stability of estimate \cref{eq:ritz-stab}, we see \cref{eq:ritz-md-stab}.

  To show the error bound, we proceed in a similar fashion to \cref{lem:ritz}, we introduce the functional $T_h \colon \V(t) \to \R$ given by
  \begin{equation*}
    T_h( \eta ) = a^\kappa( t; \mdell ( z - \pi_h z ), \eta ).
  \end{equation*}

  First, for $\eta = \chi_h^\ell \in \S_h^\ell(t)$, we can use \cref{eq:ritz-md} to see
  \begin{align*}
    T_h( \chi_h^\ell )
    & = a^\kappa( t; \mdell( z - \pi_h z ), \chi_h^\ell ) \\
    & = \big( a_h^\kappa( t; \mdh \Pi_h z, \chi_h ) - a^\kappa( t; \mdell \pi_h z, \chi_h^\ell ) \big)
      + \big( b_h^\kappa( t; \Pi_h z, \chi_h ) - b_\ell^\kappa( t; \pi_h z, \chi_h^\ell ) \big) \\
    & \qquad + b_\ell^\kappa( t; \pi_h z - z, \chi_h^\ell ).
  \end{align*}
  Then, using the perturbation estimates on $a^\kappa$ (\cref{eq:m-error,eq:a-error}), with the fact that the two discrete material derivatives and lifting commute \cref{eq:abs-md-commute}, and $b^\kappa$ (\cref{eq:c-error,eq:b-error}), the boundedness of $b_\ell^\kappa$ (\cref{eq:ctilde-bounded,eq:btilde-bounded}), the error bound \cref{eq:ritz-bound} and the stability estimates \cref{eq:ritz-stab} and \cref{eq:ritz-md-stab} gives
  \begin{equation}
    \label{eq:Th-bound1}
    \begin{aligned}
      \abs{ T_h( \chi_h^\ell ) }
      & \le c h^k \big( \norm{ \mdh \Pi_h z }_{\V_h(t)} + \norm{ \Pi_h z }_{\V_h(t)} + \norm{ z }_{\Z(t)} \big) \norm{ \chi_h^\ell }_{\V(t)} \\
      & \le c h^k \big( \norm{ z }_{\Z(t)} + \norm{ \md z }_{\Z(t)} \big) \norm{ \chi_h^\ell }_{\V(t)}.
    \end{aligned}
  \end{equation}

  Secondly, for $\eta \in \Z_0(t)$, we have using \cref{eq:ritz-md}
  \begin{align*}
    T_h( \eta )
    & = a^\kappa( t; \mdell( z - \pi_h z ), \eta - I_h \eta )
      + a^\kappa( t; \mdell( z - \pi_h z ), I_h \eta ) \\
    & = a^\kappa( t; \mdell( z - \pi_h z ), \eta - I_h \eta ) \\
    & \qquad + \big( a_{h}^\kappa( t; \mdh \Pi_h z, (I_h \eta)^{-\ell} )
      - a^\kappa( t; \mdell \pi_h z, I_h \eta ) \big) \\
    & \qquad + \big( b_h^\kappa( t; \Pi_h z, ( I_h \eta )^{-\ell} )
      - b_\ell^\kappa( t; \pi_h z, I_h \eta ) \big) \\
    & \qquad + b_\ell^\kappa( t; \pi_h z - z, I_h \eta ) \\
    & =: I_1 + I_2 + I_3 + I_4.
  \end{align*}
  We split the four terms $I_1, \ldots, I_4$ using the smooth functions $z$ and $\eta$ so that
  \begin{align*}
    I_1 & = a^\kappa( t; \mdell( z - \pi_h z ), \eta - I_h \eta ) \\
    I_2 & = \big( a_{h}^\kappa( t; \mdh \Pi_h z, (I_h \eta - \eta)^{-\ell} )
      - a^\kappa( t; \mdell \pi_h z, I_h \eta - \eta) \big) \\
    & \qquad + \big( a_{h}^\kappa( t; \mdh ( \Pi_h z - z^{-\ell} ), (\eta)^{-\ell} )
      - a^\kappa( t; \mdell ( \pi_h z - z ), \eta) \big) \\
    & \qquad + \big( a_{h}^\kappa( t; \mdh ( z^{-\ell} ) - ( \md z )^{-\ell}, (\eta)^{-\ell} )
      - a^\kappa( t; \mdell z - \md z, \eta) \big) \\
    & \qquad + \big( a_{h}^\kappa( t; ( \md z )^{-\ell}, (\eta)^{-\ell} )
      - a^\kappa( t; \md z, \eta) \big) \\
    I_3 & = \big( b_h^\kappa( t; \Pi_h z, ( I_h \eta - \eta )^{-\ell} )
      - b_\ell^\kappa( t; \pi_h z, I_h \eta - \eta ) \big) \\
    & \qquad + \big( b_h^\kappa( t; \Pi_h z - z^{-\ell}, ( \eta )^{-\ell} )
      - b_\ell^\kappa( t; \pi_h z - z, \eta ) \big) \\
    & \qquad + \big( b_h^\kappa( t; z^{-\ell}, ( \eta )^{-\ell} )
      - b_\ell^\kappa( t; z, \eta ) \big) \\
    I_4 & = b_\ell^\kappa( t; \pi_h z - z, I_h \eta - \eta )
      + \big( b_\ell^\kappa( t; \pi_h z - z, \eta )
      - b^\kappa( t; \pi_h z - z, \eta ) \big) \\
    & \qquad + b^\kappa( t; \pi_h z - z, \eta ).
  \end{align*}
  Using the boundedness of $a^\kappa$ (\cref{eq:a-bounded,eq:m-bounded}) and the interpolation estimate \cref{eq:interp-Z0}, we have
  \begin{equation*}
    \abs{ I_1 }
    \le c h \norm{ \mdell( z - \pi_h z ) }_{\V(t)} \norm{ \eta }_{\Z_0(t)}.
  \end{equation*}
  Using the perturbation errors for $a^\kappa$ \cref{eq:m-error,eq:a-error,eq:a-error2} (with the fact that discrete material derivatives and lifting commute \cref{eq:abs-md-commute}, the inclusions shown in \cref{eq:abs-C1-inclusion,lem:abs-lift-C1-inclusions}),
  as well as the estimate with material derivatives \cref{eq:amd-error2},
  together with the interpolation bound \cref{eq:interp-Z0}
  and the error in material derivatives \cref{eq:mdV-error}, we have
  \begin{align*}
    \abs{ I_2 }
    & \le c h^{k+1} \norm{ \mdh \Pi_h z }_{\V_h(t)} \norm{ \eta }_{\Z_0(t)}
    + c h^k \norm{ \mdell ( z - \pi_h z ) }_{\V(t)} \norm{ \eta }_{\V(t)} \\
    & \qquad + c h^{2k} \norm{ z }_{\Z_0(t)} \norm{ \eta }_{\V(t)}
    + c h^{k+1} \norm{ \md z }_{\Z_0(t)} \norm{ \eta }_{\Z_0(t)} \\
    & \le c h^k \norm{ \mdell( z - \pi_h z ) }_{\V(t)} \norm{ \eta }_{\Z_0(t)} \\
    & \qquad + c h^{k+1} \big( \norm{ z }_{\Z(t)} + \norm{ \md z }_{\Z(t)} + \norm{ \mdh \Pi_h z }_{\V_h(t)} \big) \norm{ \eta }_{\Z_0(t)}.
  \end{align*}
  Using the simple and improved perturbation estimate for $b^\kappa$ \cref{eq:c-error}, \cref{eq:b-error} and \cref{eq:b-error2}, the interpolation result \cref{eq:interp-Z0}, and the Ritz $\V(t)$-norm error bound \cref{eq:ritz-bound}, we have
  \begin{align*}
    \abs{ I_3 }
    & \le c h^{k+1} \norm{ \Pi_h z }_{\V_h(t)} \norm{ \eta }_{\Z_0(t)}
      + c h^{k+1} \norm{ z }_{\Z(t)} \norm{ \eta }_{\V(t)}
      + c h^{k+1} \norm{ z }_{\Z_0(t)} \norm{ \eta }_{\Z_0(t)} \\
    & \le c h^{k+1} \big( \norm{ z }_{\Z(t)} + \norm{ \Pi_h z }_{\V_h(t)} \big)
      \norm{ \eta }_{\Z_0(t)}.
  \end{align*}
  Using the boundedness of $b_\ell^\kappa$ (\cref{eq:btilde-bounded,eq:ctilde-bounded}), the perturbation estimates \cref{eq:bt-error,eq:ct-error} on $b^\kappa$, the boundedness of $b^\kappa$ (\cref{eq:b-bound2,eq:c-bounded}) and the Ritz $\V(t)$ and $\H(t)$-norm error bounds \cref{eq:ritz-bound} we have
  \begin{align*}
    \abs{ I_4 }
    & \le c h^{k+1} \norm{ z }_{\Z(t)} \norm{ \eta }_{\Z_0(t)}
      + c h^{2k} \norm{ z }_{\Z(t)} \norm{ \eta }_{\V(t)} \\
    & \qquad
      + c \bigl( \norm{ z - \pi_h z }_{\H(t)}
      + h \norm{ z - \pi_h z }_{\V(t)}
      + h^{k+1} \norm{ z }_{\Z(t)} \bigr) \norm{ \eta }_{\Z_0(t)} \\
    & \le c h^{k+1} \norm{ z }_{\Z(t)} \norm{ \eta }_{\Z_0(t)}.
  \end{align*}
  Combining the previous four bounds with the stability estimates for $\Pi_h z$ \cref{eq:ritz-stab} and $\mdh \Pi_h z$ \cref{eq:ritz-md-stab} gives
  \begin{equation}
    \label{eq:Th-bound2}
    \begin{aligned}
      \abs{ T_h( \eta ) }
      & \le c h^{k+1} \big( \norm{ z }_{\Z(t)} + \norm{ \md z }_{\Z(t)} \big) \norm{ \eta }_{\Z_0(t)} + c h \norm{ \mdell ( z - \pi_h z ) }_{\V(t)} \norm{ \eta }_{\Z_0(t)}.
    \end{aligned}
  \end{equation}

  To show the $\V(t)$-norm error bound, we start with
  \begin{align*}
    & a^\kappa( t; \mdell( z - \pi_h z ), \mdell( z - \pi_h z ) ) \\
    & \quad = a^\kappa( t; \mdell( z - \pi_h z), \mdell z - \md z )
    + a^\kappa( t; \mdell( z - \pi_h z), \md z - I_h \md z ) \\
    & \qquad + a^\kappa( t; \mdell( z - \pi_h z), I_h \md z - \mdell \pi_h z ).
  \end{align*}
  Noting that the final term on the right hand side is $T_h( I_h \md z - \mdell \pi_h z )$,
  the bounds on $a^\kappa$ (\cref{eq:a-bounded,eq:m-bounded}), the perturbation estimate \cref{eq:mdV-error}, the interpolation estimate \cref{eq:interp-Z} and the first bound on $T_h$ \cref{eq:Th-bound1} gives
  \begin{align*}
    & a^\kappa( t; \mdell ( z - \pi_h z ), \mdell ( z - \pi_h z ) ) \\
    & \le \norm{ \mdell ( z - \pi_h z ) }_{\V(t)} c h^k \norm{ z }_{\Z_0(t)}
      + \norm{ \mdell ( z - \pi_h z ) }_{\V(t)} c h^k \norm{ \md z }_{\Z(t)} \\
    & \qquad + c h^k ( \norm{ z }_{\Z(t)} + \norm{ \md z }_{\Z(t)} ) \norm{ I_h \md z - \md_h \pi_h z }_{\V(t)} \\
    & \le \norm{ \mdell ( z - \pi_h z ) }_{\V(t)} c h^k \norm{ z }_{\Z_0(t)}
      + \norm{ \mdell ( z - \pi_h z ) }_{\V(t)} c h^k \norm{ \md z }_{\Z(t)} \\
      & \qquad + c h^k \bigl( \norm{ z }_{\Z(t)} + \norm{ \md z }_{\Z(t)} \bigr) \bigl( \norm{ I_h \md z - \md z }_{\V(t)} \norm{ \md z - \mdell \pi_h z }_{\V(t)} \bigr).
  \end{align*}
  Again, using the interpolation bound \cref{eq:interp-Z} and the coercivity of $a^\kappa$ \cref{eq:akappa-coercive} and rearranging using a Young's inequality gives
  \begin{equation}
    \label{eq:ritz-md-V-bound}
    \norm{ \mdell( z - \pi_h z ) }_{\V(t)}
    \le c h^k \big( \norm{ z }_{\Z(t)} + \norm{ \md z }_{\Z(t)} \big).
  \end{equation}

  To show the $\H(t)$-norm bound, we consider the dual problem \cref{eq:dual-problem} with $\xi = e := \mdell ( z - \pi_h z ) \in \H(t)$.
  Then, there exists $\dual{\xi} \in \V(t)$ such that
  \begin{equation*}
    a^\kappa( t; \eta, \dual{\xi} ) = m( t; e, \eta ) \quad \mbox{ for all } \eta \in \V(t).
  \end{equation*}
  Furthermore, $\dual{\xi} \in \Z_0(t)$ and satisfies the bound
  \begin{equation}
    \label{eq:ritz-md-dual-reg}
    \norm{ \dual{\xi} }_{\Z_0(t)} \le c \norm{ e }_{\H(t)}.
  \end{equation}
  Then we have
  \begin{equation*}
    \norm{ \mdell ( z - \pi_h z ) }_{\H(t)}^2
    \le c_2^2 m( t; e, e )
    = c_2^2 a^\kappa( t; \mdell ( z - \pi_h z ), \dual{\xi} )
    = c_2^2 T_h( \dual{\xi} ).
  \end{equation*}
  The second bound on $T_h$ \cref{eq:Th-bound2}, the $\V(t)$-norm error bound \cref{eq:ritz-md-V-bound} and the dual regularity result \cref{eq:ritz-md-dual-reg} give
  \begin{align*}
    \norm{ e }_{\H(t)}^2
    & \le c h^{k+1} \big( \norm{ z }_{\Z(t)} + \norm{ \md z }_{\Z(t)} \big) \norm{ \dual{\xi} }_{\Z_0(t)}
    + c h \norm{ e }_{\V(t)} \norm{ \dual{\xi} }_{\Z_0(t)} \\
    & \le c h^{k+1} \big( \norm{ z }_{\Z(t)} + \norm{ \md z }_{\Z(t)} \big) \norm{ e }_{\H(t)}.
  \end{align*}
  Rearranging this inequality gives the desired $\H(t)$-norm bound.
\end{proof}

\subsection{Abstract error bound}
\label{sec:abs-error-analysis}

To show the error bound we make the following assumption on the smoothness of the continuous problem.
We assume that $\utext \in C^1_{\V}$ and that there exists a constant $C > 0$ such that $\utext$ satisfies that regularity estimate
  \begin{equation}
    \label{eq:u-regularity} \tag{R1}
    \sup_{t \in [0,T]} \norm{ \utext }_{\Z(t)}^2 + \int_0^T \norm{ \md \utext }_{\Z(t)}^2 \dd t
    \le C.
  \end{equation}

\begin{theorem}
  \label{thm:error}
  Let all the assumptions listed in \cref{sec:error-assumptions} hold as well as \cref{eq:b-bound2} and \cref{eq:u-regularity}. Denote by $\utext$ the solution of \cref{eq:cts-scheme} and by $U_h \in C^1_{\S_h}$ the solution of \cref{eq:fem} with lift $u_h \in C^1_{\S_h^\ell}$. Then, there exists constant $c > 0$ such that for $h \in (0,h_0)$ we have the error estimate
  \begin{equation}
    \label{eq:error}
    \begin{aligned}
      & \sup_{t \in [0,T]} \norm{ \utext - u_h }_{\H(t)}^2
      + h^2 \int_0^T \norm{ \utext - u_h }_{\V(t)}^2 \dd t \\
      & \qquad \le \norm{ \utext_0 - u_{h,0} }_{\H(0)}^2
      + c h^{2k+2} \left(
        \sup_{t \in [0,T]} \norm{ \utext }_{\Z(t)}^2 + \int_0^T \norm{ \md \utext }_{\Z(t)}^2 \dd t
      \right).
    \end{aligned}
  \end{equation}
\end{theorem}

To show the error bound, we start by rescaling both solutions.
Let $\check{u} = e^{-\kappa t} \utext$ and $\check{U}_h = e^{-\kappa t} U_h$, which satisfy
\begin{align}
  \label{eq:checku}
  m( t; \md \check{u}, \eta )
  + g( t; \check{u}, \eta )
  + a^\kappa( t; \check{u}, \eta )
  & = 0
  && \mbox{ for all } \eta \in L^2_\V \\
  \label{eq:checkUh}
  \dt m_h( t; \check{U}_h, \chi_h )
  + a_h^\kappa( t; \check{U}_h \chi_h )
  - m_h( t; \check{U}_h, \mdh \chi_h )
  & = 0
  && \mbox{ for all } \chi_h \in C^1_{\S_h}.
\end{align}
Our assumptions imply
\begin{multline}
  \label{eq:checku-reg}
  \sup_{t \in [0,T]} \norm{ \check{u} }_{\Z(t)}^2 + \int_0^T \norm{ \md \check{u} }_{\Z(t)}^2 \\
  \le e^{-2 \kappa t} \left( \sup_{t \in [0,T]} \norm{ \utext }_{\Z(t)}^2 + \int_0^T \norm{ \md \utext }_{\Z(t)} + \kappa^2 \norm{ \utext }_{\Z(t)}^2 \dd t \right) < C.
\end{multline}

We define $\check{u}_h := \check{U}_h^\ell$ to be the lift of $\check{U}_h$. We will decompose the error as:
\begin{equation}
  \label{eq:checku-split}
  \check{u}_h - \check{u}
  = ( \check{u}_h - \pi_h \check{u} )
  + ( \pi_h \check{u} - \check{u} )
  =: \theta + \rho.
\end{equation}
We already have bounds on $\rho$ from \cref{lem:ritz,lem:ritz-md}, thanks to assumption \cref{eq:u-regularity}, so it remains to show a bound for $\theta$.
We will denote by $\vartheta = \check{U}_h - \Pi_h \check{u}$, and by our assumptions, we know $\vartheta \in C^1_{\S_h}$.

\begin{lemma}
  Let $\chi_h \in C^1_{\S_h}$ and denote by $\chi_h^\ell \in C^1_{\S_h^\ell}$.
  Then $\vartheta = \check{U}_h - \Pi_h \check{u}$ satisfies
  \begin{equation}
    \label{eq:theta-eqn}
    \dt m_h( t; \vartheta, \chi_h ) + a_h^\kappa( t; \vartheta, \chi_h )
    - m_h( t; \vartheta, \mdh \chi_h )
    = -E_1( \chi_h ) - E_2( \chi_h ),
  \end{equation}
  where
  \begin{align*}
    E_1( \chi_h )
    & = m( t; \mdell \rho, \chi_h^\ell ) + g_\ell( t; \rho, \chi_h^\ell ) \\
    E_2( \chi_h )
    & = \big( m_h( t; \mdh \Pi_h \check{u}, \chi_h ) - m( t; \mdell \pi_h \check{u}, \chi_h^\ell ) \big)
      + \big( g_h( t; \Pi_h \check{u}, \chi_h ) - g_\ell( t; \pi_h \check{u}, \chi_h^\ell ) \big) \\
    & \quad + m( t; \check{u}, \md \chi_h^\ell - \mdell \chi_h^\ell ).
  \end{align*}
\end{lemma}

\begin{proof}
  We transform \cref{eq:checku} into variational form using \cref{eq:c-formula}, then together with the definition of the Ritz projection \cref{eq:ritz}, we see that
  \begin{align*}
    & \dt m_h( t; \Pi_h \check{u}, \chi_h )
    + a_h^\kappa( t; \Pi_h \check{u}, \chi_h )
    - m_h( t; \Pi_h \check{u}, \mdh \chi_h ) \\
    & = \dt \big( m_h( t; \Pi_h \check{u}, \chi_h )
      - m( t; \check{u}, \chi_h^\ell ) \big)
      - \big( m_h( t; \Pi_h \check{u}, \mdh \chi_h )
      - m( t; \check{u}, \md \chi_h^\ell )\big).
  \end{align*}
  We use the transport formulae \cref{eq:ch-formula} for $m_h$ and \cref{eq:ctilde-formula} for $m$ to see
  \begin{align*}
    & \dt m_h( t; \Pi_h \check{u}, \chi_h )
    + a_h^\kappa( t; \Pi_h \check{u}, \chi_h )
    - m_h( t; \Pi_h \check{u}, \mdh \chi_h ) \\
    & = \big( m_h( t; \mdh \Pi_h \check{u}, \chi_h )
      - m( t; \mdell \check{u}, \chi_h^\ell ) \big)
      + \big( g_h( t; \Pi_h u, \chi_h )
      - g_\ell( t; \check{u}, \chi_h^\ell ) \big) \\
    & \qquad + m( t; \check{u}, \md \chi_h^\ell - \mdell \chi_h^\ell ).
  \end{align*}
  Subtracting this equation from \cref{eq:checkUh} and rearranging gives \cref{eq:theta-eqn}.
\end{proof}

\begin{lemma}
  For $\chi_h \in \S_h(t)$, the consistency terms $E_1$ and $E_2$ satisfy
  \begin{equation}
    \label{eq:E-bound}
    \abs{ E_1( \chi_h ) }
    + \abs{ E_2( \chi_h ) }
    \le c h^{k+1} \big( \norm{ \check{u} }_{\Z(t)} + \norm{ \md \check{u} }_{\Z(t)} \big) \norm{ \chi_h }_{\V_h(t)}.
  \end{equation}
\end{lemma}

\begin{proof}
  For $E_1$, we use \cref{eq:m-bounded} and \cref{eq:c-bound} together with the error bounds from \cref{eq:ritz-bound} and \cref{eq:ritz-md-error} to see
  \begin{align*}
    \abs{ E_1( \chi_h ) }
    & \le c \big( \norm{ \rho }_{\H(t)} + \norm{ \mdell \rho }_{\H(t)} \big) \norm{ \chi_h }_{\H_h} \\
    & \le c h^{k+1} \big( \norm{ \check{u} }_{\Z(t)} + \norm{ \md \check{u} }_{\Z(t)} \big) \norm{ \chi_h }_{\V_h(t)}.
  \end{align*}
  For $E_2$, we use the perturbation estimates \cref{eq:m-error}, \cref{eq:c-error} and \cref{eq:md-error} together with the stability bounds on the Ritz projection \cref{eq:ritz-stab} and \cref{eq:ritz-md-stab} to see
  \begin{align*}
    \abs{ E_2( \chi_h ) }
    & \le c h^{k+1} \big( \norm{\check{u}}_{\H(t)} + \norm{ \Pi_h \check{u} }_{\V_h(t)} + \norm{ \mdh \Pi_h \check{u} }_{\V_h(t)} \big) \norm{ \chi_h }_{\V_h(t)} \\
    & \le c h^{k+1} \big( \norm{ \check{u} }_{\Z(t)} + \norm{ \md \check{u} }_{\Z(t)} \big) \norm{ \chi_h }_{\V_h(t)}.
      \qedhere
  \end{align*}
\end{proof}

\begin{lemma}
  \label{lem:theta-bound}
  The following bound holds for $\vartheta = \check{U}_h - \Pi_h \check{u}$
  \begin{multline}
    \sup_{t \in [0,T]} \norm{ \vartheta }_{\H_h(t)}^2
    + \int_0^T \norm{ \vartheta }_{\V_h(t)}^2 \dd t \\
    \le \norm{ \vartheta }_{\H_h(0)}^2 + c h^{2k+2} \int_0^T \norm{ \check{u} }_{\Z(t)}^2 + \norm{ \md \check{u} }_{\Z(t)}^2 \dd t.
  \end{multline}
\end{lemma}

\begin{proof}
  We test \cref{eq:theta-eqn} with $\chi_h = \vartheta$ to see
  \begin{equation*}
    \dt m_h( t; \vartheta, \vartheta )
    + a_h^\kappa( t; \vartheta, \vartheta )
    - m_h( t; \vartheta, \mdh \vartheta )
    = - E_1( \vartheta ) - E_2( \vartheta ).
  \end{equation*}
  The transport formula for $m_h$ \cref{eq:ch-formula} tells us that
  \begin{equation*}
    \dt m_h( t; \vartheta, \vartheta )
    - m_h( t; \vartheta, \mdh \vartheta )
    = \frac{1}{2} \dt m_h( t; \vartheta, \vartheta )
    + \frac{1}{2} g_h( t; \vartheta, \vartheta ),
  \end{equation*}
  hence, applying the bound on $E_1$ and $E_2$ \cref{eq:E-bound} we infer that
  \begin{align*}
    \frac{1}{2} \dt m_h( t; \vartheta, \vartheta )
    + a_h^\kappa( t; \vartheta, \vartheta )
    \le - \frac{1}{2} g_h( t; \vartheta, \vartheta )
    + c h^{k+1} \big( \norm{ \check{u} }_{\Z(t)} + \norm{ \md \check{u} }_{\Z(t)} \big) \norm{ \vartheta }_{\V_h(t)}.
  \end{align*}
  Applying the boundedness and coercivity estimates form $m_h, a_h^\kappa$ and $g_h$ \cref{eq:mh-bounded}, \cref{eq:ahkappa-coercive} and \cref{eq:ch-bounded} with a Young's inequality and integrating in time gives
  \begin{multline*}
    \norm{ \vartheta }_{\H_h(t)}^2
    + \int_0^t \norm{ \vartheta }_{\V_h(t')}^2 \dd t' \\
    \le \norm{ \vartheta }_{\H_h(0)}^2 + c \int_0^t \norm{ \vartheta }_{\H_h(t')}^2 \dd t'
    + c h^{2k+2} \int_0^t \big( \norm{ \check{u} }_{\Z(t')}^2 + \norm{ \md \check{u} }_{\Z(t')}^2 \big) \dd t'.
  \end{multline*}
  Finally, we use a Gr\"{o}nwall inequality to see the desired result.
\end{proof}

Finally, we can show the result of \cref{thm:error}.

\begin{proof}[Proof of \cref{thm:error}]
  \label{proof:thm:error}
  We apply the splitting \cref{eq:checku-split}, the bounds on $\rho$ from \cref{lem:ritz}, the bounds on $\theta$ from \cref{lem:theta-bound} and the estimate on $\check{u}$ from \cref{eq:checku-reg} to see
  \begin{align*}
    & \sup_{t \in (0,T)} \norm{ \utext - u_h }_{\H(t)}^2
    + h^2 \int_0^T \norm{ \utext - u_h }_{\V(t)}^2 \dd t \\
    & \le c \sup_{t \in (0,T)} \big( \norm{ \theta }_{\H(t)}^2 + \norm{ \rho }_{\H(t)}^2 \big)
      + h^2 \int_0^T \big( \norm{ \theta }_{\V(t)}^2 + \norm{ \rho }_{\V(t)}^2 \big) \\
    & \le c h^{2k+2} \left( \sup_{t \in (0,T)} \norm{ \check{u} }_{\Z(t)}^2
      + \int_0^T \big( \norm{ \check{u} }_{\Z(t)}^2 + \norm{ \md \check{u} }_{\Z(t)}^2 \big) \dd t \right)
      + c \norm{ u - u_{h,0} }_{\H(0)}^2 \\
    & \le c h^{2k+2} \left( \sup_{t \in (0,T)} \norm{ \utext }_{\Z(t)}^2 + \int_0^T \norm{ \md \utext }_{\Z(t)}^2 \dd t \right)
      + c \norm{ \utext_0 - u_{h,0} }_{\H(0)}^2 .
  \end{align*}
  The final line follows since the $L^2$ norm is bounded by the $L^\infty$ norm.
\end{proof}

\clearpage
\part{Evolving finite element spaces}\label{evolve}
\section{Evolving bulk finite element spaces}
\label{sec:bfem}

In this section, we will define families of  evolving bulk finite element spaces $\{\S_h(t)\}_{t\in[0,T]}$  on families of evolving 
triangulated bulk domains $\{ \Omega_h(t)\}_{t\in[0,T]}$ consisting of unions of elements. We have in mind  $\Omega_h(t)\subset \R^{n+1}$ approximating an open bounded domain  $\Omega(t) \subset \R^{n+1}$. We will use the word bulk in this section to emphasise the difference to the surface case considered in \cref{sec:sfem} but more common terminology would simply remove this word. 

Our work extends from standard bulk finite element theory \citep{Cia78} and the work of \citet{CiaRav72} and \citet{Ber89} for Cartesian bulk  domains with curved boundaries to the evolving case. Throughout this section we will denote global discrete quantities with a subscript $h \in (0,h_0)$, which is related to element size. We assume implicitly that these structures exist for each $h$ in this range. See also \cref{rem:small-hk-bulk}. For ease of exposition we begin with definitions without the time parameter, $t$.

\subsection{Reference finite element}

\begin{definition}
  [Reference finite element]
  \label{def:reference-element}
  The triple $( \hat{K}, \hat{\P}, \hat{\Sigma} )$ is a \emph{reference finite element} if:
  \begin{deflist}
  \item \label{def:ref-element-domain}  the \emph{element domain}  $\hat{K} \subset \R^\mathsf m$ is the closure of an open  domain with Lipschitz piecewise smooth boundary,
  \item \label{def:ref-shape-function} the set of \emph{shape functions} $\hat{\P}$ is a finite dimensional space of functions over $\hat{K}$,
  \item \label{def:ref-nodal-variables} the \emph{nodal variables} or \emph{degrees of freedom}
$\hat{\Sigma} = \{ \hat\sigma_1, \ldots, \hat\sigma_d \}$ are a basis of $\hat{\P}'$ the dual space to $\hat\P$,
\end{deflist}
  and $\hat\Sigma$ determines $\hat\P$, that is if for $\hat\chi \in \hat\P$ with $\hat\sigma(\hat\chi) = 0$ for all $\hat\sigma \in \hat\Sigma$, we have $\hat\chi = 0$.
\end{definition}

As part of this definition, we are implicitly assuming that the nodal variables live in the dual to a larger function space than $\hat\P$. We will see that this usually requires further smoothness or continuity of finite element functions.
We give an example of a simplical finite element, but this definition includes other examples such as isoparametric finite elements and brick finite elements.

 Recall that  a  (non-degenerate) $\mathsf m$-\emph{simplex}  $K$  in $\R^\mathsf m$,  is the convex hull of $\mathsf m+1$ distinct points $\{ a_i \}_{i=1}^{\mathsf m+1} \subset \R^\mathsf m$, called the \emph{vertices} of the $\mathsf m$-simplex, which are not contained in a common $(\mathsf m-1)$-dimension hyperplane. More precisely, we have
  \begin{equation*}
    K = \left\{
      x = \sum_{i=1}^{\mathsf m+1} \mu_i a_0 : 0 \le \mu_i \le 1, 1 \le i \le \mathsf m+1, \sum_{i=1}^{\mathsf m+1} \mu_i = 1
    \right\}.
  \end{equation*}
 For each $x \in K$, we call $\{ \mu_i \}_{i=1}^{\mathsf m+1}$ \emph{barycentric coordinates}.
 For any integer $l$ with $0 \le l \le \mathsf m$, an $l$-facet of an $\mathsf m$-simplex $K$ is any $l$-simplex whose $(l+1)$ vertices are also vertices of $K$. We call an $(\mathsf m-1)$-facet a \emph{boundary facet}.
  We will also use the term boundary facet for any boundary polytopes (union of simplicies) of a polytope $K$.
  For each $k \ge 0$, we shall denote by $\P_k$ the space of all polynomials of degree $k$ in the variables $x_1, \ldots, x_\mathsf m$ in $\R^\mathsf m$. For any set $A \subset \R^\mathsf m$, we let
  \begin{equation*}
    \P_k( A ) = \{ \chi |_A : \chi \in P_k \}.
  \end{equation*}

\begin{example}[Example of reference finite element]
  \label{ex:standard-fem}
  The standard piecewise linear finite element $(K, \P_1(K), \Sigma^K )$  is obtained by choosing $K$ to be a non-degenerate $\mathsf m$-simplex   in $\R^\mathsf m$ and $\Sigma^K = \{ \chi \mapsto \chi(a_i) : 1 \le i \le \mathsf m+1 \}$. We can also define higher order spaces $( K, \P_k( K ), \Sigma^K )$, for $k \ge 2$, by including extra evaluation points in $\Sigma^K$ (see, for example, \citet[Section 2.2]{Cia78}).
  The key property of the extra evaluation points is that they determine the particular function in $\P_k( K )$. It is also true that the restriction of $\Sigma^K$ to any facet determines the restriction of functions in $P_k( K )$ on that facet.
\end{example}

Given the reference element $(\hat{K}, \hat{P}, \hat{\Sigma})$ and a function $\hat{\eta} \colon \hat{K} \to \R$ for which the nodal variables $\sigma_i( \hat{\eta} )$ can be computed (e.g. in the case of Lagrange element we require $\hat{\eta}$ to be continuous), we define the nodal interpolation of $\hat{\eta}$, written $\hat{I} \hat{\eta}$, as the unique function in $\hat{P}$ which has the same nodal values as $\hat{\eta}$. Let $\{ \hat\chi_i : 1 \le i \le d \} \subset P$ be the basis of $P$ dual to $\Sigma$ then we can characterise $\hat{I} \hat{\eta}$ as
\[
  \hat{I} \hat{\eta} := \sum_{i=1}^d \sigma_i( \hat{\eta} ) \chi_i.
\]

\begin{lemma}[Bramble-Hilbert Lemma, \citealt{Cia78}, Thm.~3.1.5]
  \label{lem:bramble-hilbert}
  Let the following inclusions hold for $m, k > 0$, and $p,q \in [1,\infty]$,
  \begin{align*}
    & W^{k+1,p}( \hat{K} ) \hookrightarrow C( \hat{K} ) \\
    & W^{k+1,p}( \hat{K} ) \hookrightarrow W^{m,q}( \hat{K} ) \\
    & \P_k( \hat{K} ) \subset \hat{P} \subset W^{m,q}(\hat{K}).
  \end{align*}
  Under the above assumptions on the reference finite element we have that there exists a constant $C = C( \hat{K}, \hat{\P}, \hat{\Sigma} ) > 0$ such that for all functions $\hat\eta \in W^{k+1,p}(\hat{K})$,
  \begin{equation}
    \label{eq:bramble-hilbert}
    \abs{ \hat{\eta} - \hat{I} \hat{\eta} }_{W^{m,q}(\hat{K})}
    \le C \abs{ \hat{\eta} }_{W^{k+1,p}(\hat{K})}.
  \end{equation}
\end{lemma}

\subsection{Bulk finite element}

We start by defining a single bulk finite element in $\mathbb R^{n+1}$.
Our definition of bulk finite element combines Def.~2.1 and 2.3 from \citep{Ber89}.
\begin{definition}
  [Bulk element reference map and bulk finite element]
  \label{def:bfe}
  Let $(\hat{K}, \hat{\P}, \hat{\Sigma} )$ be a reference finite element (\cref{def:reference-element}) with $\hat{K} \subset \R^{n+1}$.
  \begin{deflist}
    \item Let $F_K \colon \hat{K} \to \R^{n+1}$ satisfy
  \begin{enumerate}
    \item
    \begin{enumerate}
    \item $F_K \in C^1( \hat{K}, \R^{n+1} )$;
    \item $\rank \nabla F_K = n+1$;
    \item $F_K$ is a bijection onto its image;
    \end{enumerate}
  \item $F_K$ can be decomposed into an affine part and smooth part
    \begin{equation*}
      F_K( \hat{x} ) = A_K \hat{x} + b_K + D_K( \hat{x} )
    \end{equation*}
    such that $A_K$ is an invertible $(n+1) \times (n+1)$ matrix, $b_K \in \R^{n+1}$, and  $D_K \in C^1(\hat{K})$
    \begin{equation}
      \label{eq:bulk-CK-defn}
      C_K := \sup_{\hat{x} \in \hat{K}} \norm{ \nabla D_K( \hat{x} ) A_K^{-1} } < 1,
    \end{equation}
    where $\norm{ \cdot }$ denotes the two-norm of the matrix.
  \end{enumerate}
  In this situation we call $F_K$ a \emph{bulk element reference map}.
  \item Let $F_K$ be a bulk element reference map and $( K, \P, \Sigma )$ be the triple given by
  \begin{subequations}
    \label{eq:bfe}
    \begin{align}
      \label{eq:bulk-element-domain}
      K & := F_K( \hat{K} ) && \mbox{(the \emph{element domain})} \\
      \label{eq:bulk-shape-functions}
      \P & := \{ \hat{\chi} \circ F_K^{-1} : \hat{\chi} \in \hat{P} \} && \mbox{(the \emph{shape functions})}\\
      \label{eq:bulk-nodal-variables}
      \Sigma & := \{ \chi \mapsto \hat{\sigma}( \chi \circ F_K ) : \hat{\sigma} \in \hat{\Sigma} \} && \mbox{(the \emph{nodal variables})}.
    \end{align}
  \end{subequations}
  Under the above assumptions, we call $( K, \P, \Sigma )$ a \emph{bulk finite element}, $( \hat{K}, \hat\P, \hat\Sigma )$ the associated \emph{reference finite element}.
\end{deflist}
\end{definition}

With the bulk reference element map $F_K$ we can compute integrals and derivatives over the reference element using the transformation identity:
\[
  \int_{K} \chi(x) \dd x = \int_{\hat{K}} \hat{\chi}(\hat{x}) \abs{ \det \nabla F_K( \hat{x} ) } \dd \hat{x}, \qquad
  \nabla \chi( x ) = \nabla F_K^{-t} ( \hat{x} ) \nabla \hat{\chi}( \hat{x} ).
\]
We denote by $\nu_K$ the outward pointing normal to $K$.

\begin{definition}[$\Theta$-bulk finite element, Def.~2.4 \citealt{Ber89}]
  \label{def:Theta-bfe}
  Let $\Theta \in \Nbb$ and $F_K$ be the bulk element reference map for a bulk finite element $(K, P, \Sigma)$.
  \begin{deflist}
    \item We say that $F_K$ is a \emph{$\Theta$-bulk finite element reference map} if
    \begin{subdeflist}
    \item the bulk element reference map  $F_K \in C^{\Theta+1}(\hat{K}; \R^{n+1})$;
    \item for $1 \le m \le \Theta+1$, there exists constants $C_m(K) > 0$ such that
      \begin{equation}
        \label{eq:bulk-CK-theta-defn}
        \sup_{\hat{x} \in \hat{K}} \abs{ \nabla^m F_K( \hat{x} ) } \norm{ A_K }^{-m}
        \le C_m(K).
      \end{equation}
    \end{subdeflist}
  \item  \label{def:Theta-bfe-funcs}
    We say that $(K, \P, \Sigma)$ is a $\Theta$-bulk finite element if $F_K$ is a $\Theta$-bulk finite element reference map and
  \begin{subdeflist}
  \item the space $\P$ contains the functions $\hat{\chi} \circ F_K^{-1}$ for all $\hat\chi \in P_\Theta( \hat{K} )$;
    \label{def:Theta-bfe-functions-smooth}
  \item the space $\P$ is contained in $C^{\Theta+1}(K)$.
  \end{subdeflist}
\end{deflist}
\end{definition}

\begin{remark}
  \label{rem:bulk-P-closed}
  The properties of $K$ allow us to define the Sobolev spaces $W^{m,p}( K )$ for $1 \le m \le \Theta+1$, $1 \le p \le \infty$.
  Since $P \subset C^{\Theta+1}(K)$ and is finite dimensional, we clearly see that $P$ is a closed subspace of $W^{m,p}( K )$ for $1 \le m \le \Theta+1$, $1 \le p \le \infty$.
\end{remark}

\begin{example}
  [Bulk finite elements]
  \label{ex:bulk-fem}

  We are thinking of two particular examples. The first is a standard Lagrange finite element and the second is an isoparametric finite element. Examples of each of these cases are shown in \cref{fig:bfem-examples}.

  \begin{exlist}
  \item \label{ex:standard-bfe}
    Let $( \hat{K}, \hat{P}, \hat\Sigma )$ be a reference finite element. Consider the affine map $F_K \colon \hat{K} \to \R^{n+1}$ given by $F_K( \hat{x} )  = A_K \hat{x} + b_K$. If $A_K$ is non-singular, then this defines a bulk finite element $(K, P, \Sigma )$. In the standard way, the element domain $K$ is defined by the location of its vertices.
    For a simplex reference element domain $\hat{K}$, we are thinking of line segments in $\R$, triangles in $\R^2$ and tetrahedra in $\R^3$.

  \item \label{ex:isoparametric-bfe}
    Let $( \hat{K}, \hat\P, \hat\Sigma )$ be a reference finite element. Let $( K, \P, \Sigma )$ be a bulk finite element which is the image of $( \hat{K}, \hat\P, \hat\Sigma )$ under a map $F_K$ which satisfies $F_K \in ( \hat{P} )^{n+1}$. We call $(K, \P, \Sigma)$ an isoparametric bulk finite element.
    We note that the functions in $\P$ will not necessarily consist of polynomials over $K$ even if $\hat{P}$ consists of polynomials over $\hat{K}$,
    however this leads to a practical scheme where integrals are computed over reference elements.
    This example is the basis for the method in \cref{BULKPDE}.
  \end{exlist}
\end{example}

\begin{figure}
  \centering

  \includegraphics{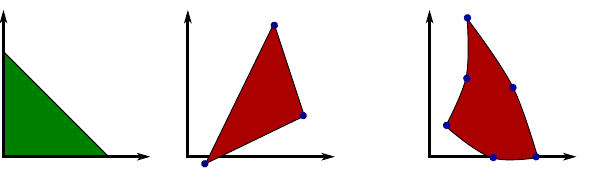}

  \caption{Examples of different bulk finite elements in the case $n=2$. Left shows a reference finite element (in green), centre shows a standard finite element (\cref{ex:standard-bfe}) and right shows an isoparametric bulk finite element (\cref{ex:isoparametric-sfe}) with quadratic $F_K$. The plot shows the element domains in red and the location of the nodes in blue.}
  \label{fig:bfem-examples}
\end{figure}

The definition of bulk finite element (\cref{def:bfe}), in particular \cref{eq:bulk-CK-defn}, is constructed to allow the following result:

\begin{lemma}[Lem.~2.1, \citealt{Ber89}]
  \label{lem:bulk-FK-AK-scale}
  Let $F_K \colon \hat{K} \to K$ be a bulk finite element reference map then $F_K$ is a $C^1$-diffeomorphism and satisfies
  \begin{align}
    \label{eq:bulk-FK-AK-scale1}
    \sup_{\hat{x} \in \hat{K}} \norm{ \nabla F_K( \hat{x} ) }
    & \le ( 1 + C_K ) \norm{ A_K } \\
    \label{eq:bulk-FK-AK-scale2}
    \sup_{\hat{x} \in \hat{K}} \norm{ (\nabla F_K (\hat{x}) )^{-1}}
    & \le ( 1 - C_K ) \norm{ A_K^{-1} },
  \end{align}
  and also for all $\hat{x} \in \hat{K}$
  \begin{equation}
    \label{eq:bulk-FK-AK-scale3}
    ( 1 - C_K )^{n+1} \abs{ \det( A_K ) }
    \le \abs{ \det( \nabla F_K(\hat{x}) ) }
    \le ( 1 + C_K )^{n+1} \abs{ \det ( A_K ) }.
  \end{equation}
\end{lemma}

To help us understand the geometry of the new element domains $K$, we introduce a new element domain $\tilde{K}$ defined by the affine part of the parametrisation: $\tilde{K} := \{ A_K \hat{x} + b_K : \hat{x} \in \hat{K} \}$.

\begin{lemma}
  \label{lem:bulk-element-geometry}
  Let $K$ be an element domain \cref{eq:bulk-element-domain} parametrised by a bulk element reference map $F_K$ over $\hat{K}$.
  Denote by
  \begin{subequations}
    \label{eq:bulk-element-geometry}
    \begin{align}
      \label{eq:bulk-hK}
      & h_K := \diam(\tilde{K}) \\
      \label{eq:bulk-rhoK}
      & \rho_K := \sup \{ \diam(B) : B \text{ is a $n$-dimensional ball contained in } \tilde{K} \}.
    \end{align}
  \end{subequations}
  We will also write $\hat{h}$ and $\hat{\rho}$ for the diameter of $\hat{K}$ and diameter of the maximum inscribed ball in $\hat{K}$.
  Then we have that
  \begin{subequations}
    \label{eq:bulk-AK-h-relations}
    \begin{align}
      \label{eq:bulk-AK-h-relation-a}
      \norm{ A_K } & \le \frac{h_K}{ \hat\rho } \\
      \label{eq:bulk-AK-h-relation-b}
      \norm{ A_K^{-1} } & \le \frac{ \hat{h} }{ \rho_K } \\
      \label{eq:bulk-AK-h-relation-c}
      \sup_{\hat{x} \in \hat{K}} \abs{ \det( \nabla F_K(\hat{x}) ) }
      & \le \frac{1}{\meas{\hat{K}}} \left( \frac{1+C_K}{1-C_K } \right)^{n+1} \meas(K).
    \end{align}
  \end{subequations}
\end{lemma}

\begin{proof}
  See \cite[Thm.~3.1.3]{Cia78}.
\end{proof}

\begin{remark}
  \label{rem:bulk-AK-h-meas}
  We note that the volume of an element $\meas(K)$ can be estimated by $h_K$ and $\rho_K$ by
  \begin{align*}
    c_1 \rho_K^{n+1} \le \meas{K} \le c_2 h_K^{n+1}.
  \end{align*}
  Here the positive constants $c_1, c_2$ depend on the volume of the unit ball in $\R^{n+1}$ and the constant $C_K$.
\end{remark}

\begin{remark}
  \label{rem:small-hk-bulk}
  In the sequel, we will assume implicitly that results hold for $h_K$ sufficiently small ($h_K < h_0$) for some particular value of $h_0$.
  In general this is always possible by subdividing a particular element using a refinement procedure and applying the result to the subdivided, smaller elements.
\end{remark}

The choice of mapping allows us to relate functions defined on $K$ to functions on $\hat{K}$.

\begin{lemma}[Lem.~2.3, \citealt{Ber89}]
  \label{lem:bulk-norm-scaling}
  Let $F_K \colon \hat{K} \to K$ be a $\Theta$-bulk finite element reference map (\cref{def:Theta-bfe}).
  Let $0 \le m \le \Theta+1$ and $p \in [1,\infty]$, then $\chi \in W^{m,p}(K)$ implies $\hat{\chi} = \chi \circ F_K$ belongs to $W^{m,p}(\hat{K})$. We have for any $\chi \in W^{m,p}(K)$ that
  \begin{equation}
    \label{eq:bulk-norm-scale-a}
    \abs{ \hat{\chi} }_{W^{m,p}(\hat{K})}
    \le c \abs{ \det A_K }^{-\frac{1}{p}}
    \norm{ A_K }^m \sum_{r=0}^m \abs{ \chi }_{W^{r,p}(K)},
  \end{equation}
  for a constant which depends on $C_K, C_2(K), \ldots, C_m(K)$.
  We also have for any $\hat{\chi} \in W^{m,p}(\hat{K})$ that $\chi = \hat{\chi} \circ F_K^{-1} \in W^{m,p}(K)$ and
  \begin{align}
    \label{eq:bulk-norm-scale-b}
    \abs{ \chi }_{W^{m,p}(K)} \le c \abs{ \det A_K }^{\frac{1}{p}}
    \sum_{r=0}^m \norm{ A_K^{-1} }^r \abs{ \hat{\chi} }_{W^{r,p}(\hat{K})},
  \end{align}
  where the constant here depends on $C_K, C_2(K), \ldots, C_m(K)$ and the product $\norm{ A_K } \norm{ A_K^{-1} }$.
\end{lemma}

Given a bulk finite element $(K, \P, \Sigma)$ (\cref{def:bfe}), let $\{ \chi_i : 1 \le i \le d \} \subset \P$ be the basis dual to $\Sigma$. This is the set of \emph{basis functions} of the finite element. If $\eta$ is a function for which all $\sigma_i( \eta )$, $1 \le i \le d$ is well defined, then we define the \emph{local interpolant} by
\begin{equation}
  \label{eq:bulk-IK}
  I_K \eta := \sum_{i=1}^d \sigma_i( \eta ) \chi_i.
\end{equation}
We can think of $I_K \eta$ as the unique shape function that has the same nodal values as $\eta$ so that, in particular, $I_K \chi = \chi$ for $\chi \in \P$.

\begin{theorem}
  [Local interpolation estimate]
  \label{thm:bulk-IK-bound}
  Let $( K, \P, \Sigma )$ be a $\Theta$-bulk finite element (\cref{def:Theta-bfe}) with reference element $(\hat{K}, \hat\P, \hat\Sigma )$ by $F_K$ which satisfies the assumptions of \cref{lem:bramble-hilbert}
    for some $0 < k,m \le \Theta$, $p,q \in [1,\infty]$.
  Then there exists a constant $C = C(\hat{K}, \hat{P}, \hat\Sigma )$ such that for all functions $\chi \in W^{k+1,p}(K)$
  \begin{equation}
    \label{eq:bulk-IK-estimate}
    \abs{ \chi - I_K \chi }_{W^{m,q}(K)}
    \le C \meas(K)^{1/q - 1/p}
    \frac{ h_K^{k+1} }{ \rho^m_K } \abs{ \chi }_{W^{k+1,p}(K)}.
  \end{equation}
\end{theorem}

\begin{proof}
  We re-scale \cref{eq:bramble-hilbert} using \cref{lem:bulk-norm-scaling} and the estimates from \cref{eq:bulk-AK-h-relations}:
  \begin{align*}
    \abs{ \chi - I_K \chi }_{W^{m,q}(K)}
    & \le c \left( \sup_{\hat{x} \in \hat{K}} \abs{ \det \nabla F_K( \hat{x} )} \right)^{1/q}
      \sum_{r=0}^m \norm{ A_K^{-1} }^r
      \abs{ \hat{\chi} - I_{\hat{K}} \hat{\chi} }_{W^{r,q}(\hat{K})} \\
    & \le c \left( \sup_{\hat{x} \in \hat{K}} \abs{ \det \nabla F_K( \hat{x} )} \right)^{1/q}
      \sum_{r=0}^m \norm{ A_K^{-1} }^r
      \abs{ \hat{\chi} }_{W^{k+1,p}(\hat{K})} \\
    & \le c \left( \sup_{\hat{x} \in \hat{K}} \abs{ \det \nabla F_K( \hat{x} )} \right)^{1/q - 1/p}
      \sum_{r=0}^m \norm{ A_K^{-1} }^r \norm{ A_K }^{k+1}
      \abs{ \chi }_{W^{k+1,p}(K)} \\
    & \le c \meas(K)^{1/q - 1/p}
      \frac{ h^{k+1}_K }{ \rho_K^m }
      \norm{ \chi }_{W^{k+1,p}(K)}.
  \end{align*}
  The last line holds if $\rho_K < 1$ (note that $\rho_K \ge h_K$ so this statement is true for $h_K$ small enough).
\end{proof}

\subsection{Triangulated bulk domain and spaces}
\label{sec:stationary-bfem}

Next, we bring together a finite family of bulk finite elements in order to create a bulk finite element space.

\begin{definition}
 \begin{deflist}
 \item \label{def:triangulated-domain}
  A \emph{triangulated (bulk) domain} is a set $\Omega_h$ equipped with an \emph{admissible subdivision} $\T_h$ consisting of bulk finite element domains such that $\bigcup_{K \in \T_h} K = {\Omega}_h$, $\mathring{K}_1 \cap \mathring{K}_2 = \emptyset$ for $K_1, K_2 \in \T_h$ with $K_1 \neq K_2$.
   \item \label{def:bulk-h}
The maximum subdivision diameter $h$ is defined by:
\begin{equation}
  \label{eq:bulk-h-defn}
  h := \max_{K \in \T_h} h_K.
\end{equation}

  \item \label{def:bulk-facet}
  Let $\Omega_h$ be a discrete bulk domain equipped with an admissible subdivision $\T_h$ such that each set $K \in \T_h$ is an element domain for a bulk finite element $(K, \P^K, \Sigma^K)$ parametrised over the same polygonal reference finite element $( \hat{K}, \hat{P}, \hat{\Sigma} )$. We say that $E \subset K$ is a \emph{facet} if $E$ is the image of a boundary facet of $\hat{K}$.
  \item \label{def:bulk-conforming-subdivision}
   We say that $\T_h$ is a \emph{conforming} subdivision of $\Omega_h$ if any facet of an element domain $K$ is either a facet of another element domain $K' \in \T_h$, in which case we say $K$ and $K'$ are \emph{adjacent}, or a portion of the boundary $\partial \Omega_h$.
 \item
   For a conforming subdivision, we denote by $\F_h$ the set of facets between adjacent elements and by $\partial \T_h$ the set of boundary facets.
   For a internal facet $F \in \F_h$ between adjacent elements $K, K' \in \T_h$, we make a choice of the two elements and denote by $\nu_F^+$ the outward normal to $K$ and by $\nu_F^-$ the outward normal to $K'$. The particular choice will not affect our calculations.
    \end{deflist}
\end{definition}

\begin{remark}
  \begin{remlist}
  \item  We recall that our elements domains (\cref{def:ref-element-domain}) are closed so that $\Omega_h$ is also a closed set.
  \item  When putting together bulk finite elements in order to form a discrete domain $\Omega_h$, we will be generally thinking of the case that only elements with more than one vertex on the boundary are curved ($D_K \neq0$).
  See \cref{BULKPDE} for more details.
  We allow for the more general case here.
\end{remlist}
\end{remark}

\begin{definition}[Broken Sobolev spaces and norms]
  \label{def:bulk-broken-sobolev-norm}

Let $\T_h$ be a subdivision of $\Omega_h$ consisting of $\Theta$-bulk finite elements.
Then for $0 \le m \le \Theta+1$, $p \in [1,\infty]$, we define the broken Sobolev space $W^{m,p}(\T_h)$ by
\begin{equation}
  \label{eq:bulk-broken-sobolev-space}
  W^{m,p}(\T_h) := \left\{
    \eta_h \in L^1( \Omega_h ) : \eta_h |_{K} \in W^{m,p}(K) \mbox{ for all } K \in \T_h
    \right\},
  \end{equation}
  with norm
\begin{equation}
  \label{eq:bulk-discrete-norm}
  \norm{ \eta_h }_{W^{m,p}(\T_h)}
  :=
  \begin{cases}
    \displaystyle
    \left(
      \sum_{K \in \T_h} \norm{ \eta_h }_{W^{m,p}(K)}^p
    \right)^{1/p}
    & p < \infty \\
    \displaystyle
    \max_{K \in \T_h} \norm{ \eta_h }_{W^{m,\infty}(K)}
    & p = \infty.
  \end{cases}
\end{equation}
\end{definition}

\begin{remark}
  This space is often used in the context of discontinuous Galerkin finite element methods. See, for example, \citet{ArnBreCoc02}.
\end{remark}

{
\begin{lemma}
  The space $W^{m,p}( \T_h )$ is complete.
\end{lemma}

\begin{proof}
  Consider a Cauchy sequence $\{ \eta_j \} \in W^{m,p}( \T_h )$. This implies
  \begin{itemize}
  \item $\eta_j$ is Cauchy in $L^p(\Omega_h)$ so there exists $\xi$ such that $\eta_j \to \xi$ in $L^p(\Omega_h)$.
  \item $\eta_j|_K$ is Cauchy for all $K$ so there exists $\xi_K$ such that $\eta_j|_{K} \to \xi_K$ in $W^{m,p}(K)$.
  \end{itemize}
  It is clear from the triangle inequality that $\xi|_K = \xi_K$:
  \[
    \norm{ \xi|_{K} - \xi_K }_{W^{m,p}(K)} \le \norm{ \xi|_{K} - \eta_j|_{K} }_{L^{m,p}(K)} + \norm{ \eta_j|_{K} - \xi_K }_{W^{m,p}(K)},
  \]
  since the right hand-side converges to $0$ as $j \to \infty$.
  Hence, we have shown that $\eta_j$ converges to a function $\xi \in W^{m,p}( \T_h )$ in the $W^{m,p}( \T_h )$ norm.
\end{proof}

Let $\Omega_h$ be a triangulated domain with conforming subdivision $\T_h$.
For $K \in \T_h$, we denote the trace of a function $\eta \in W^{1,p}(K)$ by $T_K \eta \in L^p( \partial K )$ and recall that there exists a constant $c_{T_K} > 0$ such that
\begin{equation}
  \label{eq:bulk-element-trace}
  \norm{ T_K \eta }_{L^p( \partial K )} \le c_{T_K} \norm{ \eta }_{W^{1,p}( K )} \qquad \mbox{ for all } \eta \in W^{1,p}(K).
\end{equation}
We define the space $W^{1,p}_T( \T_h )$ by
\begin{multline}
  \label{eq:bulk-trace-broken-sobolev-space}
  W^{1,p}_T( \T_h ) :=
  \Bigl\{
  \eta_h \in L^p( \Gamma_h ) : \eta_h |_K \in W^{1,p}(K) \mbox{ for all } K \in \T_h, \mbox{ and } \\
  T_K( \eta_h |_K ) = T_{K'}( \eta_h |_{K'} ) \mbox{ a.e. in } K \cap K' \mbox{ for adjacent } K, K' \in \T_h
  \Bigr\}.
\end{multline}
We equip this space with the broken norm $\norm{ \cdot }_{W^{1,p}(\T_h)}$.

\begin{lemma}
  \label{lem:bulk-W1pT-closed}
  The space $W^{1,p}_T( \T_h )$ is a closed subspace of $W^{1,p}( \T_h )$ so is complete.
\end{lemma}

\begin{proof}
  Take a sequence $\{ \eta_j \}$ which converges to $\eta_h \in W^{1,p}( \T_h )$. Then for any pair of adjacent elements $K, K' \in \T_h$ we have
  \begin{multline*}
    \norm{ \eta_h|_{K} - \eta_h|_{K'} }_{L^p( K \cap K' )}
    \le \norm{ \eta_h - \eta_j }_{L^p( \partial K )} + \norm{ \eta_h - \eta_j }_{L^p( \partial K' )} \\
    \le c_{T_K} \norm{ \eta_h - \eta_j }_{W^{1,p}(K)} + c_{T_{K'}} \norm{ \eta_h - \eta_j }_{W^{1,p}(K')}
    \le ( c_{T_K} + c_{T_{K'}} ) \norm{ \eta_h - \eta_j }_{W^{1,p}( \T_h )}.
  \end{multline*}
  Clearly the right hand side converges to $0$ as $j \to \infty$ so we have the traces of $\eta_h$ from adjacent elements coincide and $\eta_h \in W^{1,p}( \T_h )$.
\end{proof}

We will use the notation for $H^1_T( \T_h ) := W^{1,2}_T( \T_h )$ which is a Hilbert space when equipped with the obvious broken inner product.

\begin{lemma}
  \label{lem:bulk-W1pT-h-equal}
  Let $\T_h$ be a conforming subdivision of $\Omega_h$ then $W^{1,p}_T( \T_h ) = W^{1,p}( \Omega_h )$.
\end{lemma}

\begin{proof}
  First, let $\eta_h \in W^{1,p}_T( \T_h )$. Then we have that $\eta_h \in L^p( \Omega_h )$ and it is left to show that $\eta_h$ has a weak derivative in $L^p( \Omega_h )$.
  We have a candidate $\xi$ given element-wise by $\xi|_{K} = \nabla (\eta_h|_{K})$.
  It is clear that $\xi \in L^p( \Omega_h )$ and for $\varphi \in C^1_c( \Omega_h )$ and $i = 1, \ldots, n+1$, we have
  \begin{align*}
    \int_{\Omega_h} \eta_h \partial_i \varphi
    & = \sum_{K \in \T_h} \int_{K} \eta_h \partial_i \varphi \\
    & = \sum_{K \in \T_h} \left( -\int_{K} \xi_i \varphi
      + \int_{\partial K} \eta_h \varphi \nu_i^K  \right) \\
    & = - \int_{\Omega_h} \xi_i \varphi
      + \sum_{K \in \T_h} \int_{\partial K} \eta_h \varphi \nu_i^K.
  \end{align*}
  We note that we can write
  \[
    \sum_{K \in \T_h} \int_{\partial K} \eta_h \varphi \nu_i^K
    = \sum_{F \in \mathcal{F}_h} \int_{F} \eta_h \varphi ( \nu_F^{+} + \nu_F^{-} )_i
    + \sum_{F \in \partial \T_h} \int_{F} \eta_h \varphi \nu_F = 0,
  \]
  where $\mathcal{F}_h$ is the set of facets between adjacent elements in $\T_h$, since the traces of $\eta_h |_{K}$ and $\eta_h|_{K'}$ to $K \cap K'$ coincide for any adjacent pair $K, K'$.
  We note that the sum over edges is zero since the normals from adjacent elements are equal and opposite in a conforming triangulation and the integral over boundary facets is zero since $\varphi \in C^1_c( \Omega_h )$ is zero here.
  Thus, we see that $\xi$ is the weak derivative of $\eta_h$.

  Second, let $\eta_h \in W^{1,p}( \Omega_h )$ then it is clear that $\eta_h \in L^p( \Omega_h )$, $\eta_h |_{K} \in W^{1,p}( K )$ and by the trace theorem $\eta$ has common trace between adjacent elements.
\end{proof}

\begin{remark}
  The proof is based on ideas from \citet[Thm.~2.1.1]{Cia78} which showed that appropriate finite element spaces defined in that work are contained in $H^1( \Omega_h )$ (in our notation).
\end{remark}

}

\subsubsection{Bulk finite element space}
We restrict to Lagrangian finite elements over a polygonal reference finite element.
More precisely, we assume that the degrees of freedom for each element $(K, P, \Sigma)$ are given by
\[
  \Sigma = \{ \chi \mapsto \chi( a ) : a \in \N^K \},
\]
where $\N^K$ is a finite set of nodes in $K$.
We call $\N^K$ the set of \emph{Lagrange nodes} of $K$.

The set of degrees of freedom of adjacent bulk finite elements will be related as follows.
Let $( K, \P, \Sigma )$ and $( K', \P', \Sigma')$ be two bulk finite elements such that $K$ and $K'$ are adjacent with $\Sigma = \{ \chi \mapsto \chi( a ), a \in \N^K \}$ and $\Sigma' = \{ \chi \mapsto \chi( a' ), a' \in \N^{K'} \}$.
Then, we have
\begin{equation}
  \label{eq:bulk-node-agree}
  \left( \bigcup_{a \in \N^K} a \right) \cap K'
  =
  \left( \bigcup_{a' \in \N^{K'}} a' \right) \cap K.
\end{equation}

We denote the \emph{global set of Lagrange nodes} by
\begin{equation}
  \label{eq:bulk-Nh}
  \N_h = \bigcup_{K \in \T_h} \N^K.
\end{equation}
For each $a \in \N_h$, let $\T( a ) \subset \T_h$ be the local neighbourhood of elements for which $a \in \N^K$.

\begin{definition}
  [Bulk finite element space]
  \begin{deflist}
  \item \label{def:bfe-space}
    Let $\Omega_h$ be a discrete bulk domain equipped with a conforming subdivision $\T_h$ with each domain $K$ equipped with a bulk finite element $( K, \P^K, \Sigma^K )$ (\cref{def:bfe}) which satisfy \cref{eq:bulk-node-agree}.
    A \emph{bulk finite element space} is a (generally proper) subset of the product space $\prod_{K \in \T_h} \P^K$ given by
    \begin{multline*}
      \S_h :=
      \bigg\{ \chi_h = ( \chi_K )_{K \in \T_h} \in \prod_{K \in \T_h} \P^K : \\
              \chi_{K}( a ) = \chi_{K'}( a ), \mbox{ for all } K, K' \in \T( a ),
                \mbox{ for all } a \in \N_h \bigg\}.
    \end{multline*}
  \item \label{def:bulk-global-degrees-of-freedom}
    The bulk finite element space is determined by the \emph{global degrees of freedom}
    \begin{equation*}
      \Sigma_h = \left\{
        \chi_h \mapsto \chi_h( a ) : a \in \N_h
      \right\}.
    \end{equation*}
  \end{deflist}
\end{definition}

In this definition, an element $\chi_h \in \S_h$ is not, in general a ``function'' defined over $\bar\Omega_h$, since we do not necessarily have a good definition of $\chi_h$ over element boundaries: The ``function'' may be double-valued.

If it happens, however, that for each element $\chi_h \in \S_h$, the restrictions $\chi_K$ and $\chi_{K'}$ coincide along the common face of any adjacent elements $K$ and $K'$, then the function $\chi_h$ can be identified with a function defined over the set $\bar\Gamma_h$.
In this case, we call the elements $\chi_h \in \S_h$ \emph{bulk finite element functions}.
Examples of bulk finite element functions are shown in \cref{fig:bfem-func-ex}.

\begin{figure}[tb]
  \centering

  \includegraphics[width=0.3\textwidth]{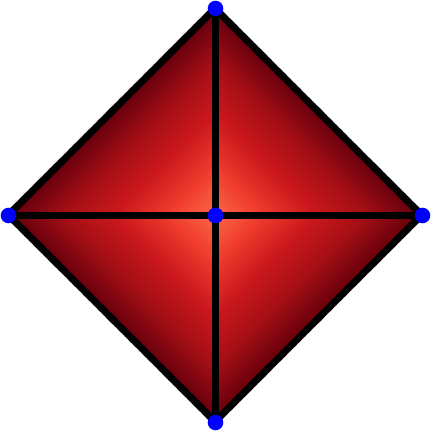}
  \hfill
  \includegraphics[width=0.3\textwidth]{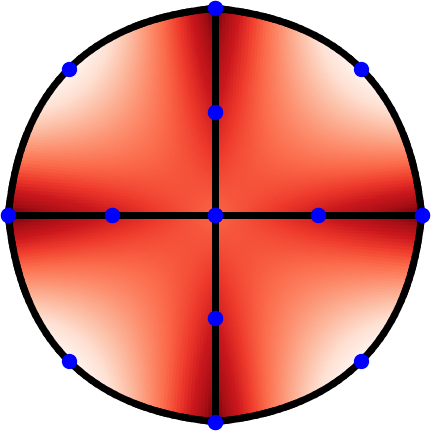}
  \hfill
  \includegraphics[width=0.3\textwidth]{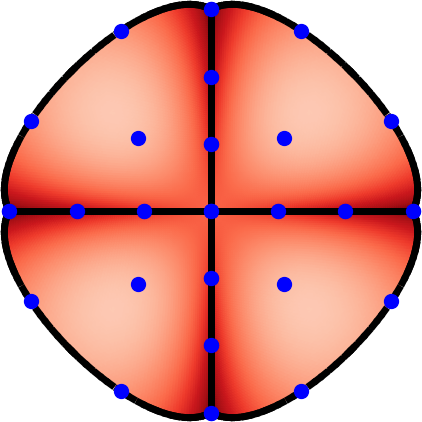}

  \caption{Examples of finite element functions.
    Left shows a piecewise linear function over a collection of standard finite elements (c.f. \cref{ex:standard-bfe}), centre and right shows a piecewise quadratic and cubic function over a collection of isoparametric (quadratic, respectively cubic) finite elements (c.f. \cref{ex:isoparametric-bfe}).
    The functions shown are interpolations into the appropriate finite element spaces of the function $x \mapsto \abs{x} \cos( 4 \arctan(x_2/x_1) )$.
    The distribution of Lagrange points is shown in blue}
  \label{fig:bfem-func-ex}
\end{figure}

We enumerate the nodes so that $\N_h = \{ a_i \}_{i=1}^N$ and take $\{ \chi_i \}_{i=1}^N$ to be the basis of $\S_h$ dual to $\Sigma_h$. Since, we have a finite basis of $\S_h$ we note that we can identify any $\chi_h \in \S_h$ with a vector $\gamma \in \R^N$ so that
\begin{equation*}
  \chi_h( x ) = \sum_{i=1}^N \gamma_i \chi_i( x ) \qquad \mbox{ for } x \in \Omega_h.
\end{equation*}

{
\begin{lemma}
  \label{lem:bulk-Sh-closed}
  Let $\S_h$ be a bulk finite element space consisting of bulk finite elements over a conforming subdivision $\T_h$ of $\Omega_h$.
  Assume further that for each $K \in \T_h$, the corresponding reference finite element is a Lagrange element of order $k \ge 1$ (\cref{ex:standard-fem}).
  Then we can identify elements of $\S_h$ as functions in $C(\Omega_h)$.
  Furthermore $\S_h$ is a closed subspace in $H^1_{T}( \T_h )$.
\end{lemma}

\begin{proof}
  Consider two adjacent elements $K, K' \in \T_h$ and $\chi_h \in \S_h$.
  The functions $\chi_K \circ F_K$ and $\chi_{K'} \circ F_{K'}$ when restricted to the appropriate edges in $\hat{K}$ are polynomials of degree $k$ which agree at the Lagrange points on this edge from \cref{eq:node-agree} and the definition of $\S_h$.
  The Lagrange points in the reference element determine polynomials of degree $k$ so we have that $\chi_{K} = \chi_{K'}$ on $K \cap K'$.
  Since $\T_h$ is a conforming subdivision we can define a global function $\chi_h \colon \bar\Omega_h \to \R$ such that $\chi_h|_{K} = \chi_K$ for each $K \in \T_h$ which is globally continuous.
  Indeed $\chi_h$ restricted to each element is continuous, as the composition of a polynomial (element of $\hat{P} = P_k( \hat{K} )$) and a smooth surface finite element reference map $F_K$, and is single valued on the facets where any two elements meet.

  In fact the restriction $\chi_h|_K$ to each element $K \in \T_h$ is a $C^1$-function hence $\chi_h|_K \in H^1(K)$ so it is clear that $\S_h \subset H^1_T( \T_h )$.
  The space $\S_h$ is closed since it is finite dimensional.
\end{proof}

\begin{remark}
  The proof is based on ideas from \citet[Thm.~2.2.3]{Cia78} which show that appropriate finite element spaces defined in that work are contained in $C^0( \bar{\Omega}_h ) \cap H^1( \Omega_h )$ (in our notation).
\end{remark}

}

The approximation property of the finite element space will be defined through an interpolation operator:
\begin{definition}[Interpolation]
  \label{def:bulk-global-interpolant}
If $\eta$ is a function on $\Omega_h$ for which all $\sigma_i( \eta )$, $1 \le i \le N$, is well defined (in case of Lagrangian finite elements, $\eta \in C(\Omega_h)$ suffices), then we can define a \emph{global interpolant} $I_h \eta$ by
\begin{equation*}
  I_h \eta := \sum_{i=1}^N \sigma_i( \eta ) \chi_i.
\end{equation*}
Note that our construction implies that
\begin{equation*}
  ( I_h \eta )|_{K} = I_K \eta |_K \mbox{ for all } K \in \T_h,
\end{equation*}
and $I_h \chi_h = \chi_h$ for all $\chi_h \in \S_h$.
\end{definition}
In order to prove estimates on the global interpolant, we will first define three further properties of our subdivision $\T_h$.

\begin{definition}
  [Regular and quasi-uniform subdivisions]
  For $h \in (0,h_0)$, let $\Omega_h$ be a triangulated bulk domain (\cref{def:triangulated-domain}) equipped with a conforming subdivision $\T_h$ (\cref{def:bulk-conforming-subdivision}).
  \begin{deflist}
  \item \label{def:bulk-regular} [Def.~3.1, \citealt{Ber89}]
    The family is said to be non-degenerate or \emph{regular} if there exists $\rho_\reg > 0$ such that for all $K \in \T_h$ and all $h \in (0,h_0)$,
    \begin{equation*}
      \rho_K \ge \rho_\reg h_K,
    \end{equation*}
    and there exists a constant $C > 0$ such that
    \[
      \sup_{h \in (0,h_0)} \max_{K \in \T_h} C_K \le C < 1.
    \]
  \item \label{def:bulk-Theta-regular} [Def.~3.2, \citealt{Ber89}]
    The family is said to be \emph{$\Theta$-regular} if it is regular, if for all $h \in (0,h_0)$ and all $K \in \T_h$, $F_K \in C^{\Theta+1}(K; \R^{n+1})$, and if, there exists a constant $C > 0$ such that
    \[
      \sup_{h \in (0,h_0)} \max_{K \in \T_h} C_m( K ) \le C < + \infty \qquad \mbox{ for } 2 \le m \le \Theta+1.
    \]
  \item \label{def:bulk-quasi-uniform}
    A regular family is said to be \emph{quasi-uniform} if there exists $\rho > 0$ such that
    \begin{equation*}
      \min \{ \rho_K : K \in \T_h \} \ge \rho h \quad \mbox{ for all } h \in (0,h_0).
    \end{equation*}
  \end{deflist}
\end{definition}
\begin{remark}
  We note that:
  \begin{itemize}
  \item for a regular subdivision there exists a constant $c > 0$ depending on the global quantities $\hat{\rho}, \hat{h}$ and $\rho_\reg$
    \begin{equation*}
      \norm{ A_K } \le c h_K \le c h
      \qquad \mbox{ and } \qquad
      \norm{ A_K^{-1} } \le c h_K^{-1},
    \end{equation*}
  \item for a quasi-uniform subdivision there exists a constant $c > 0$ depending on the global quantities $\hat{\rho}, \hat{h}$ and $\rho$
    \begin{equation*}
      \norm{ A_K } \le c h
      \qquad \mbox{ and } \qquad
      \norm{ A_K^{-1} } \le c h^{-1}.
    \end{equation*}
  \end{itemize}
\end{remark}

\begin{theorem}
  [Global interpolation estimates, c.f. Cor.~4.1, \citealt{Ber89}]
  \label{thm:bulk-Ih-bound}
  For $h \in (0,h_0)$, let $\Omega_h$ be a triangulated bulk domain (\cref{def:triangulated-domain}) equipped with a $\Theta$-regular (\cref{def:bulk-Theta-regular}), quasi-uniform (\cref{def:bulk-quasi-uniform}), conforming (\cref{def:bulk-conforming-subdivision}) subdivision $\T_h$. Let each $K \in \T_h$ be equipped with a $\Theta$-bulk finite element $( K, \P^K, \Sigma^K )$ (\cref{def:Theta-sfe}) parametrised over a reference finite element $(\hat{K}, \hat{\P}, \hat{\Sigma})$ which satisfies the assumptions of \cref{lem:bramble-hilbert}
    for some $0 < k,m \le \Theta$, $p,q \in [1,\infty]$.
Then there exists a constant $C = C( \hat{K}, \hat{\P}, \hat{\Sigma}, \rho )$ such that for all functions $\eta \in W^{k+1,p}( \T_h ) \cap C^0( \Omega_h )$,
  \begin{equation}
    \label{eq:bulk-Ih-bound}
    \norm{ \eta - I_h \eta }_{W^{m,q}(\T_h)}
    \le
    C h^{k+1-m}
    \norm{ \eta }_{W^{k+1,p}(\T_h)}.
  \end{equation}
\end{theorem}

\begin{proof}
  The proof follows by piecing together \cref{thm:bulk-IK-bound} using the fact that $\T_h$ is quasi-uniform.
\end{proof}

\begin{remark}
  The approximation property shown in \citep[Cor.~4.1]{Ber89} is a result for an $L^2$ projection for a more general class of finite element spaces.
\end{remark}

\subsection{Evolving bulk finite elements}
Let $ t \in [0,T]$ denote time. We consider families of bulk finite elements, spaces and triangulated domains parametrised by $t$.

\begin{definition}
  [Evolving bulk finite element]
  
 \begin{deflist}
 \item \label{def:ebfe}
   Let $(K(t), P(t), \Sigma(t))_{t \in [0,T]}$ be a time dependent family of bulk finite elements (\cref{def:bfe}) parametrised over a common reference element  $(\hat{K}, \hat{P}, \hat\Sigma)$.  If the constant $C_K = \sup_t C_{K(t)}$ is uniformly bounded away from 1,
  \begin{equation*}
    C_K := \max_{t \in [0,T]} \sup_{\hat{x} \in \hat{K}} \norm{ \nabla D_K( \hat{x}, t ) A_K^{-1}(t) }
    < c < 1,
  \end{equation*}
  we say that $(K(t), P(t), \Sigma(t))_{t\in[0,T]}$ is an \emph{evolving bulk finite element}.
 \item \label{def:ebfe-reference-map}
 Let  $\Phi_{(\cdot)}^K \in C^2([0,T], C^1(K_0))$ where $K_0:=K(0)$. We say that if $\Phi_t^K \colon K_0 := K(0) \to K(t)$ is such that
\begin{equation}
  \label{eq:ebfe-reference-map}
  F_{K(t)}( \hat{x} ) = \Phi_t^K( F_{K_0}(\hat{x} ) ) \qquad \mbox{ for } \hat{x} \in \hat{K}
\end{equation}
 then  $\Phi_t^K$ is the {\it flow}  defining the  evolution of the element domain and that $F_{K(t)}$ is the {\it evolving bulk element reference map}.
\item \label{def:bulk-element-velocity}
 The \emph{element velocity} $W_K$ of $K(t)$ is defined  by
\begin{equation*}
  W_K( \Phi_t^K( x ), t ) = \dt \Phi_t^K( x ) \qquad \mbox{ for } x \in K_0, t \in [0,T].
\end{equation*}

\item \label{def:e-Theta-bfe}
  If each $(K(t), \P(t), \Sigma(t))$ is a $\Theta$-bulk finite element for each $t \in [0,T]$ and the constants $C_m(K(t))$ are uniformly bounded:
  \[
    \sup_{t \in [0,T]} C_m( K(t) ) \le c < \infty \qquad \mbox{ for } 2 \le m \le \Theta+1
  \]
  then we say that $(K(t), P(t), \Sigma(t))_{t\in[0,T]}$ is an \emph{evolving $\Theta$-bulk finite element} and $F_{K(t)}$ is an \emph{evolving $\Theta$-bulk element reference map}.
  \item \label{def:bulk-temporally-quasi-uniform}
  We say that an evolving bulk finite element is \emph{temporally quasi-uniform}, if there exists $\rho_K > 0$ such that
\begin{equation*}
  \inf \{ \rho_{K(t)} : t \in [0,T] \} \ge \rho_K \sup \{ h_{K(t)} : t \in [0,T] \}.
\end{equation*}
\item \label{def:bulk-element-push-forward-map}
The family of {\it element push forward maps} denoted by $\phi_t^K( \chi ) \colon K(t) \to \R$ for $\chi \colon K_0 \to \R$, indexed by $t \in [0,T]$, is  defined to be the linear bijections defined by:
\begin{align*}
  \phi_t^K( \chi )( x ) = \chi( \Phi_{-t}^K( x ) ) \qquad \mbox{ for } x \in K(t).
\end{align*}

  \end{deflist}
\end{definition}

\begin{lemma}
  \label{lem:bulk-Kt-norm-equiv}
  Let $F_{K(t)} \colon \hat{K} \to K(t)$, for $t \in [0,T]$ be an evolving $\Theta$-bulk finite element reference map (\cref{def:e-Theta-bfe}) for a temporally quasi-uniform element domain $K(t)$ (\cref{def:bulk-temporally-quasi-uniform}) and $\phi^K_t$ the family of element push forward maps (\cref{def:bulk-element-push-forward-map}).
  Then there exists constants $c_1, c_2 > 0$, which depend only on the reference element $(\hat{K}, \hat{\P}, \hat{\Sigma})$ and the constants $C_K$ and $\rho_K$, such that for all $t \in [0,T]$,  $\chi \in W^{m,p}(K_0)$ if and only if $\phi^K_t \chi \in W^{m,p}(K(t))$ and
  \begin{align}
    \label{eq:bulk-Kt-norm-equiv}
    c_1 \norm{ \chi }_{W^{m,p}(K_0)}
    & \le \norm{ \phi^K_t \chi }_{W^{m,p}(K(t))}
      \le c_2 \norm{ \chi }_{W^{m,p}(K_0)}
    && \mbox{ for all } \chi \in W^{m,p}( K_0 ).
  \end{align}
\end{lemma}

\begin{proof}
  From \cref{lem:bulk-norm-scaling} and \cref{eq:bulk-AK-h-relations}, we have
  \begin{align*}
    \abs{ \chi }_{W^{m,p}(K_0)}
    \le c \left( \frac{ \meas(K(t)) }{ \meas (K_0) } \right)^{1/p}
    \norm{ \phi^K_t \chi }_{W^{m,p}(K(t))}
    \sum_{r=1}^m \left( \frac{ h_{K(t)} }{ \rho_{K_0} } \right)^r
  \end{align*}
  and
  \begin{align*}
    \abs{ \phi^K_t\chi }_{W^{m,p}(K(t))}
    \le c \left( \frac{ \meas(K_0) }{ \meas( K(t) )} \right)^{1/p}
    \norm{ \chi }_{W^{m,p}(K_0)}
    \sum_{r=1}^m \left( \frac{ h_{K_0} }{ \rho_{K(t)} } \right)^r.
  \end{align*}
  It can be easily seen that for a quasi-uniform evolving surface finite element that these constants only depend on allowed quantities.
\end{proof}

This result implies that $( W^{m,p}( K(t) ), \phi^K_t|_{W^{m,p}(K_0)} )_{t \in [0,T]}$ is a compatible pair (\cref{def:compatibility}).
Furthermore $( P(t), \phi^K_t|_{P(0)} )_{t \in [0,T]}$, equipped with the $W^{m,p}( K(t) )$ norm is also a compatible pair since $P(t)$ is a closed subspace of $W^{m,p}(K(t))$ (\cref{rem:compatible-subsace,rem:bulk-P-closed}).

\subsection{Evolving bulk triangulations and spaces}

We now derive definitions of an evolving bulk finite element space which is part of a compatible pair (in the sense of \cref{sec:abstract-formulation}, \cref{def:compatibility}).
For each $h \in (0,h_0)$, we are given a family of discrete bulk domains $\{ \Omega_h(t) \}_{t \in [0,T]}$ and each equipped with a bulk finite element space $\{ \S_h(t) \}_{t \in [0,T]}$.
Furthermore, we are interested in under what assumptions does the compatibility hold independently of the element diameter $h$.

\begin{definition}[Evolving bulk domain]
  \label{def:evolving-bulk-domain}
  For $t \in [0,T]$, let $\Omega_h(t)$ be a family of triangulated bulk domains (\cref{def:triangulated-domain}) each equipped with a conforming subdivision $\T_h(t)$ (\cref{def:bulk-conforming-subdivision}) such that each element domain $K(t) \in \T_h(t)$ is equipped with an element flow map $\Phi_{(\cdot)}^K \in C^2( [0, T]; C^1( K_0 ))$ (\cref{def:ebfe-reference-map}).

\begin{deflist}

\item \label{def:bulk-evolving-conforming-subdivision}
  We call $\{ \T_h(t) \}_{t \in [0,T]}$ an \emph{evolving conforming subdivision} if
  for each element $K(0) \in \T_h(0)$ and each facet $E(0)$ of $K(0)$
  either $E(0)$ is a facet of another element $K'(0) \in \T_h(0)$, in which case $E(t)$ is a common facet between $K(t)$ and $K'(t)$ for all $t \in [0,T]$ or
  $E(0)$ is a portion of the boundary $\partial \Omega(0)$, in which case $E(t)$ is a portion of the boundary $\partial \Omega(t)$ for all $t \in [0,T]$.
\item \label{def:evolving-triangulated-domain}
An \emph{evolving triangulated (bulk) domain} is  defined to be a family of triangulated bulk domains $\{ \Omega_h(t) \}_{t \in [0,T]}$ equipped with an evolving conforming subdivision. In this case, we define the mesh parameter $h$ to be
\begin{equation}
  \label{eq:bulk-def-h}
  h := \sup_{t \in [0,T]} \max_{K(t) \in \T_h(t)} h_{K(t)}.
\end{equation}
\item \label{def:bulk-global-discrete-flow}
We define a \emph{global discrete flow} $\Phi^h_{(\cdot)}(\cdot) \colon [0,T] \times \Omega_{h,0} \to \R^{n+1}$ element-wise by
\begin{equation*}
  \Phi^h_t|_{K_0} := \Phi^K_t \qquad \mbox{ for } K_0 \in \T_h(0).
\end{equation*}
Our assumptions imply that $\Phi^h_t$ is piecewise smooth and $\Phi^h_t \colon \Omega_{h,0} \to \Omega_h(t)$.

\item \label{def:bulk-global-discrete-velocity}
  We define a \emph{global discrete velocity}  ${W}_h$ given by
\begin{equation*}
  {W}_h|_{K(t)} = {W}_K.
\end{equation*}
\item \label{def:bulk-global-push-forward-map}
 The  family of linear bijections induced by  the flow $\Phi^h_t$ and called the \emph{global push forward map}
 is denoted by $\phi^h_t( \eta_h ) \colon \Omega_h(t) \to \R$ for $\eta_h \colon \Omega_{h,0} \to \R$ and defined  by
\begin{equation*}
  ( \phi^h_t \eta_h )|_{K(t)} := \phi^K_t( \eta_h|_{K_0} ) \qquad \mbox{ for all } K(t) \in \T_h(t).
\end{equation*}
\end{deflist}
\end{definition}

In order to bring together a collection of elements, we again restrict to Lagrangian finite elements over a polygonal reference finite element.
We note that our construction implies that for each element $K(t)$ each node $a(t) \in \N^{K(t)}$ is the trajectory of a point $a_0 \in \N^{K(0)}$ under the flow $\Phi_t^K$ - that is $a(t) = \Phi_t^K( a_0 )$.

We make the following extra requirement on how adjacent elements will be related.
For each $t \in [0,T]$, we denote by $\N_h(t)$ the global set of Lagrange nodes of $\Omega_h(t)$ \cref{eq:bulk-Nh} and for any $a \in \N_h(t)$, $\T(a)$ is the set of elements $K(t)$ such that $a$ is a node of $K(t)$.
We make the restriction that the global flow is single-valued at each Lagrange point:
for all $a_0 \in \N_h(0)$ we have
\begin{equation}
  \label{eq:bulk-evolve-node-agree}
  \Phi_t^K( a_0 ) = \Phi_t^{K'}( a_0 ) \qquad \mbox{ for all } K, K' \in \T( a_0 ).
\end{equation}

\begin{remark}
  We note that our construction does not imply that the global flow map is indeed a function: Its restrictions to a facet from adjacent elements may not coincide.
  The assumption \cref{eq:bulk-evolve-node-agree} imposes that the global flow map should coincide at Lagrange points along element boundaries.
\end{remark}

\begin{definition}[Evolving bulk finite element space]
  \begin{deflist}
    \item \label{def:evolving-bfe-space}
  Let $\{ \Omega_h(t) \}_{t \in [0,T]}$ be an evolving triangulated bulk domain (\cref{def:evolving-triangulated-domain}) equipped with an evolving conforming subdivision $\{ \T_h(t) \}_{t \in [0,T]}$ (\cref{def:bulk-evolving-conforming-subdivision}).
  For $t \in [0,T]$, let $\S_h(t)$ be a bulk finite element space (\cref{def:bfe-space}) over $\Omega_h(t)$.
  If each $K(t) \in \T_h(t)$ is equipped with an evolving bulk finite element $(K(t), \P(t), \Sigma(t) )_{t \in [0,T]}$ (\cref{def:ebfe}) which together satisfy \cref{eq:bulk-evolve-node-agree} then we say $\{ \S_h(t) \}_{t \in [0,T]}$ is an \emph{evolving bulk finite element space}.
\item \label{def:bulk-evolving-basis-function}
  For each $t \in [0,T]$, we will write $\Sigma_h(t)$ for the set of global nodal variables (\cref{def:bulk-global-degrees-of-freedom}). We will use the convention that
  \begin{equation*}
    \Sigma_h(t) = \{ \chi_h \mapsto \chi_h( a_i(t) ) : 1 \le i \le N \},
  \end{equation*}
  where $a_i(t)$ is the trajectory of a Lagrange point under the global flow $\Phi^h_t$. We will denote by $\{ \chi_i(\cdot,t) : 1 \le i \le N \}$ the global basis of finite element functions such that $\chi_i ( a_j(t), t ) = \delta_{ij}$ for $t \in [0,T]$ and all $i,j = 1, \ldots, N$. This implies that $\chi_i( \cdot, t ) = \phi^h_t ( \chi_i( \cdot, 0) )$.
\end{deflist}
\end{definition}

\begin{definition}
  [Uniformly regular and uniformly quasi-uniform evolving subdivisions]
  For $h \in (0,h_0)$, let $\{ \T_h(t) \}_{t \in [0,T]}$ be a family of evolving conforming subdivisions (\cref{def:bulk-evolving-conforming-subdivision}).
  \begin{deflist}
  \item \label{def:bulk-uniformly-regular}
    We say that the family is \emph{uniformly regular} if there exists $\rho > 0$ such that for all $h \in (0, h_0)$ and all times $t \in [0,T]$, we have
    \begin{equation*}
      \rho_{K(t)} \ge \rho h_{K(t)} \qquad \mbox{ for all } K(t) \in \T_h(t),
    \end{equation*}
    and there exists $C > 0$ such that
    \[
      \sup_{h \in (0,h_0)} \sup_{t \in [0,T]} \max_{K(t) \in \T_h(t)} C_{K(t)} \le C < 1.
    \]
  \item \label{def:bulk-uniformly-Theta-regular}
    We say that the family is \emph{uniformly $\Theta$-regular} if it is uniformly regular, if for each time $t \in [0,T]$, the family $\{ \T_h(t) \}_{h \in (0,h_0)}$ is $\Theta$-regular and if there exists a constant such that
    \[
      \sup_{h \in (0,h_0)} \sup_{t \in (0,T)} \max_{K(t) \in \T_h(t)} C_m(K(t)) \le C \le + \infty \qquad
      \mbox{ for } 2 \le m \le \Theta+1.
    \]
  \item \label{def:bulk-uniformly-quasi-uniform}
    We say that the family is \emph{uniformly quasi-uniform} if there exists $\rho > 0$ such that for all $h \in (0,h_0)$ and all times $t \in [0,T]$, we have
    \begin{equation*}
      \min \{ \rho_{K(t)} : K(t) \in \T_h(t) \}
      \ge \rho h.
    \end{equation*}
  \end{deflist}
\end{definition}

Note that a uniformly quasi-uniform subdivision consists of element domains for temporally quasi-uniform evolving bulk finite elements.

\begin{lemma}
  \label{lem:bulk-Vh-compatible}
  For $h \in (0,h_0)$,
  let $\{ \T_h(t) \}_{t \in [0,T]}$ be a uniformly $\Theta$-regular (\cref{def:bulk-uniformly-Theta-regular}), uniformly quasi-uniform (\cref{def:bulk-uniformly-quasi-uniform}), evolving, conforming subdivision (\cref{def:bulk-evolving-conforming-subdivision}) and let $\phi^h_t$ be the global push-forward map (\cref{def:bulk-global-push-forward-map}).
  Let $0 \le k \le \Theta+1$, $p \in [0,\infty]$.
  Then, $\eta_h \in W^{m,p}( \T_h(0) )$ if and only if $\phi^h_t \eta_h \in W^{m,p}( \T_h(t) )$ for all $t \in [0,T]$.
  Furthermore, there exists $c_1, c_2 > 0$ independent of $h \in (0,h_0)$ and $t \in [0,T]$ such that for all $\eta_h^\ell \in \S_h(t)$
  \begin{align}
    c_1 \norm{ \eta_h }_{W^{m,p}(\T_h(0))}
    \le \norm{ \phi^h_t \eta_h }_{W^{m,p}( \T_h(t) )}
    \le c_2 \norm{ \eta_h }_{W^{m,p}(\T_h(0))}.
  \end{align}

\end{lemma}

\begin{proof}
  We simply sum the element-wise result from \cref{lem:bulk-Kt-norm-equiv}. The constants are independent of $h_K$ and $\rho_K$ due to the uniform quasi-uniformity of $\{ \T_h(t) \}$.
\end{proof}

\begin{remark}
  \label{rem:bulk-Sh-compatible}
  In particular,
  the pair $( W^{m,p}( \Omega_h(t) ), \phi^h_t|_{W^{m,p}(\Omega_{h,0})} )_{t \in [0,T]}$ is compatible with respect to the broken Sobolev norm $\norm{ \cdot }_{W^{m,p}(\T_h(t))}$ (\cref{def:bulk-broken-sobolev-norm}).
  Furthermore, the pairs $( \S_h(t), \phi^h_t|_{\S_{h,0}} )_{t \in [0,T]}$, equipped with the $W^{m,p}( \T_h(t) )$ norm, and $( W^{1,p}_T( \T_h(t) ), \phi^h_t |_{W^{1,p}_T( \T_{h,0})} )_{t \in [0,T]}$ are also both compatible (\cref{rem:compatible-subsace,lem:bulk-Sh-closed,lem:bulk-W1pT-closed}).
  Note that this result implies that the spaces $L^2_{\S_h}$ and $C^1_{\S_h}$ are well defined when equipped with the appropriate norms (c.f. \cref{eq:L2X} and \cref{eq:CkX}).
\end{remark}

\section{Lifted bulk finite element spaces}
\label{sec:bulk-lift}
This section sets out a  procedure for relating functions on the 
discrete bulk  domain $\Omega_h$ to the smooth domain $\Omega$ via the construction of $\S_h^\ell(t)$ from the space $\S_h(t)$ defined in the previous section.
We start by defining a lifted bulk finite element $(K^\ell, \P^\ell, \Sigma^\ell)$ in $\Omega$ using a mapping $\Lambda_K \colon K \to K^\ell \subset \Omega$.
This process is called \emph{lifting}.
In the following we will provide the appropriate assumptions on $\Lambda_K$ to allow us to relate structures on $\Omega_h$ to their lifted counterparts on $\Omega$.

\subsection{Lifted bulk finite element}

Consider  a single bulk finite element $(K, P , \Sigma)$ (\cref{def:bfe}), with element reference map $F_K \colon \hat{K} \to K$, where the element domain approximates a portion of a domain $\Omega$ with smooth boundary in $\R^{n+1}$.
Let $\Lambda_K \colon K \to \Omega$ be a $C^1$-map which is a diffeomorphism onto its image.
We define $K^\ell := \Lambda_K( K ) \subset \Omega$.
For a function $\chi \colon K \to \R$, we call $\chi^\ell \colon K^\ell \to \R$ the \emph{lift} of $\chi$ which is given by
\[
  \chi^\ell( \Lambda_K( x ) ) = \chi( x ) \qquad \mbox{ for } x \in K.
\]
We assume that we can decompose $\Lambda_K$ into
\[
  \Lambda_K( x ) = A_\Lambda x + b_\Lambda + \tilde{\Lambda}_K( x ) \qquad \mbox{ for } x \in K,
\]
where $A_{\Lambda}$ is an invertible $(n+1) \times (n+1)$ matrix, $b_\Lambda \in \R^{n+1}$ and $\tilde{\Lambda}_K \in C^1( K, \R^{n+1} )$.
We will assume that $\tilde{\Lambda}_K$ does not affect the affine part of the parametrisation: $\tilde{\Lambda}_K( a ) = 0$ for each vertex $a \in K$.

\begin{definition}[Lifted bulk finite element]
  \label{def:lifted-bfe}
We call the triple $(K^\ell, \P^\ell, \Sigma^\ell)$ defined by
\begin{align*}
  & K^\ell := \Lambda_K( K ) \subset \Omega \\
  & \P^\ell := \{ \chi^\ell( \Lambda_K( \cdot ) ) := \chi( \cdot ) : \chi \in \P \} \\
  & \Sigma^\ell := \{ \sigma^\ell := \chi^\ell \mapsto \sigma( \chi ) : \sigma \in \Sigma \},
\end{align*}
the \emph{lift} of $(K, \P, \Sigma)$ and $\Lambda_K$ the \emph{lifting map}.
If $(K^\ell, \P^\ell, \Sigma^\ell)$ forms a bulk finite element over $(\hat{K}, \hat{P}, \hat{\Sigma})$ then we say that $(K^\ell, \P^\ell, \Sigma^\ell)$ is the \emph{lifted bulk finite element associated with} $(K, \P, \Sigma)$.
In this case, we call $F_{K^\ell}( \cdot) := \Lambda_K (F_K( \cdot ) )$ the lifted bulk finite element reference map.
\end{definition}

The next two results show under what assumptions on $\Lambda_K$ is $(K^\ell, \P^\ell, \Sigma^\ell)$ a bulk finite element (\cref{def:bfe}) or a $\Theta$-bulk finite element (\cref{def:Theta-bfe}).

\begin{lemma}
  \label{lem:bulk-LambdaK-equiv-1}
  If $\Lambda_K$ satisfies that
  \begin{equation}
    \label{eq:bulk-LambdaK-ass1}
    \sup_{x \in K} \norm{ \nabla_K \tilde{\Lambda}(x) } \le
    \frac{ \norm{ A_{\Lambda}^{-1} }^{-1} - C_K \norm{ A_{\Lambda} } }{ 1 + C_K },
  \end{equation}
  then $(K^\ell, \P^\ell, \Sigma^\ell)$ is a bulk finite element.
  Furthermore,
  for $p \in [1,\infty]$, we have that there exists $c_1, c_2 > 0$ such that
  \begin{subequations}
    \label{eq:bulk-lift-equiv}
    \begin{align}
      \label{eq:bulk-lift-equiv-Lp}
      c_1 \norm{ \chi }_{L^p(K)}
      & \le \norm{ \chi^\ell }_{L^p(K^\ell)}
      \le c_2 \norm{ \chi }_{L^p(K)}
      && \mbox{ for all } \chi \in L^p(K) \\
      \label{eq:bulk-lift-equiv-W1p}
      c_1 \norm{ \chi }_{W^{1,p}(K)}
      & \le \norm{ \chi^\ell }_{W^{1,p}(K^\ell)}
      \le c_2 \norm{ \chi }_{W^{1,p}(K)}
      && \mbox{ for all } \chi \in W^{1,p}(K),
    \end{align}
  \end{subequations}
  where the constants $c_1, c_2$ depend on $C_K$, $\norm{ A_{\Lambda}^{-1} }$ and the ratio $h_K / \rho_K$.
\end{lemma}

\begin{proof}
  To show that $( K^\ell, P^\ell, \Sigma^\ell )$ is a bulk finite element (\cref{def:bfe}),
  the conditions on the element reference map $F_{K^\ell}(\hat{x}) = \Lambda_K( F_K(\hat{x}) )$ are clear and we are left to check the curvedness condition \cref{eq:bulk-CK-defn}.
  Using the expansion of $\Lambda_K$, we see
  \begin{multline*}
    A_{K^\ell} = A_{\Lambda} A_K, \quad A_{K^\ell}^{-1} = A_K^{-1} A_{\Lambda}^{-1},
    \quad b_{K^\ell} = A_{\Lambda} b_K + b_{\Lambda} \\
    \quad \mbox{ and } \quad
    D_{K^\ell}(\hat{x}) = A_\Lambda D_K(\hat{x}) + \tilde{\Lambda}_K( F_K(\hat{x}) ).
  \end{multline*}
  So that
  \begin{align*}
    \nabla D_{K^\ell}(\hat{x}) A_{K^\ell}^{-1}
    = A_{\Lambda} \nabla D_K( \hat{x} ) A_K^{-1} A_{\Lambda}^{-1}
    + \nabla_K \tilde{\Lambda}_K( F_K(\hat{x}) ) ( \id + \nabla D_K(\hat{x}) A_K^{-1} ) A_\Lambda^{-1}.
  \end{align*}
  Applying the curvedness condition for $K$ we see
  \[
    \norm{\nabla D_{K^\ell}(\hat{x}) A_{K^\ell}^{-1}}
    \le C_K \norm{ A_{\Lambda} } \norm{ A_{\Lambda}^{-1} }
    + (1 + C_K ) \norm{ \nabla_K \tilde{\Lambda}_K( F_K(\hat{x})) } \norm{ A_{\Lambda}^{-1} }.
  \]
  The curvedness condition is shown by applying \cref{eq:bulk-LambdaK-ass1}.

  To show \cref{eq:bulk-lift-equiv-Lp} the result is clear for $p = \infty$.
  For $p < \infty$, we will apply \cref{lem:norm-scaling}. Then we see
  \begin{align*}
    c \left(\frac{\meas(K)}{\meas(K^\ell)}\right)^{1/p} \norm{ \chi }_{L^p(K)}
    \le \norm{ \chi^\ell }_{L^p(K^\ell)}
    \le c \left(\frac{\meas(K^\ell)}{\meas(K)}\right)^{1/p} \norm{ \chi }_{L^p(K)},
  \end{align*}
  where $c$ depends on $C_K$ and the bound \cref{eq:bulk-LambdaK-ass1}.
  For the $W^{1,p}$ bound \cref{eq:bulk-lift-equiv-W1p}, we note that
  \[
    \norm{ A_K } = \norm{ A_{\Lambda}^{-1} A_{\Lambda} A_K }
    \le \norm{ A_{\Lambda}^{-1} } \norm{ A_{\Lambda} A_K } = \norm{ A_{\Lambda}^{-1} } \norm{ A_{K^\ell} },
    \]
  so that we infer
  \begin{equation}
    \label{eq:bulk-AK-AKell-bound}
    \norm{ A_{K^\ell} } \ge \frac{ \norm{ A_K } }{ \norm{ A_{\Lambda^{-1}} } }.
  \end{equation}
  Then applying \cref{lem:bulk-norm-scaling} once more we see
  \begin{multline*}
    c \left(\frac{\meas(K)}{\meas(K^\ell)}\right)^{1/p}
    \frac{ \norm{ A_\Lambda^{-1} } }{ \norm{ A_K } \norm{ A_K^{-1} } }
    \abs{ \chi }_{W^{1,p}(K)}
    \le \norm{ \chi^\ell }_{W^{1,p}(K^\ell)} \\
    \le c \left(\frac{\meas(K^\ell)}{\meas(K)}\right)^{1/p}
    \norm{ A_\Lambda^{-1} } \norm{ A_K } \norm{ A_K^{-1} }
    \abs{ \chi }_{W^{1,p}(K)}.
  \end{multline*}
  The final result is given by applying \cref{eq:bulk-AK-h-relation-a,eq:bulk-AK-h-relation-b}.
\end{proof}

\begin{lemma}
  \label{lem:bulk-LambdaK-equiv-theta}
  If $(K, \P, \Sigma)$ is a $\Theta$-bulk finite element and $\Lambda_K$ satisfies $\Lambda_K \in C^{\Theta+1}( K; \R^{n} )$ and $\norm{ \Lambda_K }_{W^{\Theta+1,\infty}(K)} \le c$,
  then $(K^\ell, \P^\ell, \Sigma^\ell)$ is also a $\Theta$-bulk finite element.
  Furthermore, for $0 \le m \le \Theta+1$ and $p \in [1,\infty]$, we have that there exists $c_1, c_2 > 0$ such that
  \begin{align*}
    c_1 \norm{ \chi }_{W^{m,p}(K)}
    & \le \norm{ \chi^\ell }_{W^{m,p}(K^\ell)}
      \le c_2 \norm{ \chi }_{W^{m,p}(K)}
    && \mbox{ for all } \chi \in W^{m,p}(K),
  \end{align*}
  where the constants $c_1, c_2$ depend on $C_1(K), \ldots, C_m(K)$, $\norm{ A_{\Lambda}^{-1} }$ and the ratio $h_K / \rho_K$.
\end{lemma}

\begin{proof}
  To see that $(K^\ell, \P^\ell, \Sigma^\ell )$ is a $\Theta$-bulk finite element (\cref{def:Theta-bfe}),
  the first three conditions are clear from the smoothness assumption on $\Lambda_K$.
  We must show \cref{eq:bulk-CK-theta-defn} holds: for $2 \le m \le \Theta+1$, there exists a constant $C_m(K^\ell)$ such that
  \[
    \norm{ \nabla^m F_{K^\ell}(\hat{x}) } \le C_m(K^\ell) \norm{ A_{K^\ell} }^m.
  \]
  Computing directly using \citep[Eq. (2.9)]{Ber89}, we have
  \begin{align*}
    \nabla^m F_{K^\ell}(\hat{x})
    = \sum_{r=1}^{m} (\nabla^r \Lambda_K)( F_K( \hat{x} ) ) \left( \sum_{\vec{i} \in E(m,r)} c_{\vec{i}} \prod_{q=1}^m ( \nabla^q F_K(\hat{x}) )^{i_q} \right)
  \end{align*}
  where $c_{\vec{i}}$ are constants and
  \[
    E( m, r ) = \Bigl\{ \vec{i} \in \Nbb^m ; \sum_{q=1}^m i_q = r \mbox{ and } \sum_{q=1}^m q i_q = m \Bigr\}.
  \]
  But applying the fact that $(K, \P, \Sigma)$ is a $\Theta$-bulk finite element, the definition of $E(m,r)$, \cref{lem:bulk-norm-scaling} and the smoothness assumption on $\Lambda_K$, we see that
  \begin{align*}
    \abs{ \nabla^{m} F_{K^\ell}( \hat{x} ) }
    \le c( C_1(K), \ldots, C_m(K) ) \sum_{r=1}^{m} \abs{ \Lambda_K }_{W^{r,\infty}(K)} \norm{ A_K }^{m}.
  \end{align*}
  Finally, applying \cref{eq:bulk-AK-AKell-bound}, we have that
  \begin{align*}
    \norm{ \nabla^{m} F_{K^\ell}( \hat{x} ) }
    \le c \norm{ \Lambda_K }_{W^{m,\infty}(K)} \norm{ A_{K^\ell} }^{m} \norm{ A_{\Lambda}^{-1} }^{m}.
  \end{align*}
  This shows that $(K^\ell, \P^\ell, \Sigma^\ell)$ is a $\Theta$-surface finite element.

  To show the norm equivalence we again apply \cref{lem:bulk-norm-scaling}, recognising the geometric progression, to see
  \begin{multline*}
    c \left(
    \frac{ \meas(K) }{ \meas(K^\ell) }
    \right)
    \frac{ A_1^{m+1} - 1}{ A_1 - 1 }
    \norm{ \chi }_{W^{m,p}(K)}
    \le
    \norm{ \chi^\ell }_{W^{m,p}(K^\ell)} \\
    \le
    c \left(
    \frac{ \meas(K^\ell) }{ \meas(K) }
    \right)
    \frac{ A_2^{m+1} - 1}{ A_2 - 1 }
    \norm{ \chi }_{W^{m,p}(K)},
  \end{multline*}
  where
  \begin{equation*}
    A_1 = \norm{ A_K } \norm{ A_K^\dagger } \norm{ A_\Lambda } \quad \mbox{ and } \quad
    A_2 = \norm{ A_K } \norm{ A_K^\dagger } \norm{ A_\Lambda^{-1} }.
  \end{equation*}
  The final result is given by applying \cref{eq:bulk-AK-h-relation-a,eq:bulk-AK-h-relation-b,eq:bulk-lift-meas-scale}.
\end{proof}

\begin{remark}
  \label{rem:bulk-id-lift-ass}
  Considering the case that $\Lambda_K$ fixes the vertices of $K$ then we can write $\Lambda_K$ as
  \[
    \Lambda_K( x ) = x + \tilde{\Lambda}(x) \quad \mbox{ for } x \in K.
  \]
  That is that $A_\Lambda = \id$ and $b_{\Lambda} = 0$.
  Then the assumptions of \cref{lem:bulk-LambdaK-equiv-1} can be replaced by
  \[
    \sup_{x \in K} \norm{ \nabla_K \tilde\Lambda(x) }
    \le \frac{1 - C_K}{1 + C_K},
  \]
  and the assumptions of \cref{lem:bulk-LambdaK-equiv-theta} can be replaced by the assumption that $\norm{ \Lambda_K }_{W^{\Theta,\infty}(K)}$ is bounded.
\end{remark}

We can also use $\Lambda_K$ to define an inverse lift.
Given an element domain $K$, lifted element domain $K^\ell$ and lifting map $\Lambda_K$.
We know that $\Lambda_K$ is invertible onto its image, namely $K^\ell$.
So for $\eta \colon K^\ell \to \R$, we denote its inverse lift by $\eta^{-\ell} \colon K \to \R$ defined by
\[
  \eta^{-\ell}( x ) := \eta( \Lambda_K( x ) ) \quad \mbox{ for } x \in K.
\]
\begin{lemma}
  If the assumptions of \cref{lem:bulk-LambdaK-equiv-1,lem:bulk-LambdaK-equiv-theta} hold then,
  for $0 \le m \le \Theta$ and $p \in [1,\infty]$, we have that there exists $c_1, c_2 > 0$ such that for all $\eta \in C(K^\ell) \cap W^{m,p}(K^\ell)$, we have
  \begin{align*}
    c_1 \norm{ \eta^{-\ell} }_{W^{m,p}(K)}
    & \le \norm{ \eta }_{W^{m,p}(K^\ell)}
      \le c_2 \norm{ \eta^{-\ell} }_{W^{m,p}(K)}.
  \end{align*}
\end{lemma}

\begin{proof}
  The same proof can be applied to the case of the inverse lift as well as the lift.
\end{proof}

We next relate the geometry of the base and lifted element domains.
\begin{lemma}
  \label{lem:bulk-lift-geom}
  Using the decomposition of $\Lambda_K$, we have
  \begin{subequations}
    \begin{align}
      \label{eq:bulk-lift-h-scale}
      \norm{ A_{\Lambda}^{-1} }^{-1} h_K \le h_{K^\ell}
      & \le \norm{ A_\Lambda } h_K \\
      \label{eq:bulk-lift-rho-scale}
      \norm{ A_{\Lambda}^{-1} }^{-1} \rho_K \le \rho_{K^\ell}
      & \le \norm{ A_{\Lambda} } \rho_K \\
      \label{eq:bulk-lift-meas-scale}
      c_1 \left( \frac{\rho_K}{h_K} \right)^{n+1} \norm{ A_{\Lambda}^{-1} }^{-{n+1}} \meas(K)
      \le \meas(K^\ell)
      & \le c_2 \left( \frac{h_K}{\rho_K } \right)^{n+1} \norm{ A_\Lambda^{-1} }^{-{n+1}} \meas(K).
    \end{align}
  \end{subequations}
\end{lemma}

\begin{proof}
  For \cref{eq:bulk-lift-h-scale}, we show the second inequality. The first follows by the same reasoning applied to the inverse of $A_{\Lambda}$.
  Since $\tilde{K}^\ell$ is compact, there exists $x^\ell, y^\ell \in \tilde{K}^\ell$ such that
  \[
    h_{K^\ell} = \abs{ y^\ell - x^\ell }.
  \]
  But since $A_\Lambda$ is invertible, there exists $x, y \in \tilde{K}$ such that
  \[
    A_{\Lambda} x + b_\Lambda = x^\ell \qquad \mbox{ and } \qquad
    A_{\Lambda} y + b_{\Lambda} = y^\ell.
  \]
  Then, we can compute that
  \[
    h_{K^\ell} = \abs{ x^\ell - y^\ell }
    = \abs{ A_\Lambda ( x - y ) }
    \le \norm{ A_{\Lambda} } \abs{ x - y }
    \le \norm{ A_{\Lambda} } h_K.
  \]

   Similarly, for \cref{eq:bulk-lift-rho-scale}, we only show the first inequality.
  For each $\eps > 0$, there exists $B_\eps$ a ball in $\tilde{K}$ such that
  \[
    \rho_K - \diam B_\eps \le \eps.
  \]
  Denote by $x_\eps$ and $r_\eps$ the centre and radius of $B_\eps$ respectively. Consider the affine map $\Xi_\eps \colon \R^{n+1} \to \R^{n+1}$ given by
  \[
    \Xi_\eps(x) = \frac{1}{\norm{ A_{\Lambda}^{-1} } } ( x - x_\eps ) + b_\Lambda + A_\Lambda x_\eps.
  \]
  It is clear that $\Xi_\eps$ maps balls to balls and $B_\eps$ is mapped to a ball centred at $A_{\Lambda} x_\eps + b_\Lambda$ with radius $r_\eps / \norm{ A_{\Lambda}^{-1} }$.
  We claim $\Xi_\eps( B_\eps )$ is contained in $\tilde{K}^\ell$.
  Indeed, take $x^\ell \in \Xi_\eps( B_\eps )$, then there exists $x \in B_\eps$ such that $\Xi_\eps( x ) = x^\ell$. Denote by
  \[
    x' = x_\eps + \frac{1}{\norm{ A_{\Lambda}^{-1} } } A_\Lambda^{-1} ( x - x_\eps ),
  \]
  then
  \[
    \abs{ x' - x_\eps }
    = \frac{1}{\norm{ A_{\Lambda}^{-1} } } \abs{ A_\Lambda^{-1} ( x - x_\eps ) }
    \le \abs{ x - x_\eps } \le r_\eps,
  \]
  so $x' \in B_\eps \subset \tilde{K}$ and
  \[
    A_\Lambda x' + b_\Lambda
    = \frac{1}{\norm{ A_{\Lambda}^{-1} } } ( x - x_\eps ) + A_\Lambda x_\eps + b_\Lambda
    = x^\ell,
  \]
  so that $x^\ell \in \tilde{K}^\ell = A_\Lambda \tilde{K} + b_K$.

  Therefore, we have found a ball $\Xi_\eps( B_\eps ) \subset \tilde{K}^\ell$ with radius $r_\eps / ( \norm{ A_{\Lambda}^{-1} } )$, thus we can infer that
  \[
    \rho_{K^\ell} \ge \frac{1}{2} \norm{ A_{\Lambda}^{-1} }^{-1} \diam B_\eps
    \ge \norm{ A_{\Lambda}^{-1} }^{-1} ( \rho_K - \eps ).
  \]
  Since the proof holds for all $\eps > 0$, we see the desired result.

  Given \cref{eq:bulk-lift-h-scale} and \cref{eq:bulk-lift-rho-scale}, the final result \cref{eq:bulk-lift-meas-scale} follows directly from \cref{rem:bulk-AK-h-meas}.
\end{proof}

\subsection{Lifted bulk triangulations and spaces}

Let $\Omega$ be a smooth bulk domain and for $h \in (0,h_0)$ and let  $\Omega_h$ be a triangulated bulk domain (\cref{def:triangulated-domain}) equipped with a conforming subdivision $\T_h$ (\cref{def:bulk-conforming-subdivision}) and a bulk finite element space $\S_h$ (\cref{def:bfe-space}). Let each $K \in \T_h$ be associated with a lifted finite element $K^{\ell}$ with lifted map $\Lambda_K$.

\begin{definition}
\begin{deflist}
\item
  \label{def:bulk-exact-decomposition}
 We denote by $\T_h^\ell$ the set of all lifted element domains
\[
  \T_h^\ell := \{ K^\ell : K \in \T_h \}.
\]
If  the global map $\Lambda_h$ is single valued and $\T_h^\ell$ forms a conforming subdivision of $\bar\Omega$, we say that $\T_h^\ell$ is an \emph{exact} subdivision of the domain $\Omega$.
\item
  \label{def:bulk-global-lifting-map}
We  define a \emph{global lifting map} $\Lambda_h \colon \Omega_h \to \bar\Omega$ by $\Lambda_h|_{K} = \Lambda_K$.
We define  the inverse lift $\Lambda_h^{-1} \colon \bar\Omega \to \Omega_h$ in a similar element-wise fashion 
 by  $\Lambda^{-1}_h|_{K^\ell} = \Lambda_K^{-1}$.
\item
  We denote by $\lambda_h( \eta_h ) \colon \bar\Omega \to \R$, for $\eta_h \colon \Omega_h \to \R$, the \emph{global lift} given by
  \[
    \lambda_h( \eta_h )( \Lambda_h( x  ) ) =\eta_h( x ) \mbox{ for } x \in \Omega_h,
  \]
  and by $\varsigma_h( \eta ) \colon \Omega_h \to \R$, for $\eta \colon \bar\Omega \to \R$,, the \emph{global inverse lift} given by
  \[
    \varsigma_h( \eta )( x ) = \eta( \Lambda_h( x ) ) \mbox{ for } x \in \Omega_h.
  \]
  We will also use the notations $\eta_h^\ell = \lambda_h( \eta_h )$ and $\eta^{-\ell} = \varsigma_h( \eta )$.
\item
  \label{def:lifted-bfe-space}
  Let $\S_h$ be a bulk finite element space. If for each bulk finite element $( K, P, \Sigma )$ there is an associated lifted bulk finite element $(K^\ell, P^\ell, \Sigma^\ell)$, then, we define a \emph{lifted bulk finite element space} (c.f.\ \cref{eq:abs-Vhell}) by
\[
  \S_h^\ell := \left\{ \chi_h^\ell : \chi_h \in \S_h \right\}.
\]
\end{deflist}
\end{definition}

\begin{proposition}
  \label{prop:bulk-lift-Vh}
  Assume additionally that the family of triangulations $\{ \T_h \}_{h \in (0,h_0)}$ is regular (\cref{def:bulk-regular}),
  $\Lambda_K$ satisfies the assumptions of \cref{lem:bulk-LambdaK-equiv-1} and there exists $C_1, C_2 > 0$ such that
  \begin{equation}
    \label{eq:bulk-as-lift-linear-bounded}
    \norm{ A_\Lambda } \le C_1 \quad \mbox{ and } \norm{ A_{\Lambda}^{-1} } \le C_2
    \quad \mbox{ for all } K \in \T_h \mbox{ for all } h \in (0,h_0),
  \end{equation}
  for all $K \in T_h$ for all $h \in (0,h_0)$.
  Then $\{ \T_h^\ell \}_{h \in (0,h_0)}$ is regular and $\S_h^\ell$ is a bulk finite element space and there exists constants $c_1, c_2 > 0$, which are independent of $h \in (0,h_0)$ such that
  \begin{equation}
    \begin{aligned}
      c_1 \norm{ \eta_h }_{L^p(\T_h)} & \le
      \norm{ \eta_h^\ell }_{L^p(\T_h^\ell)}
      \le c_2 \norm{ \eta_h }_{L^p(\T_h)}
      && \mbox{ for all } \eta_h \in L^p( \T_h ) \\
      c_1 \norm{ \eta_h }_{W^{1,p}(\T_h)} & \le
      \norm{ \eta_h^\ell }_{W^{1,p}(\T_h^\ell)}
      \le c_2 \norm{ \eta_h }_{W^{1,p}(\T_h)}
      && \mbox{ for all } \eta_h \in W^{1,p}( \T_h ).
    \end{aligned}
\end{equation}
  Furthermore, if $\Lambda_K$ satisfies the assumptions of \cref{lem:bulk-LambdaK-equiv-theta} for all $K \in \T_h$ and all $h \in (0,h_0)$,
  then for $0 \le m \le \Theta$, there exists $c_1, c_2 > 0$ independent of $h \in (0,h_0)$ such that
  \begin{equation}
    c_1 \norm{ \eta_h }_{W^{m,p}(\T_h)} \le
    \norm{ \eta_h^\ell }_{W^{m,p}(\T_h^\ell)}
    \le c_2 \norm{ \eta_h }_{W^{m,p}(\T_h)}
    \quad \mbox{ for all } \eta_h \in W^{m,p}( \T_h ).
  \end{equation}
\end{proposition}

\begin{proof}
  The regularity of the family of triangulations $\{ \T_h \}_{h \in (0,h_0)}$ follows from \cref{eq:bulk-lift-h-scale} and \cref{eq:bulk-lift-rho-scale} from assumption \cref{eq:bulk-as-lift-linear-bounded}.

  The results follow by combining the previous results for each element in $\T_h$. We achieve bounds independently of $h \in (0,h_0)$ since the regularity of the subdivisions implies that $h_K / \rho_K$ is bounded independently of $h \in (0,h_0)$.
\end{proof}

We next define interpolation estimate which interpolates smooth functions over the continuous surface into the lifted surface finite element space.
We denote by $I_h$ an interpolation operator $I_h \colon C(\bar\Omega) \to \S_h^\ell$ defined by
\begin{equation}
  \label{eq:bulk-lift-Ih}
  I_h \eta |_{K^\ell} := I_{K^\ell} \eta.
\end{equation}

\begin{theorem}
  [Global lifted interpolation theorem]
  \label{thm:bulk-global-lift-interp}
  For $h \in (0,h_0)$, let $\{ \T_h \}_{h \in (0,h_0)}$ be a $\Theta$-regular (\cref{def:bulk-Theta-regular}), quasi-uniform (\cref{def:bulk-quasi-uniform}) family of subdivisions of triangulated bulk domains $\Omega_h$ equipped with a bulk finite element space $\S_h$ (\cref{def:bfe-space}) consisting of $\Theta$-bulk finite elements (\cref{def:Theta-bfe}) over a reference element which satisfies \cref{lem:bramble-hilbert} for some $0 < k,m \le \Theta$, $p,q \in [1,\infty]$.
  Let each element $K \in \T_h$ be equipped with a lifting map $\Lambda_K \in C^{\Theta+1}(K)$ such that $\norm{ \Lambda_K }_{W^{\Theta+1,\infty}(K)} \le C$ and \cref{eq:bulk-as-lift-linear-bounded} and \cref{eq:bulk-LambdaK-ass1} hold each uniformly for $h \in (0,h_0)$.
  Then $\{ \T_h^\ell \}_{h \in (0,h_0)}$ is a $\Theta$-regular, quasi-uniform family of subdivisions of $\Omega$ and $\S_h^\ell$ is a bulk finite element space consisting of $\Theta$-bulk finite elements.
  Let $\eta \in C(\bar\Omega)$ be a continuous function, then $I_h \eta \in \S_h^\ell$ is well defined.
  Furthermore, if the assumptions of \cref{thm:bulk-IK-bound} hold for the reference element $(\hat{K}, \hat{P}, \hat{\Sigma})$, there exists a constant $C = C(\hat{K}, \hat{P}, \hat{\Sigma}, \rho)$ such that for all functions $\eta \in W^{k+1,p}(\T_h^\ell) \cap C( \bar\Omega )$,
  \begin{equation}
    \label{eq:bulk-global-lift-interp}
    \norm{ \eta - I_h \eta }_{W^{m,q}(\T_h^\ell)} \le C h^{k+1-m} \norm{ \eta }_{W^{k+1,p}(\T_h^\ell)}.
  \end{equation}
\end{theorem}

\begin{proof}
  The result follows by combining previous lemmas in the appropriate way.
  We see the lifted triangulation is quasi-uniform by applying the results in \cref{lem:bulk-lift-geom} and the assumption \cref{eq:bulk-as-lift-linear-bounded}.
  The interpolation result then follows by applying \cref{thm:bulk-Ih-bound}.
\end{proof}

\begin{corollary}
  \label{cor:bulk-global-inv-lift-interp}
  Let $\eta \in C(\bar\Omega)$, then the interpolant of $\eta$ into $\S_h$, denoted by $\tilde{I}_h \eta$, given by
  \begin{equation}
    \tilde{I}_h \eta := ( I_h \eta )^{-\ell}.
  \end{equation}
  Furthermore, there exists a constant $C$ such that for all functions $\eta \in W^{k+1,p}(\T_h^\ell) \cap C( \bar\Omega )$, and all $h \in (0,h_0)$, we have
  \begin{equation}
    \label{eq:bulk-global-inv-lift-interp}
    \norm{ \eta^{-\ell} - \tilde{I}_h \eta }_{W^{m,q}(\T_h)} \le C h^{k+1-m} \norm{ \eta }_{W^{k+1,p}(\T_h^\ell)}.
  \end{equation}
\end{corollary}

\begin{proof}
  We apply the inverse lift result to the estimates in the theorem.
\end{proof}

\subsection{Evolving lifted bulk triangulations}

For $t \in [0,T]$, let $\Omega(t)$ be a smoothly evolving bulk domain with flow map $\Phi_t$ defined for the closure $\bar\Omega(t)$:
i.e.\ $\Phi_{(\cdot)} \in C^2( [0,T], C^1( \bar\Omega(0) ) )$, $\Phi_t( \Omega(0) ) = \Omega(t)$ and $\Phi_t( \partial \Omega(0) ) = \partial \Omega(t)$.
For $h \in (0,h_0)$, let $\Omega_h(t)$ be an evolving triangulated bulk domain (\cref{def:evolving-triangulated-domain}) with global discrete flow  $\Phi^h_t$ (\cref{def:bulk-global-discrete-flow}) and equipped with an evolving conforming subdivision $\T_h(t)$ (\cref{def:bulk-evolving-conforming-subdivision}) and an evolving bulk finite element space $\S_h(t)$ (\cref{def:evolving-bfe-space}).
We assume that we are given a global lifting map $\Lambda_h( \cdot, t )$ (\cref{def:bulk-global-lifting-map}) which gives an exact subdivision $\T_h^\ell(t)$ of $\Omega(t)$ (\cref{def:bulk-exact-decomposition}).

~
\begin{definition}[Lifted discrete flow map, material velocity and pushed forward map]
~
\begin{deflist}
\item \label{def:bulk-lifted-flow-map}
 The \emph{lifted flow map} $\Phi^\ell_{(\cdot)}(\cdot) \colon [0,T] \times \bar\Omega_0 \to \R^{n+1}$ of the smooth bulk domain is defined by
\begin{equation*}
  \Phi^\ell_t( \Lambda_h( x, t ) ) = \Lambda_h ( \Phi^h_t ( x ), t ) \qquad \mbox{ for } x \in \bar\Omega_0.
\end{equation*}
We note that $\Phi^\ell_t \colon \bar\Omega_0 \to \bar\Omega(t)$.

\item \label{def:bulk-lifted-discrete-material-velocity}
The \emph{lifted discrete material velocity}   ${w}_h$ on $\{ \bar\Omega(t) \}_{t \in [0,T]}$ is defined by
\begin{equation*}
  \dt \Phi^\ell_t( \cdot ) = {w}_h( \Phi^\ell_t( \cdot ), t ).
\end{equation*}
\item \label{def:bulk-lifted-push-forward-map}
The family of \emph{lifted push forward maps} denoted by $\phi^\ell_t( \eta ) \colon \Omega(t) \to \R$, for $\eta \colon \Omega_0 \to \R$, indexed by $t \in [0,T]$, are the  linear bijections defined by
\begin{equation*}
  \phi^\ell_t( \eta )( x ) = \eta( \Phi^\ell_{-t}( x ) ) \qquad \mbox{ for } x \in \bar\Omega(t).
\end{equation*}

\end{deflist}
\end{definition}
\begin{remark}
  Note that in general   $\Phi^\ell_t$ is different to $\Phi_t$, but each describes a different parametrisation of the same evolving domain.
  Also ${w}$, the material velocity of $\bar\Omega(t)$, and ${w}_h$ define the same domain $\bar\Omega(t)$ evolving from $\bar\Omega(0)$, so have the same normal components on the boundary $\partial \Omega(t)$.
\end{remark}

\begin{proposition}
  \label{prop:bulk-lift-Vht}
  For $h \in (0,h_0)$,
  let $\{ \S_h(t) \}_{t \in [0,T]}$ be an evolving bulk finite element space over a $\Theta$-regular (\cref{def:bulk-Theta-regular}), uniformly quasi-uniform (\cref{def:bulk-uniformly-quasi-uniform}), evolving conforming subdivision $\{ \T_h(t) \}_{t \in [0,T]}$ consisting of $\Theta$-evolving bulk finite elements (\cref{def:e-Theta-bfe}) and ${\phi}^h_t$ the global push-forward map (\cref{def:bulk-global-push-forward-map}).
  For each $t \in [0,T]$ and each $h \in (0,h_0)$, let each element $K(t) \in \T_h(t)$ be equipped with a lifting map $\Lambda_K( \cdot, t ) \in C^{\Theta+1}(K(t))$ such that $\norm{ \Lambda_K( \cdot, t )}_{W^{\Theta+1,\infty}(K)} \le C$ and \cref{eq:bulk-as-lift-linear-bounded} and \cref{eq:bulk-LambdaK-ass1} hold each uniformly for $h \in (0,h_0)$ and $t \in [0,T]$.
  Then $\{ \S_h^\ell(t) \}_{t \in [0,T]}$ is also an evolving bulk finite element space over a $\Theta$-regular, uniformly quasi-uniform, evolving conforming subdivision $\{ \T_h^\ell(t) \}_{t \in [0,T]}$ consisting of $\Theta$-evolving bulk finite elements.
  Furthermore, for $0 \le k \le \Theta+1$, $p \in [0,\infty]$, there exists $c_1, c_2 > 0$ independent of $h \in (0,h_0)$ and $t \in [0,T]$ such that
  \begin{align}
    \label{eq:bulk-lift-compat}
    c_1 \norm{ \eta }_{W^{k,p}(\T_h^\ell(0))}
    & \le \norm{ {\phi}^\ell_t \eta }_{W^{k,p}(\T_h^\ell(t))}
    \le c_2 \norm{ \eta }_{W^{k,p}(\T_h^\ell(0))}
    && \mbox{ for all } \eta \in W^{k,p}( \T_h^\ell(0) ).
  \end{align}
  In particular,
  the pair $( W^{k,p}( \T_h^\ell(t) ), {\phi}^\ell_t )_{t \in [0,T]}$ is compatible (\cref{def:compatibility}).
  Furthermore, $( \S_h(t), \phi^\ell_t|_{\S_{h,0}} )_{t \in [0,T]}$, equipped with the norm $\norm{ \cdot }_{W^{k,p}(\T_h^\ell(t))}$, and $( W^{1,p}( \Omega(t) ), \phi^\ell_t |_{W^{1,p}( \Omega_0 )} )_{t \in [0,T]}$ are also compatible pairs (\cref{rem:compatible-subsace,lem:bulk-Sh-closed,lem:bulk-W1pT-closed}).
\end{proposition}

\begin{proof}
  The properties of $\S_h^\ell(t)$ follow for the same reasoning as \cref{prop:bulk-lift-Vh} since the constants are bounded uniformly in time.
The bounds \cref{eq:bulk-lift-compat} and the compatibility of $~~~$
$( W^{k,p}( \T_h^\ell(t) ), \phi_h^\ell )_{t \in [0,T]}$ follow since the assumptions imply that $( W^{k,p}( \T_h(t) ), \phi^h_t )_{t \in [0,T]}$ are compatible and we can simply lift this result with uniform bounds.
\end{proof}

\section{Evolving surface finite element spaces}
\label{sec:sfem}

In this section, we will give precise definitions concerning the evolving surface finite element spaces we use.
The ideas follow in a similar manner to \cref{sec:bfem}.
The upshot will be a notion of a discrete surface $\Gamma_h(t)$ consisting of a union of elements and a finite element space $\S_h(t)$ defined on this discrete domain.
Our extensions from  standard bulk finite element theory, presented in \cref{sec:bfem}, to surface finite elements build on the work of \citet{Ned76}, \citet{Dzi88} and \citet{Hei05} for surfaces.
The subscript $h$  parametrises the constructions and will be related to the size of elements used in our computational domain.
Implicitly it is assumed that these structures exist for all $h \in (0,h_0)$ for some fixed value of $h_0$.
See also \cref{rem:small-hk-surf}.

\subsection{Surface finite elements}
\label{sec:stationary-sfem}

We next define a surface finite element which takes its inspiration from the notion of curved finite elements studied by \citet{CiaRav72} and \citet{Ber89}.
The key idea here is that a surface finite element is an $n$-dimensional parametrised surface with boundary embedded in $\R^{n+1}$.
In the following for a matrix $A$, we use  $A^\dagger$ to denote  the pseudo-inverse. For any matrix $A$ of full column rank the pseudo-inverse is given by
\begin{equation*}
  A^\dagger = ( A^t A )^{-1} A^t.
\end{equation*}
In this case $A^\dagger$ is the left inverse of $A$: $A^\dagger A = \id$.

\begin{definition}
  [Surface finite element and surface element reference map]
  \label{def:sfe}
  Let $(\hat{K}, \hat{\P}, \hat{\Sigma} )$ be a reference finite element with $\hat{K} \subset \R^n$.
  \begin{deflist}
    \item Let $F_K \colon \hat{K} \to \R^{n+1}$ satisfy
  \begin{enumerate}
    \item
    \begin{enumerate}
    \item $F_K \in C^1( \hat{K}, \R^{n+1} )$;
    \item $\rank \nabla F_K = n$;
    \item $F_K$ is a bijection onto its image;
    \end{enumerate}
  \item $F_K$ can be decomposed into an affine part, and smooth part
    \begin{equation*}
      F_K( \hat{x} ) = A_K \hat{x} + b_K + D_K( \hat{x} )
    \end{equation*}
    such that $A_K$ has full column rank,  $D_K \in C^1(\hat{K})$ and
    \begin{equation}
      \label{eq:CK-defn}
      C_K := \sup_{\hat{x} \in \hat{K}} \norm{ \nabla D_K( \hat{x} ) A_K^\dagger } < 1,
    \end{equation}
    where %
    $\norm{ \cdot }$ denotes the two-norm in this context. %
  \end{enumerate}
  In this case, we call $F_K$ a \emph{surface element reference map}.
  \item Let $F_K$ be a surface element reference map and $( K, \P, \Sigma )$ be the triple given by
  \begin{subequations}
    \label{eq:sfe}
    \begin{align}
      \label{eq:element-domain}
      K & := F_K( \hat{K} ) && \mbox{(the \emph{element domain})} \\
      \label{eq:shape-functions}
      \P & := \{ \hat{\chi} \circ F_K^{-1} : \hat{\chi} \in \hat{P} \} && \mbox{(the \emph{shape functions})}\\
      \label{eq:nodal-variables}
      \Sigma & := \{ \chi \mapsto \hat{\sigma}( \chi \circ F_K ) : \hat{\sigma} \in \hat{\Sigma} \} && \mbox{(the \emph{nodal variables})}.
    \end{align}
  \end{subequations}
  Under the above assumptions, we call $( K, \P, \Sigma )$ a \emph{surface finite element}, and $( \hat{K}, \hat\P, \hat\Sigma )$
  the associated \emph{reference finite element}.
\end{deflist}
\end{definition}

\begin{remark}
\begin{remlist}
\item
 We note that our assumptions imply that both $A_K$ and $\nabla F_K(\cdot)$ are full column rank.

\item
The first three assumptions in the definition of surface finite element imply that $K$ is a parametrised surface and the fourth \cref{eq:CK-defn} that $K$ is not too curved.
The final assumption allows the case that $\hat{K}$ is a flat simplical domain and $K$ is curved.
\end{remlist}
\end{remark}
\begin{remark}
  \begin{remlist}
  \item \label{rem:nuK}
  We denote by $\nu_K$ the unit normal vector field to $K$. It is the unique (up to sign) unit vector orthogonal to the $\hat{x}_i$ partial derivatives of $F_K$ for $i = 1, \ldots, n$ and is given by
  \begin{equation*}
    \nu_K :=
    \begin{cases}
      \frac{ \left( \frac{ \partial F_K }{ \partial \hat{x}_1 } \right)^\perp }%
      { \abs{ \left( \frac{ \partial F_K }{ \partial \hat{x}_1 } \right)^\perp } }
      & \quad \mbox{ for } n = 1 \\
      \frac{ \left( \frac{\partial F_K}{\partial \hat{x}_1} \wedge \ldots \wedge \frac{\partial F_K}{\partial \hat{x}_n} \right) }%
      { \abs{ \left( \frac{\partial F_K}{\partial \hat{x}_1} \wedge \ldots \wedge \frac{\partial F_K}{\partial \hat{x}_n} \right) } }
      & \quad \mbox{ for } n > 1.
    \end{cases}
  \end{equation*}
Here $\wedge$ denotes the wedge product.  The sign of the normal vector field is chosen by fixing a permutation of the barycentric coordinates $\hat{x}_1, \ldots \hat{x}_n$ of the reference element.
  By swapping any two elements, we reverse the sign of $\nu_K$.
  For a simplex reference element, the orientation can be fixed by ordering the labels of vertices so that $a_i = F_K(\hat{a}_i)$ where $\{ \hat{a}_i \}$ are the vertices of the reference element domain.
\item \label{rem:muK}
  We also define the outward pointing unit \emph{conormal} $\mu_K$ on the boundary of the element domain $\partial K$. This is the unique (up to sign) vector which is orthogonal to the boundary $\partial K$ and the normal $\nu_K$.
\item \label{rem:right-inv-proj}
  Our assumptions imply that $\nabla F_K(\hat{x})^\dagger$ (or $A_K^\dagger$) is a left inverse but not a right inverse of $\nabla F_K(\hat{x})$ (respectively,  $A_K$). We can compute that
  \[
    \nabla F_K( \hat{x} ) \nabla F_K( \hat{x} )^\dagger = \id - \nu_K( F_K( \hat{x}  ) ) \otimes \nu_K( F_K( \hat{x} ) ) =: P_K( F_K( \hat{x}) ),
  \]
  where $P_K$ denotes projection onto the tangent plane to $K$.
  One way to interpret this result is to note that $P_K$ is the identity operator when restricted to the tangent plane so that in some sense $\nabla F_K( \hat{x} )^\dagger$ is a right inverse of $\nabla F_K( \hat{x} )$ when we restrict to the tangent plane of $K$.
  \end{remlist}
\end{remark}

\begin{definition}[$\Theta$-surface finite element and $\Theta$-surface element reference map]
  \label{def:Theta-sfe}
  Let $\Theta \in \Nbb$ and $F_K$ a surface element reference map for a surface finite element $(K, \P, \Sigma)$.
  \begin{deflist}
  \item \label{def:Theta-sfe-map}%
    We say that $F_K$ is a \emph{$\Theta$-surface finite element reference map} if
    \begin{subdeflist}
    \item \label{def:Theta-sfe-smooth-FK}
      the surface element reference map  $F_K \in C^{\Theta+1}(\hat{K}; \R^{n+1})$ (i.e.\ $K$ is a $C^{\Theta+1}$-hypersurface);
    \item \label{def:Theta-sfe-curved}
      for $1 \le m \le \Theta+1$, there exists constants $C_m(K) > 0$ such that
      \begin{equation}
        \label{eq:CK-theta-defn}
        \sup_{\hat{x} \in \hat{K}} \abs{ \nabla^m F_K( \hat{x} ) } \norm{ A_K }^{-m}
        \le C_m(K).
      \end{equation}
    \end{subdeflist}
  \item%
    We say that $(K, \P, \Sigma)$ is a $\Theta$-surface finite element if $F_K$ is a $\Theta$-surface finite element reference map and
    \begin{subdeflist}
    \item \label{def:Theta-sfe-polynomial}
      the space $\P$ contains the functions $\hat{\chi} \circ F_K^{-1}$ for all $\hat\chi \in P_\Theta( \hat{K} )$;
    \item \label{def:Theta-sfe-smooth-func}
      the space $\P$ is contained in $C^{\Theta+1}(K)$.
    \end{subdeflist}
  \end{deflist}
\end{definition}
\begin{remark}
\begin{remlist}

\item
  The $\Theta$-surface finite element is a generalisation of a curved finite element of order $\Theta$ given by \cite{Ber89}. See also \cref{def:Theta-bfe}.
\item \label{rem:Sobolev-K}
  For a $\Theta$-surface finite element $(K, \P, \Sigma)$, we have that the Sobolev space $W^{m,p}( K )$ is well defined for $0 \le m \le \Theta+1$ and $1 \le p \le \infty$ \citep{Heb00,DziEll13a}.
\item \label{rem:surf-P-closed}
  Since $P \subset C^{\Theta+1}( K )$ and $P$ is finite dimensional, $P$ is a closed subspace of $W^{m,p}( K )$ for $0 \le m \le \Theta+1$ and $p \in [1,\infty]$.
\item
  We note that $(K, \P, \Sigma)$ is a $1$-surface finite element if $(K, \P, \Sigma)$ is a surface finite element (\cref{def:sfe}), the map $F_K \in C^2(\hat{K}, \R^{n+1})$ and $\P$ contains all affine functions on $K$. We see that the constant $C_1(K) = 1 + C_K$ (see \cref{lem:FK-AK-scale}).

  \item Our applications (\cref{sec:application1,sec:application3}) will use $F_K \in P_\theta( \hat{K} )$ in the computational method but we allow the more general case here.
  \end{remlist}
\end{remark}

Using transformation formulae, we have for $\chi \colon K \to \R$
\begin{align*}
  \int_{K} \chi \dd \sigma = \int_{\hat{K}} \hat{\chi}( \hat{x} ) \sqrt{ g( \hat{x} ) } \dd \hat{x},
  \qquad
  \nabla_K \chi ( x ) = \sum_{i,j=1}^{n} g^{ij}( \hat{x} ) \frac{\partial \hat{\chi}( \hat{x} ) }{\partial \hat{x}_j} \frac{\partial F_K(\hat{x})}{\partial \hat{x}_i},
\end{align*}
where $F_K( \hat{x} ) = x$ and $\hat\chi( \hat{x} ) = \chi( x )$, and
\begin{align*}
  G( \hat{x} ) = \big( \nabla F_K( \hat{x} ) \big)^t \big( \nabla F_K( \hat{x} ) \big),
  \qquad
  g( \hat{x} ) = \det G( \hat{x} ),
\end{align*}
and finally, $(g^{ij})$ are the components of the inverse $G^{-1}$.
Note that the surface gradient on $K$ can also be written as:
\[
  \nabla_K \chi(x) = \nabla \tilde{\chi}( x ) - ( \nabla \tilde{\chi}( x ) \cdot \nu_K( x ) ) \nu_K( x )
  = ( \id - \nu_K( x ) \otimes \nu_K( x ) ) \nabla \tilde{\chi}(x),
\]
where $\nabla \tilde{\chi}$ is the gradient of an arbitrary extension of $\chi$ away from $K$.

\begin{remark}
  If the surface finite element map $F_K$ is a $C^2$ function then we can define the extended Weingarten map $\mathbb{H}_K \colon K \to \R^{(n+1) \times (n+1)}$ by
  \begin{equation}
    \label{eq:element-Hk}
    ( \mathbb{H}_K )_{ij} = ( \nabla_K )_{i} ( \nu_K )_j \mbox{ for } i,j=1,\ldots,n+1.
  \end{equation}
\end{remark}

\begin{example}[Surface finite elements]
  \label{ex:surface-fem}
  We are thinking of three particular examples. The first is due to \citet{Dzi88} and the second due to \citet{Hei05}.  Examples of each of these first two cases are shown in \cref{fig:sfem-examples}.

  \begin{exlist}
  \item \label{ex:affine-sfe} Let $( \hat{K}, \P_1( \hat{K} ), \hat\Sigma )$ be a reference Lagrangian finite element. Consider the affine map $F_K \colon \hat{K} \to \R^{n+1}$ given by $F_K ( \hat{x} ) = A_K \hat{x} + b_K$. If $A_K$ is non-degenerate, then this defines a surface finite element $(K, \P, \Sigma)$. The element domain $K$ is determined by its vertices and $\P$ consists of affine functions over $K$. This is the surface finite element introduced by \citet{Dzi88} which we will call an \emph{affine finite element}. We think of a simplex for $\hat{K}$ so that the domains $K$ are either line segments embedded in $\R^2$, triangles embedded in $\R^3$, and tetrahedra embedded in $\R^4$.
  \item \label{ex:isoparametric-sfe} Let $( \hat{K}, \hat\P, \hat\Sigma )$ be a reference finite element. Let $( K, \P, \Sigma )$ be a surface finite element which the image of $( \hat{K}, \hat\P, \hat\Sigma )$ under a map $F_K$ which satisfies $F_K \in ( \hat{P} )^{n+1}$. We call $(K, \P, \Sigma)$ an isoparametric (surface) finite element. This construction is a generalisation of an affine finite element and was introduced by \citet{Hei05}.
    We note that the functions in $\P$ will not necessarily consist of polynomials over $K$ even if $\hat{P}$ consists of polynomials over $\hat{K}$,
    however this leads to a practical scheme where integrals are computed over reference elements.
    This example is the basis for the method in \cref{sec:application1}.
  \item \label{ex:bulk-embed-sfe} Let $( \hat{K}, \hat\P, \hat\Sigma )$ be a reference finite element. Then $( \hat{K}, \hat{\P}, \hat{\Sigma} )$ can be thought of a surface finite element $( K, \P, \Sigma )$ by defining the parametrisation $F_K$ by
    \begin{equation*}
      F_K( \hat{x} ) = ( \hat{x}_1, \ldots, \hat{x}_n, 0 ).
    \end{equation*}
    Note that $\hat{K} \subset \R^n$ but $K \subset \R^{n+1}$.
    In general, we could consider flat surface finite elements to be surface finite elements to have parametrisation $F_K$ such that $( F_K )_{n+1} \equiv 0$.
  \end{exlist}
\end{example}

\begin{figure}
  \centering
  \includegraphics{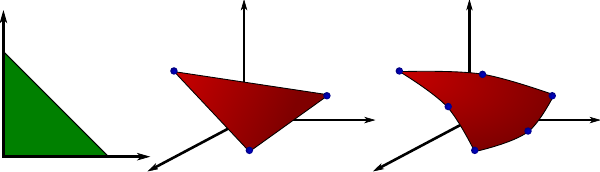}
  \caption{Examples of different surface finite elements in the case $n=2$. Left shows a reference finite element (in green), centre shows an affine finite element (\cref{ex:affine-sfe}) and right shows an isoparametric surface finite element (\cref{ex:isoparametric-sfe}) with a quadratic $F_K$. The plot shows the element domains in red and the location of nodes in blue.}
  \label{fig:sfem-examples}
\end{figure}

\begin{lemma}
  \label{lem:FK-AK-scale}
  Let $F_K \colon \hat{K} \to K$ be a surface finite element reference map then $F_K$ is a $C^1$-diffeomorphism and satisfies
  \begin{align}
    \label{eq:FK-AK-scale1}
    \sup_{\hat{x} \in \hat{K}} \norm{ \nabla F_K( \hat{x} ) }
    & \le ( 1 + C_K ) \norm{ A_K } \\
    \label{eq:FK-AK-scale2}
    \sup_{\hat{x} \in \hat{K}} \norm{ (\nabla F_K (\hat{x}) )^\dagger}
    & \le ( 1 - C_K ) \norm{ A_K^\dagger },
  \end{align}
  and also for all $\hat{x} \in \hat{K}$
  \begin{equation}
    \label{eq:FK-AK-scale3}
    ( 1 - C_K )^{2n} \det( A_K^t A_K )
    \le g( \hat{x} )^2
    \le ( 1 + C_K )^{2n} \det ( A_K^t A_K ).
  \end{equation}
\end{lemma}

\begin{proof}
  The proof of \cref{eq:FK-AK-scale1}, \cref{eq:FK-AK-scale2} and \cref{eq:FK-AK-scale3} follows immediately from \cref{eq:CK-defn} by writing $\nabla F_K$ as
  \begin{equation*}
    \nabla F_K( \hat{x} ) = ( \id + \nabla D_K( \hat{x} ) A_K^\dagger ) A_K.
  \end{equation*}
  For \cref{eq:FK-AK-scale3}, we use the fact that the determinant is an $n$-linear continuous form.
\end{proof}

To help us understand the geometry of the new elements, given an element domain $K$, we introduce a new affine element domain $\tilde{K}$ defined by the affine part of the parametrisation: $\tilde{K} := \{ A_K \hat{x} + b_K : \hat{x} \in \hat{K} \}$.

\begin{lemma}
  \label{lem:element-geometry}
  Let $K$ be an element domain \cref{eq:element-domain} parametrised by a surface element reference map $F_K$ over $\hat{K}$.
  Denote by
  \begin{subequations}
    \label{eq:element-geometry}
    \begin{align}
      \label{eq:hK}
      & h_K := \diam(\tilde{K}) \\
      \label{eq:rhoK}
      & \rho_K := \sup \{ \diam(B) : B \text{ is a $n$-dimensional ball contained in } \tilde{K} \}.
    \end{align}
  \end{subequations}
  We will also write $\hat{h}$ and $\hat{\rho}$ for the diameter of $\hat{K}$ and diameter of the maximum inscribed ball in $\hat{K}$.
  Then we have that
  \begin{subequations}
    \label{eq:AK-h-relations}
    \begin{align}
      \label{eq:AK-h-relation-a}
      \norm{ A_K } & \le \frac{h_K}{ \hat\rho } \\
      \label{eq:AK-h-relation-b}
      \norm{ A_K^\dagger } & \le \frac{ \hat{h} }{ \rho_K } \\
      \label{eq:AK-h-relation-c}
      \sup_{\hat{x} \in \hat{K}} \sqrt{g(\hat{x})}
      & \le \frac{1}{\meas{\hat{K}}} \left( \frac{1+C_K}{1-C_K } \right)^{n} \meas(K).
    \end{align}
  \end{subequations}
\end{lemma}

\begin{proof}
  To show \cref{eq:AK-h-relation-a}, we start by noticing that
  \begin{equation*}
    \norm{ A_K } = \frac{1}{\hat{\rho}} \sup \left\{ \abs{ A_K \xi } : \xi \in \R^n, \abs{ \xi } = \hat\rho \right\}.
  \end{equation*}
  From the definition of $\hat{\rho}$ we know that for all $\xi \in \R^n$, $\abs{\xi} = \hat\rho$, there exists $\hat{y}, \hat{z} \in \hat{K}$ such that $\hat{y} - \hat{z} = \xi$. Then noting that $A_K \hat{y}, A_K \hat{z} \in \tilde{K}$, we have that
  \[
    \abs{ A_K \xi } = \abs{ A_K (\hat{y} - \hat{z} ) }
    = \abs{ A_K \hat{y} - A_K \hat{z} }
    \le h_K.
  \]
  Since the choice of $\xi$ was arbitrary, we have shown \cref{eq:AK-h-relation-a}.

  For \cref{eq:AK-h-relation-b}, we proceed in a similar fashion with
  \[
    \norm{ A_K^\dagger }
    = \frac{1}{\rho_K} \sup \left\{ \abs{ A_K^\dagger \xi } : \xi \in \R^{n+1}, \abs{\xi} = \rho_K \right\}.
  \]
  Let $\xi \in \R^{n+1}$. We note that $A_K^\dagger$ has a non-trivial kernel so we decompose $\xi = \xi_1 + \xi_2$ where $A_K^\dagger \xi_2 = 0$ and $\xi_1$ is in the tangent plane to $\tilde{K}$. Then, we see that there exists $y,z \in \tilde{K}$ such that $y - z = \xi$, and noting that $A_K^\dagger ( y - b_K ), A_K^\dagger ( z - b_K ) \in \hat{K}$ (we see this using the definition of $\tilde{K}$ and that $A_K^\dagger$ is a left inverse of $A_K$), we see that
  \[
    \abs{ A_K^\dagger \xi }
    = \abs{ A_K^\dagger ( y - z ) }
    = \abs{ A_K^\dagger ( y - b_K ) - A_K^\dagger ( z - b_K ) }
    \le \hat{h}.
  \]
  Again, since the choice of $\xi$ was arbitrary, we have shown \cref{eq:AK-h-relation-b}.

  To see \cref{eq:AK-h-relation-c}
  we apply each of the previous two bounds with the result of \cref{lem:FK-AK-scale} to see
  \begin{align*}
    \sup_{\hat{x} \in \hat{K}} \sqrt{g(\hat{x})}
    & \le ( 1 + C_K )^n \sqrt{ \det ( A_K^t A_K ) } \\
    & \le \frac{1}{\meas \hat{K}} ( 1 + C_K )^{n} \int_{\hat{K}} \sqrt{ \det ( A_K^t A_K ) } \dd \hat{x} \\
    & \le \frac{1}{\meas \hat{K}} ( 1 + C_K )^n / ( 1 - C_K )^n \int_{\hat{K}} \sqrt{ g(\hat{x}) } \dd \hat{x} \\
    & \le \frac{1}{\meas \hat{K}} \left( \frac{ 1 + C_K }{ 1 - C_K } \right)^n \meas (K) \\
    & \le C(\hat{K}) \meas(K). \qedhere
  \end{align*}
\end{proof}

\begin{remark}
  \label{rem:AK-h-meas}
  We note that the volume of an element $\meas(K)$ can be estimated by $h_K$ and $\rho_K$ by
  \begin{align*}
    c_1 \rho_K^n \le \meas{K} \le c_2 h_K^n.
  \end{align*}
  Here the positive constants $c_1, c_2$ depend on the volume of the unit ball in $\R^n$ and the constant $C_K$.
\end{remark}

\begin{remark}
  \label{rem:small-hk-surf}
  In the sequel, we will assume implicitly that results hold for $h_K$ sufficiently small ($h_K < h_0$) for some particular value of $h_0$.
  In general this is always possible by subdividing a particular element using a refinement procedure and applying the result to the subdivided, smaller elements.
\end{remark}

This scaling property allows us to characterise Sobolev spaces over a surface finite element $K$ and calculate norms over $\hat{K}$ (see \cref{rem:Sobolev-K}).
\begin{lemma}
  \label{lem:norm-scaling}
  Let $F_K \colon \hat{K} \to K$ be a $\Theta$-reference finite element map \cref{def:Theta-sfe-map}.
  Let $0 \le m \le \Theta+1$ and $p \in [1,\infty]$, then $\chi \in W^{m,p}(K)$ implies $\hat{\chi} = \chi \circ F_K$ belongs to $W^{m,p}(\hat{K})$.
  We have for any $\chi \in W^{m,p}(K)$ that
  \begin{equation}
    \label{eq:norm-scale-a}
    \abs{ \hat{\chi} }_{W^{m,p}(\hat{K})}
    \le c \left( \sup_{ \hat{x} \in \hat{K} } \sqrt{ g(\hat{x} ) } \right)^{-\frac{1}{p}}
    \norm{ A_K }^m \sum_{r=1}^m \abs{ \chi }_{W^{r,p}(K)},
  \end{equation}
  for a constant which depends on $C_K, C_2(K), \ldots, C_m(K)$.
  We also have for any $\hat{\chi} \in W^{m,p}(\hat{K})$ that $\chi = \hat{\chi} \circ F_K^{-1} \in W^{m,p}(K)$ and
  \begin{align}
    \label{eq:norm-scale-b}
    \abs{ \chi }_{W^{m,p}(K)} \le c \left( \sup_{ \hat{x} \in \hat{K} } \sqrt{ g(\hat{x} ) } \right)^{\frac{1}{p}}
    \sum_{r=1}^m \norm{ A_K^\dagger }^r \abs{ \hat{\chi} }_{W^{r,p}(\hat{K})},
  \end{align}
  where the constant here depends on $C_K, C_2(K), \ldots, C_m(K)$ and the product $\norm{ A_K } \norm{ A_K^\dagger }$.
\end{lemma}

\begin{proof}
  In this proof we use the results and notation for the Fa\'a di Bruno result presented in \cref{sec:faa-di-bruno}.

  For \cref{eq:norm-scale-a} using \cref{eq:faa-di-bruno},
we see that
\begin{align*}
 \frac{\partial^m}{\partial \hat{x}_{j_m} \ldots \partial \hat{x}_{j_1}} \hat\chi ({\hat{x}})
    = \sum_{r=1}^m \sum_{\lambda_1=1}^{n+1} \cdots \sum_{\lambda_r=1}^{n+1}
    (\nabla_K)_{\lambda_r} \cdots (\nabla_K)_{\lambda_1} \chi(x)
    \sum_{\Pp_{m,r}} \prod_{s=1}^r \frac{\partial^{\abs{\sigma_s}}}{\partial \hat{x}_{j_{\sigma_s}}} (F_K(\hat{x}))_{\lambda_s}.
\end{align*}
Then
\begin{align*}
  \abs{  \frac{\partial^m}{\partial \hat{x}_{j_m} \ldots \partial \hat{x}_{j_1}} \hat\chi ({\hat{x}}) }
  & \le \sum_{r=1}^m \sum_{\lambda_1=1}^{n+1} \cdots \sum_{\lambda_r=1}^{n+1}
    \abs{ (\nabla_K)_{\lambda_r} \cdots (\nabla_K)_{\lambda_1} \chi(x) }
    \sum_{\Pp_{m,r}} \prod_{s=1}^r
    \abs{ \frac{\partial^{\abs{\sigma_s}}}{\partial \hat{x}_{j_{\sigma_s}}} (F_K(\hat{x}))_{\lambda_s} }.
\end{align*}
We see that from assumption \cref{eq:CK-theta-defn}
\begin{multline}
  \label{eq:dsigma_FK}
    \sum_{\Pp_{m,r}} \prod_{s=1}^r
    \abs{ \frac{\partial^{\abs{\sigma_s}}}{\partial \hat{x}_{j_{\sigma_s}}} (F_K(\hat{x}))_{\lambda_s} }
    \le \sum_{\Pp_{m,r}} \prod_{s=1}^r
    c(c_{\abs{\sigma_s}}(K)) \norm{ A_K }^{\abs{\sigma_s}} \\
    \le c(C_1(K), \ldots, C_{m}(K)) \sum_{\Pp_{m,r}} \norm{ A_K }^{\sum_{s=1}^r \abs{\sigma_s}}
    \le c(C_1(K), \ldots, C_{m}(K)) \norm{ A_K }^m.
\end{multline}
Then, we have that
\begin{equation*}
  \abs{ \frac{\partial^m}{\partial \hat{x}_{j_m} \ldots \partial \hat{x}_{j_1}} \hat\chi (\hat{x}) }
  \le c(C_1(K), \ldots, C_{m}(K)) \norm{ A_K }^m \sum_{r=1}^m \sum_{\lambda_1=1}^{n+1} \cdots \sum_{\lambda_r=1}^{n+1}
  \abs{ (\nabla_K)_{\lambda_r} \cdots (\nabla_K)_{\lambda_1} \chi(x) }.
\end{equation*}
The inequality \cref{eq:norm-scale-a} then follows from integration and a Minkowski inequality.

For \cref{eq:norm-scale-b}, from \cref{eq:faa-bi-bruno-inverse}, we have
\begin{align*}
  & \abs{ ( \nabla_K )_{i_m} \cdot ( \nabla_K )_{i_1} \chi( x ) } \\
  & \le \abs{ \sum_{j_m, \ldots, j_1=1}^n \frac{ \partial^m }{ \partial \hat{x}_{j_m} \cdots \partial \hat{x}_{j+1} } \hat{\chi}
    \left( \prod_{s=1}^m ( \nabla F_K( \hat{x} ) )^\dagger_{j_s, i_s} \right) } \\
  & \qquad +
    \abs{ \sum_{r=1}^{n-1} \sum_{\lambda_1, \ldots, \lambda_r=1}^{n+1} ( \nabla_\Gamma )_{\lambda_r} \cdots ( \nabla_\Gamma )_{\lambda_1} \chi( x )
    \sum_{j_m, \ldots, j_1=1}^n \sum_{\Pp_{m,r}} \prod_{s=1}^r \frac{\partial^{\abs{\sigma_s}}}{ \partial \theta_{j_{\sigma_s}} } ( F_K(\hat{x}) )_{\lambda_s} \prod_{s=1}^m ( \nabla F_K(\hat{x} ) )^\dagger_{j_s, i_s} } \\
  & =: A_1 + A_2.
\end{align*}
We bound each of $A_1$ and $A_2$ in turn. We see that for $A_1$, applying \cref{eq:FK-AK-scale2}
\begin{align*}
  A_1 & \le ( 1 - C_K )^m \norm{ A_K^\dagger }^m \sum_{j_m=1}^n \cdots \sum_{j_1=1}^n
        \abs{ \frac{\partial^m}{\partial x_{j_m} \ldots \partial x_{j_1}} \hat\chi (\hat{x}) }.
\end{align*}
For $A_2$, applying \cref{eq:FK-AK-scale2} and \cref{eq:dsigma_FK}:
\begin{align*}
  A_2 & \le
        c(C_1(K), \ldots, C_{m}(K)) ( 1 - C_K )^m \norm{ A_K }^m \norm{ A_K^\dagger }^m
        \sum_{r=1}^{m-1} \sum_{\lambda_1=1}^{n+1} \cdots \sum_{\lambda_r=1}^{n+1}
        \abs{ (\nabla_K)_{\lambda_r} \cdots (\nabla_K)_{\lambda_1} \chi(x) }.
\end{align*}
Combining the above estimates and integrating over the domain, we see
\begin{align*}
  \abs{ \chi }_{W^{m,p}(K)}
  & \le \left( \sup_{\hat{x} \in \hat{K}} \sqrt{g(\hat{x})} \right)^{1/p} ( 1 - C_K )^m \norm{ A_K^\dagger }^{m} \abs{ \hat\chi }_{W^{m,p}(\hat{K})} \\
  & \qquad + ( c(C_1(K), \ldots, C_{m}(K)) ( 1 - C_K )^m \norm{ A_K }^m \norm{ A_K^\dagger }^m ) \sum_{r=1}^{m-1} \abs{ \chi }_{W^{m-1,p}(K)}.
\end{align*}

Define $\beta_m = ( c(C_1(K), \ldots, C_{m}(K)) ( 1 - C_K )^m \norm{ A_K }^m \norm{ A_K^\dagger }^m )$. Then, a simple induction argument shows that
\begin{align*}
  \abs{ \chi }_{W^{m,p}(K)}
  \le c \left( \sup_{\hat{x} \in \hat{K}} \sqrt{g(\hat{x})} \right)^{1/p} \sum_{r=1}^m c_{m,r} \norm{A_K^\dagger}^{r} \abs{ \hat\chi }_{W^{r,p}(\hat{K})}
\end{align*}
where $c_{m,r}$ satisfies
\[
  c_{m,m} = 1 \qquad c_{m,r} = \sum_{s=r}^{m-1} \beta_m c_{s,r}. \qedhere
\]

\end{proof}

Given a surface finite element $(K, \P, \Sigma)$ (\cref{def:sfe}), let $\{ \chi_i : 1 \le i \le d \} \subset \P$ be the basis dual to $\Sigma$. This is the set of \emph{basis functions} of the finite element. If $\eta$ is a function for which all $\sigma_i( \eta )$, $1 \le i \le d$ is well defined, then we define the \emph{local interpolant} by
\begin{equation}
  \label{eq:3}
  I_K \eta := \sum_{i=1}^m \sigma_i( \eta ) \chi_i.
\end{equation}
We can think of $I_K \eta$ as the unique shape function that has the same nodal values as $\eta$ so that, in particular, $I_K \chi = \chi$ for $\chi \in \P$.

\begin{theorem}
  [Local interpolation estimate]
  \label{thm:IK-bound}
  Let $( K, \P, \Sigma )$ be a $\Theta$-surface finite element (\cref{def:Theta-sfe}) with reference element $(\hat{K}, \hat\P, \hat\Sigma )$ which satisfies the assumptions of \cref{lem:bramble-hilbert} for some $0 < k,m \le \Theta$, $p,q \in [1,\infty]$.
  Then there exists a constant $C = C(\hat{K}, \hat{P}, \hat\Sigma )$ such that for all functions $\chi \in W^{k+1,p}(K)$,
  \begin{equation}
    \label{eq:IK-estimate}
    \abs{ \chi - I_K \chi }_{W^{m,q}(K)}
    \le C \meas(K)^{1/q - 1/p}
    \frac{ h_K^{k+1} }{ \rho^m_K } \abs{ \chi }_{W^{k+1,p}(K)}.
  \end{equation}
\end{theorem}

\begin{proof}
  We re-scale \cref{eq:bramble-hilbert} using \cref{lem:norm-scaling} and the estimates from \cref{eq:AK-h-relations}:
  \begin{align*}
    \abs{ \chi - I_K \chi }_{W^{m,q}(K)}
    & \le c \left( \sup_{\hat{x} \in \hat{K}} \sqrt{g(\hat{x})} \right)^{1/q}
      \sum_{r=1}^m \norm{ A_K^\dagger }^r
      \abs{ \hat{\chi} - I_{\hat{K}} \hat{\chi} }_{W^{r,q}(\hat{K})} \\
    & \le c \left( \sup_{\hat{x} \in \hat{K}} \sqrt{g(\hat{x})} \right)^{1/q}
      \sum_{r=1}^m \norm{ A_K^\dagger }^r
      \abs{ \hat{\chi} }_{W^{k+1,p}(\hat{K})} \\
    & \le c \left( \sup_{\hat{x} \in \hat{K}} \sqrt{g(\hat{x})} \right)^{1/q - 1/p}
      \sum_{r=1}^m \norm{ A_K^\dagger }^r \norm{ A_K }^{k+1}
      \abs{ \chi }_{W^{k+1,p}(K)} \\
    & \le c \meas(K)^{1/q - 1/p}
      \frac{ h^{k+1}_K }{ \rho_K^m }
      \norm{ \chi }_{W^{k+1,p}(K)}.
  \end{align*}
  The last line holds if $\rho_K < 1$ (note that $\rho_K \ge h_K$ so this statement is true for $h_K$ small enough).
\end{proof}

\subsection{Triangulated hypersurface and surface finite element spaces}
We will next bring together several surface finite elements in order to define $\Gamma_h$ as a collection of finite element domains.

\begin{definition}
 \begin{deflist}
 \item \label{def:triangulaed-hypersurface}
  A \emph{triangulated  hypersurface} is a set $\Gamma_h$ equipped with an \emph{admissible subdivision} $\T_h$ consisting of surface finite element domains \cref{eq:element-domain} such that $\bigcup_{K \in \T_h} K = {\Gamma}_h$, $\mathring{K}_1 \cap \mathring{K}_2 = \emptyset$ for $K_1, K_2 \in \T_h$ with $K_1 \neq K_2$.
   \item \label{def:h}
The maximum subdivision diameter $h$ is defined by:
\begin{equation}
  \label{eq:h-defn}
  h := \max_{K \in \T_h} h_K.
\end{equation}

  \item \label{def:facet}
  Let $\Gamma_h$ be a discrete hypersurface equipped with an admissible subdivision $\T_h$ such that each set $K \in \T_h$ is an element domain for a surface finite element $(K, \P^K, \Sigma^K)$ parametrised over the same polygonal reference finite element $( \hat{K}, \hat{P}, \hat{\Sigma} )$. We say that $E \subset K$ is a \emph{facet} if $E$ is the image of a boundary facet of $\hat{K}$.
  \item \label{def:conforming-subdivision}
   We say that $\T_h$ is a \emph{conforming} subdivision of $\Gamma_h$ if any facet of an element domain $K$ is either a facet of another element domain $K' \in \T_h$, in which case we say $K$ and $K'$ are \emph{adjacent}, or a portion of the boundary $\partial \Gamma_h$ (if such a boundary exists).
  \item For a conforming subdivision $\T_h$ of $\Gamma_h$, we denote by $\F_h$ the set of facets between adjacent elements and $\partial \T_h$ any boundary facets.
    For a common facet $F \in \F_h$ between elements $K, K' \in \T_h$, we make a fixed choice that the conormal $\mu_K$ on $F$ will be denoted $\mu_F^+$ and the conormal $\mu_{K'}$ on $F$ will be denoted $\mu_F^-$. The choice of which element is on which side of the facet is not important in our considerations.
  \end{deflist}
\end{definition}

\begin{remark}
\begin{itemize}
\item
In the above definition we do not impose any global assumptions on the connectivity or smoothness of $\Gamma_h$.
Thus there may not be an underling smooth surface.
  
  \item
We orient a discrete hypersurface which is equipped with a conforming subdivision by choosing a particular sign to the element-wise definition of normal.
We restrict that the induced orientation of the intersection of adjacent element domains are opposite.
For example, for a simplex reference element, the vertices in facets between two elements should be ordered oppositely in each element.
\end{itemize}
\end{remark}

{
For any triangulated hypersurface $\Gamma_h$, we may define spaces of Lebesgue integrable functions $L^p( \Gamma_h )$ with the usual norms $\norm{ \cdot }_{L^p(\Gamma_h)}$ for $p \in [1,\infty]$.

\begin{definition}[Broken Sobolev spaces and norms]
  \label{def:broken-sobolev-norm}
Let $\T_h$ be a subdivision of $\Gamma_h$ consisting of $\Theta$-surface finite elements.
Then for $0 \le m \le \Theta+1$, $p \in [1,\infty]$, we define the broken Sobolev space $W^{m,p}(\T_h)$ by
\begin{equation}
  \label{eq:broken-sobolev-space}
  W^{m,p}(\T_h) := \left\{
    \eta_h \in L^1( \Gamma_h ) : \eta_h |_{K} \in W^{m,p}(K) \mbox{ for all } K \in \T_h
    \right\},
  \end{equation}
  with norm
\begin{equation}
  \label{eq:discrete-norm}
  \norm{ \eta_h }_{W^{m,p}(\T_h)}
  :=
  \begin{cases}
    \displaystyle
    \left(
      \sum_{K \in \T_h} \norm{ \eta_h }_{W^{m,p}(K)}^p
    \right)^{1/p}
    & p < \infty \\
    \displaystyle
    \max_{K \in \T_h} \norm{ \eta_h }_{W^{m,\infty}(K)}
    & p = \infty.
  \end{cases}
\end{equation}
\end{definition}

\begin{lemma}
  The space $W^{m,p}( \T_h )$ is complete.
\end{lemma}

\begin{proof}
  Consider a Cauchy sequence $\{ \eta_j \} \in W^{m,p}( \T_h )$. This implies
  \begin{itemize}
  \item $\eta_j$ is Cauchy in $L^p(\Gamma_h)$ so there exists $\xi$ such that $\eta_j \to \xi$ in $L^p(\Gamma_h)$.
  \item $\eta_j|_K$ is Cauchy for all $K$ so there exists $\xi_K$ such that $\eta_j|_{K} \to \xi_K$ in $W^{m,p}(K)$.
  \end{itemize}
  It is clear from the triangle inequality that $\xi|_K = \xi_K$:
  \[
    \norm{ \xi|_{K} - \xi_K }_{W^{m,p}(K)} \le \norm{ \xi|_{K} - \eta_j|_{K} }_{L^{m,p}(K)} + \norm{ \eta_j|_{K} - \xi_K }_{W^{m,p}(K)},
  \]
  since the right hand-side converges to $0$ as $j \to \infty$.
  Hence, we have shown that $\eta_j$ converges to a function $\xi \in W^{m,p}( \T_h )$ in the $W^{m,p}( \T_h )$ norm.
\end{proof}

We also want more connectivity between elements. This is achieved in the following space.
{
  Let $\Gamma_h$ be a triangulated hypersurface with conforming subdivision $\T_h$.
  For $K \in \T_h$, we denote the trace of a function $\chi \in W^{1,p}(K)$ by $T_K \chi \in L^p(\partial K)$ and recall that there exists a constant $c_{T_K} > 0$ such that
  \begin{equation}
    \label{eq:element-trace}
    \norm{ T_K \chi }_{L^p( \partial K )} \le c_{T_K} \norm{ \chi }_{W^{1,p}(K)} \qquad \mbox{ for all } \chi \in W^{1,p}(K).
  \end{equation}
  We define the space $W^{1,p}_T(\T_h)$ by
  \begin{multline}
    \label{eq:trace-broken-sobolev-space}
      W^{1,p}_T( \T_h ) := \Bigl\{ \eta_h \in L^p( \Gamma_h ) :
      \eta_h |_{K} \in W^{1,p}( K ) \mbox{ for all } K \in \T_h, \mbox{ and } \\
      T_{K} ( \eta_h |_{K} ) = T_{K'} ( \eta_h |_{K'} ) \mbox{ a.e. in } K \cap K' \mbox{ for adjacent } K, K' \in \T_h \Bigr\}.
  \end{multline}
  We equip this space with the broken norm $\norm{ \cdot }_{W^{1,p}(\T_h)}$.
}

\begin{lemma}
  \label{lem:W1pT-closed-subspace}
  The space $W^{1,p}_T(\T_h)$ is a closed subspace of $W^{1,p}(\T_h)$ so is complete.
\end{lemma}

\begin{proof}
  Take a sequence $\{ \eta_j \} \subset W^{1,p}_T( \T_h )$ which converges to $\eta_h \in W^{1,p}( \T_h )$. Then for any pair of adjacent elements $K, K' \in \T_h$ we have
  \begin{multline*}
    \norm{ \eta_h|_{K} - \eta_h|_{K'} }_{L^p( K \cap K' )}
    \le \norm{ \eta_h - \eta_j }_{L^p( \partial K )} + \norm{ \eta_h - \eta_j }_{L^p( \partial K' )} \\
    \le c_{T_K} \norm{ \eta_h - \eta_j }_{W^{1,p}(K)} + c_{T_{K'}} \norm{ \eta_h - \eta_j }_{W^{1,p}(K')}
    \le ( c_{T_K} + c_{T_{K'}} ) \norm{ \eta_h - \eta_j }_{W^{1,p}( \T_h )}.
  \end{multline*}
  Clearly the right hand side converges to $0$ as $j \to \infty$ so we have the traces of $\eta_h$ from adjacent elements coincide and $\eta_h \in W^{1,p}( \T_h )$.
\end{proof}
We will use the notation for $H^1_T( \T_h ) := W^{1,2}_T( \T_h )$ which is a Hilbert space when equipped with the obvious broken inner product.

}

\subsubsection{Surface finite element space}
We now restrict to Lagrangian finite elements over a polygonal reference finite element.
Here we assume that the degrees of freedom for each element $( K, \P, \Sigma )$ are given by
\begin{equation*}
  \Sigma = \{ \chi \mapsto \chi( a ) : a \in \N^K \},
\end{equation*}
where $\N^K$ is a finite set of nodes in $K$.
We call $\N^K$ the set of \emph{Lagrange nodes} of $K$.
This restriction avoids difficulties in defining the edge of elements and how to effectively bring elements together to form a global finite element space.
Extensions to other element types such as Hermite elements are left to future work.

Finally, the set of degrees of freedom of adjacent surface finite elements will be related as follows.
Let $( K, \P, \Sigma )$ and $( K', \P', \Sigma')$ be two surface finite elements such that $K$ and $K'$ are adjacent with $\Sigma = \{ \chi \mapsto \chi( a ), a \in \N^K \}$ and $\Sigma' = \{ \chi \mapsto \chi( a' ), a' \in \N^{K'} \}$.
Then, we have
\begin{equation}
  \label{eq:node-agree}
  \left( \bigcup_{a \in \N^K} a \right) \cap K'
  =
  \left( \bigcup_{a' \in \N^{K'}} a' \right) \cap K.
\end{equation}

We denote the \emph{global set of Lagrange nodes} by
\begin{equation}
  \label{eq:Nh}
  \N_h = \bigcup_{K \in \T_h} \N^K.
\end{equation}
For each $a \in \N_h$, let $\T( a ) \subset \T_h$ be the local neighbourhood of elements for which $a \in \N^K$.

\begin{definition}
  [Surface finite element space]
  \begin{deflist}
  \item \label{def:sfe-space}
    Let $\Gamma_h$ be a discrete hypersurface equipped with a conforming subdivision $\T_h$ with each domain $K$ equipped with a surface finite element $( K, \P^K, \Sigma^K )$ (\cref{def:sfe}) which satisfy \cref{eq:node-agree}.
    A \emph{surface finite element space} is a (generally proper) subset of the product space $\prod_{K \in \T_h} \P^K$ given by
    \begin{multline*}
      \S_h :=
      \bigg\{ \chi_h = ( \chi_K )_{K \in \T_h} \in \prod_{K \in \T_h} \P^K : \\
              \chi_{K}( a ) = \chi_{K'}( a ), \mbox{ for all } K, K' \in \T( a ),
                \mbox{ for all } a \in \N_h \bigg\}.
    \end{multline*}
  \item \label{def:global-degrees-of-freedom}
    The surface finite element space is determined by the \emph{global degrees of freedom}
    \begin{equation*}
      \Sigma_h = \left\{
        \chi_h \mapsto \chi_h( a ) : a \in \N_h
      \right\}.
    \end{equation*}
  \end{deflist}
\end{definition}

In this definition, an element $\chi_h \in \S_h$ is not, in general a ``function'' defined over $\bar\Gamma_h$, since we do not necessarily have a good definition of $\chi_h$ over element boundaries: The ``function'' may be double-valued.

If it happens, however, that for each element $\chi_h \in \S_h$, the restrictions $\chi_K$ and $\chi_{K'}$ coincide along the common face of any adjacent elements $K$ and $K'$, then the function $\chi_h$ can be identified with a function defined over the set $\bar\Gamma_h$.
In this case, we call the elements $\chi_h \in \S_h$ \emph{surface finite element functions}.
Examples of surface finite element functions are shown in \cref{fig:fem-func-ex}.

{
\begin{lemma}
  \label{lem:Sh-closed-subspace}
  Let $\S_h$ be a surface finite element space consisting of surface element elements over a conforming subdivision $\T_h$ of $\Gamma_h$.
  Assume further that for each $K \in \T_h$, the corresponding reference finite element is a Lagrange element of order $k \ge 1$ (\cref{ex:standard-fem}).
  Then we can identify elements of $\S_h$ as functions in $C(\Gamma_h)$.
  Furthermore $\S_h$ is a closed subspace in $H^1_{T}( \T_h )$.
\end{lemma}

\begin{proof}
  Consider two adjacent elements $K, K' \in \T_h$ and a element $\chi_h \in \S_h$.
  The functions $\chi_K \circ F_K$ and $\chi_{K'} \circ F_{K'}$ when restricted to the appropriate edges in $\hat{K}$ are polynomials of degree $k$ which agree at the Lagrange points on this edge from \cref{eq:node-agree} and the definition of $\S_h$.
  The Lagrange points in the reference element determine polynomials of degree $k$ so we have that $\chi_{K} = \chi_{K'}$ on $K \cap K'$.
  Since $\T_h$ is a conforming subdivision we can define a global function $\chi_h \colon \bar\Gamma_h \to \R$ such that $\chi_h|_{K} = \chi_K$ for each $K \in \T_h$ which is globally continuous.
  Indeed $\chi_h$ restricted to each element is continuous, as the composition of a polynomial (element of $\hat{P} = P_k( \hat{K} )$ for some $k$) and a smooth surface finite element reference map $F_K$, and is single valued on the facets where any two elements meet.

  In fact the restriction $\chi_h|_K$ to each element $K \in \T_h$ is a $C^1$-function hence $\chi_h|_K \in H^1(K)$ so it is clear that $\S_h \subset H^1_T( \T_h )$.
  The space $\S_h$ is closed since it is finite dimensional.
\end{proof}
}

\begin{figure}[tb]
  \centering

  \includegraphics[width=0.45\textwidth]{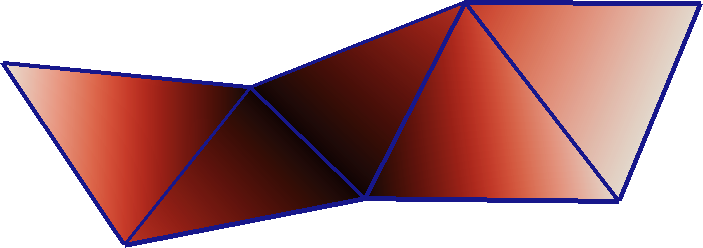}
  \includegraphics[width=0.45\textwidth]{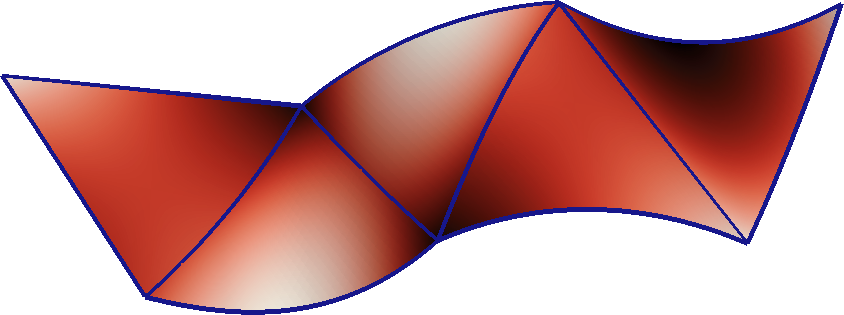}

  \caption{Examples of finite element functions. Left shows a piecewise linear function over a collection of affine finite elements (c.f. \cref{ex:affine-sfe}) and right shows a piecewise quadratic function over a collection of isoparametric (quadratic) finite elements (c.f. \cref{ex:isoparametric-sfe}).}
  \label{fig:fem-func-ex}
\end{figure}

We enumerate the nodes so that $\N_h = \{ a_i \}_{i=1}^N$ and take $\{ \chi_i \}_{i=1}^N$ to be the basis of $\S_h$ dual to $\Sigma_h$. Since, we have a finite basis of $\S_h$ we note that we can identify any $\chi_h \in \S_h$ with a vector $\gamma \in \R^N$ so that
\begin{equation*}
  \chi_h( x ) = \sum_{i=1}^N \gamma_i \chi_i( x ) \qquad \mbox{ for } x \in \Gamma_h.
\end{equation*}

\begin{definition}[Interpolation]
  \label{def:global-interpolant}
If $\eta$ is a function on $\Gamma_h$ for which all $\sigma_i( \eta )$, $1 \le i \le N$, is well defined (in case of Lagrangian finite elements, $\eta \in C(\Gamma_h)$ suffices), then we can define a \emph{global interpolant} $I_h \eta \in \S_h$ by
\begin{equation*}
  I_h \eta := \sum_{i=1}^N \sigma_i( \eta ) \chi_i.
\end{equation*}
Note that our construction implies that
\begin{equation*}
  ( I_h \eta )|_{K} = I_K \eta |_K \mbox{ for all } K \in \T_h,
\end{equation*}
and $I_h \chi_h = \chi_h$ for all $\chi_h \in \S_h$.
\end{definition}
In order to prove estimates on the global interpolant, we will first define three further properties of our subdivision $\T_h$.

\begin{definition}
  [Regular and quasi-uniform subdivisions]
  For $h \in (0,h_0)$, let $\Gamma_h$ be a triangulated hypersurface (\cref{def:triangulaed-hypersurface}) equipped with a conforming subdivision $\T_h$ (\cref{def:conforming-subdivision}).
  \begin{deflist}
  \item \label{def:regular}
    The family is said to be non-degenerate or \emph{regular} if there exists $\rho_\reg > 0$ such that for all $K \in \T_h$ and all $h \in (0,h_0)$,
    \begin{equation*}
      \rho_K \ge \rho_\reg h_K,
    \end{equation*}
    and there exists a constant $C > 0$ such that
    \[
      \sup_{h \in (0,h_0)} \max_{K \in \T_h} C_K \le C < 1.
    \]
  \item \label{def:Theta-regular}
    The family is said to be \emph{$\Theta$-regular} if it is regular, if for all $h \in (0,h_0)$ and all $K \in \T_h$, $F_K \in C^{\Theta+1}(K; \R^{n+1})$, and if, there exists a constant $C > 0$ such that
    \[
      \sup_{h \in (0,h_0)} \max_{K \in \T_h} C_m( K ) \le C < + \infty \qquad \mbox{ for } 2 \le m \le \Theta+1.
    \]
    This implies that $F_K$ is a $\Theta$-reference surface finite element map for each $K \in \T_h$.
  \item \label{def:quasi-uniform}
    A regular family is said to be \emph{quasi-uniform} if there exists $\rho > 0$ such that
    \begin{equation*}
      \min \{ \rho_K : K \in \T_h \} \ge \rho h \quad \mbox{ for all } h \in (0,h_0).
    \end{equation*}
  \end{deflist}
\end{definition}
\begin{remark}
  We note that:
  \begin{itemize}
  \item for a regular subdivision there exists a constant $c > 0$ depending on the global quantities $\hat{\rho}, \hat{h}$ and $\rho_\reg$
    \begin{equation*}
      \norm{ A_K } \le c h_K \le c h
      \qquad \mbox{ and } \qquad
      \norm{ A_K^\dagger } \le c h_K^{-1},
    \end{equation*}
  \item for a quasi-uniform subdivision there exists a constant $c > 0$ depending on the global quantities $\hat{\rho}, \hat{h}$ and $\rho$
    \begin{equation*}
      \norm{ A_K } \le c h
      \qquad \mbox{ and } \qquad
      \norm{ A_K^\dagger } \le c h^{-1}.
    \end{equation*}
  \end{itemize}
\end{remark}
\begin{theorem}
  [Global interpolation estimates]
  \label{thm:Ih-bound}
  For $h \in (0,h_0)$, let $\Gamma_h$ be a triangulated hypersurface (\cref{def:triangulaed-hypersurface}) equipped with a $\Theta$-regular (\cref{def:Theta-regular}), quasi-uniform (\cref{def:quasi-uniform}), conforming (\cref{def:conforming-subdivision}) subdivision $\T_h$. Let each $K \in \T_h$ be equipped with a $\Theta$-surface finite element $( K, \P^K, \Sigma^K )$ (\cref{def:Theta-sfe}) parametrised over a reference finite element $(\hat{K}, \hat{\P}, \hat{\Sigma})$ which satisfies the assumptions of \cref{thm:IK-bound} for some $0 < k,m \le \Theta$, $p,q \in [1,\infty]$. Then there exists a constant $C = C( \hat{K}, \hat{\P}, \hat{\Sigma}, \rho )$ such that for all functions $\eta \in W^{k+1,p}( \T_h ) \cap C( \Gamma_h )$,
  \begin{equation}
    \label{eq:Ih-bound}
    \norm{ \eta - I_h \eta }_{W^{m,q}(\T_h)}
    \le
    C h^{k+1-m}
    \norm{ \eta }_{W^{k+1,p}(\T_h)}.
  \end{equation}
\end{theorem}

\begin{proof}
  The proof follows by piecing together \cref{thm:IK-bound} using the fact that $\T_h$ is quasi-uniform.
\end{proof}

\subsection{Evolving surface finite elements} \label{sec:evolving-sfem}

Let $t\in [0,T]$. We consider  families of surface finite elements, spaces and triangulated hypersurfaces parametrised by $t$. 

\begin{definition}
  [Evolving surface finite element]
 \begin{deflist}
 \item \label{def:esfe}
   Let $(K(t), P(t), \Sigma(t))_{t \in [0,T]}$ be a time dependent family of surface finite elements (\cref{def:sfe}) parametrised over a common reference element  $(\hat{K}, \hat{P}, \hat\Sigma)$.  If the constant $C_K = \sup_t C_{K(t)}$ is uniformly bounded away from 1,
  \begin{equation*}
    C_K := \max_{t \in [0,T]} \sup_{\hat{x} \in \hat{K}} \norm{ \nabla D_K ( \hat{x}, t ) A_K^\dagger(t) }
    < c < 1,
  \end{equation*}
  we say that $(K(t), P(t), \Sigma(t))_{t\in[0,T]}$ is an \emph{evolving surface finite element}.
 \item \label{def:esfe-reference-map}
 Let  $\Phi_{(\cdot)}^K \in C^2([0,T], C^1(K_0))$ where $K_0:=K(0)$. We say that if $\Phi_t^K \colon K_0 := K(0) \to K(t)$ such that
\begin{equation}
  \label{eq:esfe-reference-map}
  F_{K(t)}( \hat{x} ) = \Phi_t^K( F_{K_0}(\hat{x} ) ) \qquad \mbox{ for } \hat{x} \in \hat{K}
\end{equation}
 then  $\Phi_t^K$ is the {\it flow}  defining the  evolution of the element domain and that $ F_{K(t)}$ is the {\it evolving surface element reference map}.
\item \label{def:element-velocity}
 The \emph{element velocity} $W_K$ of $K(t)$ is defined  by
\begin{equation*}
  W_K( \Phi_t^K( x ), t ) = \dt \Phi_t^K( x ) \qquad \mbox{ for } x \in K_0, t \in [0,T].
\end{equation*}

\item \label{def:e-Theta-sfe}
  If each $(K(t), \P(t), \Sigma(t))$ is a $\Theta$-surface finite element for each $t \in [0,T]$ and the constants $C_m(K(t))$ are uniformly bounded:
  \[
    \sup_{t \in [0,T]} C_m( K(t) ) \le c < \infty \qquad \mbox{ for } 2 \le m \le \Theta+1
  \]
  then we say that $(K(t), P(t), \Sigma(t))_{t\in[0,T]}$ is an \emph{evolving $\Theta$-surface finite element} and $F_{K(t)}$ is the evolving $\Theta$-surface element reference map.
  \item \label{def:temporally-quasi-uniform}
  We say that an evolving finite element domain is \emph{temporally quasi-uniform}, if there exists $\rho_K > 0$ such that
\begin{equation*}
  \inf \{ \rho_{K(t)} : t \in [0,T] \} \ge \rho_K \sup \{ h_{K(t)} : t \in [0,T] \}.
\end{equation*}
\item \label{def:element-push-forward-map}
The family of {\it element push forward maps} $\phi_t^K$, indexed by $t \in [0,T]$, is  defined to be the linear invertible map given for $\chi \colon K_0 \to \R$ by $\phi^K_t ( \chi ) \colon K(t) \to R$ where
\begin{align*}
  \phi_t^K( \chi )( x ) = \chi( \Phi_{-t}^K( x ) ) \qquad \mbox{ for } x \in K(t).
\end{align*}

  \end{deflist}
\end{definition}

\begin{lemma}
  \label{lem:Kt-norm-equiv}
  Let $F_{K(t)} \colon \hat{K} \to K(t)$, for $t \in [0,T]$, be an evolving $\Theta$-surface finite element reference map (\cref{def:e-Theta-sfe}) for a temporally quasi-uniform element domain $K(t)$ (\cref{def:temporally-quasi-uniform}) and $\phi^K_t$ the family of element push forward maps (\cref{def:element-push-forward-map}).
  Then there exists constant $c_1, c_2 > 0$ which depend only on the reference element domain $\hat{K}$, and the constant $C_K, C_1, \ldots, C_\Theta$ and $\rho_K$, such that for $0 \le m \le \Theta$ and all $t \in [0,T]$, $\chi \in W^{m,p}( K_0 )$ if, and only if, $\phi^K_t \chi \in W^{m,p}( K(t) )$ and
  \begin{align}
    \label{eq:Kt-norm-equiv}
    c_1 \norm{ \chi }_{W^{m,p}(K_0)}
    & \le \norm{ \phi^K_t \chi }_{W^{m,p}(K(t))}
      \le c_2 \norm{ \chi }_{W^{m,p}(K_0)}
    && \mbox{ for all } \chi \in W^{m,p}(K_0).
  \end{align}
\end{lemma}

\begin{proof}
  From \cref{lem:element-geometry,lem:norm-scaling}, we have
  \begin{align*}
    \abs{ \chi }_{W^{m,p}(K_0)}
    \le c \left( \frac{ \meas(K(t)) }{ \meas (K_0) } \right)^{1/p}
    \norm{ \phi^K_t \chi }_{W^{m,p}(K(t))}
    \sum_{r=1}^m \left( \frac{ h_{K(t)} }{ \rho_{K_0} } \right)^r
  \end{align*}
  and
  \begin{align*}
    \abs{ \phi^K_t\chi }_{W^{m,p}(K(t))}
    \le c \left( \frac{ \meas(K_0) }{ \meas( K(t) )} \right)^{1/p}
    \norm{ \chi }_{W^{m,p}(K_0)}
    \sum_{r=1}^m \left( \frac{ h_{K_0} }{ \rho_{K(t)} } \right)^r.
  \end{align*}
  It can be easily seen that for a quasi-uniform evolving surface finite element that these constants only depend on allowed quantities.
\end{proof}

  This result implies that $( W^{m,p}( K(t) ), \phi^K_t|_{W^{m,p}(K_0)} )_{t \in [0,T]}$ is a compatible pair (\cref{def:compatibility}) and in particular $( P(t), \phi^K_t|_{P(0)} )$ is a compatible pair when equipped with the $W^{m,p}( K(t) )$-norm (as $P(t)$ is a closed subspace of $W^{m,p}(K(t))$, \cref{rem:compatible-subsace,rem:surf-P-closed}).

  \subsection{Evolving surface finite element spaces}
We now formulate an  evolving surface finite element space forming  part of a compatible pair (in the sense of \cref{sec:abstract-formulation}, \cref{def:compatibility}).
For each $h \in (0,h_0)$, we are given a family of discrete hypersurfaces $\{ \Gamma_h(t) \}_{t \in [0,T]}$ and each equipped with a surface finite element space $\{ \S_h(t) \}_{t \in [0,T]}$.
Furthermore, we are interested in under what assumptions does the compatibility hold independently of the mesh size $h$.

\begin{definition}[Evolving triangulated hypersurface]
  For $t \in [0,T]$, let $\Gamma_h(t)$ be a family of triangulated hypersurfaces (\cref{def:triangulaed-hypersurface}) each equipped with a conforming subdivision $\T_h(t)$ (\cref{def:conforming-subdivision}) such that each element domain $K(t) \in \T_h(t)$ is equipped with an element flow map $\Phi_{(\cdot)}^K \in C^2( [0, T]; C^1( K_0 ))$ (\cref{def:esfe-reference-map}).

\begin{deflist}

\item \label{def:evolving-conforming-subdivision}
  We call $\{ \T_h(t) \}_{t \in [0,t]}$ an \emph{evolving conforming subdivision} if
  for each element $K(0) \in \T_h(0)$ and each facet $E(0)$ of $K(0)$
  either $E(0)$ is a facet of another element $K'(0) \in \T_h(0)$, in which case $E(t)$ is a common facet between $K(t)$ and $K'(t)$ for all $t \in [0,T]$ or
  $E(0)$ is a portion of the boundary $\partial \Omega(0)$, in which case $E(t)$ is a portion of the boundary $\partial \Omega(t)$ for all $t \in [0,T]$.
\item \label{def:evolving-triangulated-hypersurface}
An \emph{evolving triangulated hypersurface} is  defined to be a family of triangulated hypersurfaces $\{ \Gamma_h(t) \}_{t \in [0,T]}$ equipped with an evolving conforming subdivision. In this case, we define the mesh parameter $h$ to be
\begin{equation}
  \label{eq:def-h}
  h := \sup_{t \in [0,T]} \max_{K(t) \in \T_h(t)} h_{K(t)}.
\end{equation}
\item \label{def:global-discrete-flow}
We define a \emph{global discrete flow} $\Phi^h_{(\cdot)}(\cdot) \colon [0,T] \times \Gamma_{h,0} \to \R^{n+1}$ element-wise by
\begin{equation*}
  \Phi^h_t|_{K_0} := \Phi^K_t \qquad \mbox{ for } K_0 \in \T_h(0).
\end{equation*}
Our assumptions imply that $\Phi^h_t$ is piecewise smooth and $\Phi^h_t \colon \Gamma_{h,0} \to \Gamma_h(t)$.

\item \label{def:global-discrete-velocity}
  We define a \emph{global discrete velocity}  ${W}_h$ given by
\begin{equation*}
  {W}_h|_{K(t)} = {W}_K.
\end{equation*}
\item \label{def:global-push-forward-map}
 The  family of linear homeomorphisms  induced by  the flow $\Phi^h_t$ and called the \emph{global push forward map}  
 is denoted by $\phi^h_t( \eta_h ) \colon \Gamma_h(t) \to \R$ for $\eta_h \colon \Gamma_{h,0} \to \R$ and defined by
\begin{equation*}
  ( \phi^h_t \eta_h )|_{K(t)} := \phi^K_t( \eta_h|_{K_0} ) \quad \mbox{ for all } K(t) \in \T_h(t).
\end{equation*}
\end{deflist}
\end{definition}

Again, in order to bring together a collection of evolving surface finite elements we restrict to Lagrangian surface finite elements.
For each $t \in [0,T]$, we denote by $\N_h(t)$ the global set of Lagrange nodes of $\Gamma_h(t)$ \cref{eq:Nh} and for any $a \in \N_h(t)$, $\T(a)$ is the set of elements $K(t)$ such that $a$ is a node of $K(t)$.
We make the further restriction that the global flow is single-valued at each Lagrange point:
for all $a_0 \in \N_h(0)$ we have
\begin{equation}
  \label{eq:surf-evolve-node-agree}
  \Phi_t^K( a_0 ) = \Phi_t^{K'}( a_0 ) \qquad \mbox{ for all } K, K' \in \T( a_0 ).
\end{equation}

\begin{definition}[Evolving surface finite element space]
  \begin{deflist}
    \item \label{def:evolving-sfe-space}
  Let $\{ \Gamma_h(t) \}_{t \in [0,T]}$ be an evolving triangulated hypersurface (\cref{def:evolving-triangulated-hypersurface}) equipped with an evolving conforming subdivision $\{ \T_h(t) \}_{t \in [0,T]}$ (\cref{def:evolving-conforming-subdivision}).
  For $t \in [0,T]$, let $\S_h(t)$ be a surface finite element space (\cref{def:sfe-space}) over $\Gamma_h(t)$.
  If each $K(t) \in \T_h(t)$ is equipped with an evolving surface finite element $(K(t), \P(t), \Sigma(t) )_{t \in [0,T]}$ (\cref{def:esfe}) then we say $\{ \S_h(t) \}_{t \in [0,T]}$ is an \emph{evolving surface finite element space}.
\item \label{def:evolving-basis-function}
  For each $t \in [0,T]$, we will write $\Sigma_h(t)$ for the set of global nodal variables (\cref{def:global-degrees-of-freedom}). We will use the convention that
  \begin{equation*}
    \Sigma_h(t) = \{ \chi_h \mapsto \chi_h( a_i(t) ) : 1 \le i \le N \},
  \end{equation*}
  where $a_i(t)$ is the trajectory of a Lagrange point under the global flow $\Phi^h_t$. We will denote by $\{ \chi_i(\cdot,t) : 1 \le i \le N \}$ the global basis of finite element functions such that $\chi_i ( a_j(t), t ) = \delta_{ij}$ for $t \in [0,T]$ and all $i,j = 1, \ldots, N$. This implies that $\chi_i( \cdot, t ) = \phi^h_t ( \chi_i( \cdot, 0) )$. See also \cref{eq:abs-basis-def}.
\end{deflist}
\end{definition}

\begin{definition}
  [Uniformly regular and uniformly quasi-uniform evolving subdivisions]
  For $h \in (0,h_0)$, let $\{ \T_h(t) \}_{t \in [0,T]}$ be a family of evolving conforming subdivisions (\cref{def:evolving-conforming-subdivision}).
  \begin{deflist}
  \item \label{def:uniformly-regular}
    We say that the family is \emph{uniformly regular} if there exists $\rho > 0$ such that for all $h \in (0, h_0)$ and all times $t \in [0,T]$, we have
    \begin{equation*}
      \rho_{K(t)} \ge \rho h_{K(t)} \qquad \mbox{ for all } K(t) \in \T_h(t),
    \end{equation*}
    and there exists $C > 0$ such that
    \[
      \sup_{h \in (0,h_0)} \sup_{t \in [0,T]} \max_{K(t) \in \T_h(t)} C_{K(t)} \le C < 1.
    \]
  \item \label{def:uniformly-Theta-regular}
    We say that the family is \emph{uniformly $\Theta$-regular} if it is uniformly regular, if for each time $t \in [0,T]$, the family $\{ \T_h(t) \}_{h \in (0,h_0)}$ is $\Theta$-regular and if there exists a constant such that
    \[
      \sup_{h \in (0,h_0)} \sup_{t \in (0,T)} \max_{K(t) \in \T_h(t)} C_m(K(t)) \le C \le + \infty \qquad
      \mbox{ for } 2 \le m \le \Theta+1.
    \]
    This implies that $F_{K(t)}$ is an evolving $\Theta$-reference finite element map for each $K(t) \in \T_h(t)$.
  \item \label{def:uniformly-quasi-uniform}
    We say that the family is \emph{uniformly quasi-uniform} if there exists $\rho > 0$ such that for all $h \in (0,h_0)$ and all times $t \in [0,T]$, we have
    \begin{equation*}
      \min \{ \rho_{K(t)} : K(t) \in \T_h(t) \}
      \ge \rho h.
    \end{equation*}
  \end{deflist}
\end{definition}

Note that a uniformly quasi-uniform subdivision consists of element domains for temporally quasi-uniform evolving surface finite element domains.

\begin{lemma}
  \label{lem:abs-Vh-compatible}
  For $h \in (0,h_0)$,
  let $\{ \T_h(t) \}_{t \in [0,T]}$ be a uniformly $\Theta$-regular (\cref{def:uniformly-Theta-regular}), uniformly quasi-uniform (\cref{def:uniformly-quasi-uniform}), evolving, conforming subdivision (\cref{def:evolving-conforming-subdivision}) and let $\phi^h_t$ be the global push-forward map (\cref{def:global-push-forward-map}).
  Let $0 \le k \le \Theta+1$, $p \in [0,\infty]$.
  Then  $\eta_h \in W^{m,p}( \T_h(0) )$ if and only if $\phi^h_t \eta_h \in W^{m,p}( \T_h(t) )$ for all $t \in [0,T]$.
  Furthermore, there exists $c_1, c_2 > 0$ independent of $h \in (0,h_0)$ and $t \in [0,T]$ such that for all $\eta_h \in W^{m,p}( \T_h(t) )$
  \begin{align}
    c_1 \norm{ \eta_h }_{W^{m,p}(\T_h(0))}
    \le \norm{ \phi^h_t \eta_h }_{W^{k,p}( \T_h(t) )}
    \le c_2 \norm{ \eta_h }_{W^{m,p}(\T_h(0))}.
  \end{align}

\end{lemma}

{
  \begin{remark}
    \label{rem:surf-Wmp-etc-compat}
  This implies that $( W^{m,p}( \T_h(t) ), \phi^h_t|_{W^{m,p}(\T_{h}(0))} )_{t \in [0,T]}$ is a compatible pair.
  In particular the pairs $( \S_h(t), \phi^h_t )_{t \in [0,T]}$, equipped with the broken Sobolev norm $\norm{ \cdot }_{W^{m,p}(\T_h(t))}$ (\cref{def:broken-sobolev-norm}), and $( W^{1,p}_T( \T_h(t) ), \phi^h_t )_{t \in [0,T]}$ are compatible since each is a closed subspace (\cref{rem:compatible-subsace,lem:Sh-closed-subspace,lem:W1pT-closed-subspace}).
  \end{remark}
}

\begin{proof}
  We simply sum the element-wise result from \cref{lem:Kt-norm-equiv}. The constants are independent of $h_K$ and $\rho_K$ due to the uniform quasi-uniformity of $\{ \T_h(t) \}$.
\end{proof}

Note that this result implies that the spaces $L^2_{\S_h}$ and $C^1_{\S_h}$ are well defined when equipped with the appropriate norms (c.f. \cref{eq:L2X} and \cref{eq:CkX}).

\section{Lifted surface finite element spaces}
\label{sec:lift-fem}

So far we have only defined  surface finite elements without relation to approximation of a smooth hypersurface $\Gamma$ 
and function spaces on $\Gamma$ through the definition of a lifted finite element space $\S_h^\ell(t)$. In general, due to the curvature of the surface, the computational domain $\Gamma_h$ will be an approximation  to $\Gamma$. We will identify  surface finite elements  $( K, \P, \Sigma )$ on $\Gamma_h$ with a corresponding curved  surface finite elements  $( K^\ell, \P^\ell, \Sigma^\ell )$ on $\Gamma$ using a mapping  $\Lambda_K \colon K \to \Gamma$. We call this process \emph{lifting}. This provides a convenient way to  compare the smooth functions (solutions to PDEs) on $\Gamma$  with the lift of discrete functions on $\Gamma_h$.
We will also require an inverse lift that maps functions on the smooth domain to the computational domain.

In the following, we will answer the question what assumptions on $\Lambda_K$ must be made so that $( K^\ell, \P^\ell, \Sigma^\ell )$ is a $\Theta$-surface finite element and how we can put several such \emph{lifted} surface finite elements together in order to recover  $\Gamma$ and 
finite element spaces on $\Gamma$.
We will also explore some interpolation results when these assumptions hold.
Finally, we explore how we can lift the discrete push forward map.

\subsection{Lifted surface finite element}

We consider the situation of a surface finite element $( K, P, \Sigma )$ (\cref{def:sfe}), with element reference map $F_K \colon \hat{K} \to K$, where the element domain approximates a portion of $\Gamma$.

Let $\Lambda_K \colon K \to \Gamma$ be a $C^1$-map which is a diffeomorphism onto its image.
For a function $\chi \colon K \to \R$, we call $\chi^\ell$ defined by $\chi^\ell( \Lambda_K( \cdot ) ) = \chi( \cdot )$ the \emph{lift} of $\chi$.
We will assume that we can decompose $\Lambda_K$ into
\[
  \Lambda_K( x )  = A_{\Lambda} x + b_\Lambda + \tilde{\Lambda}_K(x) \qquad \mbox{ for } x \in K,
\]
where $A_{\Lambda}$ is an invertible $(n+1) \times (n+1)$ matrix, $b_\Lambda \in \R^{n+1}$ and $\tilde{\Lambda}_K \in C^1( K, \R^{n+1} )$.
We will assume that $\tilde{\Lambda}_K$ does affect the affine part of the parametrisations:
\[
  A_{K^\ell} \hat{x} + b_{K^\ell} = A_\Lambda ( A_K \hat{x} + b_K ) + b_{\Lambda} \quad \mbox{ for all } \hat{x} \in \hat{K}.
\]

\begin{definition}[Lifted surface finite element]
  \label{def:lifted-sfe}
We call the triple $(K^\ell, \P^\ell, \Sigma^\ell)$ defined by
\begin{align*}
  & K^\ell := \Lambda_K( K ) \subset \Gamma \\
  & \P^\ell := \{ \chi^\ell( \Lambda_K( \cdot ) ) := \chi( \cdot ) : \chi \in \P \} \\
  & \Sigma^\ell := \{ \sigma^\ell := \chi^\ell \mapsto \sigma( \chi ) : \sigma \in \Sigma \},
\end{align*}
the \emph{lift} of $(K, \P, \Sigma)$ and $\Lambda_K$ the \emph{lifting map}.
If $(K^\ell, \P^\ell, \Sigma^\ell)$ forms a surface finite element over $(\hat{K}, \hat{P}, \hat{\Sigma})$ then we say that $(K^\ell, \P^\ell, \Sigma^\ell)$ is the \emph{lifted surface finite element associated with} $(K, \P, \Sigma)$.
In this case we call $F_{K^\ell}( \cdot ) = \Lambda_K( F_K( \cdot ) )$ the \emph{lifted surface finite element reference map}.
\end{definition}

The next two results show under what assumptions on $\Lambda_K$ is $(K^\ell, \P^\ell, \Sigma^\ell)$ a surface finite element (\cref{def:sfe}) or a $\Theta$-surface finite element (\cref{def:Theta-sfe}).

\begin{lemma}
  \label{lem:LambdaK-equiv-1}
  If $\Lambda_K$ satisfies that
  \begin{equation}
    \label{eq:LambdaK-ass1}
    \sup_{x \in K} \norm{ \nabla_K \tilde{\Lambda}(x) } \le
    \frac{ \norm{ A_{\Lambda}^{-1} }^{-1} - C_K \norm{ A_{\Lambda} } }{ 1 + C_K },
  \end{equation}
  then $(K^\ell, \P^\ell, \Sigma^\ell)$ is a surface finite element.
  Furthermore,
  for $p \in [1,\infty]$, we have that there exists $c_1, c_2 > 0$ such that
  \begin{equation}
    \label{eq:lift-equiv}
    \begin{aligned}
      c_1 \norm{ \chi }_{L^p(K)}
      & \le \norm{ \chi^\ell }_{L^p(K^\ell)}
      \le c_2 \norm{ \chi }_{L^p(K)}
      && \mbox{ for all } \chi \in L^p(K) \\
      c_1 \norm{ \chi }_{W^{1,p}(K)}
      & \le \norm{ \chi^\ell }_{W^{1,p}(K^\ell)}
      \le c_2 \norm{ \chi }_{W^{1,p}(K)}
      && \mbox{ for all } \chi \in W^{1,p}(K),
    \end{aligned}
  \end{equation}
  where the constants $c_1, c_2$ depend on $C_K$, $\norm{ A_{\Lambda}^{-1} }$ and the ratio $h_K / \rho_K$.
\end{lemma}

\begin{proof}
  To show that $( K^\ell, P^\ell, \Sigma^\ell )$ is a surface finite element (\cref{def:sfe}),
  the conditions on the element reference map $F_{K^\ell}(\hat{x}) = \Lambda_K( F_K(\hat{x}) )$ are clear and we are left to check the curvedness condition \cref{eq:CK-defn}.
  Using the expansion of $\Lambda_K$, we see
  \begin{multline*}
    A_{K^\ell} = A_{\Lambda} A_K, \quad A_{K^\ell}^\dagger = A_K^\dagger A_{\Lambda}^{-1},
    \quad b_{K^\ell} = A_{\Lambda} b_K + b_{\Lambda} \\
    \quad \mbox{ and } \quad
    D_{K^\ell}(\hat{x}) = A_\Lambda D_K(\hat{x}) + \tilde{\Lambda}_K( F_K(\hat{x}) ).
  \end{multline*}
  So that
  \begin{align*}
    \nabla D_{K^\ell}(\hat{x}) A_{K^\ell}^\dagger
    = A_{\Lambda} \nabla D_K( \hat{x} ) A_K^\dagger A_{\Lambda}^{-1}
    + \nabla_K \tilde{\Lambda}_K( F_K(\hat{x}) ) ( \id + \nabla D_K(\hat{x}) A_K^\dagger ) A_K A_K^\dagger A_\Lambda^{-1}.
  \end{align*}
  Applying the curvedness condition for $K$ and the fact that $A_K A_K^\dagger$ is a projection (\cref{rem:right-inv-proj} applied to the flat element $\tilde{K}$) we see
  \[
    \norm{\nabla D_{K^\ell}(\hat{x}) A_{K^\ell}^\dagger}
    \le C_K \norm{ A_{\Lambda} } \norm{ A_{\Lambda}^{-1} }
    + (1 + C_K ) \norm{ \nabla_K \tilde{\Lambda}_K( F_K(\hat{x})) } \norm{ A_{\Lambda}^{-1} }.
  \]
  The curvedness condition is shown by applying \cref{eq:LambdaK-ass1}.

  To show \cref{eq:lift-equiv} the result is clear for $p = \infty$.
  For $p < \infty$, we will apply \cref{lem:norm-scaling}. Then we see
  \begin{align*}
    c \left(\frac{\meas(K)}{\meas(K^\ell)}\right)^{1/p} \norm{ \chi }_{L^p(K)}
    \le \norm{ \chi^\ell }_{L^p(K^\ell)}
    \le c \left(\frac{\meas(K^\ell)}{\meas(K)}\right)^{1/p} \norm{ \chi }_{L^p(K)},
  \end{align*}
  where $c$ depends on $C_K$ and the bound \cref{eq:LambdaK-ass1}.
  For the higher order bound, we note that
  \[
    \norm{ A_K } = \norm{ A_{\Lambda}^{-1} A_{\Lambda} A_K }
    \le \norm{ A_\Lambda^{-1} } \norm{ A_{\Lambda} A_K } = \norm{ A_{\Lambda^{-1}} } \norm{ A_{K^\ell} },
    \]
  so that we infer
  \begin{equation}
    \label{eq:AK-AKell-bound}
    \norm{ A_{K^\ell} } \ge \frac{ \norm{ A_K } }{ \norm{ A_{\Lambda^{-1}} } }.
  \end{equation}
  Then applying \cref{lem:norm-scaling} once more we see
  \begin{multline*}
    c \left(\frac{\meas(K)}{\meas(K^\ell)}\right)^{1/p}
    \frac{ \norm{ A_\Lambda^{-1} } }{ \norm{ A_K } \norm{ A_K^\dagger } }
    \abs{ \chi }_{W^{1,p}(K)}
    \le \norm{ \chi^\ell }_{W^{1,p}(K^\ell)} \\
    \le c \left(\frac{\meas(K^\ell)}{\meas(K)}\right)^{1/p}
    \norm{ A_\Lambda^{-1} } \norm{ A_K } \norm{ A_K^\dagger }
    \abs{ \chi }_{W^{1,p}(K)}.
  \end{multline*}
  The final result is given by applying \cref{eq:AK-h-relation-a,eq:AK-h-relation-b,eq:lift-meas-scale}.
\end{proof}

\begin{lemma}
  \label{lem:LambdaK-equiv-theta}
  If $(K, \P, \Sigma)$ is a $\Theta$-surface finite element and $\Lambda_K$ satisfies $\Lambda_K \in C^{\Theta+1}( K; \R^{n+1} )$ and $\norm{ \Lambda_K }_{W^{\Theta+1,\infty}(K)} \le c$,
  then $(K^\ell, \P^\ell, \Sigma^\ell)$ is also a $\Theta$-surface finite element.
  Furthermore, for $0 \le m \le \Theta+1$ and $p \in [1,\infty]$, we have that there exists $c_1, c_2 > 0$ such that
  \begin{align*}
    c_1 \norm{ \chi }_{W^{m,p}(K)}
    & \le \norm{ \chi^\ell }_{W^{m,p}(K^\ell)}
      \le c_2 \norm{ \chi }_{W^{m,p}(K)}
    && \mbox{ for all } \chi \in W^{m,p}(K),
  \end{align*}
  where the constants $c_1, c_2$ depend on $C_1(K), \ldots, C_m(K)$, $\norm{ A_{\Lambda}^{-1} }$ and the ratio $h_K / \rho_K$.
\end{lemma}

\begin{proof}
  To see that $(K^\ell, \P^\ell, \Sigma^\ell )$ is a $\Theta$-surface finite element (\cref{def:Theta-sfe}),
  the first three conditions are clear from the smoothness assumption on $\Lambda_K$.
  We must show \cref{eq:CK-theta-defn} holds: for $2 \le m \le \Theta+1$, there exists a constant $C_m(K^\ell)$ such that
  \[
    \norm{ \nabla^m F_{K^\ell}(\hat{x}) } \le C_m(K^\ell) \norm{ A_{K^\ell} }^m.
  \]
  Computing directly using the Fa\'a di Bruno formula \cref{eq:faa-di-bruno}, we have
  \begin{align*}
    \frac{\partial^{r'}}{\partial \hat{x}_{j_{m}} \ldots \partial \hat{x}_{j_1}} F_{K^\ell}(\hat{x})
    = \sum_{r=1}^{m} \sum_{\lambda_1, \ldots, \lambda_1=1}^{n+1} ( \nabla_K )_{\lambda_r} \cdots ( \nabla_K )_{\lambda_1} \Lambda_K( F_K(\hat{x}) )
    \sum_{\Pp( m, r ) } \prod_{s=1}^r \frac{\partial^{\abs{\sigma_s}}}{\partial \hat{x}_{j_{\sigma_s}}} (F_K)_{\lambda_s}( \hat{x} ).
  \end{align*}
  But applying the fact that $(K, \P, \Sigma)$ is a $\Theta$-surface finite element and the smoothness assumption on $\Lambda_K$, we see that
  \begin{align*}
    \abs{ \nabla^{m} F_{K^\ell}( \hat{x} ) }
    \le c( C_1(K), \ldots, C_m(K) ) \sum_{r=1}^{m} \abs{ \Lambda_K }_{W^{r,\infty}(K)} \norm{ A_K }^{m}.
  \end{align*}
  Finally, applying \cref{eq:AK-AKell-bound}, we have that
  \begin{align*}
    \norm{ \nabla^{m} F_{K^\ell}( \hat{x} ) }
    \le \norm{ \Lambda_K }_{W^{m,\infty}(K)} \norm{ A_{K^\ell} }^{m} \norm{ A_{\Lambda}^{-1} }^{m}.
  \end{align*}
  This shows that $(K^\ell, \P^\ell, \Sigma^\ell)$ is a $\Theta$-surface finite element.

  To show the norm equivalence we again apply \cref{lem:norm-scaling}, recognising the geometric progression, to see
  \begin{multline*}
    c \left(
    \frac{ \meas(K) }{ \meas(K^\ell) }
    \right)
    \frac{ A_1^{m+1} - 1 }{A_1 - 1}
    \norm{ \chi }_{W^{m,p}(K)}
    \le
    \norm{ \chi^\ell }_{W^{m,p}(K^\ell)} \\
    \le
    c \left(
    \frac{ \meas(K^\ell) }{ \meas(K) }
    \right)
    \frac{ A_2^{m+1} - 1 }{A_2 - 1}
    \norm{ \chi }_{W^{m,p}(K)},
  \end{multline*}
  where
  \begin{equation*}
    A_1 = \norm{ A_K } \norm{ A_K^\dagger } \norm{ A_\Lambda } \quad \mbox{ and } \quad
    A_2 = \norm{ A_K } \norm{ A_K^\dagger } \norm{ A_\Lambda^{-1} }.
  \end{equation*}
  The final result is given by applying \cref{eq:AK-h-relation-a,eq:AK-h-relation-b,eq:lift-meas-scale}.
\end{proof}

\begin{remark}
  \label{rem:id-lift-ass}
  Considering the case that $\Lambda_K$ fixes the vertices of $K$ then we can write $\Lambda_K$ as
  \[
    \Lambda_K( x ) = x + \tilde{\Lambda}(x) \quad \mbox{ for } x \in K.
  \]
  That is that $A_\Lambda = \id$ and $b_{\Lambda} = 0$.
  Then the assumptions of \cref{lem:LambdaK-equiv-1} can be replaced by
  \[
    \sup_{x \in K} \norm{ \nabla_K \tilde\Lambda(x) }
    \le \frac{1 - C_K}{1 + C_K},
  \]
  and the assumptions of \cref{lem:LambdaK-equiv-theta} can be replaced by the assumption that $\norm{ \Lambda_K }_{W^{\Theta,\infty}(K)}$ is uniformly bounded.
\end{remark}

We next relate the geometry of the base and lifted element domains.
\begin{lemma}
  \label{lem:lift-geom}
  Using the decomposition of $\Lambda_K$, we have
  \begin{subequations}
    \begin{align}
      \label{eq:lift-h-scale}
      \norm{ A_{\Lambda}^{-1} }^{-1} h_K \le h_{K^\ell}
      & \le \norm{ A_\Lambda } h_K \\
      \label{eq:lift-rho-scale}
      \norm{ A_{\Lambda}^{-1} }^{-1} \rho_K \le \rho_{K^\ell}
      & \le \norm{ A_{\Lambda} } \rho_K \\
      \label{eq:lift-meas-scale}
      c_1 \left( \frac{\rho_K}{h_K} \right)^n \norm{ A_{\Lambda}^{-1} }^{-n} \meas(K)
      \le \meas(K^\ell)
      & \le c_2 \left( \frac{h_K}{\rho_K } \right)^n \norm{ A_\Lambda^{-1} }^{-n} \meas(K).
    \end{align}
  \end{subequations}
\end{lemma}

\begin{proof}
  For \cref{eq:lift-h-scale}, we show the second inequality. The first follows by the same reasoning applied to the inverse of $A_{\Lambda}$.
  Since $\tilde{K}^\ell$ is compact, there exists $x^\ell, y^\ell \in \tilde{K}^\ell$ such that
  \[
    h_{K^\ell} = \abs{ y^\ell - x^\ell }.
  \]
  But since $A_\Lambda$ is invertible, there exists $x, y \in \tilde{K}$ such that
  \[
    A_{\Lambda} x + b_\Lambda = x^\ell \qquad \mbox{ and } \qquad
    A_{\Lambda} y + b_{\Lambda} = y^\ell.
  \]
  Then, we can compute that
  \[
    h_{K^\ell} = \abs{ x^\ell - y^\ell }
    = \abs{ A_\Lambda ( x - y ) }
    \le \norm{ A_{\Lambda} } \abs{ x - y }
    \le \norm{ A_{\Lambda} } h_K.
  \]

   Similarly, for \cref{eq:lift-rho-scale}, we only show the first inequality.
  For each $\eps > 0$, there exists $B_\eps$ a ball in $\tilde{K}$ such that
  \[
    \rho_K - \diam B_\eps \le \eps.
  \]
  Denote by $x_\eps$ and $r_\eps$ the centre and radius of $B_\eps$ respectively. Consider the affine map $\Xi_\eps \colon \R^{n+1} \to \R^{n+1}$ given by
  \[
    \Xi_\eps(x) = \frac{1}{\norm{ A_{\Lambda}^{-1} } } ( x - x_\eps ) + b_\Lambda + A_\Lambda x_\eps.
  \]
  It is clear that $\Xi_\eps$ maps balls to balls and $B_\eps$ is mapped to a ball centred at $A_{\Lambda} x_\eps + b_\Lambda$ with radius $r_\eps / \norm{ A_{\Lambda}^{-1} }$.
  We claim $\Xi_\eps( B_\eps )$ is contained in $\tilde{K}^\ell$.
  Indeed, take $x^\ell \in \Xi_\eps( B_\eps )$, then there exists $x \in B_\eps$ such that $\Xi_\eps( x ) = x^\ell$. Denote by
  \[
    x' = x_\eps + \frac{1}{\norm{ A_{\Lambda}^{-1} } } A_\Lambda^{-1} ( x - x_\eps ),
  \]
  then
  \[
    \abs{ x' - x_\eps }
    = \frac{1}{\norm{ A_{\Lambda}^{-1} } } \abs{ A_\Lambda^{-1} ( x - x_\eps ) }
    \le \abs{ x - x_\eps } \le r_\eps,
  \]
  so $x' \in B_\eps \subset \tilde{K}$ and
  \[
    A_\Lambda x' + b_\Lambda
    = \frac{1}{\norm{ A_{\Lambda}^{-1} } } ( x - x_\eps ) + A_\Lambda x_\eps + b_\Lambda
    = x^\ell,
  \]
  so that $x^\ell \in \tilde{K}^\ell = A_\Lambda \tilde{K} + b_K$.

  Therefore, we have found a ball $\Xi_\eps( B_\eps ) \subset \tilde{K}^\ell$ with radius $r_\eps / ( \norm{ A_{\Lambda}^{-1} } )$, thus we can infer that
  \[
    \rho_{K^\ell} \ge \frac{1}{2} \norm{ A_{\Lambda}^{-1} }^{-1} \diam B_\eps
    \ge \norm{ A_{\Lambda}^{-1} }^{-1} ( \rho_K - \eps ).
  \]
  Since the proof holds for all $\eps > 0$, we see the desired result.

  Given \cref{eq:lift-h-scale} and \cref{eq:lift-rho-scale}, the final result \cref{eq:lift-meas-scale} follows directly from \cref{rem:AK-h-meas}.
\end{proof}

We can also use $\Lambda_K$ to define an inverse lift.
Given an element domain $K$, lifted element domain $K^\ell$ and lifting map $\Lambda_K$.
We know that $\Lambda_K$ is invertible onto its image, namely $K^\ell$.
So for $\eta \colon K^\ell \to \R$, we denote its inverse lift by $\eta^{-\ell} \colon K \to \R$ defined by
\[
  \eta^{-\ell}( x ) := \eta( \Lambda_K( x ) ) \quad \mbox{ for } x \in K.
\]
\begin{lemma}
  If the assumptions of \cref{lem:LambdaK-equiv-1,lem:LambdaK-equiv-theta} hold then,
  for $0 \le m \le \Theta + 1$ and $p \in [1,\infty]$, we have that there exists $c_1, c_2 > 0$ such that for all $\eta \in W^{m,p}(K^\ell)$, we have
  \begin{align*}
    c_1 \norm{ \eta^{-\ell} }_{W^{m,p}(K)}
    & \le \norm{ \eta }_{W^{m,p}(K^\ell)}
      \le c_2 \norm{ \eta^{-\ell} }_{W^{m,p}(K)}.
  \end{align*}
\end{lemma}

\begin{proof}
  The same proof can be applied to the case of the inverse lift as well as the usual lift.
\end{proof}

\subsection{Lifted surface finite element space}

Let $\Gamma$ be a smooth hypersurface and for $h \in (0,h_0)$ and let  $\Gamma_h$ be a triangulated hypersurface (\cref{def:triangulaed-hypersurface}) equipped with a conforming subdivision $\T_h$ (\cref{def:conforming-subdivision}) and a surface finite element space $\S_h$ (\cref{def:sfe-space}). Let each $K \in \T_h$ be associated with a lifted finite element $K^{\ell}$ with lifted map $\Lambda_K$.

\begin{definition}
\begin{deflist}
\item
  \label{def:exact-decomposition}
 We denote by $\T_h^\ell$ the set of all lifted element domains
\[
  \T_h^\ell := \bigcup_{K \in \T_h} K^\ell.
\]
If  the global map $\Lambda_h$ is single valued and $\T_h^\ell$ forms a conforming subdivision of $\Gamma$, we say that $\T_h^\ell$ is an \emph{exact} subdivision of the surface $\Gamma$.
In this case we set $\F_h^\ell$ to be the set of facets between adjacent elements.
\item
  \label{def:global-lifting-map}
We  define a \emph{global lifting map} $\Lambda_h \colon \Gamma_h \to \Gamma$ by $\Lambda_h|_{K} = \Lambda_K$.
We define  the inverse lift $\Lambda_h^{-1} \colon \Gamma \to \Gamma_h$ in a similar element-wise fashion 
by  $\Lambda^{-1}_h|_{K^\ell} = \Lambda_K^{-1}$.
\item
  We denote by $\lambda_h$ the \emph{global lift} given for $\eta_h \colon \Gamma_h \to \R$ by
  \[
    \lambda_h( \eta_ h )( \Lambda_h( x  ) ) = \eta_h( x ) \mbox{ for } x \in \Gamma_h,
  \]
  and by $\varsigma_h$ the \emph{global inverse lift} given for $\eta \colon \Gamma \to \R$ by
  \[
    \varsigma_h( \eta )( x ) = \eta( \Lambda_h( x ) ) \mbox{ for } x \in \Gamma_h.
  \]
  We will also use the notation $\eta_h^\ell = \lambda_h( \eta_h )$ and $\eta^{-\ell} = \varsigma_h( \eta )$.
\item
  \label{def:lifted-sfe-space}
  If for each surface finite element $( K, P, \Sigma )$ there is an associated lifted surface finite element $(K^\ell, P^\ell, \Sigma^\ell)$, then, we define a \emph{lifted surface finite element space} (c.f.\ \cref{eq:abs-Vhell}) by
\[
  \S_h^\ell := \left\{ \chi_h^\ell : \chi_h \in \S_h \right\}.
\]
\end{deflist}
\end{definition}

\begin{proposition}
  \label{prop:lift-Vh}
  Assume additionally that the family of triangulations $\{ \T_h \}_{h \in (0,h_0)}$ is regular (\cref{def:regular}),
  $\Lambda_K$ satisfies the assumptions of \cref{lem:LambdaK-equiv-1} and there exists $C_1, C_2 > 0$ such that
  \begin{equation}
    \label{eq:as-lift-linear-bounded}
    \norm{ A_\Lambda } \le C_1 \quad \mbox{ and } \quad \norm{ A_{\Lambda}^{-1} } \le C_2
    \quad \mbox{ for all } K \in \T_h \mbox{ for all } h \in (0,h_0),
  \end{equation}
  for all $K \in T_h$ for all $h \in (0,h_0)$.
  Then $\{ \T_h^\ell \}_{h \in (0,h_0)}$ is regular and $\S_h^\ell$ is a surface finite element space and there exists constants $c_1, c_2 > 0$, which are independent of $h \in (0,h_0)$ such that
  \begin{equation}
    \begin{aligned}
      c_1 \norm{ \eta_h }_{L^p(\T_h)} & \le
      \norm{ \eta_h^\ell }_{L^p(\T_h^\ell)}
      \le c_2 \norm{ \eta_h }_{L^p(\T_h)}
      && \mbox{ for all } \eta_h \in L^p( \T_h )\\
      c_1 \norm{ \eta_h }_{W^{1,p}(\T_h)} & \le
      \norm{ \eta_h^\ell }_{W^{1,p}(\T_h^\ell)}
      \le c_2 \norm{ \eta_h }_{W^{1,p}(\T_h)}
      && \mbox{ for all } \eta_h \in W^{1,p}( \T_h ).
    \end{aligned}
\end{equation}
  Furthermore, if $\Lambda_K$ satisfies the assumptions of \cref{lem:LambdaK-equiv-theta} for all $K \in \T_h$ and all $h \in (0,h_0)$,
  then for $0 \le m \le \Theta$, there exists $c_1, c_2 > 0$ independent of $h \in (0,h_0)$ such that
  \begin{equation}
    c_1 \norm{ \eta_h }_{W^{m,p}(\T_h)} \le
    \norm{ \eta_h^\ell }_{W^{m,p}(\T_h^\ell)}
    \le c_2 \norm{ \eta_h }_{W^{m,p}(\T_h)}
    \quad \mbox{ for all } \eta_h \in W^{m,p}( \T_h ).
  \end{equation}
\end{proposition}

\begin{proof}
  The regularity of the family of triangulations $\{ \T_h \}_{h \in (0,h_0)}$ follows from \cref{eq:lift-h-scale} and \cref{eq:lift-rho-scale} from assumption \cref{eq:as-lift-linear-bounded}.

  The results follow by combining the previous results for each element in $\T_h$. We achieve bounds independently of $h \in (0,h_0)$ since the regularity of the subdivisions implies that $h_K / \rho_K$ is independent of $h \in (0,h_0)$.
\end{proof}

{

  \begin{lemma}
    \label{lem:W1pThell-exact}
    Let $\T_h^\ell$ be an exact decomposition of $\Gamma$ then $W^{1,p}_T( \T_h^\ell ) = W^{1,p}( \Gamma )$.
  \end{lemma}

  \begin{proof}
    First, let $\eta \in W^{1,p}_T( \T_h )$. Then we have that $\eta \in L^p( \Gamma )$ and it is left to show that $\eta$ has a weak derivative in $L^p( \Gamma )$ \citep[Def.~2.11]{DziEll13a}.
    We have a candidate $\xi$ given element-wise by $\xi|_{K^\ell} = \nabla_\Gamma (\eta)|_{K^\ell}$.
    It is clear that $\xi \in L^p( \Gamma )$ and for $\varphi \in C^1( \Gamma )$ with compact support and $i = 1, \ldots, n+1$, we have
    \begin{align*}
      \int_\Gamma \eta \underline{D}_i \varphi
      & = \sum_{K^\ell \in \T_h^\ell} \int_{K^\ell} \eta \underline{D}_i \varphi \\
      & = \sum_{K^\ell \in \T_h^\ell} \left( -\int_{K^\ell} \xi_i \varphi
        + \int_{K^\ell} \eta \varphi H \nu_i + \int_{\partial K^\ell} \eta \varphi \mu_i  \right) \\
      & = - \int_\Gamma \xi_i \varphi + \int_\Gamma \eta \varphi H \nu_i
        + \sum_{K^\ell \in \T_h^\ell} \int_{\partial K^\ell} \eta \varphi \mu_i.
    \end{align*}
    We note that we can write
    \[
      \sum_{K^\ell \in \T_h^\ell} \int_{\partial K^\ell} \eta \varphi \mu_i
      = \sum_{F^\ell \in \F_h^\ell} \int_{F^\ell} \eta \varphi ( \mu_{F^\ell}^+ + \mu_{\F^\ell}^-)_i = 0,
    \]
    where $\F_h^\ell$ is the set of facets between adjacent elements in $\T_h^\ell$, since the traces of $\eta |_{K'}$ and $\eta|_{(K^\ell)'}$ to $K^\ell \cap ( K^\ell )'$ coincide for any adjacent pair $K, K'$.
    We note that the sum over edges is zero since the conormals from adjacent elements are equal and opposite in the exact triangulation and we see that $\xi$ is the weak derivative of $\eta$.

    Second, let $\eta \in W^{1,p}( \Gamma )$ then it is clear that $\eta \in L^p( \Gamma )$, $\eta |_{K^\ell} \in W^{1,p}( K^\ell )$ and by the trace theorem $\eta$ has common trace between adjacent elements.
  \end{proof}
}

We can also show a special interpolation estimate which interpolates smooth functions over the continuous surface into the lifted surface finite element space.
We denote by $I_h$ an interpolation operator $I_h \colon C(\Gamma) \to \S_h^\ell$ defined by
\begin{equation}
  \label{eq:lift-Ih}
  I_h \eta |_{K^\ell} := I_{K^\ell} \eta.
\end{equation}

\begin{theorem}
  [Global lifted interpolation theorem]
  \label{thm:global-lift-interp}
  For $h \in (0,h_0)$, let $\{ \T_h \}_{h \in (0,h_0)}$ be a $\Theta$-regular (\cref{def:Theta-regular}), quasi-uniform (\cref{def:quasi-uniform}) family of triangulations of $\Gamma_h$ equipped with a surface finite element space $\S_h$ (\cref{def:sfe-space}) consisting of $\Theta$-surface finite elements (\cref{def:Theta-sfe}) over a reference element which satisfies the assumptions of \cref{lem:bramble-hilbert} for some $0 < k,m \le \Theta$, $p,q \in [1,\infty]$.
  Let each element $K \in \T_h$ be equipped with a lifting map $\Lambda_K \in C^{\Theta+1}(K)$ such that $\norm{ \Lambda_K }_{W^{\Theta+1,\infty}(K)} \le C$ and \cref{eq:as-lift-linear-bounded} and \cref{eq:LambdaK-ass1} hold each uniformly for $h \in (0,h_0)$,
  then $\{ \T_h^\ell \}_{h \in (0,h_0)}$ is a $\Theta$-regular, quasi-uniform family of triangulations of $\Gamma$ and $\S_h^\ell$ is a surface finite element space consisting of $\Theta$-surface finite elements.
  Let $\eta \in C(\Gamma)$ be a continuous function, then $I_h \eta \in \S_h^\ell$ is well defined.
  Furthermore, if the assumptions of \cref{thm:IK-bound} hold for the reference element $(\hat{K}, \hat{P}, \hat{\Sigma})$, there exists a constant $C = C(\hat{K}, \hat{P}, \hat{\Sigma}, \rho)$ such that for all functions $\eta \in W^{k+1,p}(\T_h^\ell) \cap C( \Gamma )$,
  \begin{equation}
    \label{eq:global-lift-interp}
    \norm{ \eta - I_h \eta }_{W^{m,q}(\T_h^\ell)} \le C h^{k+1-m} \norm{ \eta }_{W^{k+1,p}(\T_h^\ell)}.
  \end{equation}
\end{theorem}

\begin{proof}
  The result follows by combining previous lemmas in the appropriate way.
  We see the lifted triangulation is quasi-uniform by applying the results in \cref{lem:lift-geom} and the assumption \cref{eq:as-lift-linear-bounded}.
  The interpolation result then follows by applying \cref{thm:Ih-bound}.
\end{proof}

\Cref{thm:global-lift-interp} will be used to show the abstract approximation properties \cref{eq:interp-Z0,eq:interp-Z}.

\begin{corollary}
  \label{cor:global-inv-lift-interp}
  Let $\eta \in C(\Gamma)$, then the interpolant of $\eta$ into $\S_h$, denoted by $\tilde{I}_h \eta$, is given by
  \begin{equation}
    \tilde{I}_h \eta := ( I_h \eta )^{-\ell}.
  \end{equation}
  Furthermore, there exists a constant $C$ such that for all functions $\eta \in W^{k+1,p}(\T_h^\ell) \cap C( \Gamma )$, and all $h \in (0,h_0)$, we have
  \begin{equation}
    \label{eq:global-inv-lift-interp}
    \norm{ \eta^{-\ell} - \tilde{I}_h \eta }_{W^{m,q}(\T_h)} \le C h^{k+1-m} \norm{ \eta }_{W^{k+1,p}(\T_h^\ell)}.
  \end{equation}
\end{corollary}

\begin{proof}
  We apply the inverse lift result to the estimates in the theorem.
\end{proof}

\subsection{Evolving lifted finite element spaces}
For $t \in [0,T]$, let $\Gamma(t)$ be a smoothly evolving hypersurface with flow map  $\Phi_t$: i.e.\ $\Phi_{(\cdot)} \in C^2( [0,T], C^1( \Gamma(0) ) )$ and $\Phi_t( \Gamma(0) ) = \Gamma(t)$.
For $h \in (0,h_0)$, let $\Gamma_h(t)$ be an evolving triangulated hypersurface (\cref{def:evolving-triangulated-hypersurface}) with global discrete flow  $\Phi^h_t$ (\cref{def:global-discrete-flow}) and equipped with an evolving conforming subdivision $\T_h(t)$ (\cref{def:evolving-conforming-subdivision}) and an evolving surface finite element space $\S_h(t)$ (\cref{def:evolving-sfe-space}).
We assume that we are given a global lifting map $\Lambda_h( \cdot, t )$ (\cref{def:global-lifting-map}) which gives an exact subdivision $\T_h^\ell(t)$ of $\Gamma(t)$ (\cref{def:exact-decomposition}).

\begin{definition}[Lifted discrete flow map, material velocity and pushed forward map]
~

\begin{deflist}
\item \label{def:lifted-flow-map}
 The \emph{lifted flow map} $\Phi^\ell_{(\cdot)}(\cdot) \colon [0,T] \times \Gamma_0 \to \R^{n+1}$ of the smooth hypersurface is defined by
\begin{equation*}
  \Phi^\ell_t( \Lambda_h( x, t ) ) = \Lambda_h ( \Phi^h_t ( x ), t ) \qquad \mbox{ for } x \in \Gamma_0.
\end{equation*}
We note that $\Phi^\ell_{t} \colon \Gamma_0 \to \Gamma(t)$.

\item \label{def:lifted-discrete-material-velocity}
The \emph{lifted discrete material velocity}   ${w}_h$ on $\{ \Gamma(t) \}_{t \in [0,T]}$ is defined by
\begin{equation*}
  \dt \Phi^\ell_t( \cdot ) = {w}_h( \Phi^\ell_t( \cdot ), t ).
\end{equation*}
\item \label{def:lifted-push-forward-map}
  The family of \emph{lifted push forward maps} is denoted by $\phi^\ell_t( \eta ) \colon \Gamma(t) \to \R$ for $\eta \colon \Gamma_0 \to \R$ and is defined by
\begin{equation*}
  \phi^\ell_t( \eta )( x ) = \eta( \Phi^\ell_{-t}( x ) ) \qquad \mbox{ for } x \in \Gamma(t).
\end{equation*}
See also \cref{eq:phiellt}.
\end{deflist}
\end{definition}
\begin{remark}
  Note that in general   $\Phi^\ell_t$ is different to $\Phi_t$, but each describes a different parametrisation of the same evolving surface.
  Also ${w}$, the material velocity of $\Gamma(t)$, and ${w}_h$ define the same surface $\Gamma(t)$ evolving from $\Gamma(0)$, so have the same normal components whilst the tangential components may not agree.
\end{remark}

\begin{proposition}
  \label{prop:lift-Vht}
  For $h \in (0,h_0)$,
  let $\{ \S_h(t) \}_{t \in [0,T]}$ be an evolving surface finite element space over a $\Theta$-regular (\cref{def:Theta-regular}), uniformly quasi-uniform (\cref{def:uniformly-quasi-uniform}), evolving conforming subdivision $\{ \T_h(t) \}_{t \in [0,T]}$ consisting of $\Theta$-evolving surface finite elements (\cref{def:e-Theta-sfe}) and ${\phi}^h_t$ the global push-forward map (\cref{def:global-push-forward-map}).
  For each $t \in [0,T]$ and each $h \in (0,h_0)$, let each element $K \in \T_h$ be equipped with a lifting map $\Lambda_K( \cdot, t ) \in C^{\Theta+1}(K(t))$ such that $\norm{ \Lambda_K( \cdot, t )}_{W^{\Theta+1,\infty}(K)} \le C$ and \cref{eq:as-lift-linear-bounded} and \cref{eq:LambdaK-ass1} hold each uniformly for $h \in (0,h_0)$ and $t \in [0,T]$.
  Then $\{ \S_h^\ell(t) \}_{t \in [0,T]}$ is also an evolving surface finite element space over a $\Theta$-regular, uniformly quasi-uniform, evolving conforming subdivision $\{ \T_h^\ell(t) \}_{t \in [0,T]}$ consisting of $\Theta$-evolving surface finite elements.
  Furthermore, for $0 \le m \le \Theta+1$, $p \in [0,\infty]$, there exists $c_1, c_2 > 0$ independent of $h \in (0,h_0)$ and $t \in [0,T]$ such that
  \begin{align}
    \label{eq:lift-compat}
    c_1 \norm{ \eta }_{W^{m,p}(\T_h^\ell(0))}
    & \le \norm{ {\phi}^\ell_t \eta }_{W^{m,p}(\T_h^\ell(t))}
    \le c_2 \norm{ \eta }_{W^{m,p}(\T_h^\ell(0))}
    && \mbox{ for all } \eta \in W^{m,p}( \T_h^\ell(0) ).
  \end{align}
  In particular,
  the pair $( W^{m,p}( \T_h^\ell(t) ), {\phi}^\ell_t|_{W^{m,p}( \T_h^\ell(0) )} )_{t \in [0,T]}$ is compatible (\cref{def:compatibility}).
  Furthermore the pairs $( \S_h^\ell(t), \phi^\ell_t|_{\S_{h,0}^\ell} )_{t \in [0,T]}$,  equipped with the norm $\norm{ \cdot }_{W^{m,p}(\T_h^\ell(t))}$, and $( W^{1,p}( \Gamma), \phi^\ell_t|_{W^{1,p}( \Gamma_0 )} )_{t \in [0,T]}$ are also compatible.
\end{proposition}

\begin{proof}
  The initial properties of $\S_h(t)$ follow for the same reasoning as \cref{prop:lift-Vh} since the constants are bounded uniformly in time.
  The bounds \cref{eq:lift-compat} and the compatibility of $( W^{m,p}(\T_h^\ell), \phi_h^\ell )_{t \in [0,T]}$ follow since the assumptions imply that $( W^{m,p}( \T_h^\ell ), \phi^h_t )_{t \in [0,T]}$ are compatible and we can simply lift this result with uniform bounds.
  The compatibility of $( \S_h^\ell(t), \phi^\ell_t|_{\S_{h,0}^\ell} )_{t \in [0,T]}$ and $( W^{1,p}( \Gamma(t) ), \phi^\ell_t|_{W^{1,p}( \Gamma_0) } )_{t \in [0,T]}$ follow since they are closed subspaces of $W^{1,p}( \T_h^\ell(t) )$ (\cref{lem:Sh-closed-subspace,lem:W1pT-closed-subspace,lem:W1pThell-exact}) and \cref{rem:compatible-subsace}.
\end{proof}

\clearpage

\part{Application to parabolic equations on evolving domains}
\section{Application I: Parabolic equation on an evolving bulk domain}
\label{BULKPDE}

In this section, we will formulate and analyse a finite element method for a parabolic problem posed in an evolving bulk domain \cref{eq:bulk-eqn}.
We will begin with notation and definitions for the initial value problem.
The numerical method is based on the general theory of \cref{sec:bfem} applied here with isoparametric bulk finite element spaces of order $k$.
The analysis is based on the abstract theory presented in \cref{sec:abstract-formulation,AbstractNA}.

\subsection{The domain and function spaces}
We set $\H(t) = L^2(\Omega(t))$, $\V(t) = H^1(\Omega(t))$ and  $\V^*(t) = ( H^1(\Omega(t)) )^*$.
We will also make use of the spaces $\Z_0(t) = H^2(\Omega(t))$ and $\Z(t) = H^{k+1}( \Omega(t) )$.
We see that $Z(t) \subset \Z_0(t) \subset \V(t)$ for each $t \in [0,T]$ and the inclusions are uniformly continuous.
We define the push-forward operator $\phi_t$ by
\begin{equation}
  \label{eq:bulk-pushforward}
  (\phi_t \eta)( \cdot, t ) := \eta( \Phi_{-t}( \cdot) ) \qquad \mbox{ for } \eta \in \H_0.
\end{equation}
 $( L^2( \Omega(t) ), \phi_t )_{t \in [0,T]}$ and    $( H^1( \Omega(t) ), \phi_t )_{t \in [0,T]}$ 
are  compatible pairs (\cref{def:compatibility}), the spaces $L^2_\H$, $L^2_\V$ and  $L^2_{\V^*}$ are well defined (c.f. \cref{eq:L2X}) 
and we have a well defined strong material derivative, denoted $\md$ \cref{eq:strong-md}.
$( \V(t), \H(t), \V^*(t) )_{t \in [0,T]}$ form an evolving Hilbert triple (\cref{def:evolving-hilbert-triple}).
We also have that $( \Z_0(t), \phi_t|_{\Z_0(t)} )_{t \in [0,T]}$ and $( \Z(t), \phi_t|_{\Z_0(t)} )_{t \in [0,T]}$ are compatible pairs.
For details see \citet{AlpEllSti15b-pp}.

{
  \begin{remark}
    \label{rem:bulk-smoothness}
    For the well posedness of the partial differential equation \cref{Bulkprob} we require that the boundary 
    $\Gamma$ is a $C^{2}$-hypersurface and that the flow map $\Phi_{(\cdot)}( \cdot ) \in C^2( 0,T; C^2( \Omega_0 ) )$. 
    See \citet{AlpEllSti15b-pp} for more details.
    For the approximation properties we derive we require that $\Gamma$ is a $C^{k+2}$-hypersurface and that 
    $\Phi_{(\cdot)}( \cdot ) \in C^{k+2}( 0,T; C^{k+2}( \Omega_0) )$ with $\Phi_{t} \colon \Omega_0 \to \Omega(t)$ and $\Phi_{-t} \colon \Omega(t) \to \Omega$ both of class $C^{k+2}$.
  \end{remark}
}

{ %
  We introduce a signed distance function for the boundary surface $\Gamma(t) = \partial \Omega(t)$. The oriented signed distance function for $\Gamma(t)$ is given by
  \[
    d( x, t ) = \begin{cases}
      - \inf\{ \abs{ x - y } : y \in \Gamma(t) \} & \mbox{ for } x \in \bar\Omega(t) \\
      \inf\{ \abs{ x - y } : y \in \Gamma(t) \} & \mbox{ otherwise.}
    \end{cases}
  \]
  For each $t \in [0,T]$, we orient $\Gamma(t)$ by choosing the unit normal $\nu( x, t ) = \nabla d( x, t )$ for $x \in \Gamma(t)$. Our assumptions on $\Gamma(t)$ 
  imply that there exists a neighbourhood $\N(t)$ of $\Gamma(t)$ and normal projection operator $p( \cdot, t ) \colon \N(t) \to \Gamma(t)$ given as the unique solution of
  \begin{equation}
    \label{eq:bulk-cp}
    x = p(x,t) + d( x, t ) \nu( p( x, t ), t ).
  \end{equation}
  See \citet[Lem. 14.16]{GilTru01}; \citet{Foote1984} for more details.

The final result we show in this section will be useful in the error analysis of our methods.

\begin{lemma}
  [Narrow band trace inequality]
  \label{lem:narrow-band}
  For $t \in [0,T]$, let $B_\eps(t) \subset \N(t)$ be the band given by
  \begin{equation*}
    B_\eps(t) = \{ x \in \Omega(t) : - \eps < d( x, t ) \le 0 \}.
  \end{equation*}
  Then there exists a constant $c$ such that for all $t \in [0,T]$,
  \begin{equation}
    \label{eq:narrow-band}
    \norm{ \eta }_{L^2(B_\eps(t))} \le c \eps^{1/2} \norm{ \eta }_{H^1(\Omega(t))}
    \qquad
    \mbox{ for all } \eta \in H^1(\Omega(t)).
  \end{equation}
\end{lemma}

\begin{proof}
  The proof for stationary domains in given by \citet[Lem.~4.10]{EllRan13} which can be easily extended to the evolving case.
\end{proof}

}

\subsection{The initial value problem}
\label{sec:bulk-prob}

We assume that $\A_\Omega  \in C^1( \Omega_T; \R^{(n+1)\times(n+1)} )$ is an $(n+1) \times (n+1)$ symmetric diffusion tensor, 
$\b_\Omega \in C^1( \Omega_T ; \R^{n+1})$ a smooth  vector field  and  $\c_\Omega \in C^1( \Omega_T )$ is a smooth scalar field. 
We assume that for each $t \in [0,T]$, $\A_\Omega(\cdot, t)$  is uniformly positive definite: There exists $a_0 > 0$ such that for all $t \in [0,T]$
\begin{equation*}
  \A_\Omega(\cdot, t) \xi \cdot \xi \ge a_0 \abs{ \xi }^2
  \quad
  \mbox{ for all } \xi \in \R^{n+1}.
\end{equation*}
We consider the initial value  problem
\begin{problem}\label{Bulkprob} Find $\emph \utext$ such that
  \begin{subequations}
    \begin{align}
      \md \emph \utext   +\nabla \cdot (\b_\Omega  \emph \utext)  -\nabla \cdot (\A_\Omega\nabla \emph \utext) +\c_\Omega \emph \utext + (\nabla\cdot w)\utext &=0
      && \mbox{on} ~~\Omega(t)\\
       (-\b_\Omega  \emph \utext  +\A_\Omega\nabla \emph \utext)\cdot \nu &= 0
          && \mbox{on}~\Gamma(t)\\
       \emph \utext(0) &=  \emph \utext_0 &&\mbox{on}~~\Omega_0.
    \end{align}
\end{subequations}

\end{problem}
\begin{remark}
The problem  (\ref{eq:bulk-eqn}) is recovered by setting
$\A_\Omega=\la_\Omega, \b_\Omega=\lb_\Omega-w$ and $\c_\Omega=\lc_\Omega$.
\end{remark}
\subsubsection{Transport formulae}
  The following transport formulae hold on portions 
  $\{ \omega(t)\subseteq \Omega(t) \}_{t \in [0,T]}$ of the domain $\{ \Omega(t) \}_{t \in [0,T]}$, which follow the flow $\omega(t) = \Phi_t( \omega(0) )$ for $t \in [0,T]$, and have Lipschitz boundaries at each time.
:-
\begin{itemize}
\item
 For $\eta, \zeta \in C^1_\H$ by
\begin{equation}
  \label{eq:bulk-transport}
  \dt \int_{\omega(t)} \eta \zeta \dd x
  = \int_{\omega(t)} \md \eta \zeta + \md \eta \zeta \dd x
  +  \int_{\omega(t)} \eta \zeta \nabla \cdot {w} \dd x.
\end{equation}

\item
For $\eta, \zeta \in C^1_\V$, we have the identity
\begin{equation}
  \label{eq:bulk-transport-dirichlet}
  \begin{aligned}
  \dt \int_{\omega(t)} \A_\Omega \nabla \eta \cdot \nabla \zeta \dd x
  & = \int_{\omega(t)} \A_\Omega\nabla \md \eta \cdot \nabla \zeta
  + \A_\Omega \nabla \eta \cdot \nabla \md \zeta \dd x \\
  & \qquad + \int_{\omega(t)} \B( {w}, \A_\Omega ) \nabla \eta \cdot \nabla \zeta \dd x,
  \end{aligned}
\end{equation}
where $\B( {w}, \A_\Omega )$ is given by
\begin{equation}
  \label{eq:bulk-B-defn}
  \B( {w}, \A_\Omega ) = \md \A_\Omega + \nabla \cdot {w} \A_\Omega - 2 D( {w},\A_\Omega)
\end{equation}
and $D( {w},\A_\Omega)$ is the rate of deformation tensor
\begin{equation*}
  D( {w},\A_\Omega)_{ij} = \frac{1}{2} \sum_{k=1}^{n+1} \{(\A_\Omega)_{ik} (\nabla)_k {w}_j + (\A_\Omega)_{jk} (\nabla)_k {w}_i\}
  \qquad \mbox{ for } i,j = 1, \ldots, n+1.
\end{equation*}

\item
For $\eta \in C^1_\H$, $\zeta \in C^1_\V$, we have
\begin{equation}
  \label{eq:bulk-transport-advection}
  \begin{aligned}
    \dt
    \int_{\omega(t)} \b_\Omega \eta \cdot \nabla \zeta \dd x
    & = \int_{\omega(t)} \{\b_\Omega \md \eta \cdot \nabla \zeta
    + \b_\Omega \eta \cdot \nabla \md \zeta \} \dd x \\
    & \qquad + \int_{\omega(t)} \B_{\adv}( {w}, \b_\Omega ) \eta \cdot \nabla \zeta \dd x,
  \end{aligned}
\end{equation}
where $\B_{\adv}( {w}, \b_\Omega ) $ is given by
\begin{align*}
  \B_{\adv}( {w}, \b_\Omega )
  = \md \b_\Omega + \b_\Omega \nabla \cdot {w}
  - \sum_{j=1}^{n+1}  (\b_\Omega)_j (\nabla)_j {w}.
\end{align*}

\end{itemize}

\subsection{The bilinear forms and transport formulae}

We define
  \begin{align*}
    m( t; \eta, \zeta )
    & := \int_{\Omega(t)} \eta \zeta \dd x
    && \eta, \zeta \in \H(t) \\
    g( t; \eta, \zeta )
    & := \int_{\Omega(t)} \eta \zeta \nabla \cdot {w} \dd x
    && \eta, \zeta \in \H(t) \\
    a( t; \eta, \zeta )
    & := a_s( t; \eta, \zeta ) + a_n( t; \eta, \zeta )
    && \eta, \zeta \in \V(t) \\
    a_s( t; \eta, \zeta )
    & := \int_{\Omega(t)} \A_\Omega \nabla \eta \cdot \nabla \zeta + \c_\Omega \eta \zeta \dd x
    && \eta, \zeta \in \V(t) \\
    a_n( t; \eta, \zeta )
    & := \int_{\Omega(t)} \b_\Omega \eta \cdot \nabla \zeta \dd x
    && \eta \in \H(t), \zeta \in \V(t).
  \end{align*}

   We can apply \cref{eq:bulk-transport}, \cref{eq:bulk-transport-dirichlet} and \cref{eq:bulk-transport-advection} to see that we have the transport laws
\begin{align}
  \label{eq:bulk-m-transport}
  \dt m( t; \eta, \zeta )
  & = m( t; \md \eta, \zeta )
  + m( t; \eta, \md \zeta )
  + g( t; \eta, \zeta )
  && \mbox{ for all } \eta, \zeta \in C^1_{\H} \\
  \label{eq:bulk-as-transport}
  \dt a_s( t; \eta, \zeta )
  & = a_s( t; \md \eta, \zeta )
  + a_s( t; \eta, \md \zeta )
  + b_s( t; \eta, \zeta )
  && \mbox{ for all } \eta, \zeta \in C^1_{\V} \\
  \label{eq:bulk-an-transport}
  \dt a_n( t; \eta, \zeta )
  & = a_n( t; \md \eta, \zeta )
  + a_n( t; \eta, \md \zeta )
  + b_n( t; \eta, \zeta )
  && \mbox{ for all } \eta \in C^1_{\H}, \zeta \in C^1_{\V},
\end{align}
with the new forms
\begin{align*}
  b_s( t; \eta, \zeta )
  & = \int_{\Omega(t)} \B_\Omega( {w}, \A_\Omega ) \nabla \eta \cdot \nabla \zeta
    + ( \md \c + \c_\Omega \nabla \cdot {w} ) \eta \zeta \dd x
  && \mbox{ for } \eta, \zeta \in \V(t) \\
  b_n( t; \eta, \zeta )
  & = \int_{\Omega(t)}
    \B_{\adv}( {w}, \b_\Omega) \eta \cdot \nabla \zeta \dd x
  && \mbox{ for } \eta \in \H(t), \zeta \in \V(t),
\end{align*}
where
\begin{align*}
  \B_\Omega( w, \A_\Omega ) & :=
  \md \A_\Omega + \nabla \cdot w \A_\Omega - 2 D_h( w ) \\
  \B_{\adv}( w, \b_\Omega ) & :=
  \md \b_\Omega + \b \nabla \cdot w - \sum_{j=1}^{n+1} ( \b_\Omega )_j \partial_{x_j} w,
\end{align*}
and $D( w )$ is the rate of deformation tensor:
\begin{align*}
  D( w )_{ij}
  = \frac{1}{2} \sum_{l=1}^{n+1} ( \A_\Omega )_{il} \partial_{x_l} w_j + ( \A_\Omega )_{jl} \partial_{x_l} w_i
  \quad \mbox{ for } i,j = 1, \ldots, n+1.
\end{align*}

We define $b( t; \cdot, \cdot ) \colon \V(t) \times \V(t) \to \R$ to be
\begin{equation*}
  b( t; \eta, \zeta ) := b_s( t ; \eta, \zeta ) + b_n( t; \eta, \zeta ) \qquad \mbox{ for } \eta, \zeta \in \V(t).
\end{equation*}

 The smoothness assumptions on the coefficients and the velocity imply that  $g$, $b_s$ and $b_n$ are uniformly bounded. There exists a constant $c > 0$ such that
\begin{align}
  \label{eq:bulk-g-bounded}
  \abs{ g( t; \eta, \zeta ) }
  & \le c \norm{ \eta }_{\H(t)} \norm{ \zeta }_{\H(t)}
  && \mbox{ for all } \eta, \zeta \in \H(t) \\
  \label{eq:bulk-bs-bounded}
  \abs{ b_s( t; \eta, \zeta ) }
  & \le c \norm{ \eta }_{\V(t)} \norm{ \zeta }_{\V(t)}
  && \mbox{ for all } \eta, \zeta \in \V(t) \\
  \label{eq:bulk-bn-bounded}
  \abs{ b_n( t; \eta, \zeta ) }
  & \le c \norm{ \eta }_{\V(t)} \norm{ \zeta }_{\V(t)}
  && \mbox{ for all } \eta \in \H(t), \zeta \in \V(t).
\end{align}

\subsection{Variational formulation}
 We consider the following variational formulation of \cref{eq:bulk-eqn}:
\begin{problem}
  \label{pb:weak-bulk}
  Given $\emph \utext_0 \in L^2(\Omega_0)$, find $\emph \utext  \in \W(\V,\V^*)$%
  such that for almost every  $t \in (0,T)$ we have
  \begin{equation}
    \label{eq:weak-bulk}
    \begin{aligned}
      m( t; \md \emph \utext, \zeta ) + g( t; \emph \utext, \zeta ) + a( t; \emph \utext, \zeta )
      & = 0 && \mbox{ for all } \zeta \in H^1(\Omega(t)) \\
      \emph \utext( 0 ) & = \emph \utext_0.    \end{aligned}
  \end{equation}
 \end{problem}

\begin{theorem}
  \label{thm:bulk-exist}
  There exists a unique solution $\emph \utext$  to \cref{pb:weak-bulk} which satisfies the stability bound:
  \begin{align}
    \sup_{t\in[0,T]}\norm{ \emph \utext }_{L^2(\Omega(t))}^2 +   \int_0^T \norm{ \emph \utext }_{H^1(\Omega(t))}^2 \dd t
    & \le c(T) \norm{ \emph \utext_0 }_{L^2(\Omega_0)}^2 
    \end{align}
    and if $u_0\in H^1(\Omega_0)$
    \begin{align}
    \sup_{t\in[0,T]}\norm{ \emph \utext }_{H^1(\Omega(t))}^2 +   \int_0^T \norm{ \md \emph \utext }_{L^2(\Omega(t))}^2 \dd t
    & \le c(T) \norm{ \emph \utext_0 }_{H^1(\Omega_0)}^2.
  \end{align}
\end{theorem}

\begin{proof}
  We simply apply the abstract theory of \cref{thm:abs-exist}. For \cref{as:c-exists,ass:evolving-space-equivalence}    
  we refer to  \citet[Sec. 4 and Sec. 5]{AlpEllSti15b-pp}.  \cref{ass:u0-approx} is a consequence of \cref{rem:ass:u0-approx}. %
Also Assumptions \cref{eq:m-symmetric}, \cref{eq:m-bounded}, \cref{eq:c-formula}, 
  \cref{eq:c-bounded}, \cref{eq:a-meas}, \cref{eq:a-decomp}, \cref{eq:as-coercive}, \cref{eq:as-bounded}, \cref{eq:bs-exist}, \cref{eq:bn-exist}, \cref{eq:bs-bound} and  \cref{eq:bn-bound} hold.
   It is clear that \cref{eq:m-symmetric} and \cref{eq:m-bounded} hold since $m( t; \cdot, \cdot )$ is equal to the $\H(t)$-inner product.
 Similarly  assumptions \cref{eq:c-formula} and \cref{eq:c-bounded} follow from the transport formula  \cref{eq:bulk-m-transport} and the boundedness of the velocity. %
  We know that the map $t \mapsto a( t; \cdot, \cdot )$ is differentiable hence measurable which shows \cref{eq:a-meas}.
  The coercivity \cref{eq:as-coercive} and boundedness \cref{eq:as-bounded} of $a_s$ and the boundedness of \cref{eq:an-bounded} follow from standard arguments.
  The existence of the bilinear forms $b_s$ \cref{eq:bs-exist} and $b_n$ \cref{eq:bn-exist} has been shown in \cref{eq:bulk-as-transport} 
  and \cref{eq:bulk-an-transport} and the estimates \cref{eq:bs-bound} and \cref{eq:bn-bound} are shown in \cref{eq:bulk-bs-bounded} and \cref{eq:bulk-bn-bounded}.
\end{proof}

\subsection{Discretisation of the domain and finite element spaces}
\label{sec:fem-bulk}

The first stage in constructing  our finite element method is to define an approximate computation domain $\{ \Omega_h(t) \}$.
Our construction satisfies that the boundary Lagrange points of $\Omega_h(t)$ lie on the boundary of $\Omega(t)$ and all Lagrange points evolve with the prescribed velocity ${w}$.
We will consider $\Omega_h(t)$ as an interpolant of $\Omega(t)$.
We recall $k$ is the order of isoparametric bulk finite element spaces we wish to use.
Throughout the remainder of this section we will denote global discrete quantities with a subscript $h \in (0,h_0)$, which is related to element size. 
We assume implicitly that these structures exist for each $h$ in this range (see also \cref{rem:small-hk-bulk}).

We will use the simplical, Lagrangian reference element $(\hat{K}, \hat{P}, \hat{\Sigma})$ over order $k$ from \cref{ex:standard-fem}.
We start by constructing a family of time dependent element reference maps (\cref{def:ebfe-reference-map}) which will define an evolving conforming 
subdivision $\{ \T_h(t) \}$ (\cref{def:bulk-evolving-conforming-subdivision}) of an evolving triangulated open domain $\{ \Omega_h(t) \}$ (\cref{def:evolving-triangulated-domain}).

Let $\tilde\Omega_{h,0}$ be a polyhedral approximation of $\Omega_0$ equipped with a quasi-uniform, conforming subdivision into simplicies $\tilde{\T}_{h,0}$ (see \cref{sec:stationary-bfem} for details).
We denote by $\tilde{\Gamma}_{h,0} = \partial \tilde{\Omega}_{h,0}$.
We restrict that the vertices of $\tilde{\Gamma}_{h,0}$ lie on the surface $\Gamma_0$.
We assume that the normal projection operator \cref{eq:bulk-cp}, $p(\cdot, 0)$ is a homomorphism from $\tilde\Gamma_{h,0}$ onto $\Gamma_0$.
More precisely, for each $\tilde{K} \in \tilde\T_{h,0}$, there exists an affine map $\hat{K} \to \R^{n+1}$ which satisfies the assumption of \cref{def:bfe} so that we can define a bulk finite element $( \tilde{K}, \tilde{P}, \tilde{\Sigma} )$ using \cref{eq:bfe}.
We assume the Lagrange points satisfy \cref{eq:bulk-node-agree}.

We extend $p$ to construct a bijection $\Psi_h \colon \tilde\Omega_{h,0} \to \bar\Omega_0$ which we will define element-wise.
A similar construction is used by \citet{EllRan13}.
We first decompose $\tilde\T_{h,0}$ into boundary elements, which have more than one vertex on the boundary, and interior elements.
For an interior element $\tilde{K} \in \tilde{\T}_{h,0}$, we define
\begin{align*}
  \Psi_h( \tilde{x} ) = \tilde{x} \qquad \mbox{ for } \tilde{x} \in \tilde{K}.
\end{align*}
Otherwise, let $\tilde{K}$ be a boundary element and consider $\tilde{x} \in \tilde{K}$.
Denote by $\{ \tilde{a}_i \}_{i=1}^{n+2}$ the vertices of $\tilde{K}$ ordered so that $\{ \tilde{a}_i \}_{i=1}^L$ lie 
on $\Gamma_0$ (recall that $\Omega(t) \subset \R^{n+1}$).
First, decompose $\tilde{x}$ into barycentric coordinates:
\begin{equation*}
  \tilde{x} = \sum_{j=1}^{n+2} \mu_j( \tilde{x} ) \tilde{a}_j.
\end{equation*}
We introduce the function $\mu^*( \tilde{x} )$ and the singular set $\sigma$ by
\begin{equation*}
  \mu^*( \tilde{x} ) = \sum_{j=1}^L \mu_j( \tilde{x} ),
  \qquad
  \sigma = \{ \tilde{x} \in \tilde{K} : \mu^*( \tilde{x} ) = 0 \}.
\end{equation*}
The scalar $\mu^*$ represents how far we are from $\tilde\Gamma_{h,0}$ with $\mu^* = 1$ on $\Gamma_h$.
Note that we have $0 \le \mu^* \le 1$.
The set $\sigma$ is the set of points in $\tilde{K}$ furthest from $\tilde{\Gamma}_{h,0}$.
If $\tilde{x} \not\in \sigma$, we denote the projection onto $\tilde\Gamma_{h,0} \cap \tilde{K}$ by $y( \tilde{x} )$. We see that $y(\tilde{x})$ is given by
\begin{equation*}
  y( \tilde{x} ) = \sum_{j=1}^{L} \frac{\mu_j( \tilde{x} )}{ \mu^*( \tilde{x} ) } \tilde{a}_j.
\end{equation*}
Then, we define $\Psi_h|_{\tilde{K}}$ by
\begin{equation*}
  \Psi_h|_{\tilde{K}}( \tilde{x} )
  =
  \begin{cases}
    \tilde{x} + ( \mu^*( \tilde{x} ) )^{k+2} ( p( y( \tilde{x} ), 0 ) - y( \tilde{x} ) )
    & \mbox{ if } \tilde{x} \not\in \sigma \\
    \tilde{x} & \mbox{ otherwise}.
  \end{cases}
\end{equation*}
\begin{remark}
  \label{rem:why-Psih}
This mapping takes points on $\tilde\Gamma_{h,0}$ onto $\Gamma_0$, since here $\mu^* = 1$ and $y(\tilde{x}) = \tilde{x}$ and we recover $\Psi_h( \tilde{x} ) = p(\tilde{x},0)$.
Points in $\sigma$ remain unchanged by this mapping, since $\mu^* = 0$ here and furthermore the power $k+2$ on $\mu^*$ ensures that the mapping is $C^{k+1}$ on the closed domain $\tilde{K}$.
\end{remark}

We write $\tilde{I}$ for interpolation over $( \tilde{K}, \tilde{\P}, \tilde{\Sigma} )$ \cref{eq:bulk-IK} and define an initial element reference map $F_{K_0} \colon \hat{K} \to \R^{n+1}$ by
\[
  F_{K_0}( \hat{x} ) = [ \tilde{I} \Psi_h ]( F_{\tilde{K}}( \hat{x} ) )
  \qquad \mbox{ for } \hat{x} \in \hat{K}.
\]
We call the union of all domains $K$ constructed in this way $\T_{h,0}$, which is a conforming subdivision (\cref{def:bulk-conforming-subdivision}), 
and call the union of element domains $\Omega_{h,0}$, which is a triangulated bulk domain (\cref{def:triangulated-domain}).
Finally, we call $\{ a_i \}_{i=1}^N$ the Lagrange nodes of $\Omega_{h,0}$ which satisfy \cref{eq:bulk-node-agree} since Lagrange points 
on the boundaries of each element $\tilde{K}$ are not moved by $\Psi_h$.
An example of domains constructed using the above are shown in \cref{fig:bulk-ex-fem}.

\begin{figure}
  \centering
  \includegraphics[width=0.22\textwidth]{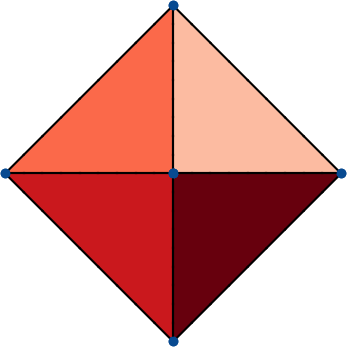}
  \includegraphics[width=0.22\textwidth]{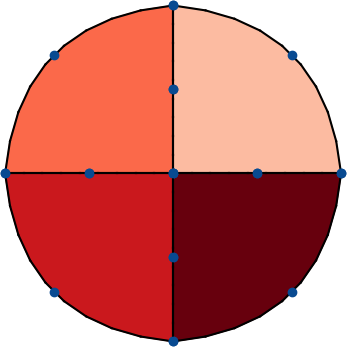}
  \includegraphics[width=0.22\textwidth]{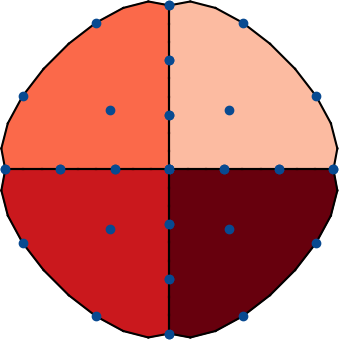}
  \includegraphics[width=0.22\textwidth]{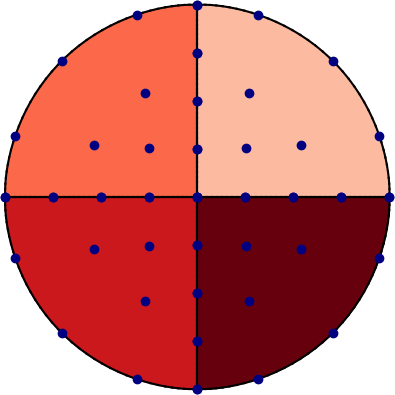}

  \caption{Examples of collections of isoparametric bulk finite elements. In each example we have four finite elements which approximate the unit disc in $\R^2$. 
  The element domains are shown in red and the Lagrange points in blue. From left to right we use $k=1, 2, 3, 4$.}
  \label{fig:bulk-ex-fem}
\end{figure}

To complete the construction, for $K_0 \in \T_{h,0}$, we consider the discrete domain $K(t)$ given by the 
discrete flow $\Phi_t^K \colon K_0 \to K(t)$ defined by
\begin{equation*}
  \Phi_t^K |_{K_0} = I_{K_0} [ \Phi_t ( \Psi_h( \cdot ) ) ],
\end{equation*}
which is a bijection onto its image for $h$ small enough. We denote its inverse by $\Phi_{-t}^K$.
Using \cref{eq:ebfe-reference-map} $\Phi^K_t$ defines an evolving reference element map and with \cref{eq:bfe} we have a bulk finite element $( K(t), P(t)^K, \Sigma(t)^K )$ (\cref{def:ebfe}).
We call the set of such domains at each time $\T_h(t)$ and we define $\Omega_h(t)$ by
\begin{align*}
  \Omega_h(t) := \bigcup_{K(t) \in \T_h(t)} K(t),
\end{align*}
and write a global discrete flow map $\Phi_t^h \colon \Omega_{h,0} \to \Omega_h(t)$ defined element-wise 
by $\Phi_t^h |_{K_0} := \Phi^K_t$ (\cref{def:bulk-global-discrete-flow}).
By $h$ we denote the maximum mesh diameter over time \cref{eq:bulk-def-h}:
\[
  h := \max_{t \in [0,T]} \max_{K(t) \in \T_h(t)} \diam( \tilde{K}(t) ).
\]
We introduce the Hilbert spaces $\H_h(t) := L^2( \Omega_h(t) )$ and $\V_h(t) := H^1_T( \T_h(t) )$ (c.f. \cref{lem:bulk-W1pT-h-equal}).
Since $\Phi^h_t$ is a globally continuous function, it is clear that the Lagrange points in each element satisfy \cref{eq:bulk-evolve-node-agree}.
We define a global discrete function space (\cref{def:evolving-bfe-space}) by
  \begin{multline}
    \label{eq:bulk-Sh}
  \S_h(t) := \Bigl\{
  \chi_h = ( \chi_h )_{K(t) \in \T_h(t)} \in \prod_{K \in \T_h(t)} \bigl\{ \hat\chi \circ F_{K(t)}^{-1} : \hat\chi \in P_k(\hat{K}) \bigr\} : \\
  \chi_K( a(t) ) = \chi_{K'(t)} \mbox{ for all } K(t), K'(t) \in \T(a(t)),
  \mbox{ for all } a(t) \in \N_h(t)
  \Bigr\}.
\end{multline}
Using \cref{lem:bulk-Sh-closed} we can identify elements in $\S_h(t)$ as continuous functions and $\S_h(t) \subset \V_h(t)$.

We will assume that this construction results in a uniformly quasi-uniform evolving subdivision 
$\{ \T_h(t) \}_{t \in [0,T]}$ (\cref{def:bulk-uniformly-quasi-uniform}) - that is that the velocity $w$ is such that the simplicies in $\T_h(t)$ do not become too distorted.
In order to show the result of this construction satisfies the other assumptions we require we first state a result shown by \citet{EllRan13}:
\begin{lemma}[Prop.~4.4, \citealt{EllRan13}]
  The mapping $\Psi_K = \Psi_h|_{\tilde{K}}$ is of class $C^{k+1}$ when restricted to each element $\tilde{K} \in \tilde{\T}_{h,0}$ and the composed map $\Psi_K \circ F_{\tilde{K}}$ satisfies
  \begin{equation}
    \label{eq:PsiK-bound}
    \norm{ \nabla^m (\Psi_K \circ F_{\tilde{K}} - F_{\tilde{K}} ) }_{L^\infty(\hat{K})} \le c h^{2} \qquad \mbox{ for } 0 \le m \le k+1.
  \end{equation}
\end{lemma}

\begin{proposition}
  \label{prop:bulk-fe}
  The space $\{ \S_h(t) \}_{t \in [0,T]}$ defined by the above construction is an evolving bulk finite element space consisting of $k$-surface finite elements over a uniformly $k$-regular evolving subdivision.
\end{proposition}

\begin{proof}
  Let $K(t)$ be a single element in $\T_h(t)$.
  We can write the element parametrisation $F_{K(t)}$ as
  \begin{align*}
    F_{K(t)}( \hat{x} )
    & = \Phi_t^K( F_{K_0}( \hat{x} ) ) \\
    & = \Phi_t^{\tilde{K}}( F_{\tilde{K}_0}( \hat{x} ) )
      + \big(
      \Phi_t^K( F_{K_0}( \hat{x} ) ) - \Phi^{\tilde{K}}_t( F_{\tilde{K}_0}( \hat{x} ) )
      \big),
  \end{align*}
  where $\Phi^{\tilde{K}}_t$ is the flow map from $\tilde{K}_0 = F_{\tilde{K}_0}( \hat{x} )$ defined as the piecewise linear interpolant of $\Phi^K_t$.
  We note that $\hat{x} \mapsto \Phi^{\tilde{K}}_t( F_{\tilde{K}_0}( \hat{x} ) )$ is linear so we define the splitting
  \begin{align*}
    F_{K(t)}( \hat{x} ) =
    \underbrace{\Phi^{\tilde{K}}_t( F_{\tilde{K}_0}( \hat{x} ) )}_{\displaystyle =: A_{K(t)} \hat{x} + b_{K(t)}}
    + \underbrace{\big(
    \Phi^K_t( F_{K_0}( \hat{x} ) ) - \Phi^{\tilde{K}}_t( F_{\tilde{K}_0}( \hat{x} ) )
    \big)}_{\displaystyle =: \Phi_{K(t)}( \hat{x} )}.
  \end{align*}
  Since $\Phi^K_t( F_{K_0}( \cdot ) )$ and $\Phi^K_t( F_{\tilde{K}}( F_{\tilde{K}_0}( \cdot )) )$ agree at the vertices of $\hat{K}$, 
  we have that $\Phi^K_t ( F_{\tilde{K}} ( \cdot ) )$ is a linear interpolant of $\Phi^K_t( F_{K_0}( \cdot ) )$ over $\hat{K}$. Thus, we infer that
  \begin{align*}
    \abs{ \nabla_{\hat{x}} \left(
        \Phi^K_t( F_{K_0}( \hat{x} ) ) - \Phi^K_t( F_{\tilde{K}_0}( \hat{x} ) )
      \right) }
    & = \abs{ \nabla_{\hat{x}} \left(
        \Phi^K_t( F_{K_0}( \hat{x} ) ) - \hat{I}_1 \Phi^K_t( F_{K_0}( \hat{x} ) )
      \right) } \\
    & \le c \norm{ \Phi^K_t ( F_{K_0}( \cdot ) ) }_{W^{2,\infty}(\hat{K})} \\
    & \le c h^{2}_K \norm{ \Phi^K_t }_{W^{2,\infty}(K_0)}.
  \end{align*}
  Here, we have used the notation $\hat{I}_1$ for piecewise linear interpolation over $\hat{K}$ and applied the 
  Bramble-Hilbert Lemma (\cref{lem:bramble-hilbert}) and the rescaling \cref{eq:bulk-norm-scale-a}.

  Hence we have
  \begin{align*}
    C_{K(t)} = \sup_{\hat{x} \in \hat{K}} \norm{ \nabla D_K(\hat{x},t) A_K^\dagger }
    \le c h_K^2 \norm{ \Phi^K_t }_{W^{2,\infty}(K_0)} \norm{ A_K^\dagger }
    \le c \frac{ h^2_{K} }{ \rho_K } \abs{\Phi^K_t }_{W^{2,\infty}(K_0)}.
  \end{align*}
  Since we have assumed that $\{ \T_h(t) \}_{t \in [0,T]}$ is uniformly quasi-uniform we see that $h_K^2 / \rho_K$ is uniformly 
  bounded, it remains to show that $\abs{\Phi^K_t }_{W^{2,\infty}(K_0)}$ is uniformly bounded.
  However, this follows from the definition of $\Phi_t^h$ and the smoothness of $\Psi_K$ \cref{eq:PsiK-bound} and  the smooth flow map $\Phi_t$.
  Therefore, for $h_K$ sufficiently small, $C_{K(t)} < 1$.

  To show the higher order bounds, we compute similarly that for $2 \le m \le k+1$
  \begin{align*}
    & \abs{ \nabla_{\hat{x}}^m F_K( \hat{x} ) } \\
    & \le \abs{ \nabla_{\hat{x}}^m ( A_{K(t)} \hat{x} + b_{K(t)} ) } +
    \abs{ \nabla_{\hat{x}}^m \left(
        \Phi^K_t( F_{K_0}( \hat{x} ) ) - \Phi^K_t( F_{\tilde{K}_0}( \hat{x} ) )
      \right) } \\
    & = \abs{ \nabla_{\hat{x}}^m \left(
        \Phi^K_t( F_{K_0}( \hat{x} ) ) - \hat{I}_1 \Phi^K_t( F_{K_0}( \hat{x} ) )
      \right) } \\
    & \le c \norm{ \Phi^K_t ( F_{K_0}( \cdot ) ) }_{W^{m,\infty}(\hat{K})} \\
    & \le c h^{m}_K \norm{ \Phi^K_t }_{W^{m,\infty}(K_0)}.
  \end{align*}
  Hence, we see that
  \[
    \sup_{\hat{x} \in \hat{K}} \abs{ \nabla^m F_K( \hat{x} ) } \norm{ A_K }^{-m}
    \le c h^{m}_K \norm{ A_K }^{-m} \norm{ \Phi^K_t }_{W^{m,\infty}(K_0)}
    \le c \left( \frac{h_K}{\rho_K} \right)^m \norm{ \Phi^K_t }_{W^{m,\infty}(K_0)}.
  \]
  In the final inequality, we use that $\norm{ A_K }^{-1} \le \norm{ A_K^{-1} } \le \rho_K^{-1}$.
  The bounds on $C_m(K(t))$ then follow since we have assumed that the mesh is uniformly quasi-uniform and that $\Phi^K_t$ is sufficiently smooth.
\end{proof}

The element flow map $\Phi_t^K$ defines a velocity on each element ${W}_K$ (\cref{def:bulk-element-velocity}) by
\begin{equation*}
  \dt \Phi_t^K( \cdot ) = {W}_K( \Phi_t^K( \cdot ), t ) \qquad \mbox{ for } t \in [0,T].
\end{equation*}
This can be combined into a global velocity ${W}_h$ (\ref{def:bulk-global-discrete-velocity}).
We note that the global velocity is determined purely by the velocity of the vertices $\{ a_i(t) \}_{i=1}^N$:
\begin{equation}
  \label{eq:bulk-defn-Wh}
  {W}_h( x, t ) = \sum_{i=1}^N {w}( a_i(t), t ) \chi_i( x, t ) \qquad
  \mbox{ for } x \in \Omega_h(t).
\end{equation}
We also have a discrete push forward map $\phi^h_t$ (\cref{def:bulk-global-push-forward-map}) for $\eta_h \colon \Omega_{h,0} \to \R$ given by
\[
  \phi^h_t( \eta_h )( x ) = \eta_h ( \Phi^h_{-t}( x ) ), \qquad
  \mbox{ for } x \in \Omega_h(t), t \in [0,T].
\]
Since we have constructed a uniformly $k$-regular mesh, we infer that $( \H_h(t), \phi^h_t )_{t \in [0,T]}$ is a compatible pair (\cref{lem:bulk-Vh-compatible}) so we may define the space $C^1_{\H_h}$ \cref{eq:CkX} and a discrete material derivative by  \cref{eq:abs-mdh}:%
\[
  \mdh \eta_h = \phi^h_t \left( \dt \left( \phi^h_{-t} \eta_h \right) \right) \qquad \mbox{ for } \eta_h \in C^1_{\H_h}.
\]
We note that the pairs $( \V_h(t), \phi^h_t|_{\V_h(t)} )_{t \in [0,T]}$ and $( \S_h(t), \phi^h_t|_{\S_h(t)} )_{t \in [0,T]}$ are also compatible (\cref{lem:bulk-Vh-compatible,rem:bulk-Sh-compatible}) so we may also define the spaces $C^1_{\V_h}$ and $C^1_{\S_h}$.

  \begin{lemma}
  \label{lem:omegah-transport}
  For $t \in [0,T]$ and for each $K(t) \in \T_h(t)$,
  let $\A_\Omega^K$ be a smooth, positive-definite, diffusion tensor on $K(t)$ and $\b_\Omega^K$ be a smooth vector field on $K(t)$ for each $K(t) \in \T_h(t)$ and $t \in [0,T]$.
  For $\eta_h \in C^1_{\H_h}$ we have the transport formula:
  \begin{align}
    \label{eq:omegah-transport}
    \dt \int_{\Omega_h(t)} \eta_h \dd x
    = \sum_{K(t) \in \T_h(t)} \int_{K(t)} \mdh \eta_h + \eta_h \nabla \cdot {W}_h \dd x.
  \end{align}
  Furthermore for $\eta_h, \zeta_h \in C^1_{\V_h}$ we have
  \begin{align}
    \label{eq:bulk-dirichleth-transport}
    & \dt \sum_{K(t) \in \T_h(t)} \int_{K(t)} \A_\Omega^K \nabla \eta_h \cdot \nabla \zeta_h \dd x \\
    \nonumber
    & = \sum_{K(t) \in \T_h(t)} \int_{K(t)} \A_\Omega^K \nabla \mdh \eta_h \cdot \nabla \zeta_h
      + \A_\Omega^K \nabla \eta_h \cdot \nabla \mdh \zeta_h  \dd x \\
    \nonumber
    & \qquad + \sum_{K(t) \in \T_h(t)} \int_{K(t)} \B_K( {W}_K, \A_\Omega^K ) \nabla \eta_h \cdot \nabla \eta_h \dd x,
  \end{align}
  and for $\eta_h \in C^1_{\V_h}, \zeta_h \in C^1_{\H_h}$
  \begin{align}
    \label{eq:bulk-advectionh-transport}
    & \dt \sum_{K(t) \in \T_h(t)} \int_{K(t)} \b_\Omega^K \eta_h \cdot \nabla \zeta_h \dd x \\
    \nonumber
    & = \sum_{K(t) \in \T_h(t)} \int_{K(t)} \b_\Omega^K ( \mdh \eta_h \cdot \nabla \zeta_h +  \eta_h \cdot \nabla \mdh \zeta_h ) \dd x \\
    \nonumber
    & \qquad + \sum_{K(t) \in \T_h(t)} \int_{K(t)} \B_{\adv,K}( {W}_K, \b_\Omega^K )
      \eta_h \cdot \nabla \zeta_h \dd x,
  \end{align}
  where $\B_K$ and $\B_{\adv,K}$ are given by
  \begin{align*}
    \B_K( {W}_K, \A_\Omega^K ) & = \md_K \A_\Omega^K + \nabla \cdot {W}_K \A_\Omega^K - 2 D_h( {W}_h ) \\
    \B_{\adv, K}( {W}_K, \b_\Omega^K )
    & = \mdK \b_\Omega^K + \b_\Omega^K \nabla \cdot {W}_K -
      \sum_{j=1}^{n+1} ( \b_\Omega^K )_j ( \nabla )_j {W}_K,
  \end{align*}
  and $D_h$ is the rate of deformation tensor
  \begin{equation*}
    D_h( {W_h} )_{ij} = \frac{1}{2}
    \sum_{k=1}^{n+1} (\A_\Omega^K)_{ik} ( \nabla )_k ( {W}_K )_j
    + (\A_\Omega^K)_{jk} ( \nabla )_k ( {W}_K )_i
    \quad \mbox{ for } i,j = 1, \ldots, n+1.
  \end{equation*}
\end{lemma}

\begin{proof}
  We note that the left hand side maybe decomposed into individual elements then apply \cref{eq:bulk-transport,eq:bulk-transport-dirichlet,eq:bulk-transport-advection} on each element.%
\end{proof}

\subsection{Construction of the lifted finite element space}
\label{sec:bulk-lift-prop}

We construct a bijection between the computation domain $\Omega_h(t)$ and the continuous problem domain $\Omega(t)$ 
which we will call the lifting operator.
We do this using a similar construction to $\Psi_h$ used to define $\Omega_h(t)$ at the start of \cref{sec:fem-bulk}.
It will again be based on using an extension of the normal projection operator used as a lifting operator in \cref{sec:surf-lift-prop}.

Fix $t \in [0,T]$.
We wish to construct a bijection $\Lambda_h( \cdot, t ) \colon \Omega_h(t) \to \bar\Omega(t)$ which we will define element-wise.
We decompose $\T_h(t)$ into boundary elements, which have more than one vertex on the boundary, 
and interior elements. For an interior element $K(t) \in \T_h(t)$, we define
\begin{equation*}
  \Lambda_h( x, t ) = x \qquad \mbox{ for } x \in K(t).
\end{equation*}
Otherwise, let $K(t)$ be a boundary element and consider $x \in K(t)$.
Denote by $\{ a_i(t) \}_{i=1}^{n+2}$ the vertices of $K(t)$ ordered so that $\{ a_i(t) \}_{i=1}^L$ lie on $\Gamma(t)$.
We recall that the element $K(t)$ is given by a parametrisation $F_K$ over a reference element $\hat{K}$ so that, fixing a point $x \in K$, we 
can define points $\hat{x}$ and vertices $\{ \hat{a}_i \}_{i=1}^{n+2}$ in $\hat{K}$ by
\begin{equation*}
  x = F_K( \hat{x}, t ) \quad \mbox{ and } \quad
  a_i(t) = F_K( \hat{a}_i, t ) \mbox{ for } 1 \le i \le n+2.
\end{equation*}
We decompose $\hat{x}$ into barycentric coordinates on $\hat{K}$:
\begin{equation*}
  \hat{x} = \sum_{i=1}^{n+2} \hat\mu_i( \hat{x} ) \hat{a}_i.
\end{equation*}
We introduce the function $\hat{\mu}^*( \hat{x} )$ and the singular set $\hat\sigma$ by
\begin{equation*}
  \hat\mu^*( \hat{x} ) = \sum_{i=1}^L \hat\mu_i( \hat{x} ),
  \qquad
  \hat\sigma = \{ \hat{x} \in \hat{K} : \hat\mu^*( \hat{x} ) = 0 \}.
\end{equation*}
Again, we note that $0 \le \hat{\mu}^* \le 1$, $\abs{ D_{\hat{x}} \hat{\mu}^* } \le 1$ and $D_{\hat{x}}^m \hat{\mu}^* = 0$ for $m \ge 2$.
If $x \not\in F_K( \hat{\sigma}, t )$, we denote the projection onto $\Gamma_h(t) \cap K(t)$ by $y( x, t  )$ given by
\begin{equation}
  \label{eq:Lambda-y-defn}
  y( \hat{x}, t ) = F_K( \hat{y}( \hat{x} ), t ),
  \qquad
  \hat{y}( \hat{x} ) = \sum_{i=1}^L \frac{\hat\mu_i(\hat{x})}{\hat{\mu}^*(\hat{x})} \hat{a}_i.
\end{equation}
We then define $\Lambda_h( \cdot, t )|_{K(t)}$ by
\begin{equation}
  \label{eq:Lambdah-defn}
  \Lambda_h( \cdot, t )|_{K(t)}( x )
  =
  \begin{cases}
    x + ( \hat\mu^*( \hat{x} ) )^{k+2} ( p( y( x, t ), t ) - y( x, t ) )
    & \mbox{ if } x \not\in F_K( \hat{\sigma}, t ) \\
    x & \mbox{ otherwise}.
  \end{cases}
\end{equation}
The lifting map $\Lambda_h$ induces a lifted subdivision $\{ \T_h^\ell(t) \}_{t \in [0,T]}$ given as the set of all elements
\[
  K^\ell(t) := \Lambda_K( K(t), t ) \quad \mbox{ for all } K(t) \in \T_h(t).
\]

\begin{remark}
  The ideas behind this construction are similar to the construction of $\Psi_h$.
  See \cref{rem:why-Psih} for more discussion and also \cref{lem:Lambdah-good}.
\end{remark}

We next follow a sequence of calculations to show the properties of $\Lambda_h$.
These estimates are based on previous work by \citet{Ber89} and \citet{EllRan13}.
It is useful to recall the following Fa\'a di Bruno formula \cite[Eq. 2.9]{Ber89} for two smooth functions $f,g$
\begin{equation}
  \label{eq:multi-chain}
  \nabla^m ( f \circ g )
  = \sum_{r=1}^m \nabla^r f \left(
    \sum_{\underline{i} \in E(m,r)} c_{\underline{i}} \prod_{q=1}^m ( \nabla^q g )^{i_q}
    \right),
\end{equation}
where $\{ c_{\underline{}{i}} \}$ are constants and $E( m, r )$ is the set given by
\begin{equation*}
  E( m, r ) = \left\{
    \underline{i} \in \Nbb^m : \sum_{q=1}^m i_q = r \mbox{ and } \sum_{q=1}^m q i_q = m
  \right\}.
\end{equation*}

We wish to calculate derivatives of $\Lambda_h$ of order $m$ for $m \le k+1$.
A direct calculation shows that
\begin{equation*}
  \norm{ D^m_{\hat{x}} \hat{y} }_{L^\infty(\hat{K} \setminus \hat\sigma)}
  \le \frac{c}{( \hat\mu^*(\hat{x}) )^m},
\end{equation*}
for a constant $c$ independent of $\hat{x}$ and $K(t)$.
Then, we have that
\begin{equation}
  \label{eq:y-bound}
  \begin{aligned}
    \norm{ \nabla^m_{\hat{x}} y( \cdot, t ) }_{L^\infty(\hat{K} \setminus \hat{\sigma})}
    & \le c \sum_{r=1}^m
    \norm{ \nabla^r_{\hat{x}} F_K( \cdot, t ) }_{L^\infty(\hat{K})}
    \left(
      \sum_{\underline{i} \in E(m,r)}
      \prod_{q=1}^m \norm{ \nabla_{\hat{x}}^q \hat{y} }_{L^\infty(\hat{K} \setminus \sigma)}^{i_q}
    \right) \\
    & \le \sum_{r=1}^m \frac{h_K^r}{(\hat\mu^*(\hat{x}))^m} \le c \frac{h_K}{( \hat\mu^*(\hat{x}))^m },
  \end{aligned}
\end{equation}
where in the second line, we have used that $K$ is a $k$-surface finite element so
\begin{equation*}
  \norm{ \nabla^r_{\hat{x}} F_K( \cdot, t ) }_{L^\infty(\hat{K})}
  \le c \norm{ A_K }^r \le c h_K^r.
\end{equation*}
Next, applying \cref{eq:multi-chain} and \cref{eq:y-bound}, we have
\begin{align*}
  \norm{ \nabla^m_{\hat{x}} \big( p(y(\cdot,t),t) - y(\cdot,t) \big) }_{L^\infty(\hat{K}\setminus \hat\sigma)}
  \le c \sum_{r=1}^m \norm{ \nabla_y^r \big( p(\cdot,t) - \id ) }_{L^\infty(K(t))} \frac{ h_K^r }{ ( \hat\mu^*(\hat{x} ) )^m }.
\end{align*}
Using a similar geometric construction to \cref{lem:surf-geom-est}, we have
\begin{equation*}
  \norm{ \nabla^r \big( p(\cdot,t) - \id ) }_{L^\infty(K(t))}
  = \norm{ \nabla^r d }_{L^\infty(K(t))} \le c h_K^{k+1-r}.
\end{equation*}
Hence, we infer that
\begin{align*}
  \norm{ \nabla^m_{\hat{x}} \big( p(y(\cdot,t),t) - y(\cdot,t) \big) }_{L^\infty(\hat{K}\setminus \hat\sigma)}
  \le c \frac{h_K^{k+1}}{( \hat\mu^*(\hat{x}) )^m}.
\end{align*}
Finally, using the Leibniz formula and the properties of $\hat{\mu}^*$, we have
\begin{align*}
  & \abs{ \nabla^{m}_{\hat{x}} \big( ( \hat\mu^*(\hat{x} ) )^{k+2} ( p(y({x},t),t) - y({x},t ) ) \big) } \\
  & \quad \le c \sum_{r=0}^m ( \hat\mu^*(\hat{x} ) )^{k+2-r} \norm{ \nabla_{\hat{x}}^{m-r}\big( p(y(\cdot),t) - y(\cdot) \big) }_{L^\infty(\hat{K} \setminus \hat\sigma)} \\
  & \quad \le c \sum_{r=0}^m ( \hat\mu^*(\hat{x} ) )^{k+2-r} \frac{ h_K^{k+1} }{ ( \hat\mu^*(\hat{x} ) )^{m-r} } \\
  & \quad \le c h_K^{k+1} ( \hat\mu^*(\hat{x} ) )^{k+2-m}.
\end{align*}

\begin{lemma}
  \label{lem:Lambdah-good}
  The lifting function $\Lambda_h( \cdot, t ) \colon \Omega_h(t) \to \bar\Omega(t)$ is an element-wise $C^{k+1}$-diffeomorphism and satisfies
  \begin{equation}
    \sup_{h \in (0,h_0)} \sup_{t \in [0,T]} \norm{ \Lambda_h( \cdot, t ) }_{W^{k+1,\infty}(\T_h(t))}
    \le C.
  \end{equation}
\end{lemma}

\begin{proof}
  The first result is clear for all internal elements.
  Consider a time $t \in [0,T]$ and a single boundary element $K(t) \in \T_h(t)$.
  The smoothness away from $\sigma(t) = F_K( \hat{\sigma}, t )$ follows from the fact that $\Lambda_h$ restricted to each element is the composition of smooth functions.
  The above calculations, combined with \cref{eq:bulk-norm-scale-b} and \cref{eq:bulk-AK-h-relation-b}, show that for $x = F_K( \hat{x}, t ) \in K(t) \setminus \sigma(t)$,
  \begin{multline}
    \label{eq:bulk-Lambdah-m-bound}
    \abs{ \nabla^m \Lambda_h|_{K(t)}( x, t ) - \id }
    \le \frac{c}{\rho_K^m} \abs{ \nabla^m \Lambda_h|_{K(t)}( F_K( \hat{x}, t ), t ) - \id }
    \le c \frac{ h_K^{k+1} }{ \rho_K^m } \abs{( \hat\mu^*(\hat{x}) )^m}.
  \end{multline}
  The mapping $\Lambda_h$ is of class $C^{k+1}$ on $K(t) \setminus \sigma(t)$ with derivatives of order less than or equal to $k+1$ tending 
  to zero when $x$ tends to a point in $\sigma(t)$. Hence, it can be extended to a $C^{k+1}$ mapping on $K(t)$.
\end{proof}

\begin{lemma}
  \label{lem:bulk-Thell-good}
  Furthermore, we have that the lifted triangulation $\{ \T_h^{\ell}(t) \}_{t \in [0,T]}$ is a uniformly $k$-regular evolving subdivision (\cref{def:bulk-uniformly-Theta-regular}).
\end{lemma}

\begin{proof}
  We use the splitting
  \[
    \Lambda_h( x,t ) = x + \tilde{\Lambda}_h(x,t),
  \]
  and write $\tilde{\Lambda}_K(\cdot,t) = \tilde{\Lambda}_h( \cdot , t )|_{K(t)} = ( \Lambda_K( x, t ) - x )$.
  We start by taking $m = 1$ in \cref{eq:bulk-Lambdah-m-bound} to give
  \[
    \abs{ \nabla^m \Lambda_h|_{K(t)}( x, t ) - \id }
    \le c \frac{h_K^{k+1}}{\rho_K} \abs{ (\hat\mu^*(\hat{x}) ) }.
  \]
  Clearly the right hand side is less than $(1-C_K)/(1+C_K)$ for $h_K$ small enough for a uniformly regular subdivision.
  This inequality can clearly be translated to the required estimate on $\Lambda_h$.
\end{proof}

For $t \in [0,T]$ and a function $\eta_h \colon \bar\Omega_h(t) \to \R$, we define its lift $\eta_h^\ell \colon \bar\Omega(t) \to \R$ by
\begin{equation*}
  \eta_h^\ell( \Lambda_h( x, t ) ) = \eta_h( x ) \qquad \mbox{ for } x \in \Omega_h(t).
\end{equation*}
We will also make use of an inverse lift for functions on $\bar\Omega(t)$. For $\eta \colon \bar\Omega(t) \to \R$, we define the inverse lift of $\eta$, denoted by $\eta^{-\ell} \colon \bar\Omega_h(t) \to \R$ by
\begin{equation*}
  \eta^{-\ell}( x ) = \eta( \Lambda_h( x, t ), t ) \qquad \mbox{ for } x \in \Omega_h(t).
\end{equation*}

\begin{lemma}
  \label{lem:bulk-lift-stab}
  Let $\eta_h \in \H_h(t)$ and denote its lift by $\eta_h^\ell$. Then there exists constants $c_1, c_2 > 0$ such that
  \begin{align}
    c_1 \norm{ \eta_h^\ell }_{\H(t)}
    \le \norm{ \eta_h }_{\H_h(t)}
    \le c_2 \norm{ \eta_h^\ell }_{\H(t)} \\
    c_1 \norm{ \eta_h^\ell }_{\V(t)}
    \le \norm{ \eta_h }_{\V_h(t)}
    \le c_2 \norm{ \eta_h^\ell }_{\V(t)} && \mbox{ if } \eta_h \in \V_h(t).
  \end{align}
\end{lemma}

\begin{proof}
  We apply \cref{prop:bulk-lift-Vh,lem:bulk-Thell-good}.
\end{proof}

For each $K(t) \in \T_h(t)$, we use \cref{def:lifted-bfe} to construct an associated lifted bulk finite element $(K^\ell(t), P^\ell(t), \Sigma^\ell(t) )$.
We assume that the domains $\{ \Omega_h(t) \}_{t \in [0,T]}$ are such that the set of lifted element domains 
$\T_h^\ell(t)$ defines an exact decomposition of $\Omega(t)$ (\cref{def:bulk-exact-decomposition}) for each $t \in [0,T]$.
For $t \in [0,T]$, we define the space of lifted functions to be $\S_h^\ell(t)$ given by (c.f.\ \cref{eq:abs-Vhell})
\[
  \S_h^\ell(t) := \{ \chi_h^\ell : \chi_h \in \S_h(t) \}.
\]

Using the lift $\Lambda_h$ and the discrete flow $\Phi^h_t$, we can define the lifted flow map $\Phi^\ell_t$ (\cref{def:bulk-lifted-flow-map}), 
lifted discrete velocity $w_h$ (\cref{def:bulk-lifted-discrete-material-velocity}) and lifted push forward 
maps $\phi^\ell_t$ (\cref{def:bulk-lifted-push-forward-map}).
The results of \cref{lem:bulk-Thell-good} imply with \cref{prop:bulk-lift-Vht} that $( \H(t), \phi^\ell_t )_{t \in [0,T]}$ 
and $( \V(t), \phi^\ell_t|_{\V_0} )_{t \in [0,T]}$ and $( \S_h^\ell(t), \phi^\ell_t|_{\S_{h,0}^\ell} )_{t \in [0,T]}$ form compatible pairs (\cref{def:compatibility}).
We will use the notations $C^1_{(\H,\phi^\ell)}$ and $C^1_{(\V,\phi^\ell)}$ for the spaces of functions smoothly evolving in 
time \cref{eq:CkX} with respect to the push-forward map $\phi^\ell_t$ in $\H(t)$ and $\V(t)$ respectively.
We may define a lifted strong material derivative \cref{eq:lift-md}:%
\[
  \mdell \eta = \phi_t^\ell \left( \dt \left( \phi^\ell_{-t} \eta \right) \right) \quad \mbox{ for } \eta \in C^1_{(\H,\phi^\ell)}.
\]

\begin{lemma}
  \label{lem:bulk-lifted-transport}
  The flow map $\Phi_t^\ell$ induces a new transport formula on $\{ \Omega(t) \}$. For $\eta \in C^1_{(\H, \phi^\ell)}$ we have
  \begin{equation}
    \dt \int_{\Omega(t)} \eta \dd x
    = \sum_{K^\ell(t) \in \T_h^\ell(t)} \int_{K^\ell(t)} \mdell \eta + \eta \nabla \cdot {w}_h \dd x.
  \end{equation}
  Let $\A_\Omega$ be a smooth, positive-definite, diffusion tensor on $\Omega(t)$ and $\b_\Omega$ be a smooth vector field on $\Omega(t)$ for all $t \in [0,T]$.  Furthermore, we have for $\eta, \zeta \in C^1_{(\V, \phi^\ell)}$
  \begin{multline}
    \dt \int_{\Omega(t)} \A_\Omega \nabla\eta \cdot \nabla\zeta \dd x \\
    = \sum_{K^\ell(t) \in \T_h^\ell(t)}
    \int_{K^\ell(t)}
    \A_\Omega \left( \nabla\mdell \eta \cdot \nabla\zeta + \nabla\eta \cdot \nabla\mdell \zeta \right)
    + \B_{K^\ell}( {w}_K, \A_\Omega ) \nabla\eta \cdot \nabla\zeta \dd x
  \end{multline}
  and for $\eta \in C^1_{(\V, \phi^\ell)}, \zeta \in C^1_{(\H, \phi^\ell)}$ we have
  \begin{multline}
    \dt \int_{\Omega(t)} \b_\Omega \eta \cdot \nabla \zeta \dd x \\
    = \sum_{K^\ell(t) \in \T_h^\ell(t)}
    \int_{K^\ell(t)}
    \b_\Omega \left( \mdell \eta \cdot \nabla\zeta + \eta \cdot \nabla\mdell \zeta \right)
    + \B_{K^\ell,\adv}( {w}_K, \b_\Omega ) \eta \cdot \nabla \zeta \dd \sigma,
  \end{multline}
  where $\B_{K^\ell}$ and $\B_{K^\ell,\adv}$ are defined as in \cref{lem:omegah-transport}.
\end{lemma}

\begin{proof}
  We note that the left hand side may be decomposed into individual elements then apply \cref{eq:bulk-transport,eq:bulk-transport-dirichlet,eq:bulk-transport-advection} on each element.%
\end{proof}

For the later analysis, we will also require bounds on the time derivative of $\Lambda_h$.
We consider an element $K(t) \in \T_h(t)$ and the trajectory of a point $X(t)$ which follows the velocity field ${W}_h$.
From the definition of ${W}_h$, we have that
\begin{equation*}
  X(t) = F_K( \hat{x}, t ), \qquad \hat{x} = \sum_{j=1}^{n+2} \hat\mu_j \hat{a}_i.
\end{equation*}
In particular, the barycentric coordinate representation of $X(t)$ does not depend on time.
Therefore, writing $x = X(t)$ and using $y$ from \cref{eq:Lambda-y-defn}, we have
\begin{equation*}
  \mdh y( x, t ) = \dt y( X(t), t )
  = \dt F_K( \hat{y}(\hat{x}), t ) = \frac{ \partial F_K }{ \partial t } ( \hat{y}(\hat{x}), t )
  = {W}_h( y( x, t ), t ).
\end{equation*}
Then we can compute that if $K(t)$ is a boundary element, recalling the definition of $\Lambda_h$  \cref{eq:Lambdah-defn}, we have
\begin{align*}
  & \mdh \Lambda_h|_{K(t)}( x, t ) = \dt \Lambda_h|_{K(t)}( X(t), t ) \\
  & =
  \begin{cases}
    {W}_h( x, t ) + ( \hat\mu^*(x) )^{k+2} \left( \frac{\partial p}{\partial t}( y, t ) + \nabla p ( y, t ) {W}_h( y( x, t ), t )  - {W}_h( y( x, t ), t ) \right)
    & \mbox{ if } x \not\in F_K( \hat{\sigma}, t ) \\
    0 & \mbox{ otherwise}
  \end{cases} \\
  & =
  \begin{cases}
    {W}_h( x, t ) +
    ( \hat\mu^*(x) )^{k+2}
    \left(
      \big( {W}_h( y, t ) - {w}( y, t ) \big) \cdot \nu( y, t ) \nu( y,t )
      - d(y,t) \nu_t( y,t )
 \right)
    & \mbox{ if } x \not\in F_K( \hat{\sigma}, t ) \\
    0 & \mbox{ otherwise}.
  \end{cases}
\end{align*}
A similar calculation to those prior to \cref{lem:Lambdah-good} show that for $m \le k+1$
\begin{equation}
  \label{eq:mdLambdah-estimate}
  \norm{ \mdh \Lambda_h( \cdot, t ) - {W}_h( \cdot, t ) }_{W^{m,\infty}(K(t))}
  \le c h_K^{k+1-m} ( \hat\mu^*( \hat{x} ) )^{k+1-m}
  \le c h_K^{k+1-m}.
\end{equation}
This follows by using the smoothness of the surface $\Gamma(t)$ along with the fact that ${W}_h$ is an interpolant of ${w}$ (\cref{cor:bulk-global-inv-lift-interp}).

We conclude this section by  showing some geometric estimates arising from the use of the lifting function $\Lambda_h$.
\begin{lemma}
  \label{lem:Lambdah-bound}
  Under the above assumptions, we have the estimates that
  \begin{align}
    \label{eq:DLambdah-bound}
    \sup_{t \in [0,T]} \norm{ \nabla \Lambda_h( \cdot, t ) - \id }_{L^\infty( \Omega_h(t) )}
    & \le c h^{k} \\
    \label{eq:mdDLambdah-bound}
    \sup_{t \in [0,T]} \norm{ \mdh \nabla \Lambda_h( \cdot, t ) }_{L^\infty( \Omega_h(t) )}
    & \le c h^{k}.
  \end{align}
  Additionally, writing $J_h(\cdot, t) = \sqrt{ \det\big( ( \nabla \Lambda_h( \cdot, t ) )^t ( \nabla \Lambda_h( \cdot, t ) ) \big) }$, we have
  \begin{align}
    \label{eq:detDLambdah-bound}
    \sup_{t \in [0,T]} \norm{ J_h(\cdot,t) - 1 }_{L^\infty(\Omega_h(t))}
    & \le c h^k \\
    \label{eq:mddetDLambdah-bound}
    \sup_{t \in [0,T]} \norm{ \mdh J_h(\cdot, t) }_{L^\infty(\Omega_h(t))}
    & \le c h^k.
  \end{align}
\end{lemma}

\begin{proof}
  We have already shown \cref{eq:DLambdah-bound} as part of the proof of \cref{lem:Lambdah-good}.
 Inequality  \cref{eq:detDLambdah-bound} then follows by a result of \citep[Cor. 2.11]{IpsReh08}.

 {
   To show the time derivative bounds, we note that for each $K(t) \in \T_h(t)$ we can write
   \[
     \mdK \eta_h( x, t ) = \partial_t \eta_h( x, t ) + W_K( x, t ) \cdot \nabla \eta_h( x, t ) \qquad \mbox{ for } x \in \interior{K}(t).
   \]
   Then we can compute, for $i = 1, \ldots, n+1$
 \begin{align*}
   \mdh \partial_{x_i} \Lambda_h
   & = ( \partial_{x_i} \Lambda_h )_t + \sum_{j=1}^{n+1} \partial_{x_j} \partial_{x_i} \Lambda_h ( W_h )_{j} \\
   & = \partial_{x_i} ( \partial_t \Lambda_h ) + \partial_{x_i} \left( \sum_{j=1}^{n+1} \partial_{x_j} \Lambda_h (W_h)_j \right)
     - \sum_{j=1}^{n+1} \partial_{x_j} \Lambda_h \partial_{x_i} ( W_h )_j \\
   & \qquad - \sum_{j=1}^{n+1} \partial_{x_i} \partial_{x_j} \Lambda_h ( W_h )_j
     + \sum_{j=1}^{n+1} \partial_{x_i} \partial_{x_j} \Lambda_h ( W_h )_j.
 \end{align*}
 Hence we have that (up to a set of measure zero)
  \begin{align*}
    \mdh \nabla \Lambda_h
    = \nabla \mdh \Lambda_h - ( \nabla {W}_h ) ( \nabla \Lambda_h ).
  \end{align*}
  }
  Then applying \cref{eq:mdLambdah-estimate}, \cref{eq:DLambdah-bound} together with \cref{lem:bulk-Wh-bounded}, we have
  \begin{align*}
    \norm{ \mdh \nabla \Lambda_h }_{L^\infty(\Omega_h(t))}
    & \le \norm{ \nabla ( \mdh \Lambda_h - {W}_h )}_{L^\infty(\Omega_h(t))}
    + \norm{ \nabla {W}_h ( \id - \nabla \Lambda_h )}_{L^\infty(\Omega_h(t))} \\
    & \le c h^{k} + \norm{ \nabla {W}_h }_{L^\infty(\Omega_h(t))} c h^k \le c h^k.
  \end{align*}
  This shows \cref{eq:mdDLambdah-bound}.

  For \cref{eq:mddetDLambdah-bound} we have, applying \cref{eq:DLambdah-bound} and \cref{eq:mdDLambdah-bound}, that
  \begin{align*}
    & \abs{ \mdh \sqrt{ \det\big( \nabla \Lambda_h( x, t )^t \nabla \Lambda_h( x, t ) \big) } } \\
    & = \abs{
      \frac{1}{2} \sqrt{
      \det \big( \nabla \Lambda_h( x, t )^t \nabla \Lambda_h( x, t ) \big)
      }
      \trace \left(
      \big( \nabla \Lambda_h( x, t )^t \nabla \Lambda_h( x, t ) \big)^{-1}
      \mdh \big( \nabla \Lambda_h( x, t )^t \nabla \Lambda_h( x, t ) \big)
      \right)
      } \\
    & \le c \norm{ \mdh \nabla \Lambda_h } \le c h^k.
      \qedhere
  \end{align*}
\end{proof}

\subsection{The discrete problem and stability}
\label{sec:bulk-discrete-pr}

The practical finite element method is based on the variational formulation \cref{eq:abs-var-form} of \cref{pb:weak-bulk}.
We introduce a element-wise smooth $(n+1)\times(n+1)$-diffusion tensor $\A^h_\Omega$, an element-wise smooth $(n+1)$ 
dimensional vector field $\b^h_\Omega$ and an element-wise smooth scalar field $\c^h_\Omega$.
We will use the notation $\A^K_\Omega := \A^h_\Omega|_{K(t)}, \b^K_\Omega := \b^h_\Omega|_{K(t)}$ and $\c^K_\Omega := \c^h_\Omega|_{K(t)}$, for all $K(t) \in \T_h(t)$, which we assume satisfy:
\begin{align*}
  \sup_{h \in (0,h_0)} \sup_{t \in [0,T]} \sup_{K(t) \in \T_h(t)} \left( \norm{ \A^K_\Omega }_{L^\infty(K(t))}
  + \norm{ \b^K_\Omega }_{L^\infty(K(t))}
  + \norm{ \c^K_\Omega }_{L^\infty(K(t))}
  \right)
  \le C.
\end{align*}
We assume that $\A^K_\Omega$ is uniformly positive definite: There exists $\bar{a}_0 > 0$ such that for all $h \in (0,h_0), t \in [0,T]$ and $K(t) \in \T_h(t)$
\[
  \A_\Omega^K(t, \cdot) \xi \cdot \xi \ge \bar{a}_0 \abs{ \xi }^2
  \qquad
   \mbox{ for all } \xi \in \R^{n+1}.
\]

\begin{problem}
  \label{pb:fem-bulk}
  Given $U_{h,0} \in \S_{h,0}$, find $U_h \in C^1_{\S_h}$ such that for every $t \in [0,T]$
  \begin{equation}
    \label{eq:fem-bulk}
    \begin{aligned}
      \dt m_h( t; U_h, \zeta_h )
      + a_h( t; U_h, \zeta_h )
      & = m_h( t; U_h, \mdh \zeta_h )
      && \mbox{ for all } \zeta_h \in C^1_{\S_{h}} \\
      U_h( \cdot ) & = U_{h,0},
    \end{aligned}
  \end{equation}
  where we have
  \begin{align*}
    m_h( t; \eta_h, \zeta_h )
    & = \int_{\Omega_h(t)} \eta_h \zeta_h \dd x
    && \mbox{ for } \eta_h, \zeta_h \in \H_h(t) \\
    a_h( t; \eta_h, \zeta_h )
    & = \sum_{K(t) \in \T_h(t)} \int_{K(t)} \A^K_\Omega \nabla \eta_h \cdot \nabla \zeta_h + \b^K_\Omega \eta_h \cdot \nabla \zeta_h + \c^K_\Omega \eta_h \zeta_h \dd x
    && \mbox{ for } \eta_h, \zeta_h \in \V_h(t).
  \end{align*}
\end{problem}

To show the properties of these bilinear forms we require one further lemma:
\begin{lemma}
  \label{lem:bulk-Wh-bounded}
  The discrete velocity ${W}_h$ of the discrete evolving domain $\{ \Omega_h(t) \}$ is uniformly bounded in $W^{1,\infty}(\Omega_h(t))$. 
  That is, there exists a constant $C > 0$ such that for all $h \in (0,h_0)$
  \begin{equation}
    \label{eq:bulk-Wh-bound}
    \sup_{t \in [0,T]} \norm{ {W}_h }_{W^{1,\infty}(\T_h(t))} \le C.
  \end{equation}
\end{lemma}
\begin{proof}
  The bound follows using the characterisation \cref{eq:bulk-defn-Wh} by using the interpolation bound 
  shown in \cref{cor:bulk-global-inv-lift-interp}.
\end{proof}

We have a transport formula for the domain $\{ \Omega_h(t) \}$.

\begin{lemma}
  \label{lem:bulk-fem-transport}
  There exists bilinear forms 
  $g_h( t; \cdot, \cdot ) \colon \H_h(t) \times \H_h(t) \to \R$, and
  $b_h( t; \cdot, \cdot ) \colon \V_h(t) \times \V_h(t) \to \R$ such that
  \begin{subequations}
    \begin{align}
      \label{eq:8}
      \dt m_h( t; \eta_h, \zeta_h )
      & = m_h( t; \mdh \eta_h, \zeta_h ) + m_h( t; \eta_h; \mdh \zeta_h )
      + g_h( t; \eta_h, \zeta_h )
      && \mbox{ for } \eta_h, \zeta_h \in C^1_{\H_h} \\
      \dt a_h( t; \eta_h, \zeta_h )
      & = a_h( t; \mdh \eta_h, \zeta_h ) + a_h( t; \eta_h; \mdh \zeta_h )
      + b_h( t; \eta_h, \zeta_h )
      && \mbox{ for } \eta_h, \zeta_h \in C^1_{\V_h},
    \end{align}
  \end{subequations}
  where
  \begin{equation*}
    g_h( t; \eta_h, \zeta_h )
    := \int_{\Omega_h(t)} \eta_h \zeta_h \nabla \cdot {W}_h \dd x,
  \end{equation*}
  and
  \begin{align*}
    b_h( t; \eta_h, \zeta_h )
    := \sum_{K(t) \in \T_h(t)}
    \int_{K(t)} & \B_h( {W}_h, \A_\Omega^K ) \nabla \eta_h \cdot \nabla \zeta_h
    + \B_{\adv,h}( {W}_h, \b_\Omega^K ) \eta_h \cdot \nabla \zeta_h \\
    & + ( \mdh \c_\Omega^K + \c_\Omega^K \nabla \cdot {W}_h ) \eta_h \zeta_h \dd x.
  \end{align*}

  Furthermore, there exists a constants $c > 0$ such that for all $t \in [0,T]$ and all $h \in (0,h_0)$ we have
  \begin{align*}
    \abs{ g_h( t; \eta_h, \zeta_h ) }
    & \le c \norm{ \eta_h }_{\H_h(t)} \norm{ \zeta_h }_{\H_h(t)}
    && \mbox{ for all } \eta_h, \zeta_h \in \H_h(t) \\
    \abs{ b_h( t; \eta_h, \zeta_h ) }
    & \le c \norm{ \eta_h }_{\V_h(t)} \norm{ \zeta_h }_{\V_h(t)}
    && \mbox{ for all } \eta_h, \zeta_h \in \V_h(t).
  \end{align*}
\end{lemma}

\begin{proof}
  The bilinear forms exist due to the more general \cref{lem:omegah-transport}.
  The estimates follow from \cref{lem:bulk-Wh-bounded}.
\end{proof}

\begin{theorem}
  \label{thm:bulk-fem-stability}
  There exists a unique solution of the finite element scheme \cref{eq:fem-bulk}. The solution $U_h$ satisfies the stability bound:
  \begin{equation}
    \sup_{t \in (0,T)} \norm{ U_h }_{\H_h(t)}^2 + \int_0^T \norm{ U_h }_{\V_h(t)}^2 \dd t
    \le c \norm{ U_{h,0} }_{\H_h(t)}^2.
  \end{equation}
\end{theorem}

\begin{proof}
  We apply the abstract result of \cref{thm:abs-stab}. It is left to check the required assumptions.
  The assumptions on $m_h$, \cref{eq:mh-symmetric} and \cref{eq:mh-bounded}, follow directly since $m_h$ is equal to the $\H_h(t)$ inner-product.
  That $t \mapsto a_h(t; \cdot, \cdot)$ is measurable \cref{eq:ah-meas} and the estimates on $a_h$, \cref{eq:ah-coercive} and \cref{eq:ah-bounded} follow in the same manner as \cref{thm:bulk-exist}.
  The transport formulae and estimates for $g_h$ and $b_h$, \cref{eq:ch-formula}, \cref{eq:ch-bounded} \cref{eq:bh-formula} and \cref{eq:bh-bounded}, 
  are shown in \cref{lem:bulk-fem-transport}.
\end{proof}

\subsection{Error analysis}
\label{sec:bulk-error}

The space $\S_h^\ell(t)$ is equipped with the following approximation property.
\begin{lemma}
  \label{lem:bulk-approx}
  For $\eta \in C(\Omega(t))$ there exists a Lagrangian interpolation operator $I_h \eta \in \S_h^\ell(t)$ that is well defined. 
  Furthermore, the following bounds hold for constants independent of $h$ and time:
  \begin{align}
    \label{eq:bulk-approx}
    \norm{ \eta - I_h \eta }_{L^2(\Omega(t))}
    + h \norm{ \nabla( \eta - I_h \eta )}_{L^2(\Omega(t))}
    & \le c h^{k+1} \norm{ \eta }_{\Z(t)}
    && \mbox{ for } \eta \in \Z(t) \\
    \label{eq:bulk-approx-0}
    \norm{ \eta - I_h \eta }_{L^2(\Omega(t))}
    + h \norm{ \nabla( \eta - I_h \eta )}_{L^2(\Omega(t))}
    & \le c h^{2} \norm{ \eta }_{\Z_0(t)}
    && \mbox{ for } \eta \in \Z_0(t).
  \end{align}
\end{lemma}

\begin{proof}
  We apply \cref{thm:bulk-global-lift-interp}.
  The second result applies the theorem in the obvious way. The first result applies the theorem with 
  $k=1$ noting that $P_1(\hat{K}) \subset P_k(\hat{K})$ and the inclusions for \cref{lem:bramble-hilbert} still hold.
\end{proof}

We can also use the lift to define an evolving lifted triangulation. For each $t \in [0,T]$ and $h \in (0,h_0)$, we define
\begin{equation*}
  \T_h^\ell(t) := \{ K^\ell(t) : K(t) \in \T_h(t) \}.
\end{equation*}
The edges of these curvilinear-simplicies evolve with a velocity ${w}_h$ which can be characterised as follows.
Let $X(t)$ be the trajectory of a point on $\Omega_h(t)$ according to the flow $\Phi^h_t$.
Then we have that
\begin{equation*}
  {W}_h( X(t), t ) = \dt X(t).
\end{equation*}
Now consider a point $Y(t) = \Lambda_h( X(t), t )$. The trajectory of $Y(t)$ defines the velocity field ${w}_h$ by
\begin{equation}
  \label{eq:bulk-wh}
  {w}_h( Y(t), t ) := \dt Y(t)
  = ( \partial_t \Lambda_h )( X(t), t ) + ( \nabla \Lambda_h )( X(t), t ) {W}_h( X(t), t ).
\end{equation}
Equivalently this means the flow $\Phi_t^\ell$ is given by
\begin{equation*}
  \Phi_{t^*}^\ell( y_0 ) = Y(t^*) \quad \mbox{ such that } \quad
  \dt Y(t) = {w}_h( Y(t), t ), \, Y(0) = y_0.
\end{equation*}

\begin{lemma}
  \label{lem:bulk-lift-transport}
  The pairs $\{ \H(t), \phi^\ell_t \}_{t \in [0,T]}$ and $\{ \V(t), \phi^\ell_t \}_{t \in [0,T]}$ are compatible and we may define a 
  material derivative $\mdell \eta$ for $\eta \in C^1_{(\H,\phi^\ell)}$ and transport formula: There exists a bilinear for $g_\ell( t; \cdot, \cdot ) \colon \H(t) \times \H(t) \to \R$ given by
  \begin{align*}
    g_\ell( t; \eta, \zeta )
    & := \int_{\Omega(t)} \eta \zeta \nabla \cdot w_h \dd x,
  \end{align*}
  such that
  \begin{equation}
    \dt m( t; \eta, \zeta )
    = m( t; \mdell \eta, \zeta ) + m( t; \eta, \mdell \zeta )
    + g_\ell( t; \eta, \zeta )
    \qquad
    \mbox{ for } \eta, \zeta \in C^1_{(\H,\phi^\ell)},
  \end{equation}
  and there exists a constant $c > 0$ such that for all $t \in [0,T]$ and $h \in (0,h_0)$ we have
  \begin{equation}
    \abs{ g_\ell( t; \eta, \zeta ) }
    \le c \norm{ \eta }_{\H(t)} \norm{ \zeta }_{\H(t)}
    \qquad
    \mbox{ for all } \eta, \zeta \in \H(t).
  \end{equation}
  Furthermore, we have a new transport formula for the $a$ bilinear form. There exists a bilinear form $b_\ell( t; \cdot, \cdot ) \colon \V(t) \times \V(t) \to \R$ given by
  \begin{align*}
    b_\ell( t; \eta, \zeta )
    & := \int_{\Omega(t)} \B( w_h, \A_\Omega ) \nabla \eta \cdot \nabla \zeta
      + \B_{\adv}( w_h, \b_\Omega ) \eta \cdot \nabla \zeta
      + ( \mdell \c_\Omega + \c_\Omega \nabla \cdot w_h ) \eta \zeta \dd x,
  \end{align*}
  such that
  \begin{equation}
    \dt a( t; \eta, \zeta )
    = a( t; \mdell \eta, \zeta ) + a( t; \eta, \mdell \zeta )
    + b_\ell( t; \eta, \zeta )
    \qquad
    \mbox{ for } \eta, \zeta \in C^1_{(\V,\phi^\ell)},
  \end{equation}
  and there exists a constant $c > 0$ such that for all $t \in [0,T]$ and $h \in (0,h_0)$ we have
  \begin{equation}
    \abs{ b_\ell( t; \eta, \zeta ) }
    \le c \norm{ \eta }_{\V(t)} \norm{ \zeta }_{\V(t)}
    \qquad
    \mbox{ for all } \eta, \zeta \in \V(t).
  \end{equation}
\end{lemma}

\begin{proof}
  We simply apply \cref{lem:bulk-lifted-transport}.
\end{proof}

The bounds in \cref{lem:Lambdah-bound} allow to show some of the abstract error bounds required use the result of \cref{thm:error}.

\begin{lemma}
  \label{lem:bulk-w-md-errors}
  We have the estimate:
  \begin{equation}
    \label{eq:bulk-w-error}
    \abs{ {w} - {w}_h }_{L^\infty(\Omega(t))}
    + h \abs{ \nabla ( {w} - {w}_h ) }_{L^\infty(\Omega(t))}
    \le c h^{k+1}.
  \end{equation}
  If $\eta \in C^1_\H \cap C^0_\V$ then $\eta \in C^1_{(\H,\phi^\ell)}$ and if $\eta \in C^1_\V \cap C^0_{\Z_0}$ then $\eta \in C^1_{(\V,\phi^\ell)}$. Moreover
  \begin{subequations}
    \label{eq:bulk-md-error}
    \begin{align}
      \norm{ \md \eta - \mdell \eta }_{L^2(\Omega(t))}
      & \le c h^{k+1} \norm{ \eta }_{H^1(\Omega(t))}
      && \mbox{ for } \eta \in C^1_{\V} \\
      \norm{ \nabla( \md \eta - \mdell \eta ) }_{L^2(\Omega(t))}
      & \le c h^{k} \norm{ \eta }_{H^2(\Omega(t))}
      && \mbox{ for } \eta \in C^1_{\Z_0}.
    \end{align}
  \end{subequations}
\end{lemma}

\begin{proof}
  We write for $x \in \Omega_h(t)$ that
  \begin{align*}
    & {w}( \Lambda_h( x, t), t )
    - {w}_h( \Lambda_h( x, t ) ) \\
    & = \big( {w}( \Lambda_h( x, t ), t )
    - {W}_h( x, t ) \big)
    + \big( {W}_h( x, t ) - {w}_h( x, t ) \big) \\
    & = \big( {w}( \Lambda_h( x, t ), t )
    - {W}_h( x, t ) \big)
    + \big( {W}_h( x, t ) - \mdh \Lambda_h( x,t ) \big).
  \end{align*}
  This allows us to apply the interpolation theorem (\cref{cor:bulk-global-inv-lift-interp}) to the embedded velocity with the 
  estimate \cref{eq:mdLambdah-estimate} to achieve the estimate \cref{eq:bulk-w-error}.
  The inclusions and bounds \cref{eq:bulk-md-error} follow from a simple calculation finding $\mdell \eta - \md \eta$ for appropriate $\eta$ and \cref{eq:bulk-w-error}.
\end{proof}

For the remainder of this section, we will take $\A_\Omega^h = \A_\Omega^{-\ell}$, $\b_\Omega^h = \b_\Omega^{-\ell}$ and $\c_\Omega^h = \c_\Omega^{-\ell}$.
and assume that $\A_\Omega$ is of class $C^2$ in space.

\begin{lemma}
  \label{lem:bulk-geom-pert}
  There exists a constant $c > 0$ such that for all $t \in [0,T]$ and all $h \in (0,h_0)$ the following holds for all $\eta_h, \zeta_h \in \V_h(t)$ with lifts $\eta_h^\ell, \zeta_h^\ell$:
  \begin{align}
    \label{eq:m-error-bulk}
    \abs{ m( t; \eta_h^\ell, \zeta_h^\ell ) - m_h( t; \eta_h, \zeta_h ) }
    & \le c h^{k+1} \norm{ \eta_h^\ell }_{\V(t)} \norm{ \zeta_h^\ell }_{\V(t)} \\
    \label{eq:g-error-bulk}
    \abs{ g_\ell( t; \eta_h^\ell, \zeta_h^\ell ) - g_h( t; \eta_h, \zeta_h ) }
    & \le c h^{k+1} \norm{ \eta_h^\ell }_{\V(t)} \norm{ \zeta_h^\ell }_{\V(t)} \\
    \label{eq:a-error-bulk}
    \abs{ a( t; \eta_h^\ell, \zeta_h^\ell ) - a_{h}( t; \eta_h, \zeta_h ) }
    & \le c h^{k} \norm{ \eta_h^\ell }_{\V(t)} \norm{ \zeta_h^\ell }_{\V(t)} \\
    \label{eq:b-error-bulk}
    \abs{ b_\ell( t; \eta_h^\ell, \zeta_h^\ell ) - b_{h}( t; \eta_h, \zeta_h ) }
    & \le c h^{k} \norm{ \eta_h^\ell }_{\V(t)} \norm{ \zeta_h^\ell }_{\V(t)} \\
    \label{eq:gt-error-bulk}
    \abs{ g_\ell( t; \eta_h^\ell, \zeta_h^\ell ) - g( t; \eta_h^\ell, \zeta_h^\ell ) }
    & \le c h^{k} \norm{ \eta_h^\ell }_{\V(t)} \norm{ \zeta_h^\ell }_{\V(t)} \\
    \label{eq:bt-error-bulk}
    \abs{ b_\ell( t; \eta_h^\ell, \zeta_h^\ell ) - b( t; \eta_h^\ell, \zeta_h^\ell ) }
    & \le c h^{k} \norm{ \eta_h^\ell }_{\V(t)} \norm{ \zeta_h^\ell }_{\V(t)}.
  \end{align}
  For $\eta, \zeta \in \Z_0(t)$ with inverse lifts $\eta^{-\ell}, \zeta^{-\ell}$, we have
  \begin{align}
    \label{eq:a-error2-bulk}
    \abs{ a( t; \eta, \zeta ) - a_{h}( t; \eta^{-\ell}, \zeta^{-\ell} ) }
    & \le c h^{k+1} \norm{ \eta }_{\Z_0(t)} \norm{ \zeta }_{\Z_0(t)} \\
    \label{eq:b-error2-bulk}
    \abs{ b_\ell( t; \eta, \zeta ) - b_{h}( t; \eta^{-\ell}, \zeta^{-\ell} ) }
    & \le c h^{k+1} \norm{ \eta }_{\Z_0(t)} \norm{ \zeta }_{\Z_0(t)} \\
    \label{eq:amd-error-bulk}
    \abs{ a( t; \mdell \eta, \zeta ) - a_{h}( t; \mdh \eta^{-\ell}, \zeta^{-\ell} ) }
    & \le c h^{k+1} ( \norm{ \eta }_{\Z_0(t)} + \norm{ \md \eta }_{\Z_0(t)} ) \norm{ \zeta }_{\Z_0(t)}.
  \end{align}
\end{lemma}

\begin{proof}
  We use the notation $J_h = \sqrt{ \det\big( ( \nabla \Lambda_h )^t ( \nabla \Lambda_h ) \big) }$. We note that $\Lambda_h$ is the identity on elements away from the boundary and hence $J_h = 1$ on these elements.
  We denote by $[ J_h \neq 1 ]$ the union of boundary elements where $J_h \neq 1$ and note that we have
  \[
    [ J_h \neq 1 ]^\ell := \{ \Lambda_h( x, t ) : x \in [ J_h \neq 1 ] \} \subset \{ x \in \Omega : -h < d( x, t ) \le 0 \}.
  \]
  For \cref{eq:m-error-bulk}, we have
  \begin{equation*}
    \int_{\Omega(t)} \eta_h^\ell \zeta_h^\ell \dd x
    = \int_{\Omega_h(t)} \eta_h \zeta_h J_h \dd x.
  \end{equation*}
  Hence, we have, applying \cref{eq:detDLambdah-bound} and \cref{lem:bulk-lift-stab}, that
  \begin{multline*}
    \abs{ m( t; \eta_h^\ell, \zeta_h^\ell ) - m_h( t; \eta_h, \zeta_h ) }
    = \abs{ \int_{\Omega_h(t)} \eta_h \zeta_h \left(J_h -1 \right) } \dd x \\
    \le \int_{[ J_h \neq 1]} \abs{ \eta_h } \abs{ \zeta_h } \abs{ J_h -1 } \dd x
    \le c h^k \norm{ \eta_h }_{\H_h(t)} \norm{ \zeta_h^\ell }_{\H_h(t)}.
  \end{multline*}
  Applying the narrow band trace inequality \cref{lem:narrow-band} this can be improved to
  \begin{multline*}
    \abs{ m( t; \eta_h^\ell, \zeta_h^\ell ) - m_h( t; \eta_h, \zeta_h ) }
    \le \int_{[ J_h \neq 1]} \abs{ \eta_h } \abs{ \zeta_h } \abs{ J_h -1 } \dd x \\
    \le c \abs{ [ J_h \neq 1 ]^\ell} h^k \norm{ \eta_h^\ell }_{L^2( [J_h \neq 1]^\ell )} \norm{ \zeta_h^\ell }_{L^2( [J_h \neq 1]^\ell)}
    \le c h^{k+1}\norm{ \eta_h }_{\V_h(t)} \norm{ \zeta_h^\ell }_{\V_h(t)}.
  \end{multline*}

  Similarly we have
  \begin{align*}
    & \int_{\Omega(t)} \A_\Omega \nabla \eta_h^\ell \cdot \nabla \zeta_h^\ell
    + \b_\Omega \eta_h^\ell \cdot \nabla \zeta_h^\ell
    + \c_\Omega \eta_h^\ell \zeta_h^\ell \dd x \\
    & = \int_{[ J_h \neq 1]} J_h ( \nabla \Lambda_h ) \A_\Omega^h ( \nabla \Lambda_h )^t \nabla \eta_h \cdot \nabla \zeta_h
      + J_h ( \nabla \Lambda_h ) \b_\Omega^h \eta_h \cdot \nabla \zeta_h \\
    & \qquad\qquad + J_h \c_\Omega^h \eta_h \zeta_h \dd x.
  \end{align*}
  Applying \cref{eq:DLambdah-bound}, \cref{eq:detDLambdah-bound} and \cref{lem:bulk-lift-stab} we can see \cref{eq:a-error-bulk}.
  Again, by applying \cref{lem:narrow-band} we show the improved bound in \cref{eq:a-error2-bulk}.

  We apply a similar process to the proof of \citep[Lemma 3.3.14]{Ran13} combined with the results of \cref{lem:Lambdah-bound} 
  and the narrow band trace inequality (\cref{lem:narrow-band}) to show the estimates \cref{eq:g-error-bulk}, \cref{eq:b-error-bulk} and \cref{eq:b-error2-bulk}.

  Finally, \cref{eq:gt-error-bulk,eq:bt-error-bulk} follow from the estimate \cref{eq:bulk-w-error}.
  The bound \cref{eq:amd-error-bulk} follows from \cref{eq:a-error2-bulk}, the fact that $(\mdell \eta)^{-\ell} = \mdh \eta^{-\ell}$ and the estimate \cref{eq:bulk-md-error}.
  Indeed we can compute that
  \begin{align*}
    & \abs{ a( t; \mdell \eta, \zeta ) - a_{h}( t; \mdh \eta^{-\ell}, \zeta^{-\ell} ) } \\
    & \le \abs{ a( t; ( \mdell \eta - \md \eta ), \zeta ) - a_{h}( t; ( \mdh \eta^{-\ell} - ( \md \eta )^{-\ell} ), \zeta^{-\ell} ) } \\
    & \qquad\qquad + \abs{ a( t; \md \eta, \zeta ) - a_{h}( t; ( \md \eta )^{-\ell}, \zeta^{-\ell} ) } \\
    & \le c h^{2k} \norm{ \eta }_{\Z_0(t)} \norm{ \zeta }_{\V(t)} + c h^{k+1} \norm{ \md \eta }_{\Z_0(t)} \norm{ \zeta }_{\Z_0(t)}.
      && \qedhere
  \end{align*}
 \end{proof}

The remaining assumtion to verify  is the  estimate \cref{eq:b-bound2}.
\begin{lemma}
  \label{lem:bulk-b-bound2}
  For any $t \in [0,T]$, let $z, \pi_h z$ be as in \cref{eq:ritz} and $\eta = z - \pi_h z$ then for all $\zeta \in \Z_0(t)$, we have
  \begin{equation}
    \label{eq:bulk-b-bound2}
    \abs{ b( t; \eta, \zeta ) }
    \le c \left( \norm{ \eta }_{\H(t)} + h \norm{ \eta }_{\V(t)} + h^{k+1} \norm{ z }_{\Z(t)} \right) \norm{ \zeta }_{\Z_0(t)}.
  \end{equation}
\end{lemma}

\begin{proof}
  The proof is based on a duality argument from \citet{DouDup73}.
  Fix $t \in [0,T]$.
We use the Hilbert triple  $( (H^{1/2}(\Gamma(t)) )', L^2( \Gamma(t) ), H^{1/2}(\Gamma(t)) )$  (at each time $t \in [0,T]$) and identify
\begin{equation}
  \label{eq:L2-H12Gamma}
  {}_{( H^{1/2}( \Gamma(t) ))'} \langle \xi_\eta, \zeta \rangle_{H^{1/2}(\Gamma(t))} = ( \eta, \zeta )_{L^2( \Gamma(t) )} \mbox{ when } \eta \in L^2(\Gamma(t)), \zeta \in H^{1/2}( \Gamma(t) ).
\end{equation}
 For any $\eta \in H^1(\Omega(t))$, let $\xi_\eta \in H^{1/2}(\Gamma(t))$ be the solution of the variational problem
  \[
    ( \xi_\eta, \zeta )_{H^{1/2}(\Gamma(t))}
    = {}_{(H^{1/2}(\Gamma(t)))'} \langle \eta, \zeta \rangle_{H^{1/2}(\Gamma(t))}
    \qquad \mbox{ for all } \zeta \in H^{1/2}(\Gamma(t)).
  \]
  By the trace theorem the trace of $\eta$ lies in $L^2(\Gamma(t))$ so that interpreting the duality product 
  on the right-hand side as an $L^2(\Gamma(t))$ inner-product \cref{eq:L2-H12Gamma}, we see that
  \begin{subequations}
    \begin{align}
      \label{eq:bulk-new-dual-1}
      \norm{ \xi_\eta }_{H^{1/2}(\Gamma(t))} & = \norm{ \eta }_{(H^{1/2}(\Gamma(t)))'} \\
      \label{eq:bulk-new-dual-2}
      \int_{\Gamma(t)} \xi_\eta \eta \dd \sigma & = \norm{ \eta }_{(H^{1/2}(\Gamma(t)))'}^2.
    \end{align}
  \end{subequations}
  This is a simple consequence of the trace theorem and the Riesz representation theorem.

  As in \cref{sec:ritz}, we introduce $\kappa > 0$ such that there exists a constant $c>0$ such that
  \[
    a^\kappa( t; \zeta, \zeta ) := a( t; \zeta, \zeta ) + \kappa m( t; \zeta, \zeta )
    \ge c \norm{ \zeta }_{\V(t)}^2
    \quad \mbox{ for all } \zeta \in \V(t).
  \]
  We wish to estimate $\eta=z - \pi_h z$ in the $(H^{1/2}(\Gamma(t)))'$-norm. We consider the dual problem: Given $\xi_\eta$, find $\dual{\xi_\eta} \in H^1(\Omega(t))$ such that
  \begin{equation}
    \label{eq:bulk-new-dual}
    a^\kappa( \zeta, \dual{\xi_\eta} ) = \int_{\Gamma(t)} \xi_\eta \zeta \dd \sigma
    \qquad \mbox{ for all } \zeta \in H^1(\Omega(t)).
  \end{equation}
  The problem is a weak form of an elliptic problem with inhomogenous Neumann boundary data and has a unique solution which satisfies the regularity estimate \citep{LadUra68,GilTru01}
  \begin{equation}
    \label{eq:bulk-new-dual-reg}
    \norm{ \dual{\xi_\eta} }_{H^2(\Omega(t))}
    \le C \norm{ \xi_\eta }_{H^{1/2}(\Gamma(t))}
    = C \norm{ \eta }_{(H^{1/2}(\Gamma(t)))'},
  \end{equation}
  for a constant independent of time $t$.

  We see using \cref{eq:bulk-new-dual-2} and $\zeta = \eta$ in  \cref{eq:bulk-new-dual}, that
  \begin{equation}
    \label{eq:bulk-new-dual-3}
    \norm{ \eta }_{(H^{1/2}(\Gamma(t)))'}
    = \int_{\Gamma(t)} \xi_\eta \eta \dd \sigma
    = a^\kappa( \eta, \dual{\xi_\eta} )
    = a^\kappa( \eta, \dual{\xi_\eta} - I_h \dual{\xi_\eta} ) - a^\kappa( \eta, I_h \dual{\xi_\eta} ).
  \end{equation}
  For the first term on the right hand side of \cref{eq:bulk-new-dual-3}, we apply the boundedness of 
  $a^\kappa$ and the interpolation estimate \cref{eq:bulk-approx-0} to see
  \begin{multline*}
    \abs{ a^\kappa( \eta, \dual{\xi_\eta} - I_h \dual{\xi_\eta} ) } \\
    \le c \norm{ \eta }_{H^1(\Omega(t))} \norm{ \dual{\xi_\eta} - I_h \dual{\xi_\eta} }_{H^1(\Omega(t))}
    \le c h \norm{ \eta }_{H^1(\Omega(t))} \norm{ \dual{\xi_\eta} }_{H^2(\Omega(t))}.
  \end{multline*}
  For the second term on the right hand side, we apply the geometric estimates \cref{eq:a-error-bulk}, \cref{eq:m-error-bulk} 
  and \cref{eq:a-error2-bulk} and the interpolation estimate \cref{eq:bulk-approx-0} to see that
  \begin{align*}
    \abs{ a^\kappa( \eta, I_h \dual{\xi_\eta} ) }
    & = \abs{ a^\kappa( \pi_h z, I_h \dual{\xi_\eta} ) - a^\kappa_h( \Pi_h z, \tilde{I}_h \dual{\xi_\eta} ) } \\
    & \le \abs{ a^\kappa( \pi_h z - z, I_h \dual{\xi_\eta} ) - a^\kappa_h( \Pi_h z - z^{-\ell}, \tilde{I}_h \dual{\xi_\eta} ) } \\
    & \quad + \abs{ a^\kappa( z, I_h \dual{\xi_\eta} - \dual{\xi_\eta} ) - a^\kappa_h( z^{-\ell}, \tilde{I}_h \dual{\xi_\eta} - \dual{\xi_\eta}^{-\ell} ) } \\
    & \quad + \abs{ a^\kappa( z, \dual{\xi_\eta} ) - a^\kappa_h( z^{-\ell}, \dual{\xi_\eta}^{-\ell} ) } \\
    & \le c h^{k+1} \norm{ z }_{\Z_0(t)} \norm{ \dual{\xi_\eta} }_{\Z_0(t)}.
  \end{align*}
  Hence, combining the previous estimates and applying the dual regularity \cref{eq:bulk-new-dual-reg}, we infer that
  \begin{align*}
    \norm{ \eta }_{(H^{1/2}(\Gamma(t)))'}^2
    & \le c \left( h \norm{ \eta }_{H^1(\Omega(t))} + h^{k+1} \norm{ z }_{\Z_0(t)} \right) \norm{ \dual{\xi_\eta} }_{\Z_0(t)} \\
    & \le c \left( h \norm{ \eta }_{H^1(\Omega(t))} + h^{k+1} \norm{ z }_{\Z_0(t)} \right) \norm{ \eta }_{(H^{1/2}(\Gamma(t)))'},
  \end{align*}
  and we have shown that
  \begin{equation}
    \label{eq:bulk-eta-H12-est}
    \norm{ \eta }_{(H^{1/2}(\Gamma(t)))'}
    \le c \left( h \norm{ \eta }_{H^1(\Omega(t))} + h^{k+1} \norm{ z }_{\Z_0(t)} \right).
  \end{equation}

  Returning to \cref{eq:b-bound2}, we now see that for $\eta = z - \pi_h z$ and $\varphi \in \Z_0(t)$ that
  \begin{align*}
    \abs{ b( t; \eta, \zeta ) }
    & \le \abs{ \int_{\Omega(t)} \B( w, \A_\Omega ) \nabla \eta \cdot \nabla \zeta \dd x }
      + \abs{ \int_{\Omega(t)} \B_{\adv}( w, \b_\Omega ) \eta \cdot \nabla \zeta \dd x } \\
    & \quad + \abs{ \int_{\Omega(t)} (\md \c_\Omega + \c_\Omega \nabla \cdot w ) \eta \zeta \dd x } \\
    & =: I_1 + I_2 + I_3.
  \end{align*}
  For $I_1$, using the smoothness of $\A_\Omega$ and the divergence theorem, we have
  \begin{align*}
    I_1
    & = \int_{\Omega(t)} \B( w, \A_\Omega ) \nabla \eta \cdot \nabla \zeta \dd x \\
    & = \sum_{i,j=1}^{n+1} \int_{\Omega(t)} \B( w, \A_\Omega )_{ij} \partial_{x_j} \eta \partial_{x_i} \zeta \dd x \\
    & = \sum_{i,j=1}^{n+1} \int_{\Omega(t)} \partial_{x_j} \big( \B( w, \A_\Omega )_{ij} \eta \partial_{x_i} \zeta \big) \dd x
      - \sum_{i,j=1}^{n+1} \int_{\Omega(t)} \eta \sum_{i,j=1}^{n+1} \partial_{x_j} \big( \B( w, \A_\Omega )_{ij} \partial_i \zeta \big) \dd x \\
    & = \sum_{i,j=1}^{n+1} \int_{\Gamma(t)} \eta \B( w, \A_\Omega )_{ij} \partial_{x_i} \zeta \nu_j \dd \sigma
      - \sum_{i,j=1}^{n+1} \int_{\Omega(t)} \eta \sum_{i,j=1}^{n+1} \partial_{x_j} \big( \B( w, \A_\Omega )_{ij} \partial_i \zeta \big) \dd x \\
    & =: I_{11} + I_{12}.
  \end{align*}
  We interpret the first term, $I_{11}$, on the right hand side as the duality pairing between $H^{1/2}(\Gamma(t))$ and its dual \cref{eq:L2-H12Gamma} and apply \cref{eq:bulk-eta-H12-est} and the trace theorem to see
  \[
    \abs{ I_{11} }
    \le c \sum_{i=1}^{n+1} \norm{ \eta }_{(H^{1/2}(\Gamma(t)))'} \norm{ \partial_{x_i} \zeta }_{H^{1/2}(\Gamma(t))}
    \le  c \left( h \norm{ \eta }_{H^1(\Omega(t))} + h^{k+1} \norm{ z }_{\Z_0(t)} \right) \norm{ \zeta }_{\Z_0(t)}.
  \]
  For $I_{12}$, we use the smoothness of $\A$ and $w$ to see
  \[
    \abs{ I_{12} }
    \le c \norm{ \eta }_{\H(t)} \norm{ \zeta }_{\Z_0(t)}.
  \]
  Similarly for $I_2$ and $I_3$ we have
  \[
    \abs{ I_{2} } + \abs{ I_3 }
    \le c \norm{ \eta }_{\H(t)} \norm{ \zeta }_{\Z_0(t)}.
    \qedhere
  \]
\end{proof}

Finally, we have collected all the estimates we require to show the error bound.

\begin{theorem}
  \label{thm:bulk-error}
  Let $\A_\Omega \in C^2( \Omega_T )$, $\A_\Omega^h = \A_\Omega^{-\ell}$, $\b_\Omega^h = \b_\Omega^{-\ell}$ and $\c_\Omega^h = \c_\Omega^{-\ell}$ and let $\utext \in L^2_\V$ be the solution of \cref{eq:weak-bulk} which we assume satisfies
  \begin{equation}
    \sup_{t \in (0,T)} \norm{ \utext }_{\Z(t)}^2 + \int_0^T \norm{ \md \utext }_{\Z(t)}^2 \dd t
    \le C_u.
  \end{equation}
  Let $U_h \in C^1_{\V_h}$ be the solution of the finite element scheme \cref{eq:fem-bulk} and denote its lift by $u_h = U_h^\ell$.
  Then we have the following error estimate
  \begin{equation}
    \sup_{t \in (0,T)} \norm{ \utext - u_h }_{\H(t)}^2
    + h^2 \int_0^T \norm{ \utext - u_h }_{\V(t)}^2 \dd t
    \le c \norm{ \utext - u_{h,0} }_{\H(t)}^2 + c h^{2k+2} C_u.
  \end{equation}
\end{theorem}

\begin{proof}
  The proof is performed by applying the abstract result from \cref{thm:error}.
  We know the lift is stable from \cref{lem:bulk-lift-stab}.
  The existence and boundedness of $g_\ell$ and $b_\ell^\kappa$ are dealt with in \cref{lem:bulk-lift-transport}.
  The interpolation properties \cref{eq:interp-Z0} and \cref{eq:interp-Z} are shown in \cref{lem:bulk-approx}.
  The geometric perturbation estimates 
  \cref{eq:m-error,eq:c-error,eq:ct-error,eq:a-error,eq:b-error,eq:bt-error,eq:a-error2,eq:b-error2,eq:md-error,eq:mdV-error} 
  are shown in \cref{lem:bulk-geom-pert,lem:bulk-w-md-errors}  and \cref{eq:amd-error2} follow from 
  \cref{eq:m-error,eq:a-error,eq:a-error2} and that fact that lifting and taking material derivatives commute (\cref{lem:commutation-md-lift}).
  We have shown \cref{eq:b-bound2} in \cref{lem:bulk-b-bound2}.
\end{proof}

\section{Application II: Parabolic equation on an evolving surface}\label{SURFPDE}
\label{sec:application1}

In this section, we will formulate and analyse a finite element method for a parabolic equation posed on an evolving 
surface \cref{eq:surf-eqn}. We begin with some notation and a definition of the initial value problem. Our numerical  approach will be to 
first discretise the domain and construct isoparametric surface  finite element spaces based on the general theory in \cref{sec:sfem,sec:lift-fem}.
We will consider isoparametric finite elements of order $k$, which will be fixed throughout this section.
We will analyse this method using both techniques from the general surface finite element constructions in \cref{sec:sfem} and the abstract theory from \cref{AbstractNA}.

\subsection{The domain and function spaces}
Let $0<T < \infty$ and $\Gamma_0 \subset \R^{n+1}$ be a compact
sufficiently smooth hypersurface without boundary. Let  $\{ \Gamma(t) \}_{t \in [0,T)}$ be  a family of evolving
hypersurfaces such that  there exists a sufficiently smooth mapping (called the \emph{flow map})
$\Phi_{(\cdot)}(\cdot)\colon [0,T] \times\Gamma_0  \rightarrow \R^{n+1}$ such that:
\begin{enumerate}
\item $\Phi_t(\cdot)$ is a diffeomorphism of $\Gamma_0$ onto $\Gamma(t)$ for each $t \in [0,T)$;
\item $\Phi_0(\cdot)=\id_{|\Gamma_0}$.
\end{enumerate}

We call
\begin{multline*}
w \colon S_T \left( :=  \bigcup_{t \in (0,T)} \Gamma(t) \times \lbrace t \rbrace \right)
 \rightarrow \R^{n+1}, \\
w(x,t):= \frac{\partial \Phi_t}{\partial t}\bigl( (\Phi_t(\cdot))^{-1}(x) \bigr), \quad x \in \Gamma(t),
t\in (0,T)
\end{multline*}
the material velocity field of $\Gamma(t)$ which also satisfies
\begin{equation} \label{veldef}
\frac{\partial}{\partial t}\Phi_t(p) = w(\Phi_t(p),t), \; t \in (0,T), \quad
\Phi_0(p) = p.
\end{equation}
We denote  $\H(t) = L^2(\Gamma(t))$, $\V(t) = H^1(\Gamma(t))$ and  $\V^*(t) = ( H^1(\Gamma(t)) )^*$. %
We will also make use of the spaces $\Z_0(t) = H^2(\Gamma(t))$ and $\Z(t) = H^{k+1}( \Gamma(t) )$.
We see that $Z(t) \subset \Z_0(t) \subset \V(t)$ for each $t \in [0,T]$ and the inclusions are uniformly continuous.
We define the push-forward operator $\phi_t$ by
\begin{equation}
  \label{eq:surface-pushforward}
  (\phi_t \eta)( \cdot, t ) := \eta( \Phi_{-t}( \cdot ) ) \qquad \mbox{ for } \eta \in \H_0.
\end{equation}
 $( L^2( \Gamma(t) ), \phi_t )_{t \in [0,T]}$ and    $( H^1( \Gamma(t) ), \phi_t )_{t \in [0,T]}$ 
are  compatible pairs (\cref{def:compatibility}), the spaces $L^2_\H$, $L^2_\V$ and  $L^2_{\V^*}$ are well defined (c.f. \cref{eq:L2X}) and 
we have a well defined strong material derivative denoted by $\md$ \cref{eq:strong-md} (see \citealt[Lem.~3.2,3.3]{Vie14}).
$( \V(t), \H(t), \V^*(t) )_{t \in [0,T]}$ form a Hilbert triple (\cref{def:evolving-hilbert-triple}).
We also have that $( \Z_0(t), \phi_t|_{\Z_0(t)} )_{t \in [0,T]}$ and $( \Z(t), \phi_t|_{\Z_0(t)} )_{t \in [0,T]}$ are compatible pairs.
For details see \citet{AlpEllSti15b-pp}.

{
  \begin{remark}
    \label{rem:surf-smoothness}
    For the well posedness of the partial differential equation \cref{Surfprob} we require that the boundary $\Gamma$ is a $C^{2}$-hypersurface 
    and that the flow map $\Phi_{(\cdot)}( \cdot ) \in C^2( 0,T; C^2( \Omega_0 ) )$. See \citet{AlpEllSti15b-pp} for more details.
    For the approximation properties we derive we require that $\Gamma$ is a $C^{k+2}$-hypersurface and that
     $\Phi_{(\cdot)}( \cdot ) \in C^{k+2}( 0,T; C^{k+2}( \Omega_0) )$ with $\Phi_{t} \colon \Omega_0 \to \Omega(t)$ and $\Phi_{-t} \colon \Omega(t) \to \Omega$ both of class $C^{k+2}$.
  \end{remark}
}

{
We introduce a signed distance function for a closed surface $\Gamma(t)$.
We assume that $\Gamma(t) = \partial \Omega(t)$ is the boundary of an open bounded domain.
The oriented distance function for $\Gamma(t)$ is defined by
\begin{equation*}
  d(x,t) =
  \begin{cases}
    \inf\{ \abs{ x - y } : y \in \Gamma(t) \} & \mbox{ for } x \in \R^{n+1} \setminus \bar\Omega(t) \\
    -\inf\{ \abs{ x - y } : y \in \Gamma(t) \} & \mbox{ for } x \in \Omega(t).
  \end{cases}
\end{equation*}
We can orient $\Gamma(t)$ by choosing the unit normal $\nu$ as
\begin{equation}
  \label{eq:nu-defn}
  \nu( x, t ) = \nabla d ( x, t ) \qquad \mbox{ for } x \in \Gamma(t).
\end{equation}
This allows us to define the (extended) Weingarten map by $\mathbb{H} := \{\mathbb H_{ij}\}=\{d_{x_ix_j}\}$ and 
the mean curvature by $H_\kappa := \trace \mathbb H$. 
Note that, for example when $n=2$, this definition of $H_\kappa$ is the sum of the 
principal curvatures   rather than the mean.
For each time $t \in [0,T]$, there exists a narrow band $\N(t)$ such that the distance function $d( \cdot, t )$ is smooth and, for each $x \in \N(t)$ 
there exists a unique point $p(x,t) \in \Gamma(t)$ such that (see Lemma 14.16 of \cite{GilTru01} and \citet{Foote1984})
\begin{equation}
  \label{eq:cp}
  x = p(x,t) + d( x, t ) \nu( p(x,t), t ).
\end{equation}
We call the operator $p(\cdot,t) \colon \N(t) \to \Gamma(t)$ the normal projection operator and note that $p(\cdot,t)$ is also smooth.
We use this projection to extend the unit normal and Weingarten map to be defined in $\N(t)$ by $\nu( x, t ) = \nu( p(x,t), t )$ and $\mathbb{H}( x, t ) = \mathbb{H}( p(x,t), t )$.
See \citet[Lem. 14.16]{GilTru01}; \citet{Foote1984} for more details.
}

 \subsection{The initial value problem}
 \label{sec:surf-prob}
  We assume that $\A_\Gamma \in C^1( S_T; \R^{(n+1)\times(n+1)} )$ is  a symmetric diffusion tensor which 
maps the tangent space of $\Gamma(t)$ at a point into itself, and there exists $a_0 > 0$ such that for all $t \in [0,T]$
\begin{equation*}
  \A_\Gamma(\cdot, t) \xi \cdot \xi \ge a_0 \abs{ \xi }^2
  \quad
  \mbox{ for all } \xi \in \R^{n+1}, \xi \cdot \nu( \cdot, t ) = 0.
\end{equation*}
$\b_\Gamma \in C^1( S_T ; \R^{n+1})$ is a smooth tangential vector field  and  $\c_\Gamma\in C^1( S_T )$ is a smooth scalar field. 

We consider the initial value  problem
\begin{problem}
  \label{Surfprob} Find $\emph \utext$ such that
  \begin{subequations}
    \begin{align}
      \md \emph \utext
      +\nabla_\Gamma\cdot (\b_\Gamma  \emph \utext)  -\nabla_\Gamma\cdot (\A_\Gamma \nabla_\Gamma \emph \utext) +\c_\Gamma  \emph \utext +(\nabla_\Gamma\cdot w )\utext &=0
      && \mbox{on} ~~\Gamma(t) \\
       \emph \utext(0)
    &  =\emph \utext_0
      ~~&&\mbox{on} ~~\Gamma_0.
    \end{align}
  \end{subequations}
\end{problem}
\begin{remark}
  Writing $w = w_\tau + w_\nu$ for a decomposition of $w$ into tangential and normal components,
the problem  \cref{eq:surf-eqn} is recovered by setting
$\A_\Gamma=\la_\Gamma, \b_\Gamma=\lb_\Gamma-w_\tau$ and $\c_\Gamma=\lc_\Gamma - \nabla_\Gamma \cdot w_\nu$.
\end{remark}

\subsubsection{Transport formulae}
  The following transport formulae hold  on $\omega(t)\equiv \Gamma(t)$ and on portions $\{ \omega(t)\subset \Gamma(t)\}_{t \in [0,T]}$ of the domain $\{ \Gamma(t) \}_{t \in [0,T]}$, 
  which follow the flow $\omega(t) = \Phi_t( \omega(0) )$ for $t \in [0,T]$, and have Lipschitz boundaries at each time 
:-
\begin{itemize}
\item
 For $\eta, \zeta \in C^1_\H$
\begin{equation}
  \label{eq:surface-transport}
  \dt \int_{\omega(t)} \eta \zeta \dd \sigma
  = \int_{\omega(t)} ( \md \eta \zeta + \eta \md \zeta ) \dd \sigma
  + g( t; \eta, \zeta ),
\end{equation}
where
\begin{equation*}
  g( t; \eta, \zeta ) = \int_{\omega(t)} \eta \zeta \nabla_{\Gamma} \cdot {w} \dd \sigma.
\end{equation*}
More precisely, we can show that \cref{as:c-exists} holds for $\H(t) = L^2(\Gamma(t))$ \citep[Sec.~4.1]{AlpEllSti15b-pp}.

\item
For $\eta, \zeta \in C^1_\V$, we have the identity
\begin{equation}
  \label{eq:transport-dirichlet}
  \begin{aligned}
  \dt \int_{\omega(t)} \A_\Gamma \nabla_\Gamma \eta \cdot \nabla_\Gamma \zeta \dd \sigma
  & = \int_{\Gamma(t)} \A_\Gamma\nabla_\Gamma \md \eta \cdot \nabla_\Gamma \zeta
  + \A_\Gamma \nabla_\Gamma \eta \cdot \nabla_\Gamma \md \zeta \dd \sigma \\
  & \qquad + \int_{\Gamma(t)} \B( {w}, \A_\Gamma ) \nabla_\Gamma \eta \cdot \nabla_\Gamma \zeta \dd \sigma,
  \end{aligned}
\end{equation}
where $\B( {w}, \A_\Gamma )$ is given by
\begin{equation}
  \label{eq:B-defn}
  \B( {w}, \A_\Gamma ) = \md \A_\Gamma + \nabla_\Gamma \cdot {w} \A_\Gamma - 2 D( {w},\A_\Gamma)
\end{equation}
and $D( {w},\A_\Gamma)$ is the rate of deformation tensor
\begin{equation*}
  D( {w},\A_\Gamma)_{ij} = \frac{1}{2} \sum_{k=1}^{n+1} \{(\A_\Gamma)_{ik} (\nabla_\Gamma)_k {w}_j + (\A_\Gamma)_{jk} (\nabla_\Gamma)_k {w}_i\}
  \qquad \mbox{ for } i,j = 1, \ldots, n+1.
\end{equation*}

\item
For $\eta \in C^1_\H$, $\zeta \in C^1_\V$, we have
\begin{equation}
  \label{eq:transport-advection}
  \begin{aligned}
    \dt
    \int_{\omega(t)} \b_\Gamma \eta \cdot \nabla_\Gamma \zeta \dd \sigma
    & = \int_{\omega(t)} \{\b_\Gamma \md \eta \cdot \nabla_\Gamma \zeta
    + \b_\Gamma \eta \cdot \nabla_\Gamma \md \zeta \} \dd \sigma \\
    & \qquad + \int_{\omega(t)} \B_{\adv}( {w}, \b_\Gamma ) \eta \cdot \nabla_\Gamma \zeta \dd \sigma,
  \end{aligned}
\end{equation}
where $\B_{\adv}( {w}, \b_\Gamma ) $ is given by
\begin{align*}
  \B_{\adv}( {w}, \b_\Gamma )
  = \md \b_\Gamma + \b_\Gamma \nabla_\Gamma \cdot {w}
  - \sum_{j=1}^{n+1}  (\b_\Gamma)_j (\nabla_\Gamma)_j {w}.
\end{align*}

\end{itemize}

\begin{remark}
The identity \cref{eq:transport-advection} is equivalent to \citep[Lem.~A.1]{EllVen15}.
The proof of \cref{eq:transport-dirichlet} and \cref{eq:transport-advection} follows from applying \cref{eq:surface-transport} with the identity
(in which $\nu$ is the unit normal to $\Gamma(t)$)%
\begin{align}
  \label{eq:md-grad-comm}
  \md (\nabla_\Gamma)_i \chi
  = (\nabla_\Gamma)_i \md \chi
  - \sum_{j=1}^{n+1} (\nabla_\Gamma)_i {w}_j (\nabla_\Gamma)_j \chi
  + \left(
  \nabla_\Gamma ( {w} \cdot \nu ) \cdot \nabla_\Gamma \chi
  - \sum_{j,l=1}^{n+1} {w}_j (\nabla_\Gamma)_l \chi (\nabla_\Gamma)_l \nu_j
  \right) \nu_i.
\end{align}
    \label{rem:surface-transport}
\end{remark}

\subsection{The bilinear forms and transport formulae}
We define   %
  \begin{align*}
    m( t; \eta, \zeta )
    & := \int_{\Gamma(t)} \eta \zeta \dd \sigma
    && \eta, \zeta \in \H(t)  \\
    g( t; \eta, \zeta )
    & := \int_{\Gamma(t)} \eta \zeta \nabla_\Gamma \cdot {w} \dd \sigma
    && \eta, \zeta \in \H(t)  \\
    a( t; \eta, \zeta )
    & := a_s( t; \eta, \zeta ) + a_n( t; \eta, \zeta )
    && \eta, \zeta \in \V(t) \\
    a_s( t; \eta, \zeta )
    & := \int_{\Gamma(t)} \A_\Gamma \nabla_\Gamma \eta \cdot \nabla_\Gamma \zeta
      + \c_\Gamma \eta \zeta \dd \sigma
    && \eta, \zeta \in \V(t) \\
    a_n( t; \eta, \zeta )
    & := \int_{\Gamma(t)} \b_\Gamma \eta \cdot \nabla_\Gamma \zeta \dd \sigma
    && \eta \in \H(t), \zeta \in \V(t).
  \end{align*}

We can apply  \cref{eq:surface-transport}  \cref{eq:transport-dirichlet} and \cref{eq:transport-advection} to see that:%
\begin{align}
  \label{eq:surface-m-transport}
  \dt m( t; \eta, \zeta )
  & = m( t; \md \eta, \zeta ) + m( t; \eta, \md \zeta )+ g( t; \eta, \zeta )
  && \mbox{ for all } \eta, \zeta \in C^1_\H \\
  \label{eq:bs-exist-surf}
  \dt a_s( t; \eta, \zeta )
  & = a_s( t; \md \eta, \zeta ) + a_s(t; \eta, \md \zeta )
    + b_s( t; \eta, \zeta )
  && \mbox{ for all } \eta, \zeta \in C^1_\V \\
  \label{eq:bn-exist-surf}
  \dt a_n( t; \eta, \zeta )
  & = a_n( t; \md \eta, \zeta ) + a_n(t; \eta, \md \zeta )
    + b_n( t; \eta, \zeta )
  && \mbox{ for all } \eta \in C^1_\H, \zeta \in C^1_\V
\end{align}
with
\begin{align*}
  b_s( t; \eta, \zeta )
  & := \int_{\Gamma(t)} \B( {w}, \A_\Gamma ) \nabla_\Gamma \eta \cdot \nabla_\Gamma \zeta + ( \md \c_\Gamma + \c_\Gamma \nabla_\Gamma \cdot {w} ) \eta \zeta \dd \sigma \\
  b_n( t; \eta, \zeta )
  & := \int_{\Gamma(t)} \B_{\adv}( {w}, \b_\Gamma ) \eta \cdot \nabla_\Gamma \zeta \dd \sigma.
\end{align*}
We define $b( t; \cdot, \cdot ) \colon \V(t) \times \V(t) \to \R$ to be
\begin{equation*}
  b( t; \eta, \zeta ) := b_s( t ; \eta, \zeta ) + b_n( t; \eta, \zeta ) \qquad \mbox{ for } \eta, \zeta \in \V(t).
\end{equation*}

\subsection{Variational formulation}
The weak formulation of \cref{Surfprob} becomes
\begin{problem}
  \label{pb:weak-surf}
  Given $\emph \utext_0 \in L^2(\Gamma_0)$, find $\emph \utext$ %
  such that for almost every  $t \in (0,T)$ we have
  \begin{equation}
    \label{eq:weak-surf}
    \begin{aligned}
      m( t; \md \emph \utext, \zeta )
      + g( t; \emph \utext, \zeta )
      + a( t; \emph \utext, \zeta )
      & = 0 \qquad \mbox{ for all } \zeta \in H^1(\Gamma(t)) \\
      \emph \utext( 0 ) & = \emph \utext_0.
    \end{aligned}
  \end{equation}
 \end{problem}

We have the following  well-posedness result.

\begin{theorem}
  \label{thm:exist-surf}
  There exists a unique solution $\emph \utext$ to \cref{eq:weak-surf} which satisfies
  \begin{align}
  \label{Hstabsurf}
 \sup_{t\in[0,T]}\norm{ \emph \utext }_{L^2(\Gamma(t))}^2  + \int_0^T  \norm{ \emph \utext }_{H^1(\Gamma(t))}^2 \dd t
    \le c \norm{ \emph \utext_0 }^2_{L^2(\Gamma_0)},\end{align}
    and if $u_0\in H^1(\Gamma_0)$ then
    \begin{align}
    \label{eq:2}
    \sup_{t\in[0,T]}\norm{ \emph \utext }_{H^1(\Gamma(t))}^2 +\int_0^T  \norm{ \md \emph \utext }_{L^2(\Gamma(t))}^2 \dd t
    \le c \norm{ \emph \utext_0 }^2_{H^1(\Gamma_0)}.
  \end{align}
\end{theorem}

\begin{proof}
  We simply check the assumptions required for \cref{thm:abs-exist}.  For \cref{as:c-exists,ass:evolving-space-equivalence}    
  we refer to  \citet[Sec. 4 and Sec. 5]{AlpEllSti15b-pp}.  \cref{ass:u0-approx} is a consequence of \cref{rem:ass:u0-approx}.
  The assumptions \cref{eq:m-symmetric} and \cref{eq:m-bounded} follow simply since $m( t; \cdot, \cdot )$ is the $\H(t) = L^2(\Gamma(t))$-inner product.
  \cref{eq:c-formula} holds from \cref{eq:surface-transport} and \cref{eq:c-bound} from the assumption that $\nabla_\Gamma \cdot {w}$ is uniformly bounded in space and time.
  The bilinear form $a(t;\cdot, \cdot)$ is differentiable in time, hence measurable \cref{eq:a-meas}.
  The coercivity of $a$ \cref{eq:as-coercive} follows from a standard calculation  since $\A_\Gamma$ is positive definite and $\c_\Gamma$ is bounded..
  The smoothness of $\A_\Gamma, \b_\Gamma$ and $\c_\Gamma$ imply that $a_s$ and $a_n$ are bounded \cref{eq:as-bounded}, \cref{eq:an-bounded}.
  The existence of the bilinear forms $b_s$ \cref{eq:bs-exist} and $b_n$ \cref{eq:bn-exist} follow from \cref{eq:bs-exist-surf} 
  and \cref{eq:bn-exist-surf} and the bounds \cref{eq:bs-bound} and \cref{eq:bn-bound} from the smoothness of $\A_\Gamma, \b_\Gamma$ and $\c_\Gamma$.
\end{proof}

\subsection{Discretisation of the domain and finite element spaces}

The first stage in developing our  finite element method is to define the approximate 
computational domain $\{ \Gamma_h(t) \}$.
We do this by constructing an isoparametric approximation of $\Gamma_0$ which is then pushed-forwards under an approximation to the flow $\Phi_t$.
The result will be that the Langrange points of $\Gamma_h(t)$ lie on the surface $\Gamma(t)$ for all times and evolving according to the velocity ${w}$.
In this sense, $\Gamma_h(t)$ can be considered as an interpolation of $\Gamma(t)$. Recall that $p$ denotes the normal 
projection operator \cref{eq:cp} and that $k$ is the order of isoparametric finite elements we will use.
Throughout the remainder of this section we will denote global discrete quantities with a subscript $h \in (0,h_0)$, which is related to 
element size. We assume implicitly that these structures exist for each $h$ in this range (see also \cref{rem:small-hk-surf}).

We will use the simplical, Lagrange reference element $(\hat{K}, \hat{P}, \hat{\Sigma})$ from \cref{ex:standard-fem}.
We start by constructing a family of time dependent element reference maps (\cref{def:sfe}) which defines an evolving conforming 
subdivision $\{ \T_h(t) \}$ (\cref{def:evolving-conforming-subdivision}) of an evolving triangulated hypersurface $\{ \Gamma_h(t) \}$ (\cref{def:evolving-triangulated-hypersurface}).

Let $\tilde\Gamma_{h,0}$ be a polyhedral approximation of $\Gamma_0$ equipped with a quasi-uniform, conforming subdivision $\tilde\T_{h,0}$ (\cref{def:conforming-subdivision}).
We restrict  the vertices of $\tilde\Gamma_{h,0}$ to  lie on the surface $\Gamma_0$ and denote by $\tilde{h}_0$ the maximum element domain diameter on $\tilde\Gamma_{h,0}$ (\cref{def:h}).
We assume that $\tilde{h}_0$ is sufficiently small and $\tilde{\Gamma}_{h,0}$ is such that $p( \cdot, 0 )$ is a smooth bijection from $\tilde{\Gamma}_{h,0}$ onto $\Gamma_0$.
More precisely, for each $\tilde{K} \in \tilde{\T}_{h,0}$, there exists an affine map $F_{\tilde{K}} \colon \hat{K} \to \R^{n+1}$ which satisfies 
the assumptions of \cref{def:sfe} so that we can define a surface finite element $(\tilde{K}, \tilde{P}, \tilde{\Sigma})$ using \cref{eq:sfe}.
Note that the vertices of $\tilde{K}$ lie on $\Gamma$ but the other Lagrange points may not.
We assume the collection of all Lagrange points satisfy \cref{eq:node-agree}.
We write $\tilde{I}$ for the local interpolation operator over $( \tilde{K}, \tilde\P, \tilde\Sigma )$ \cref{eq:3} %
and define an initial element reference map $F_{K_0} \colon \hat{K} \to \R^{n+1}$ by
\[
  F_{K_0}( \hat{x} ) = [ \tilde{I} (p(\cdot,0)) ]( F_{\tilde{K}}( \hat{x} ) ) \qquad \mbox{ for } \hat{x} \in \hat{K}.
\]
An example of element domains \cref{eq:element-domain} defined by $F_{K_0}$ is given for $k=1,2,3$ in \cref{fig:surf-domain}.
We denote the union of all elements constructed in this way $\Gamma_{h,0}$ which is a triangulated hypersurface (\cref{def:triangulaed-hypersurface}) 
equipped with a conforming subdivision $\T_{h,0}$ (\cref{def:conforming-subdivision}), the set of all element domains $K$.

\begin{figure}[tb]
  \centering
  \includegraphics{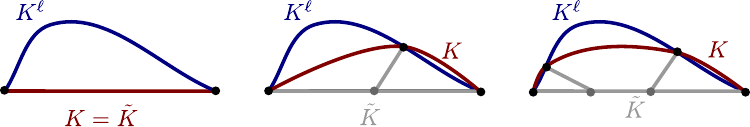}
  \caption{Examples of construction of an isoparametric surface finite element for $k=1$ (left), $k=2$ (centre), $k=3$ (right). The Lagrange nodes $\tilde{a}_i$ are shown in grey on $\tilde{K}$ which are lifted to $K \subset \Gamma_{h,0}$ (red) to the Lagrange nodes $\tilde{a}_i$ (black) which lie on the smooth surface $K^{\ell} \subset \Gamma_0$ (blue).}
  \label{fig:surf-domain}
\end{figure}

To complete the construction of $F_{K(t)} \colon \hat{K} \to \R^{n+1}$, for $K_0 \in \T_{h,0}$, we consider the element flow map $\Phi_t^K \colon K_0 \to K(t)$ (\cref{def:esfe-reference-map}) defined by
\begin{equation*}
  \Phi_t^K := I_{K_0} [ \Phi_t ( p( \cdot, 0 ) ) ] \qquad \mbox{ for all } K_0 \in \T_{h,0},
\end{equation*}
which is a bijection onto its image and we denote its inverse by $\Phi_{-t}^K$. An example is shown in \cref{fig:surf-domain2}.
Using \cref{eq:esfe-reference-map}, $\Phi_t^K$ defines an evolving reference element map:
\begin{equation*}
  F_{K(t)}( \hat{x} ) = \Phi_t^K( F_{K_0}(\hat{x} ) ) \qquad \mbox{ for } \hat{x} \in \hat{K}.
\end{equation*}
Using \cref{eq:sfe}, this defines an evolving surface finite element $( K(t), \P^K(t), \Sigma^K(t) )_{t \in [0,T]}$ (\cref{def:esfe}).

\begin{figure}[tb]
  \centering
  \includegraphics{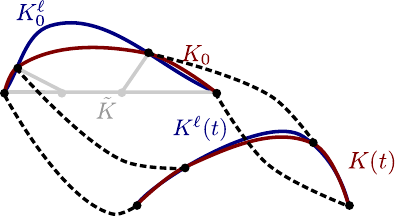}
  \caption{Examples of construction of an isoparametric evolving surface finite element for $k=3$. The Lagrange nodes ${a}_i(t)$ 
  follow the dashed black trajectories from the initial element $K_0 \subset \Gamma_{h,0}$ to a element $K(t) \subset \Gamma_h(t)$.}
  \label{fig:surf-domain2}
\end{figure}

We call the set of such elements $\T_h(t)$ and the union $\Gamma_h(t)$
and write a global discrete flow map $\Phi^h_t \colon \Gamma_{h,0} \to \Gamma_h(t)$ given element-wise by $\Phi^h_t |_{K_0} := \Phi^K_t$ (\cref{def:global-discrete-flow}).
By $h$ we denote the maximum mesh diameter over time \cref{eq:def-h}:
\[
  h := \max_{t \in [0,T]} \max_{K(t) \in \T_h(t)} \diam( \tilde{K}(t) ).
\]
Our construction implies that \cref{eq:surf-evolve-node-agree} holds since 
each of the transformation from $\tilde{\Gamma}_{h,0}$ are continuous.

We will use two Hilbert spaces defined over $\Gamma_h(t)$. First we will 
denote by $\H_h(t) := L^2( \Gamma_h(t) )$ and by $\V_h(t) := H^1_T( \T_h )$ \cref{eq:trace-broken-sobolev-space}.
We equip each space with a norm:
\begin{align}
  \label{eq:surf-Hh-norm}
  \norm{ \eta_h }_{\H_h(t)} & :=
  \norm{ \eta_h }_{L^2(\Gamma_h(t))}
  =
  \left(
    \int_{\Gamma_h(t)} \eta_h^2 \dd \sigma_h
                              \right)^{\frac{1}{2}} \\
  \label{eq:surf-Vh-norm}
  \norm{ \eta_h }_{\V_h(t)} & :=
  \norm{ \eta_h }_{H^1(\T_h(t))}
  =
  \left(
    \sum_{K(t) \in \T_h(t)}
    \int_{K(t)} \abs{ \nabla_K \eta_h }^2 + \eta_h^2 \dd \sigma_h
  \right)^{\frac{1}{2}}.
\end{align}

We note that each element reference map $F_{K(t)}$ is an element of $\hat{P}$ so that at 
each time $t$ the triple $(K(t), P(t), \Sigma(t))$ is an isoparametric surface finite element.
Furthermore, we can use the basis functions of $\hat{P}$ to decompose $F_{K(t)}$:
\[
  F_{K(t)}( \hat{x} )  = \sum_{i=1}^{N_K} F_{K(t)}( \hat{a}_i ) \hat\chi_i( \hat{x} ).
\]
In particular, this implies that the evolving triangulated surface $\Gamma_h(t)$ only depends on 
the evolving of the Lagrange points $\{ a_i(t) \}_{i=1}^N$ which we can infer satisfy
\[
  \dot{a}_i(t) = w( a_i(t), t ).
\]

We define a global evolving surface finite element space (\cref{def:evolving-sfe-space}) by
  \begin{multline}
    \label{eq:surf-Sh}
  \S_h(t) := \Bigl\{
  \chi_h = ( \chi_h )_{K(t) \in \T_h(t)} \in \prod_{K \in \T_h(t)} \bigl\{ \hat\chi \circ F_{K(t)}^{-1} : \hat\chi \in P_k(\hat{K}) \bigr\} : \\
  \chi_K( a(t) ) = \chi_{K'(t)} \mbox{ for all } K(t), K'(t) \in \T(a(t)),
  \mbox{ for all } a(t) \in \N_h(t).
  \Bigr\}
\end{multline}
Using \cref{lem:Sh-closed-subspace} we can identify elements in $\S_h(t)$ as continuous functions and $\S_h(t) \subset \V_h(t)$.

\begin{remark}
  This construction is a generalisation of the construction of \citet{DziEll07}.
  Indeed, in the case that we wish to consider affine finite elements, it is worth noting that 
  $\tilde{I} p( \tilde{x}, 0 ) = \tilde{x}$ for $\tilde{x} \in \tilde{K}$ and $K = \tilde{K}$.
  A different view of the same construction is given by \citet{Kov18} and at each time our construction coincides with the work of \citet{Dem09}.
\end{remark}

We will assume that the evolving subdivision $\{ \T_h(t) \}_{t \in [0,T]}$ is uniformly quasi-uniform (\cref{def:uniformly-quasi-uniform}).
It is clear that our construction maintains the conformity of the initial base triangulation $\tilde\T_{h,0}$.

\begin{proposition}
  \label{prop:surf-fe}
  The evolving surface finite element space $\{ \S_h(t) \}_{t \in [0,T]}$ defined by the above construction consists of evolving $k$-surface 
  finite elements (\cref{def:e-Theta-sfe}) over a uniformly $k$-regular subdivision (\cref{def:uniformly-Theta-regular}).
\end{proposition}

\begin{proof}
  The proof follows in the same way as \cref{prop:bulk-fe} and we do not give full details here.
  The only part to check is that the discrete flow map $\Phi_t^h$ is uniformly bounded in $W^{2,\infty}(\T_{h,0})$.
  However, this follows directly from the definition of $\Phi^K_t$ as an interpolation of $\Phi_t ( p( \cdot, 0 ) )$ which is a smooth function.
\end{proof}

The element flow map $\Phi^K_t$ defines a velocity on each element $W_K$ (\cref{def:element-velocity}) by
\[
  \dt \Phi^K_t( \cdot ) = W_K( \Phi^K( \cdot ), t ) \qquad \mbox{ for } t \in [0,T].
\]
This can be combined into a global velocity $W_h$ (\cref{def:global-discrete-velocity}).
We note that the global velocity is determined purely by the velocity of the vertices $\{ a_i(t) \}_{i=1}^N$:
\begin{equation}
  \label{eq:surf-defn-Wh}
  W_h( x, t ) = \sum_{i=1}^n w( a_i(t), t) \chi_i( x, t ) \qquad \mbox{ for } x \in \Gamma_h(t).
\end{equation}
We also have a discrete push forward map (\cref{def:global-push-forward-map}) $\phi^h_t$ on $\H_h(t)$ by
\[
  \phi^h_t( \eta_h )( x ) = \eta_h( \Phi^h_{-t}( x  ) ) \qquad \mbox{ for } x \in \Gamma_h(t), t \in [0,T].
\]
Since we have constructed a uniformly $k$-regular mesh, we infer that $( \H_h(t), \phi^h_t )_{t \in [0,T]}$ is a compatible 
pair (\cref{lem:abs-Vh-compatible}) and we may define the space $C^1_{\H_h}$ \cref{eq:CkX} and a discrete material derivative by \cref{eq:abs-mdh}:%
\[
  \mdh \eta_h = \phi^h_t \left( \dt \left( \phi^h_{-t} \eta_h \right) \right).
\]
We also have that $( \V_h(t), \phi^h_t|_{\V_{h,0}} )_{t \in [0,T]}$ and $( \S_h(t), \phi^h_t|_{\S_{h,0}} )_{t \in [0,T]}$ 
are compatible pairs (\cref{rem:surf-Wmp-etc-compat}) so we may define the spaces $C^1_{\V_h}$ and $C^1_{\S_h}$.

\begin{lemma}
  \label{lem:gammah-transport}
  For $t \in [0,T]$ and for each $K(t) \in \T_h(t)$,
  let $\A_\Gamma^K$ be a smooth, positive-definite, diffusion tensor on $K(t)$, which maps the tangent space of $K(t)$ to itself, and $\b_\Gamma^K$ be a smooth tangential vector field on $K(t)$ for each $K(t) \in \T_h(t)$ and $t \in [0,T]$.
  Then for $\eta_h \in C^1_{\H_h}$:
  \begin{align}
    \label{eq:gammah-transport}
    \dt \int_{\Gamma_h(t)} \eta_h \dd \sigma_h
    = \int_{\Gamma_h(t)} \mdh \eta_h + \eta_h \nabla_{\Gamma_h} \cdot {W}_h \dd \sigma_h.
  \end{align}
  For all $\eta_h, \zeta_h \in C^1_{\V_h}$, we have
  \begin{multline}
    \label{eq:dirichleth-transport}
    \dt \sum_{K(t) \in \T_h(t)} \int_{K(t)} \A_\Gamma^K \nabla_K \eta_h \cdot \nabla_K \zeta_h \dd \sigma_h \\
    = \sum_{K(t) \in \T_h(t)} \int_{K(t)} \A_\Gamma^K \nabla_{K} \mdh \eta_h \cdot \nabla_{K} \zeta_h
      + \A_\Gamma^K \nabla_{K} \eta_h \cdot \nabla_{K} \mdh \zeta_h  \dd \sigma_h \\
    + \sum_{K(t) \in \T_h(t)} \int_{K(t)} \B_K( {W}_K, \A_\Gamma^K ) \nabla_{K} \eta_h \cdot \nabla_{K} \eta_h \dd \sigma_h,
  \end{multline}
  and for all $\eta_h \in \C^1_{\V_h}, \zeta_h \in C^1_{\H_h}$, we have
  \begin{multline}
    \label{eq:advectionh-transport}
    \dt \sum_{K(t) \in \T_h(t)} \int_{K(t)} \b_\Gamma^K \eta_h \cdot \nabla_{K} \zeta_h \dd \sigma_h \\
    = \sum_{K(t) \in \T_h(t)} \int_{K(t)} \b_\Gamma^K ( \mdh \eta_h \cdot \nabla_{K} \zeta_h +  \eta_h \cdot \nabla_{K} \mdh \zeta_h ) \dd \sigma_h \\
    + \sum_{K(t) \in \T_h(t)} \int_{K(t)} \B_{\adv,K}( {W}_K, \b_\Gamma^K )
      \eta_h \cdot \nabla_{K} \zeta_h \dd \sigma_h,
  \end{multline}
  where $\B_K$ and $\B_{\adv,K}$ are given by
  \begin{align*}
    \B_K( {W}_K, \A_\Gamma^K ) & = \mdK \A_\Gamma^K + \nabla_K \cdot {W}_K \A_\Gamma^K - 2 D_h( {W}_h ) \\
    \B_{\adv, K}( {W}_K, \b_\Gamma^K )
    & = \mdK \b_\Gamma^K + \b_\Gamma^K \nabla_K \cdot {W}_K -
      \sum_{j=1}^{n+1} ( \b_\Gamma^K )_j ( \nabla_K )_j {W}_K,
  \end{align*}
  and $D_h$ is the rate of deformation tensor
  \begin{equation*}
    D( {w} )_{ij} = \frac{1}{2}
    \sum_{k=1}^{n+1} (\A_\Gamma^K)_{ik} ( \nabla_K )_k ( {W}_K )_j
    + (\A_\Gamma^K)_{jk} ( \nabla_K )_k ( {W}_K )_i
    \quad \mbox{ for } i,j = 1, \ldots, n+1.
  \end{equation*}
\end{lemma}

\begin{proof}
  We note that the left hand side may be decomposed into individual elements then apply \cref{eq:surface-transport,eq:transport-dirichlet,eq:transport-advection} on each element.%
\end{proof}

\subsection{Construction of the lifted finite element space}
\label{sec:surf-lift-prop}

Recalling the normal projection operator \cref{eq:cp}, we define the global lifting map 
$\Lambda_h( \cdot, t ) \colon \Gamma_h(t) \to \Gamma(t)$ (\cref{def:global-lifting-map}) by
\begin{equation*}
  \Lambda_h( x, t ) := p( x, t )
  \qquad \mbox{ for } x \in \Gamma_h(t),
\end{equation*}
and denote the restriction to each element by $\Lambda_K( \cdot, t ) := \Lambda_h( \cdot, t ) |_{K(t)}$ for each $K(t) \in \T_h(t)$.
For $\eta_h \colon \Gamma_h(t) \to \R$, we denote its lift by $\eta_h^\ell( x )$ given by
\begin{equation*}
  \eta_h^\ell( \Lambda_h( x, t ) ) = \eta_h( x ) \qquad \mbox{ for } x \in \Gamma_h(t)
\end{equation*}
and for $\eta \colon \Gamma(t) \to \R$, we denote its inverse lift by $\eta^{-\ell}( x )$ given by
\begin{equation*}
  \eta^{-\ell}( x ) = \eta_h( \Lambda_h( x, t ) ) \qquad \mbox{ for } x \in \Gamma_h(t).
\end{equation*}

For each $t \in [0,T]$ and each $K(t) \in T_h(t)$, we use \cref{def:lifted-sfe} to construct an 
associated lifted surface finite element $(K^\ell(t), \P^\ell(t), \Sigma^\ell(t) )$.
We assume the domains $\{ \Gamma_h(t) \}_{t \in [0,T]}$ are such that the set of lifted element domains 
$\T_h^\ell(t)$ defines an exact decomposition of $\Gamma(t)$ (\cref{def:exact-decomposition}) at each $t \in [0,T]$.

We use the lift to define the space of lifted finite element functions $\S_h^\ell(t)$ (\cref{def:lifted-sfe-space}) by
\begin{equation*}
  \S_h^\ell(t) := \{ \chi_h^\ell : \chi_h \in \S_h(t) \}.
\end{equation*}

We continue by showing some basic geometric estimates.
These geometric estimates have been given by \cite[Lemma 5.2]{Kov18}.

\begin{lemma}
  \label{lem:surf-geom-est}
  Under the above smoothness assumptions, we have
  \begin{equation}
    \label{eq:d-estimate}
    \sup_{t \in [0,T]} \max_{K(t) \in \T_h(t)} \norm{ d }_{L^\infty(K(t))}
    \le c h^{k+1}.
  \end{equation}
  \begin{align}
    \label{eq:surf-nunuh}
    \sup_{t \in [0,T]} \max_{K(t) \in \T_h(t)} \norm{ \nu - \nu_K }_{L^\infty(K(t))}
    & \le c h^k \\
    \label{eq:surf-HHh}
    \sup_{t \in [0,T]} \max_{K(t) \in \T_h(t)} \norm{ \mathbb{H} - \mathbb{H}_K }_{L^\infty(K(t))}
    & \le c h^{k-1},
  \end{align}
  where $\mathbb{H}_K := \nabla_K \nu_K$.
  Writing $\delta_h$ for the quotient between discrete and continuous surface measures so that $\mathrm{d} \sigma = \delta_h \dd \sigma_h$, we have
  \begin{align}
    \label{eq:closed-deltah-esimate}
    \sup_{t \in [0,T]} \max_{K(t) \in \T_h(t)} \norm{ 1 - \delta_h }_{L^\infty(K(t))}
    & \le c h^{k+1}.
  \end{align}
\end{lemma}

\begin{lemma}
  \label{lem:surf-lift-bounds}
  Let $\{ \T_h^\ell(t) \}_{t \in [0,T]}$ be the evolving subdivision defined by the lifting map $\Lambda_h$
  and assume that $\{ \T_h(t) \}_{t \in [0,T]}$ is a uniformly $k$-regular, uniformly quasi-uniform, evolving conforming subdivision.
  Then $\{ \T_h^\ell(t) \}_{t \in [0,T]}$ is also a uniformly $k$-regular, uniformly quasi-uniform, evolving conforming exact subdivision of $\Gamma(t)$.
\end{lemma}

\begin{proof}
  We use the decomposition that
  \[
    \Lambda_h( \cdot, t ) = x + ( p(x,t) - x ) =: x + \tilde{\Lambda}_h( x, t ),
  \]
  and write $\tilde{\Lambda}_K( \cdot, t )$ for $\tilde{\Lambda_h}(\cdot, t)|_{K(t)}$.
  From \cref{rem:id-lift-ass} and \cref{prop:lift-Vht}, we have to show
  \begin{align}
    \sup_{x \in K} \norm{ \nabla_{K(t)} \tilde\Lambda_K(x) }
    & \le \frac{1 - C_K}{1+C_K} && \mbox{ for all } K \in \T_h(t), \mbox{ all } t \in [0,T], \mbox{ and all } h \in (0,h_0),
  \end{align}
  and
  \begin{align}
    \sup_{h \in (0,h_0)} \sup_{t \in [0,T]} \norm{ \Lambda_h( \cdot, t ) }_{W^{k+1,\infty}(\Gamma_h(t))}
    & \le C.
  \end{align}

  The second estimate follows from the smoothness of $p$ \cref{eq:cp} (which follows from the smoothness of $\Gamma(t)$ \citep{Foote1984}).
  Next we compute directly that
  \begin{equation*}
    \partial_{x_j} ( p( x, t ) )_i = \delta_{ij} - \nu_i( x, t ) \nu_j( x, t ) - d( x, t ) \mathbb{H}( x, t )_{ij}
  \end{equation*}
  so that using the notations $P_{ij} = \delta_{ij} - \nu_i \nu_j$ and $( P_h )_{ij} = \delta_{ij} - (\nu_h)_i (nu_h)_j$ for $j=1,\ldots,n+1$, we have
  \begin{align*}
    \nabla_{K(t)} \tilde\Lambda_K( x, t )
    & = P_h( x, t ) ( P( x, t ) - d( x, t ) \mathbb{H}( x, t ) ) - P_h( x, t ) \\
    & = P_h( x, t ) ( P( x, t ) - P_h( x, t ) ) - d( x, t ) P_h( x, t ) \mathbb{H}( x, t ).
  \end{align*}
  Applying \cref{eq:d-estimate} and \cref{eq:surf-nunuh}, we see that
  \begin{align*}
    \norm{ \nabla_{K(t)} \tilde\Lambda_K( x, t ) }
    \le \norm{ P( x, t ) - P_h( x, t ) } + c h^{k+1} \norm{ \mathbb{H}(x,t) }
    \le c h^k.
  \end{align*}
  Clearly the right hand side of this equation is less that $(1-C_K)/(1+C_K)$ for $h$ sufficiently small.
\end{proof}

\begin{lemma}
  \label{lem:surf-lift-stab}
  Let $\eta_h \in \H_h(t)$ and denote its lift by $\eta_h^\ell$. Then there exists constants $c_1, c_2 > 0$ such that
  \begin{align}
    \label{eq:surf-lift-stab-H}
    c_1 \norm{ \eta_h^\ell }_{\H(t)}
    & \le \norm{ \eta_h }_{\H_h(t)}
    \le c_2 \norm{ \eta_h^\ell }_{\H(t)}.
    \intertext{Furthermore, if $\eta_h \in \V_h(t)$, there exists constants $c_3, c_4 > 0$ such that}
    \label{eq:surf-lift-stab-V}
    c_3 \norm{ \eta_h^\ell }_{\V(t)}
    & \le \norm{ \eta_h }_{\V_h(t)}
    \le c_4 \norm{ \eta_h^\ell }_{\V(t)}.
  \end{align}
\end{lemma}

\begin{proof}
  We apply \cref{prop:lift-Vh} using \cref{lem:surf-lift-bounds}.
\end{proof}

Using the lift $\Lambda_h$ and the discrete flow $\Phi^h_t$, we can define the lifted flow map $\Phi^\ell_t$ (\cref{def:lifted-flow-map}), 
lifted discrete velocity $w_h$ (\cref{def:lifted-discrete-material-velocity}) and lifted push forward maps $\phi^\ell_t$ (\cref{def:lifted-push-forward-map,eq:phiellt}).
\cref{lem:surf-lift-bounds} implies, with \cref{prop:lift-Vht}, that $( \H(t), \phi^\ell_t )_{t \in [0,T]}$, $( \V(t), \phi^\ell_t|_{\V_0} )_{t \in [0,T]}$ 
and $( \S^\ell_h(t), \phi^\ell_t|_{\S_h(0)} )_{t \in [0,T]}$ are compatible pairs uniformly for $h \in (0,h_0)$.
We will use the notations $C^1_{(\H,\phi^\ell)}$ and $C^1_{(\V,\phi^\ell)}$ for the spaces of functions smoothly evolving 
in time \cref{eq:CkX} with respect to the push-forward map $\phi^\ell_t$ in $\H(t)$ and $\V(t)$ respectively.
We may also define a lifted material derivative for functions $\eta \in C^1_{(\H,\phi^{\ell})}$ using \cref{eq:lift-md} by %
\begin{equation}
  \label{eq:surf-lift-md}
  \mdell \eta = \phi^\ell_{t} \left( \dt \left( \phi^\ell_{-t} \eta \right) \right).
\end{equation}

\begin{lemma}
  \label{lem:lifted-transport}
  The push forward map $\phi_t^\ell$ induces a new transport formula on $\{ \Gamma(t) \}$. For $\eta \in C^1_{(\H, \phi^\ell)}$ we have
  \begin{equation}
    \dt \int_{\Gamma(t)} \eta \dd \sigma
    = \sum_{K^\ell(t) \in \T_h^\ell(t)} \int_{K^\ell(t)} \mdell \eta + \eta \nabla_\Gamma \cdot {w}_h \dd \sigma.
  \end{equation}
  Let $\A_\Gamma$ be a smooth, positive-definite, diffusion tensor on $\Gamma(t)$, which maps the tangent space of $\Gamma(t)$ to itself, and $\b_\Gamma$ be a smooth tangential vector field on $\Gamma(t)$ for all $t \in [0,T]$.  Furthermore, we have for $\eta, \zeta \in C^1_{( \V, \phi^\ell)}$
  \begin{multline}
    \dt \int_{\Gamma(t)} \A_\Gamma \nabla_\Gamma \eta \cdot \nabla_\Gamma \zeta \dd \sigma \\
    = \sum_{K^\ell(t) \in \T_h^\ell(t)}
    \int_{K^\ell(t)}
    \A_\Gamma \left( \nabla_\Gamma \mdell \eta \cdot \nabla_\Gamma \zeta + \nabla_\Gamma \eta \cdot \nabla_\Gamma \mdell \zeta \right)
    + \B_{K^\ell}( {w}_K, \A_\Gamma ) \nabla_\Gamma \eta \cdot \nabla_\Gamma \zeta \dd \sigma,
  \end{multline}
  and for $\eta \in C^1_{(\V, \phi^\ell)}, \zeta \in C^1_{(\H, \phi^\ell)}$ we have
  \begin{multline}
    \dt \int_{\Gamma(t)} \b_\Gamma \eta \cdot \nabla_\Gamma \zeta \dd \sigma \\
    = \sum_{K^\ell(t) \in \T_h^\ell(t)}
    \int_{K^\ell(t)}
    \b_\Gamma \left( \mdell \eta \cdot \nabla_\Gamma \zeta + \eta \cdot \nabla_\Gamma \mdell \zeta \right)
    + \B_{K^\ell,\adv}( {w}_K, \b_\Gamma ) \eta \cdot \nabla_\Gamma \zeta \dd \sigma,
  \end{multline}
  where $\B_{K^\ell}$ and $\B_{K^\ell,\adv}$ are defined as in \cref{lem:gammah-transport}.
\end{lemma}

\begin{proof}
  We note that the left hand side may be decomposed into individual elements then apply \cref{eq:surface-transport,eq:transport-dirichlet,eq:transport-advection} on each element.%
\end{proof}

\begin{lemma}
  \label{lem:surf-md-estimates}
  \begin{align}
    \label{eq:md-d-estimate}
    \sup_{t \in [0,T]} \norm{ \mdh d }_{L^\infty(\Gamma_h(t))} & \le c h^{k+1} \\
    \label{eq:md-P-nu-h-estimate}
    \sup_{t \in [0,T]} \norm{ \mdh P \nu_h }_{L^\infty(\Gamma_h(t))} & \le c h^k \\
    \label{eq:md-delta-h-estimate}
    \sup_{t \in [0,T]} \norm{ \mdh \delta_h }_{L^\infty(\Gamma_h(t))} & \le c h^{k+1}.
  \end{align}
\end{lemma}

\begin{proof}
  See \cite[Lem.~5.2]{Kov18}.
\end{proof}

We can characterise the lifted discrete material derivatives \cref{eq:surf-lift-md} using the difference between smooth, $w$, and lifted, $w_h$, velocities.
The result also shows the abstract inclusions \cref{eq:abs-C1-inclusion}.
\begin{lemma}
  \label{lem:md-lift-char}
  If $\eta \in C^1_{\H} \cap C^0_\V$, then $\eta \in C^1_{(\H,\phi^\ell)}$ and conversely if $\eta \in C^1_{(\H,\phi^\ell)} \cap C^0_{(\V,\phi^\ell)}$ then $\eta \in C^1_{\H}$.
  In either case, we have the identity
  \begin{equation}
    \label{eq:md-lift-char}
    \md \eta - \mdh \eta = \nabla_\Gamma \eta \cdot ( w - w_h ).
  \end{equation}
  Furthermore, if $\eta \in C^1_{\V} \cap C^0_{\Z_0}$ then $\eta \in C^1_{(\V,\phi^\ell)}$.
\end{lemma}

\begin{proof}
  We use a chart $X(\cdot, t) \colon U \to V \subset \Gamma(t)$ and write $\eta( X( \theta, t ), t ) = F( \theta, t )$.
  Note that $X_t( \theta, t ) = w( X(\theta,t),t)$.
  Then
  \begin{align*}
    \md \eta(x,t) - \mdell \eta(x,t)
    & = \frac{\partial}{\partial t}
      \left. \left(
      \eta( \Phi_t( x_0 ), t ) - \eta( \Phi_t^{\ell}( y_0 ), t )
      \right) \right|_{x_0 = \Phi_{-t}( x ), y_0 = \Phi_{-t}^\ell( x )}.
  \end{align*}
  Using the notation $X( \theta(t), t ) = \Phi_t( x_0 )$ and $X( \theta^\ell(t), t ) = \Phi_t^\ell( y_0 )$, we have
  \begin{multline*}
    \md \eta( x, t ) - \mdell \eta( x, t ) \\
    = \frac{\partial F}{\partial t}( \theta(t), t ) - \frac{\partial F}{\partial t}( \theta^\ell(t), t )
    + \sum_{i=1}^n \left(
      \frac{\partial F}{\partial \theta_i}( \theta(t), t ) \frac{\partial \theta_i}{\partial t}(t) -
      \frac{\partial F}{\partial \theta_i}( \theta^\ell(t), t ) \frac{\partial \theta^\ell_i}{\partial t}(t)
    \right).
  \end{multline*}
  Using the substitutions $X(\theta(t), t ) = \Phi_t( x_0 )$, $X( \theta^\ell(t), t ) = \Phi_t^\ell( y_0 )$, $x_o = \Phi_{-t}(x)$, and $y_0 = \Phi_{-t}^\ell( x )$, we see
  \[
    \frac{\partial F}{\partial t}( \theta(t), t ) - \frac{\partial F}{\partial t}( \theta^\ell(t), t )
    = 0.
  \]

  Next, we take a time derivative of $X( \theta(t), t ) = \Phi_t( x_0 )$ to see that
  \[
    \sum_{i=1}^n \frac{\partial X}{\partial \theta_i}( \theta(t) ) \frac{\partial \theta_i}{\partial t}(t) = 0.
  \]
  By multiplying by $\partial X / \partial \theta_j$, summing over $i$ and multiplying $G^{-1}$ we see that $\partial \theta / \partial t = 0$.

  Next, we take a time derivative of $X( \theta^\ell(t), t ) = \Phi_t^\ell( x_0 )$, to see that
  \[
    w( X( \theta^\ell(t), t ) ) + \sum_{i=1}^n \frac{\partial X}{\partial \theta_i}( \theta^\ell(t) ) \frac{\partial \theta^\ell_i}{\partial t} (t) = w_h( \Phi_t^\ell(y_0), t ).
  \]
  Multiplying by $\partial X / \partial \theta_j ( \theta^\ell(t) )$ we see that
  \[
    \sum_{i=1}^n g_{ij}( \theta^\ell(t) ) \frac{\partial \theta^\ell_i}{\partial t}( t )
    = \frac{\partial X}{\partial \theta_j}( \theta^\ell(t) ) \left( w_h( \Phi_t^\ell(y_0), t ) - w( \theta^\ell(t), t ) \right),
  \]
  from which we infer that
  \[
    \frac{\partial \theta^\ell_i}{\partial t}( t )
    = \sum_{j=1}^n g^{ij}( \theta^\ell(t) ) \frac{\partial X}{\partial \theta_j}( \theta^\ell(t) ) \left( w_h( \Phi_t^\ell(y_0), t ) - w( \theta^\ell(t), t ) \right).
  \]

  Combining the previous expressions we see that
  \[
    \md \eta( x, t ) - \mdell \eta( x, t )
    = \sum_{i,j=1}^n g^{ij}( \theta^\ell(t) ) \frac{\partial X}{\partial \theta_j}( \theta^\ell(t) ) \frac{\partial F}{\partial \theta_i}( \theta^\ell(t), t ) \left( w_h( \Phi_t^\ell(y_0), t ) - w( \theta^\ell(t), t ) \right).
  \]
  Using the substitutions as above and the parametric definition of tangential gradient gives the result.

  The second result can be shown by applying the tangential gradient to the basic result.
\end{proof}

\subsection{The discrete problem and stability}
\label{sec:surf-discrete-pr}

For each $t \in [0,T]$, and $h \in (0,h_0)$, we assume that $\A_\Gamma^h(t)$ is an element-wise smooth $(n+1) \times (n+1)$ symmetric 
diffusion tensor defined element-wise with $\A_\Gamma^h(t)|_{K(t)} = \A_\Gamma^K(t)$ for each $K(t) \in \T_h(t)$. We assume that 
$ \A_\Gamma^K(t)$ maps the tangent space of $K(t)$ at a point into itself and is uniformly positive definite on the tangent 
space: There exists $\bar{a}_0 > 0$ such that for all $h \in (0,h_0)$, $t \in [0,T]$, and $K(t) \in \T_h(t)$
\begin{equation*}
 \A_\Gamma^K(t)(\cdot) \xi \cdot \xi \ge \bar{a}_0 \abs{ \xi }^2 \quad
  \mbox{ for all } \xi \in \R^{n+1}, \xi \cdot \nu_K( \cdot, t ) = 0.
\end{equation*}
We assume we are also given a element-wise smooth tangential vector field $\b_\Gamma ^h(t)$ (with $\b_\Gamma ^h(t)|_{K} = \b_\Gamma^K(t)$) 
and element-wise smooth scalar field $\c_\Gamma ^h(t)$ (with $\c_\Gamma ^h(t)|_{K(t)} = \c_\Gamma ^K(t))$). We assume that
\begin{equation}
  \sup_{t \in [0,T]} \max_{K(t) \in \T_{h}(t)} \left(
    \norm{ \A_\Gamma^K }_{C^1(K(t))} + \norm{ \b_\Gamma^K }_{C^1(K(t))} + \norm{ \c_\Gamma^K }_{C^1(K(t))}
  \right)
  < C.
\end{equation}

\begin{example}
  Here we are thinking of the case that $\A_\Gamma^h = \A_\Gamma^{-\ell}$, $\b_\Gamma^h = \b_\Gamma^{-\ell}$ and $\c_\Gamma^h = \c_\Gamma^{-\ell}$.
\end{example}

We consider the following semi-discrete problem (c.f. \cref{pb:fem}):
\begin{problem}
  \label{pb:surf-fem}
  Given $U_{h,0} \in \S_h(0)$, find $U_h \in C^1_{\S_h}$ with $U_h(0) = U_{h,0}$ and such that for every $t \in (0,T)$,
  \begin{equation}
    \label{eq:surf-fem}
    \dt m_h( t; U_h, \zeta_h ) + a_h( t; U_h, \zeta_h )
    = m_h( t; U_h, \mdh \zeta_h )
    \qquad \mbox{ for all } \zeta_h \in C^1_{\S_h},
  \end{equation}
  where
  \begin{align*}
    m_h( t; \eta_h, \zeta_h ) & :=
    \int_{\Gamma_h(t)} \eta_h \zeta_h \dd \sigma_h && \mbox{ for } \eta_h, \zeta_h \in \H_h(t) \\
    a_h( t; \eta_h, \zeta_h ) & :=
    \sum_{K(t) \in \T_h(t)}
                                \int_{K(t)} \A_\Gamma^K \nabla_K \eta_h \cdot \nabla_K \zeta_h \\
    & \qquad\qquad\qquad + \eta_h \nabla_K \zeta_h \cdot \b_\Gamma^K + \c_\Gamma^K \eta_h \zeta_h \dd \sigma_h
    && \mbox{ for } \eta_h, \zeta_h \in \V_h(t).
  \end{align*}
\end{problem}

We note that the assumption that $\{ \T_h(t) \}_{t \in [0,T]}$ is uniformly quasi-uniform regular implies that 
$\{ \S_h(t), \phi^h_t \}_{t \in [0,T]}$ is a compatible pair when equipped with the $\V_h(t)$ or $\H_h(t)$-norms (\cref{lem:abs-Vh-compatible}).

To show the properties of these bilinear forms we require one further lemma:
\begin{lemma}
  \label{lem:surf-Wh-bounded}
  The discrete velocity $W_h$ of the discrete evolving surface $\{ \Gamma_h(t) \}$ \cref{eq:surf-defn-Wh} is uniformly bounded in $W^{1,\infty}( \Gamma_h(t) )$.
  That is, there exists a constant $C>0$ such that for all $h \in (0,h_0)$
  \begin{equation}
    \label{eq:surf-Wh-bound}
    \norm{ \nabla_{\Gamma_h} {W}_h }_{L^\infty(\Gamma_h(t))}
    \le c \norm{ {w} }_{W^{1,\infty}(\Gamma(t))}.
  \end{equation}
\end{lemma}

\begin{proof}
  The result follows from the interpolation bound (\cref{cor:global-inv-lift-interp}) and the stability of the lift (\cref{prop:lift-Vh}).
\end{proof}

We have transport formulae on the surface $\{ \Gamma_h(t) \}$.
\begin{lemma}
  \label{lem:surf-transport-bounds}
  There exists a bilinear forms $g_h( t; \cdot, \cdot ) \colon \H_h(t) \times \H_h(t) \to \R$ and $b_h( t; \cdot, \cdot ) \colon \V_h(t) \times \V_h(t) \to \R$ such that
  \begin{align}
    \label{eq:surf-discrete-transport-mh}
    \dt m_h( t; \eta_h, \zeta_h )
    & = m_h( t; \mdh \eta_h, \zeta_h ) + m_h( t; \eta_h, \mdh \zeta_h )
      + g_h( t; \eta_h, \zeta_h )
    && \mbox{ for } \eta_h, \zeta_h \in C^1_{\H_h} \\
    \label{eq:surf-discrete-transport-ah}
    \dt a_h( t; \eta_h, \zeta_h )
    & = a_h( t; \mdh \eta_h, \zeta_h ) + a_h( t; \eta_h, \mdh \zeta_h )
      + b_h( t; \eta_h, \zeta_h )
    && \mbox{ for } \eta_h, \zeta_h \in C^1_{\V_h},
  \end{align}
  where
  \begin{align}
    \label{eq:gh-surface}
    g_h( t; \eta_h, \zeta_h )
    & = \sum_{K(t) \in \T_h(t)} \int_{K(t)} \eta_h \zeta_h \nabla_{K} \cdot {W}_h \dd \sigma_h
  \end{align}
  and
  \begin{equation}
    \label{eq:bh-surface}
    \begin{aligned}
    b_h( t; \eta_h, \zeta_h )
    = \sum_{K(t) \in \T_h(t)} \int_{K(t)}
    & \Big( \B_K( {W}_K, \A_\Gamma^K ) \nabla_K  \eta_h \cdot \nabla_K \zeta_h \\
    & + \B_{\adv,K}( {W}_K, \b_\Gamma^K ) \eta_h \cdot \nabla_K \zeta_h \\
    & + (\mdK \c_\Gamma^K + \c_\Gamma^K \nabla_K {W}_K ) \eta_h \zeta_h \Big)
    \dd \sigma_h.
    \end{aligned}
  \end{equation}
  Furthermore, there exist a constant $c > 0$ such that for all $t \in [0,T]$ and all $h \in (0,h_0)$ we have
  \begin{align}
    \label{eq:surf-gh-bounded}
    \abs{ g_h( t; \eta_h, \zeta_h ) }
    & \le c \norm{ \eta_h }_{\H_h(t)} \norm{ \zeta_h }_{\H_h(t)}
    && \mbox{ for all } \eta_h, \zeta_h \in \H_h(t) \\
    \label{eq:surf-bh-bounded}
    \abs{ b_h( t; \eta_h, \zeta_h ) }
    & \le c \norm{ \eta_h }_{\V_h(t)} \norm{ \zeta_h }_{\V_h(t)}
    && \mbox{ for all } \eta_h, \zeta_h \in \V_h(t).
  \end{align}
\end{lemma}

\begin{proof}
  The transport theorem \cref{eq:gammah-transport} directly gives \cref{eq:gh-surface} and 
  additionally \cref{eq:dirichleth-transport} and \cref{eq:advectionh-transport} give \cref{eq:bh-surface}.
  To see the boundedness properties, we directly apply \cref{lem:surf-Wh-bounded}.
\end{proof}

\begin{theorem}
  \label{thm:surf-fem-stability}
  There exists a unique solution of the finite element scheme \cref{eq:surf-fem}. The solution satisfies the stability bound:
  \begin{equation}
    \sup_{t \in (0,T)} \norm{ U_h }_{\H_h(t)}^2
    + \int_0^T \norm{ U_h }_{\V_h(t)}^2 \dd t
    \le c \norm{ U_{h,0} }_{\H_h(t)}^2.
  \end{equation}
\end{theorem}

\begin{proof}
  The result is shown in the abstract setting in \cref{thm:abs-stab} so we are left to check the assumptions.
  The assumptions \cref{eq:mh-symmetric} and \cref{eq:mh-bounded} follow since $m_h$ is simply the $\H_h(t) = L^2( \Gamma_h(t) )$ inner product.
  For \cref{eq:ch-formula}, we use \cref{eq:gammah-transport} and the product rule $\mdh ( v_h )^2 = 2 v_h \mdh v_h$.
  The bound \cref{eq:ch-bounded} is shown in \cref{eq:surf-gh-bounded}.
  The map $t \mapsto a_h(t ; \cdot, \cdot )$ is differentiable and hence measurable \cref{eq:ah-meas}.
  The bounds \cref{eq:ah-coercive} and \cref{eq:ah-bounded} follow from standard calculations and our assumptions on $\A_\Gamma^h, \b_\Gamma^h, \c_\Gamma^h$.
  Finally we have shown \cref{eq:bh-formula} and \cref{eq:bh-bounded} in \cref{eq:bh-surface} and \cref{eq:surf-bh-bounded}.
\end{proof}

\subsection{Error analysis}
\label{sec:fem-surf-error}

The space $\S_h^\ell(t)$ is equipped with the following approximation property:
\begin{lemma}
  \label{lem:surf-approx}
  The interpolation operator $I_h \colon \Z_0(t) \to \S_h^\ell(t)$ is well defined and satisfies
  \begin{align}
    \norm{ z - I_h z }_{\H(t)} + h \norm{ z - I_h z }_{\V(t)}
    & \le c h^2 \norm{ z }_{\Z_0(t)}
    && \mbox{ for } z \in \Z_0(t) \\
    \norm{ z - I_h z }_{\H(t)} + h \norm{ z - I_h z }_{\V(t)}
    & \le c h^{k+1} \norm{ z }_{\Z(t)}
    && \mbox{ for } z \in \Z(t).
  \end{align}
\end{lemma}

\begin{proof}
  We simply apply \cref{thm:global-lift-interp}. The second result applies the theorem in the obvious way. 
  The first result applies the theorem with $k=1$ noting that $P_1(\hat{K}) \subset P_k(\hat{K})$ and the 
  inclusions for \cref{lem:bramble-hilbert} still hold.
\end{proof}

We can further relate the lifted material velocity with the discrete material velocity.
Let $x = X(t) \in \Gamma_h(t)$ evolve with velocity $W_h(X(t),t)$ and $Y(t) = \Lambda_h( X(t), t )$. Then
\begin{equation}
  \label{eq:surf-wh}
  {w}_h( p( x, t ), t )
  = \dt Y( t ) = \dt p( X(t), t )
  = \nabla p ( X(t), t ) {W}_h( X(t), t )
  + p_t( X(t), t ).
\end{equation}
Then from the above calculation of $\partial_{x_j} p$ we have
\begin{equation}
  \label{eq:surf-wh-calc}
  {w}_h( p( x, t ) )
  = ( P( x, t ) - d( x, t ) \mathbb{H}( x, t ) ) {W}_h( x, t )
  - d_t( x, t ) \nu( x, t ) - d( x, t ) \nu_t( x, t ).
\end{equation}

\begin{lemma}
  \label{lem:surf-wh-estimate}
  We have the estimate:
  \begin{equation}
    \label{eq:surf-wh-estimate}
    \norm{ {w} - {w}_h }_{L^\infty(\Gamma(t))}
    + h \norm{ \nabla_\Gamma ( {w} - {w}_h ) }_{L^\infty(\Gamma(t))}
    \le c h^{k+1}.
  \end{equation}
\end{lemma}

\begin{proof}
  See \citep[Lem~5.4]{Kov18}.
\end{proof}

The lifting operator also defines transport formulae:
\begin{lemma}
  \label{lem:surf-lift-transport}
  There exists bilinear forms $g_\ell \colon \H(t) \times \H(t) \to \R$ and $b_\ell \colon \V(t) \times \V(t)$ given by
  \begin{align}
    g_\ell( t; \eta, \eta )
    & = \int_{\Gamma(t)} \eta \zeta \nabla_{\Gamma} \cdot {w}_h \dd \sigma
    && \mbox{ for all } \eta, \zeta \in \H(t) \\
    b_\ell( t; \eta, \zeta )
    & = \sum_{K^\ell(t) \in \T_h^\ell(t)}
      \int_{K^\ell(t)}  \Big( \B( {w}_h, \A_\Gamma ) \nabla_\Gamma \eta \cdot \nabla_\Gamma \zeta \\
    \nonumber
    & \qquad\qquad\qquad\qquad\qquad + \B_{\adv}( {w}_h, \b_\Gamma ) \eta \cdot \nabla_\Gamma \zeta \\
    \nonumber
    & \qquad\qquad\qquad\qquad\qquad + ( \mdell \c_\Gamma + \c_\Gamma \nabla_\Gamma \cdot {w}_h ) \eta \zeta \Big) \dd \sigma
    && \mbox{ for all } \eta, \zeta \in \V(t).
  \end{align}
  These bilinear forms satisfy the following transport formulae on $\Gamma(t)$:
  \begin{align}
    \label{eq:surf-lift-transport-m}
    \dt m( t; \eta, \zeta )
    & = m( t; \mdell \eta, \zeta )
    + m( t; \eta, \mdell \zeta )
    + g_\ell( t; \eta, \zeta )
    && \mbox{ for } \eta, \zeta \in C^1_{(\H,\phi^\ell)} \\
    \label{eq:surf-lift-transport-a}
    \dt a( t; \eta, \zeta )
    & = a( t; \mdell \eta, \zeta )
    + a( t; \eta, \mdell \zeta )
    + b_\ell( t; \eta, \zeta )
    && \mbox{ for } \eta, \zeta \in C^1_{(\V,\phi^\ell)}.
  \end{align}
  Furthermore, the two new bilinear forms are uniformly bounded in the sense that there 
  exists a constant $c > 0$ such that for all $t \in [0,T]$ and all $h \in (0,h_0)$,
  \begin{align}
    \abs{ g_\ell( t; \eta, \zeta ) }
    & \le c \norm{ \eta }_{\H(t)} \norm{ \zeta }_{\H(t)}
    && \mbox{ for all } \eta, \zeta \in \H(t) \\
    \abs{ b_\ell( t; \eta, \zeta ) }
    & \le c \norm{ \eta }_{\V(t)} \norm{ \zeta }_{\V(t)}
    && \mbox{ for all } \eta, \zeta \in \V(t).
  \end{align}
\end{lemma}

\begin{proof}
  The transport formulae are direct translations of \cref{lem:lifted-transport}.
  The bounds follow from the fact that $\norm{ {w}_h }_{W^{1,\infty}(\Gamma_h(t))}$ is bounded uniformly from \cref{lem:surf-wh-estimate}.
\end{proof}

For the remainder of this section, we will additionally assume that $\A_\Gamma$ is uniformly $C^2$ smooth in space.
 We set $\A_\Gamma^h = \A_\Gamma^{-\ell}, \b_\Gamma^h = \b_\Gamma^{-\ell}, \c_\Gamma^h = \c_\Gamma^{-\ell}$.%

\begin{lemma}
  \label{lem:surface-geom-pert}
  There exists a constant $c > 0$ such that for all $t \in [0,T]$ and all $\eta_h, \zeta_h \in \V_h(t)$ with lifts $\eta_h^\ell, \zeta_h^\ell \in \V(t)$ we have
  \begin{subequations}
    \begin{align}
      \label{eq:m-error-surf}
      \abs{ m( t; \eta_h^\ell, \zeta_h^\ell ) - m_h( t; \eta_h, \zeta_h ) }
      & \le c h^{k+1} \norm{ \eta_h^\ell }_{\V(t)} \norm{ \zeta_h^\ell }_{\V(t)} \\
      \label{eq:a-error-surf}
      \abs{ a( t; \eta_h^\ell, \zeta_h^\ell ) - a_{h}( t; \eta_h, \zeta_h ) }
      & \le c h^{k+1} \norm{ \eta_h^\ell }_{\V(t)} \norm{ \zeta_h^\ell }_{\V(t)} \\
      \label{eq:g-error-surf}
      \abs{ g_\ell( t; \eta_h^\ell, \zeta_h^\ell ) - g_h( t; \eta_h, \zeta_h ) }
      & \le c h^{k+1} \norm{ \eta_h^\ell }_{\V(t)} \norm{ \zeta_h^\ell }_{\V(t)} \\
      \label{eq:b-error-surf}
      \abs{ b_\ell( t; \eta_h^\ell, \zeta_h^\ell ) - b_{h}( t; \eta_h, \zeta_h ) }
      & \le c h^{k+1} \norm{ \eta_h^\ell }_{\V(t)} \norm{ \zeta_h^\ell }_{\V(t)} \\
      \label{eq:gt-error-surf}
      \abs{ g_\ell( t; \eta_h^\ell, \zeta_h^\ell ) - g( t; \eta_h^\ell, \zeta_h^\ell ) }
      & \le c h^{k} \norm{ \eta_h^\ell }_{\V(t)} \norm{ \zeta_h^\ell }_{\V(t)} \\
      \label{eq:bt-error-surf}
      \abs{ b_\ell( t; \eta_h^\ell, \zeta_h^\ell ) - b( t; \eta_h^\ell, \zeta_h^\ell ) }
      & \le c h^{k} \norm{ \eta_h^\ell }_{\V(t)} \norm{ \zeta_h^\ell }_{\V(t)}.
    \end{align}
  \end{subequations}
\end{lemma}

\begin{proof}
  The results \cref{eq:m-error-surf,eq:a-error-surf,eq:g-error-surf,eq:b-error-surf} easily follow using ideas from \citep[Lem~5.6]{Kov18} and \cref{lem:bulk-geom-pert}.
  \cref{eq:gt-error-surf,eq:bt-error-surf} follow directly from \cref{lem:surf-wh-estimate}.
\end{proof}

\begin{lemma}
  \label{lem:closed-md-error}
  \begin{align}
    \label{eq:closed-md-error}
    \norm{ \md \eta - \mdell \eta }_{L^2(\Gamma(t))}
    & \le c h^{k+1} \norm{ \eta }_{H^1(\Gamma(t))}
    && \mbox{ for } \eta \in H^1( \Gamma(t) ) \\
    \label{eq:closed-mdV-error}
    \norm{ \nabla_\Gamma (\md \eta - \mdell \eta) }_{L^2(\Gamma(t))}
    & \le c h^{k} \norm{ \eta }_{H^2(\Gamma(t))}
    && \mbox{ for } \eta \in H^2( \Gamma(t) ).
  \end{align}
  Furthermore for all $\eta \in C^1_{\Z_0}$ and $\zeta \in \Z_0(t)$, we have
  \begin{equation}
    \label{eq:amd-error-surf}
    \abs{ a( t; \mdell \eta, \zeta ) - a_h( t; \mdh \eta^{-\ell}, \zeta^{-\ell} ) }
    \le c h^{k+1} \bigl( \norm{ \eta }_{\Z_0(t)} + \norm{ \md \eta }_{\Z_0(t)} \bigr) \norm{ \zeta }_{\Z_0(t)}.
  \end{equation}
\end{lemma}

\begin{proof}
  We recall \cref{eq:md-lift-char}:
  \begin{equation*}
    \md \eta - \mdell \eta
    = ( {w} - {w}_h ) \cdot \nabla_\Gamma \eta.
  \end{equation*}
  We combine this calculation with \cref{eq:surf-wh-estimate} to see \cref{eq:closed-md-error}.

  We may apply the tangential gradient to the above equation and use \cref{eq:surf-wh-estimate} again to obtain
  \begin{equation*}
    \norm{ \nabla_\Gamma ( \md \eta - \mdell \eta ) }_{L^2(\Gamma(t))}
    \le c h^k \norm{ \eta }_{H^1(\Gamma(t))} + c h^{k+1} \norm{ \eta }_{H^2(\Gamma(t))}.
  \end{equation*}
  \cref{eq:amd-error-surf} follows from \cref{eq:a-error-surf} and \cref{eq:closed-mdV-error}.
\end{proof}

\begin{theorem}
  \label{thm:error-surf}
  Let $\utext \in L^2_\V$ be the solution of \cref{eq:weak-surf} which we assume satisfies the further regularity requirement
  \begin{equation}
    \label{eq:surf-u-reg}
    \sup_{t \in (0,T)} \norm{ \utext }_{\Z(t)}^2 + \int_0^T \norm{ \md \utext }_{\Z(t)}^2 \dd t
    \le C_u.
  \end{equation}
  Let $U_h \in C^1_{\S_h}$ be the solution of \cref{eq:surf-fem} and denote its lift by $u_h = U_h^\ell$. Then we have the following error estimate
  \begin{equation}
    \begin{aligned}
      & \sup_{t \in (0,T)} \norm{ \utext - u_h }_{\H(t)}^2
      + h^2 \int_0^T \norm{ \utext - u_h }_{\V(t)}^2 \dd t
       \le c \norm{ \utext - u_{h,0} }_{\H(t)}^2 + c h^{2k+2} C_u.
    \end{aligned}
  \end{equation}
\end{theorem}

\begin{proof}
  We simply check the assumptions of \cref{thm:error}.
  We know the lift is stable from \cref{lem:surf-lift-stab}.
  The existence and boundedness of $g_\ell$ and $b_\ell^\kappa$ are dealt with in \cref{lem:surf-lift-transport}.
  The interpolation properties \cref{eq:interp-Z0,eq:interp-Z} are shown in \cref{lem:surf-approx}.
  The geometric perturbation estimates \cref{eq:m-error,eq:c-error,eq:ct-error,eq:a-error,eq:b-error,eq:bt-error,eq:a-error2,eq:b-error2,eq:amd-error2,eq:md-error,eq:mdV-error} 
  are shown in the sequence of \cref{lem:surface-geom-pert,lem:surf-wh-estimate,lem:closed-md-error}.  Finally, to show \cref{eq:b-bound2} we follow a 
  calculation given in the proof of \citep[Thm.~6.2]{DziEll13}. In this setting the simpler version  \cref{eq:b-bound2prime}  holds; see  \cref{B3prime}. Observe  that for any $\eta \in \V(t)$, $\zeta \in \Z_0(t)$, that
  \begin{align*}
    b( t; \eta, \zeta )
    & = \int_{\Gamma(t)} \B( w, \A_\Gamma ) \nabla_\Gamma \eta \cdot \nabla_\Gamma \zeta \dd \sigma
      + \int_{\Gamma(t)} \B_{\adv}( w, \b_\Gamma ) \eta \cdot \nabla \zeta \dd \sigma \\
    & \qquad + \int_{\Gamma(t)} ( \md \c_\Gamma + \c_\Gamma \nabla_{\Gamma} \cdot w ) \eta \zeta \dd \sigma
    =: I_1 + I_2 + I_3.
  \end{align*}
  For $I_1$, using the additional smoothness assumptions on $\A_\Gamma$, we can apply integration by parts to see that
  \begin{align*}
    I_1
    & = \int_{\Gamma(t)} \B( w, \A_\Gamma ) \nabla_\Gamma \eta \cdot \nabla_\Gamma \zeta \dd \sigma \\
    & = \sum_{i,j=1}^{n+1} \int_{\Gamma(t)} \B( w, \A_\Gamma )_{ij} ( \nabla_{\Gamma} )_j \eta ( \nabla_\Gamma )_i \zeta \dd \sigma \\
    & = \sum_{i,j=1}^{n+1} \int_{\Gamma(t)} (\nabla_\Gamma)_j \big( \B( w, \A_\Gamma )_{ij} \eta ( \nabla_\Gamma )_i \zeta \big) \dd \sigma
      - \int_{\Gamma(t)} \eta \sum_{i,j=1}^{n+1} ( \nabla_\Gamma )_j \big( \B( w, \A_\Gamma )_{ij} ( \nabla_\Gamma )_i \zeta \big) \dd \sigma \\
    & = \int_{\Gamma(t)} \sum_{i,j=1}^{n+1} {H} \nu_j \B( w, \A_\Gamma )_{ij} \eta ( \nabla_\Gamma )_i \zeta \dd \sigma
      - \int_{\Gamma(t)} \eta \sum_{i,j=1}^{n+1} ( \nabla_\Gamma )_j \big( \B( w, \A_\Gamma )_{ij} ( \nabla_\Gamma )_i \zeta \big) \dd \sigma.
  \end{align*}
  The bounds for $I_2$ and $I_3$ are obvious. This implies that
  \begin{equation}
    \label{eq:surf-b-bound2}
    \abs{ b( t; \eta, \zeta ) }
    \le \abs{ I_1 } + \abs{ I_2 } + \abs{ I_3 }
    \le c \norm{ \eta }_{\H(t)} \norm{ \zeta }_{\Z_0(t)}.
    \qedhere
  \end{equation}
\end{proof}

\section{Application III: A coupled bulk-surface parabolic system}
\label{sec:application3}

In this section we will consider a finite element method for the coupled bulk-surface problem \cref{eq:bulksurf-eqn}.
The method is based on combining the isoparametric approaches from the problem in a bulk domain (\cref{BULKPDE}) and the problem on a surface (\cref{SURFPDE}).
The discretisation will be posed in the product of a bulk isoparametric finite element space of order $k$ and a surface 
isoparametric finite element space of order $k$.
We take our notation from the previous sections (\cref{SURFPDE,BULKPDE}).

\subsection{The domain and function spaces}

We set $\H(t) = L^2(\Omega(t)) \times L^2(\Gamma(t))$, $\V(t) = H^1(\Omega(t)) \times H^1(\Gamma(t))$ and $\V^*(t) = ( H^1(\Omega(t)) )^*  \times ( H^1(\Gamma(t)) )^*.$
We will also make use of the spaces $\Z_0(t) = H^2(\Omega(t)) \times H^2(\Gamma(t))$ and $\Z(t) = H^{k+1}(\Omega(t)) \times H^{k+1}(\Gamma(t))$.
We see that $Z(t) \subset \Z_0(t) \subset \V(t)$ for each $t \in [0,T]$ and the inclusions are uniformly continuous.
We define the push forward operator $\phi_t$ by
\begin{equation}
  \label{eq:coupled-pushforwards}
  ( \phi_t ( \underline{\eta} ) )( x, y )
  := ( \eta( \Phi_{-t}( x ) ), \xi( \Phi_{-t}( y ) ) )
  ~~~\mbox{~~for } \underline{\eta} = (\eta, \xi) \in \H_0, ~x \in \Omega(t), y \in \Gamma(t).
\end{equation}
$(\H(t), \phi_t )_{ t \in [0,T]}$ and $( \V(t), \phi_t )_{t \in [0,T]}$ are compatible pairs (\cref{def:compatibility}) and the spaces $L^2_\H, L^2_\V$ 
and $L^2_{\V^*}$ (c.f. \cref{eq:L2X}) and $C^1_\H, C^1_\V$ and $C^1_{\V^*}$ (c.f. \cref{eq:CkX}) are well defined.
For $\underline{\eta} = ( \eta, \xi ) \in C^1_\H$, we define the strong material derivative, denoted $\md \underline{\eta}$, 
using \cref{eq:strong-md}. The product material derivative coincides with the product of surface and bulk material derivatives:
\[
  \md \underline{\eta} = ( \md \eta, \md \xi ) \qquad \mbox{ for } \underline{\eta} = ( \eta, \xi ) \in C^1_\H.
\]
$(\V(t), \H(t), \V^*(t))_{t \in [0,T]}$ form an evolving Hilbert triple (\cref{def:evolving-hilbert-triple}).
We also have that $( \Z_0(t), \phi_t|_{\Z_0(t)} )_{t \in [0,T]}$ and $( \Z(t), \phi_t|_{\Z_0(t)} )_{t \in [0,T]}$ are compatible pairs.
For further information on this functional analytic setting see \citet[Sec.~5.3]{AlpEllSti15b-pp}.

{
  \begin{remark}
    \label{rem:coupled-smoothness}
    For the well posedness of the partial differential equation \cref{BulkSurfprob} we require that the boundary $\Gamma$ is 
    a $C^{2}$-hypersurface and that the flow map $\Phi_{(\cdot)}( \cdot ) \in C^2( 0,T; C^2( \Omega_0 ) )$. See \citet{AlpEllSti15b-pp} for more details.
    For the approximation properties we derive we require that $\Gamma$ is a $C^{k+2}$-hypersurface and that 
    $\Phi_{(\cdot)}( \cdot ) \in C^{k+2}( 0,T; C^{k+2}( \Omega_0) )$ with $\Phi_{t} \colon \Omega_0 \to \Omega(t)$ and $\Phi_{-t} \colon \Omega(t) \to \Omega$ both of class $C^{k+2}$.
  \end{remark}
}

{
  We introduce a signed distance function for the boundary surface $\Gamma(t) = \partial \Omega(t)$. The oriented signed distance function for $\Gamma(t)$ is given by
  \[
    d( x, t ) = \begin{cases}
      - \inf\{ \abs{ x - y } : y \in \Gamma(t) \} & \mbox{ for } x \in \bar\Omega(t) \\
      \inf\{ \abs{ x - y } : y \in \Gamma(t) \} & \mbox{ otherwise.}
    \end{cases}
  \]
  For each $t \in [0,T]$, we orient $\Gamma(t)$ by choosing the unit normal $\nu( x, t ) = \nabla d( x, t )$ for $x \in \Gamma(t)$. 
  Our assumptions on $\Gamma(t)$ imply that there exists a neighbourhood $\N(t)$ of $\Gamma(t)$ and normal projection 
  operator $p( \cdot, t ) \colon \N(t) \to \Gamma(t)$ given as the unique solution of
  \begin{equation}
    \label{eq:coupled-cp}
    x = p(x,t) + d( x, t ) \nu( p( x, t ), t ).
  \end{equation}
  See \citet[Lem. 14.16]{GilTru01}; \citet{Foote1984} for more details.
}

\subsection{The initial value problem}

Given $\A_\Omega, \b_\Omega, \c_\Omega$ as in \cref{sec:bulk-prob} and $\A_\Gamma, \b_\Gamma, \c_\Gamma$ as in \cref{sec:surf-prob},
we consider the initial value  problem
\begin{problem}
  \label{BulkSurfprob}
  Find $\emph \utext$ and $\emph \vtext$ such that
  \begin{subequations}
  \begin{align}
  \md \emph \utext +\nabla \cdot (\b_\Omega  \emph \utext)  -\nabla \cdot (\A_\Omega\nabla \emph \utext) +\c_\Omega \emph \utext  +(\nabla\cdot w)\utext&=0
  && \mbox{on} ~~\Omega(t)\\
  \md \emph \vtext +\nabla \cdot (\b_\Gamma \emph \vtext) - \nabla \cdot ( \A_\Gamma \nabla_\Gamma \emph \vtext ) +\c_\Gamma  \emph \vtext  +( \nabla_\Gamma \cdot w )\emph \vtext&=0
  && \mbox{on} ~~\Gamma(t)\\
 (-\b_\Omega  \emph \utext  +\A_\Omega\nabla_\Gamma \emph \utext)\cdot \nu + ( \alpha \emph \utext - \beta \emph \vtext ) &=0
  && \mbox{on} ~~\Gamma(t) \\
            \emph \utext(0) =\emph \utext_0 ~~ \mbox{on}~~\Omega_0, \qquad
                     \emph \vtext(0)&=\emph \vtext_0 ~~ \mbox{on}~~\Gamma_0.
  \end{align}
  \end{subequations}
\end{problem}
\begin{remark}
The problem  \cref{eq:bulksurf-eqn} is recovered by setting 
$\A_\Omega=\la_\Omega, \b_\Omega=\lb_\Omega-w$ and $\c_\Omega=\lc_\Omega$ and
$\A_\Gamma=\la_\Gamma, \b_\Gamma=\lb_\Gamma-w_\tau$ and $\c_\Gamma=\lc_\Gamma - \nabla_\Gamma \cdot w_\nu$, where $w = w_\tau + w_\nu$ is a decomposition of $w$ into tangential and normal components on $\Gamma(t)$.
\end{remark}

\subsection{The bilinear forms and transport formulae}

We set
  \begin{align*}
    m\big( t; (\eta, \xi), ( \zeta, \rho) \big)
    & = \alpha \int_{\Omega(t)} \eta \zeta \dd x
      + \beta \int_{\Gamma(t)} \xi \rho \dd \sigma
    && ( \eta, \xi ), ( \zeta, \rho ) \in \H(t) \\
    g\big( t; (\eta, \xi), ( \zeta, \rho) \big)
    & = \alpha \int_{\Omega(t)} \eta \zeta \nabla \cdot {w} \dd x
      + \beta \int_{\Gamma(t)} \xi \rho \nabla_\Gamma \cdot {w} \dd \sigma
    && ( \eta, \xi ), ( \zeta, \rho ) \in \H(t) \\
    a_s\big( t; (\eta, \xi), ( \zeta, \rho) \big)
    & = \alpha \int_{\Omega(t)} \A_\Omega \nabla \eta \cdot \nabla \zeta
      + \c_\Omega \eta \zeta \dd x \\
    & \quad + \beta \int_{\Gamma(t)} \A_\Gamma \nabla_\Gamma \xi \cdot \nabla \rho
      + \c_\Gamma \xi \rho \dd \sigma \\
    & \quad + \int_{\Gamma(t)} ( \alpha \eta - \beta \xi ) ( \alpha \zeta - \beta \rho ) \dd \sigma.
    && ( \eta, \xi ), ( \zeta, \rho ) \in \V(t) \\
    a_n\big( t; (\eta, \xi), ( \zeta, \rho) \big)
    & = \alpha \int_{\Omega(t)}\b_\Omega \eta \cdot \nabla \zeta \dd x + \beta \int_{\Gamma(t)} \b_\Gamma \xi \cdot \nabla_\Gamma \rho\dd \sigma
    && ( \eta, \xi ) \in H(t), \\
    &&& \quad ( \zeta, \rho ) \in \V(t).
  \end{align*}
  We can combine the transport formula for the bulk and surface only cases
  \cref{eq:bulk-m-transport} and \cref{eq:surface-m-transport} for $m$,
  \cref{eq:bulk-as-transport} and \cref{eq:bs-exist-surf} for $a_s$ and \cref{eq:bulk-an-transport} and \cref{eq:bn-exist-surf}
  to derive transport laws for these coupled bilinear forms
First, for $\underline{\eta}, \underline{\zeta} \in C^1_\H$ we have
\begin{align}
  \label{eq:coupled-m-transport}
  \dt m\big( t; \underline{\eta}, \underline{\zeta} \big)
  & = m\big( t; \md \underline{\eta}, \underline{\zeta} \big)
    + m\big( t; \underline{\eta}, \md \underline{\zeta} \big)
    + g\big( t; \underline{\eta}, \underline{\zeta} \big),
\end{align}
for $\underline{\eta}, \underline{\zeta} \in C^1_\V$, we have
\begin{align}
  \label{eq:coupled-as-transport}
  \dt a_s\big( t; \underline{\eta}, \underline{\zeta} \big)
  & = a_s\big( t; \md \underline{\eta}, \underline{\zeta} \big)
    + a_s\big( t; \underline{\eta}, \md \underline{\zeta} \big)
    + b_s\big( t; \underline{\eta}, \underline{\zeta} \big),
\end{align}
and for $\underline{\eta} \in C^1_\H, \underline{\zeta} \in C^1_\V$, we have
\begin{align}
  \label{eq:coupled-an-transport}
  \dt a_n\big( t; \underline{\eta}, \underline{\zeta} \big)
  & = a_n\big( t; \md \underline{\eta}, \underline{\zeta} \big)
    + a_n\big( t; \underline{\eta}, \md \underline{\zeta} \big)
    + b_n\big( t; \underline{\eta}, \underline{\zeta} \big),
\end{align}
where $b_s( t; \cdot, \cdot ) \colon \V(t) \times \V(t) \to \R$ is given for $( \eta, \xi ), ( \zeta, \rho ) \in \V(t)$ by
\begin{align*}
  & b_s\big( t; ( \eta, \xi ), ( \zeta, \rho ) \big) \\
  & = \alpha \int_{\Omega(t)} \B( {w}, \A_\Omega ) \nabla \eta \cdot \nabla \zeta
    + ( \md \c_\Omega + \c_\Omega \nabla \cdot {w} ) \eta \zeta \dd x \\
  & \quad + \beta \int_{\Gamma(t)} \B( {w}, \A_\Gamma ) \nabla_\Gamma \xi \cdot \nabla_\Gamma \rho
    + ( \md \c_\Gamma + \c_\Gamma \nabla_\Gamma \cdot {w} ) \xi \rho \dd \sigma \\
  & \qquad + \int_{\Gamma(t)} ( \alpha \eta - \beta \xi ) ( \alpha \zeta - \beta \rho )
    \nabla_\Gamma \cdot {w} \dd \sigma,
\end{align*}
and $b_n( t; \cdot, \cdot ) \colon \H(t) \times \V(t) \to \R$ is given for $( \eta, \xi ) \in \H(t), ( \zeta, \rho ) \in \V(t)$ by
\begin{align*}
  & b_n\big( t; ( \eta, \xi ), ( \zeta, \rho ) \big) \\
  & = \alpha \int_{\Omega(t)}
    \B_{\adv}( {w}, \b_\Omega ) \eta \cdot \nabla \zeta
  \dd x
  + \beta \int_{\Gamma(t)}
    \B_{\adv}( {w}, \b_\Gamma ) \xi \cdot \nabla_\Gamma \rho
  \dd \sigma.
\end{align*}
We define $b( t; \cdot, \cdot ) \colon \V(t) \times \V(t) \to \R$ to be
\begin{equation*}
  b( t; \underline{\eta}, \underline{\zeta} ) := b_s( t ; \underline{\eta}, \underline{\zeta} ) + b_n( t; \underline{\eta}, \underline{\zeta} ) \qquad \mbox{ for } \underline{\eta}, \underline{\zeta} \in \V(t).
\end{equation*}
We also have the estimates that there exists a constant $c > 0$ such that for all $t \in (0,T)$ we have
\begin{align}
  \label{eq:coupled-g-bounded}
  \abs{ g\big( t; \underline{\eta}, \underline{\zeta} \big) }
  & \le c \norm{ \underline{\eta} }_{\H(t)} \norm{ \underline{\zeta} }_{\H(t)}
  && \mbox{ for all } \underline{\eta}, \underline{\zeta} \in \H(t) \\
  \label{eq:coupled-bs-bounded}
  \abs{ b_s\big( t; \underline{\eta}, \underline{\zeta} \big) }
  & \le c \norm{ \underline{\eta} }_{\V(t)} \norm{ \underline{\zeta} }_{\V(t)}
  && \mbox{ for all } \underline{\eta}, \underline{\zeta} \in \V(t) \\
  \label{eq:coupled-bn-bounded}
  \abs{ b_n\big( t; \underline{\eta}, \underline{\zeta} \big) }
  & \le c \norm{ \underline{\eta} }_{\H(t)} \norm{ \underline{\zeta} }_{\V(t)}
  && \mbox{ for all } \underline{\eta} \in \H(t), \underline{\zeta} \in \V(t).
\end{align}

\subsection{Variational formulation}
We consider a weak form of \cref{eq:bulksurf-eqn}.

\begin{problem}
  \label{pb:coupled}
  Given $(\emph \utext_0, v_0) \in \H_0$, find $(\emph \utext,\emph \vtext) \in \W(\V,\V^*)$  such that for almost every  $t \in [0,T]$ we have
  \begin{equation}
    \label{eq:weak-coupled}
    \begin{aligned}
      m( t; \md (\emph \utext, \emph \vtext), \underline{\zeta} ) + g( t; ( \emph \utext, \emph \vtext ), \underline{\zeta} ) + a( t; ( \emph \utext, \emph \vtext ), \underline{\zeta} )
      & = 0
      && \mbox{ for all } \underline{\zeta} \in H^1(\Omega(t))\times H^1(\Gamma(t)) \\
      \emph \utext( 0 ) = \emph \utext_0, \qquad
      \emph \vtext( 0 ) & = \emph \vtext_0.
    \end{aligned}
  \end{equation}
  \end{problem}

\begin{theorem}
  \label{thm:coupled-exists}
  There exists a unique solution pair $( \utext,  \vtext)$ %
  which satisfies the stability bound:
  \begin{align}
   & \sup_{t\in[0,T]} \norm{ ( \utext,  \vtext) }_ {(L^2(\Omega(t)) \times L^2(\Gamma(t)))}^2 + \int_0^T \norm{ ( \utext,  \vtext) }_{ (H^1(\Omega(t)) \times H^1(\Gamma(t)))}^2 \dd t \\
    \nonumber
   & \hspace{7cm} \le c \norm{ (\utext_0,  \vtext_0) }_{(L^2(\Omega_0)\times L^2(\Gamma_0))}^2   \end{align}
   and if $( \utext,  \vtext)\in H^1(\Omega_0)\times H^1(\Gamma_0))$ then
   \begin{align}
    & \sup_{t\in[0,T]} \norm{ ( \utext,  \vtext) }_ {(H^1(\Omega(t)) \times H^1(\Gamma(t)))}^2 + \int_0^T \norm{ \md( \utext,  \vtext) }_{ (L^2(\Omega(t)) \times L^2(\Gamma(t)))}^2 \dd t \\
    \nonumber
   & \hspace{7cm} \le c \norm{ ( \utext_0,  \vtext_0) }_{(H^1(\Omega_0)\times H^1(\Gamma_0))}^2.
  \end{align}
\end{theorem}

\begin{proof}
  We again apply the abstract theory of \cref{thm:abs-exist} and check the assumptions.
  \cref{ass:evolving-space-equivalence,ass:u0-approx} are shown in \citet[Sec 5.3]{AlpEllSti15b-pp}.
  It is clear that \cref{eq:m-symmetric} and \cref{eq:m-bounded} hold since $m( t; \cdot, \cdot )$ is equal to the $\H(t)$-inner product.
  The assumptions \cref{eq:c-formula} and \cref{eq:c-bounded} are shown in \cref{eq:coupled-m-transport} and \cref{eq:coupled-g-bounded}.
  We know that the map $t \mapsto a( t; \cdot, \cdot )$ is differentiable hence measurable which shows \cref{eq:a-meas}.
  The coercivity \cref{eq:as-coercive} and boundedness \cref{eq:as-bounded} of $a_s$ and boundedness of $a_n$ \cref{eq:an-bounded} follow from standard arguments since the extra cross term is clearly positive (see also \citet[Thm.~3.2]{EllRan13}).
  The existence of the bilinear forms $b_s$ \cref{eq:bs-exist} and $b_n$ \cref{eq:bn-exist} has been shown in \cref{eq:coupled-as-transport} and \cref{eq:coupled-an-transport} and the estimates \cref{eq:bs-bound} and \cref{eq:bn-bound} are shown in \cref{eq:coupled-bs-bounded} and \cref{eq:coupled-bn-bounded}.
\end{proof}

\subsection{Discretisation of the domain and finite element spaces}

In order to define our computational method we use the construction of the isoparametric domain of order $k$ used in \cref{sec:fem-bulk}.
This defines an evolving triangulated bulk domain $\{ \Omega_h(t) \}_{t \in [0,T]}$ (\cref{def:evolving-triangulated-domain}) equipped with an evolving conforming subdivision $\{ \T_h(t) \}_{t \in [0,T]}$ (\cref{def:bulk-evolving-conforming-subdivision}).
We will assume that $\{ \T_h(t) \}_{t \in [0,T]}$ is a uniformly quasi-uniform subdivision (\cref{def:bulk-uniformly-regular}).
We will make the same assumptions on the domain as in \cref{sec:fem-bulk} which allow us to show \cref{prop:bulk-fe}.
Namely, we  conclude that we can define an evolving bulk finite element space $\{ \S_h^\Omega(t) \}$ (\cref{eq:bulk-Sh,def:evolving-bfe-space}) consisting of Lagrange finite elements of order $k$ (\cref{ex:isoparametric-bfe}) over a uniformly $k$-regular evolving subdivision (\cref{def:bulk-uniformly-Theta-regular}).
By $h$ we denote the maximum mesh diameter over time \cref{eq:bulk-def-h}:
\[
  h := \max_{t \in [0,T]} \max_{K(t) \in \T_h(t)} \diam( \tilde{K}(t) ).
\]
Throughout the remainder of this section we will denote global discrete quantities with a subscript $h \in (0,h_0)$. We assume implicitly that these structures exist for each $h$ in this range (see also \cref{rem:small-hk-bulk,rem:small-hk-surf}).

For $t \in [0,T]$, we write $\Gamma_h(t) = \partial\Omega_h(t)$ and $\T_h^\Gamma(t)$ for the boundary faces of $\T_h(t)$:
\begin{equation*}
  \T_h^\Gamma(t) := \left\{
    K(t) \cap \Gamma_h(t) : K(t) \in \T_h(t)
    \right\}.
\end{equation*}
We note that $\{ \T_h^\Gamma(t) \}_{t \in [0,T]}$ is also an evolving conforming subdivision (\cref{def:evolving-conforming-subdivision}) which we assume is uniformly quasi-uniform (\cref{def:uniformly-quasi-uniform}).
In fact, this is the construction we have previously used for an evolving surface in \cref{sec:application1} so that we may make the same conclusions as \cref{prop:surf-fe}.
That is that we can define an evolving surface finite element space $\{ \S_h^\Gamma(t) \}$ (\cref{eq:surf-Sh,def:evolving-sfe-space}) consisting of Lagrange finite elements of order $k$ (\cref{ex:isoparametric-sfe}) over a uniformly $k$-regular, evolving subdivision $\{ \T_h^\Gamma(t) \}_{t \in [0,T]}$ (\cref{def:uniformly-Theta-regular}).

Our analysis will make use of the product Hilbert spaces $\H_h(t) := L^2( \Omega_h(t) ) \times L^2( \Gamma_h(t) )$ and $\V_h(t) := H^1_T( \T_h(t) ) \times H^1_T( \T_h^\Gamma(t) )$ and the product finite element space $\S_h(t) := \S_h^\Omega(t) \times \S_h^\Gamma(t)$.
Using \cref{lem:bulk-Sh-closed,lem:Sh-closed-subspace} we can identify elements of $\S_h(t)$ as a product of continuous functions on $\Omega(t)$ and $\Gamma(t)$ and that $\S_h(t) \subset \V_h(t)$.
We equip $\S_h(t)$ with the norms:
\begin{align*}
  \norm{ ( \chi_h, \myvarrho_h ) }_{\V_h(t)}
  & := \left( \norm{ \chi_h }_{H^1(\T_h(t))}^2 + \norm{ \myvarrho_h }_{H^1(\T_h^\Gamma(t))}^2 \right)^{1/2} \\
  \norm{ ( \chi_h, \myvarrho_h ) }_{\H_h(t)}
  & := \left( \norm{ \chi_h }_{L^2(\Omega_h(t))}^2 + \norm{ \myvarrho_h }_{L^2(\Gamma_h(t))}^2 \right)^{1/2}.
\end{align*}

The previous constructions define a global flow map $\Phi_t^h \colon \Omega_{h,0} \to \Omega_h(t)$ (\cref{def:bulk-global-discrete-flow}) and discrete velocity $W_h$ (\cref{def:bulk-global-discrete-velocity}), with well defined surface restrictions $\Phi_t^h|_{\Gamma_{h,0}}$ (\cref{def:global-discrete-flow}) and $W_h|_{\Gamma_{h}(t)}$ (\cref{def:global-discrete-velocity}).
For each $t \in [0,T]$, we define the discrete push forward map $\phi^h_t \colon \H_h(0) \to \H_h(t)$ by
\[
  ( \phi_t( \eta_h, \xi_h ) )( x, y ) :=
  ( \eta_h( \Phi_{-t}^h( x) ), \xi_h( \Phi_{-t}^h(y) ) )
  \qquad
  \mbox{ for } ( \eta_h, \xi_h ) \in \H_h(0), x \in \Omega_h(t), y \in \Gamma_h(t).
\]
Since we have shown that $\{ \T_h \}_{t \in [0,T]}$ and $\{ \T_h^\Gamma(t) \}_{t \in [0,T]}$ are both uniformly $k$-regular and uniformly quasi-uniform, the spaces $( \H_h(t), \phi^h_t|_{\H_{h,0}} )_{t \in [0,T]}$, $( \V_h(t), \phi^h_t|_{\V_{h,0}} )_{t \in [0,T]}$ and $( \S_h(t), \phi^h_t|_{\S_{h,0}} )_{t \in [0,T]}$ form a compatible pair (\cref{def:compatibility}).
Further, we can define the spaces $L^2_{\H_h}$ \cref{eq:L2X} and $C^1_{\H_h}$ \cref{eq:CkX} and we can define a material derivative for functions $\underline{\eta_h} = (\eta_h, \xi_h) \in C^1_{\H_h}$ which can be identified as \cref{eq:abs-mdh}: %
\begin{equation*}
  \mdh \underline{\eta_h} = ( \mdh \eta_h, \mdh \xi_h ).
\end{equation*}

\subsection{Construction of lifted finite element space}

We have already constructed a bijection between the computational domain $\Omega_h(t)$ and the continuous domain $\bar\Omega(t)$.
In \cref{sec:bulk-lift-prop}, for each $t \in [0,T]$, we constructed element-wise a bijection $\Lambda_h( \cdot, t ) \colon \Omega_h(t) \to \bar\Omega(t)$.
Furthermore, we note that the restriction of the lifting operator to $\Gamma_h(t)$, $\Lambda_h( \cdot, t )|_{\Gamma_h(t)}$, is simply the normal projection operator which is the lifting operator used in \cref{sec:surf-lift-prop}.

For each bulk finite element $K(t) \in \T_h(t)$, we can use \cref{def:lifted-bfe} to construct an associated lifted bulk 
finite element $(K^\ell(t), P^\ell(t), \Sigma^\ell(t))$. We will denote the set of lifted bulk finite elements by $\T_h^\ell(t)$.
For each surface finite element $E(t) \in \T_h^\Gamma(t)$, we can use \cref{def:lifted-sfe} to construct 
an associated lifted surface finite element $(E^\ell(t), P_E^\ell(t), \Sigma_E^\ell(t))$. We will denote the set of lifted 
surface finite elements by $\T_h^{\Gamma,\ell}(t)$.

For $t \in [0,T]$ and a function pair $\underline{\eta_h} = ( \eta_h, \xi_h ) \colon \Omega_h(t) \times \Gamma_h(t) \to \R^2$, 
we define the lift $\underline{\eta_h}^\ell = ( \eta_h, \xi_h )^\ell \colon \Omega(t) \times \Gamma(t) \to \R^2$ by
\begin{equation*}
  \underline{\eta_h}^\ell( \Lambda_h( x, t ), p( y, t ) )
  := \big( \eta_h( x ), \xi_h( y ) \big)
  \qquad \mbox{ for } x \in \Omega_h(t), y \in \Gamma_h(t).
\end{equation*}
We will often write $\underline{\eta_h}^\ell = ( \eta_h, \xi_h )^{\ell} = ( \eta_h^\ell, \xi_h^\ell )$ to signify that 
the lifting process is simply a combination the previous lifts for the surface and bulk components.

We will also make use of an inverse lift for functions on $\Omega(t) \times \Gamma(t)$. 
For $\underline{\eta} = ( \eta, \xi ) \colon \Omega_h \times \Gamma(t) \to \R^2$, we define 
he inverse lift of $\underline{\eta}$, denoted by $\underline{\eta}^{-\ell} = ( \eta, \xi )^{-\ell} \colon \Omega_h(t) \times \Gamma_h(t) \to \R^2$ by
\begin{equation*}
  \underline{\eta}^{-\ell}( x, y )
  = \big( \eta( \Lambda_h( x,t ) ), \xi_h( p(x,t) ) \big)
\qquad \mbox{ for } x \in \Omega_h(t), y \in \Gamma_h(t).
\end{equation*}

\begin{lemma}
  \label{lem:coupled-lift-stab}
  Let $\underline{\eta_h} \in \H_h(t)$ and denote their lift by $\underline{\eta_h}^\ell$. Then there exists constants $c_1, c_2 > 0$ such that
  \begin{align}
    c_1 \norm{ \underline{\eta_h}^\ell }_{\H(t)}
    & \le \norm{ \underline{\eta_h} }_{\H_h(t)}
    \le c_2 \norm{ \underline{\eta_h}^\ell }_{\H(t)} \\
    c_1 \norm{ \underline{\eta_h}^\ell }_{\V(t)}
    & \le \norm{ \underline{\eta_h} }_{\V_h(t)}
    \le c_2 \norm{ \underline{\eta_h}^\ell }_{\V(t)}.
  \end{align}
\end{lemma}

\begin{proof}
  We simply combine the results of \cref{lem:bulk-lift-stab} and \cref{lem:surf-lift-stab}.
\end{proof}

We use the lifts to define a product space of lifted finite element functions $\S_h^\ell(t)$ (\cref{def:lifted-bfe-space,def:lifted-sfe-space}) by
\[
  \S_h^\ell(t) := \{ \underline{\chi_h}^\ell : \underline{\chi_h} \in \S_h \}.
\]

Using the lift $\Lambda_h$ and the discrete flow $\Phi^h_t$, we can defined a lifted flow map 
$\Phi^\ell_t \colon \bar{\Omega}(0) \to \bar{\Omega}(t)$ (\cref{def:bulk-lifted-flow-map}) and lifted discrete velocity 
$w_h$ (\cref{def:bulk-lifted-discrete-material-velocity}), with well defined restrictions 
$\Phi^\ell_t|_{\Gamma(0)}$ (\cref{def:lifted-flow-map}) and $w_h|_{\Gamma(t)}$ (\cref{def:lifted-discrete-material-velocity}).
This allows us to define a lifted push forward map (\cref{eq:phiellt,def:lifted-push-forward-map,def:bulk-lifted-push-forward-map}):
\[
  \phi^\ell_t( ( \eta, \xi ) )( x, y )
  := \bigr( \eta( \Phi_{-t}^\ell( x ) ), \xi( \Phi_{-t}^\ell( y ) ) \bigl)
  \qquad \mbox{ for } ( \eta, \xi ) \in \H_0, x \in \bar\Omega(t), y \in \Gamma(t).
\]
Using the previous constructions, applying \cref{lem:surf-lift-bounds,lem:bulk-Thell-good},
we see that $(\H(t), \phi^\ell_t)_{t \in [0,T]}$, $( \V(t), \phi_t^\ell|_{\V(t)} )_{t \in [0,T]}$ and $( \S_h^\ell(t), \phi_t^\ell|_{\S_h^\ell(t)} )_{t \in [0,T]}$  are compatible pairs uniformly for $h \in (0,h_0)$.
We will use the notations $C^1_{(\H,\phi^\ell)}$ and $C^1_{(\V,\phi^\ell)}$ for the spaces of functions smoothly evolving in time \cref{eq:CkX} with respect to the push-forward map $\phi^\ell_t$ in $\H(t)$ and $\V(t)$ respectively.
This implies we have a well defined strong material derivative \cref{eq:lift-md}:
\[
  \mdell \underline{\eta} = \phi_t^\ell \left( \dt \left( \phi_{-t}^\ell \underline{\eta} \right) \right) \qquad \mbox{ for } \underline{\eta} \in C^1_{(\H,\phi^\ell)}.
\]

\subsection{The discrete problem  and stability}

The finite element method is based on the variation form \cref{eq:abs-var-form} of \cref{pb:coupled}.
We assume we have $\A^K_\Omega, \b^K_\Omega$ and $\c^K_\Omega$ for $K \in \T_h(t)$ as 
in \cref{sec:bulk-discrete-pr} and $\A^E_\Gamma, \b^E_\Gamma$ and $\c^E_\Gamma$ for $E \in \T_h^\Gamma(t)$ as in \cref{sec:surf-discrete-pr}.

\begin{problem}
  \label{pb:fem-coupled}
  Given $(U_{h,0}, V_{h,0} ) \in \S_{h,0}$, find $(U_h, V_h) \in C^1_{\S_h}$ such that for all $t \in [0,T]$
  \begin{equation}
    \label{eq:fem-coupled}
    \begin{aligned}
      \dt m_h\big( t; ( U_h, V_h ), \underline{\chi_h} \big)
      + a_h\big( t; ( U_h, V_h ), \underline{\chi_h} \big)
      & = m_h\big( t; ( U_h, V_h ), \mdh \underline{\chi_h} \big) \\
      & \qquad\qquad \mbox{ for all } \underline{\chi_h} \in C^1_{\S_h} \\
      U_h( 0 ) = U_{h,0},
      V_h( 0 ) & = V_{h,0},
    \end{aligned}
  \end{equation}
  where for $( \eta_h, \xi_h ), ( \zeta_h, \rho_h ) \in \H_h(t)$, we define
  \begin{align*}
    m_h\big( t; ( \eta_h, \xi_h ), ( \zeta_h, \rho_h ) \big)
    & = \alpha \int_{\Omega_h(t)} \eta_h \zeta_h \dd x
    + \beta \int_{\Gamma_h(t)} \xi_h \rho_h \dd \sigma_h,
  \end{align*}
  and, for $( \eta_h, \xi_h ), ( \zeta_h, \rho_h ) \in \V_h(t)$, we define
  \begin{align*}
    a_h\big( t; ( \eta_h, \xi_h ), ( \zeta_h, \rho_h ) \big)
    & = \alpha \sum_{K(t) \in \T_h(t)} \int_{K(t)}
      \A_\Omega^K \nabla \eta_h \cdot \nabla \zeta_h
      + \b_\Omega^K \eta_h \cdot \nabla \zeta_h
      + \c_\Omega^K \eta \zeta_h \dd x \\
    & \quad + \beta \sum_{E(t) \in \T_h^\Gamma(t)} \int_{E(t)}
      \A_\Gamma^E \nabla_E \xi_h \cdot \nabla \rho_h
      + \b_\Gamma^E \xi_h \cdot \nabla_E \rho_h
      + \c_\Gamma^E \xi_h \rho_h \dd \sigma_h \\
    & \quad + \int_{\Gamma_h} ( \alpha \eta_h - \beta \xi_h ) ( \alpha \zeta_h - \beta \rho_h ) \dd \sigma_h.
  \end{align*}
\end{problem}

We also have discrete transport formula from the bulk and surface cases:
\begin{lemma}
  \label{lem:coupled-fem-transport}
  There exists bilinear forms 
  $g_h( t; \cdot, \cdot ) \colon \H_h(t) \times \H_h(t) \to \R$ and $b_h( t; \cdot, \cdot ) \colon \V_h(t) \times \V_h(t) \to \R$ 
  such that for all $\underline{\eta_h} = ( \eta_h, \xi_h ), \underline{\zeta_h} = ( \zeta_h, \rho_h ) \in \C^1_{\H_h}$ we have
  \[
    \dt m_h\big( t; \underline{\eta_h}, \underline{\zeta_h} \big)
    = m_h\big( t; \mdh \underline{\eta_h}, \underline{\zeta_h} \big)
      + m_h\big( t; \underline{\eta_h}, \mdh \underline{\zeta_h} \big)
    + g_h\big( t; \underline{\eta_h}, \underline{\zeta_h} \big),
  \]
  and for all $\underline{\eta_h}, \underline{\zeta_h} \in \C^1_{\V_h}$ we have
  \[
    \dt a_h\big( t; \underline{\eta_h}, \underline{\zeta_h} \big)
    = a_h\big( t; \mdh \underline{\eta_h}, \underline{\zeta_h} \big)
      + a_h\big( t; \underline{\eta_h}, \mdh \underline{\zeta_h} \big)
    + b_h\big( t; \underline{\eta_h}, \underline{\zeta_h} \big),
  \]
  where
  \begin{equation*}
    g_h\big( t; ( \eta_h, \xi_h ), ( \zeta_h, \rho_h ) \big)
    = \alpha \int_{\Omega_h(t)} \eta_h \zeta_h \nabla \cdot {W}_h \dd x
    + \beta \int_{\Gamma_h(t)} \xi_h \rho_h \nabla_{\Gamma_h} \cdot {W}_h \dd \sigma_h,
  \end{equation*}
  and
  \begin{align*}
    & b_h\big( t; ( \eta_h, \xi_h ), ( \zeta_h, \rho_h ) \big) \\
    & = \sum_{K(t) \in \T_h(t)}
    \int_{K(t)} \B_h( {W}_h, \A_\Omega^K ) \nabla \eta_h \cdot \nabla \zeta_h
    + \B_{\adv,h}( {W}_h, \b_\Omega^K ) \eta_h \cdot \nabla \zeta_h \\
    & \qquad\qquad\qquad\qquad
      + ( \mdh \c_\Omega^K + \c_\Omega^K \nabla \cdot {W}_h ) \eta_h \zeta_h \dd x \\
    & + \sum_{E(t) \in \T_h^\Gamma(t)}
      \int_{E(t)} \B_h( {W}_h, \A_\Gamma^K ) \nabla_E \xi_h \cdot \nabla_E \rho_h
      + \B_{\adv,h}( {W}_h, \b_\Gamma^K ) \xi_h \cdot \nabla_E \rho_h \\
    & \qquad\qquad\qquad\qquad
      + ( \mdh \c_\Gamma^K + \c_\Gamma^K \nabla_E \cdot {W}_h ) \xi_h \rho_h \dd \sigma_h.
  \end{align*}
  Further, there exists a constant $c > 0$ such that, for all $t \in [0,T]$,
  \begin{subequations}
    \begin{align}
      \abs{ g_h( t; \underline{\eta_h}, \underline{\zeta_h} ) }
      & \le c \norm{ \underline{\eta_h} }_{\H_h(t)} \norm{ \underline{\zeta_h} }_{\H_h(t)}
      && \mbox{ for all } \underline{\eta_h}, \underline{\zeta_h} \in \H_h(t) \\
      \abs{ b_h( t; \underline{\eta_h}, \underline{\zeta_h} ) }
      & \le c \norm{ \underline{\eta_h} }_{\V_h(t)} \norm{ \underline{\zeta_h} }_{\V_h(t)}
      && \mbox{ for all } \underline{\eta_h}, \underline{\zeta_h} \in \V_h(t).
    \end{align}
  \end{subequations}
\end{lemma}

\begin{proof}
  We simply combine \cref{lem:surf-transport-bounds,lem:bulk-fem-transport}.
\end{proof}

\begin{theorem}
  \label{thm:coupled-fem-stability}
  There exists a unique solution pair $(U_h, V_h)$ of the finite element scheme (\cref{pb:fem-coupled}) which satisfies the stability bound
  \begin{equation}
    \sup_{t \in (0,T)} \norm{ (U_h, V_h) }_{\H_h(t)}^2 + \int_0^T \norm{ (U_h, V_h) }_{\V_h(t)}^2 \dd t
    \le c \norm{ (U_{h,0}, V_{h,0}) }_{\H_h(t)}^2.
  \end{equation}
\end{theorem}

\begin{proof}
  We apply the abstract result of \cref{thm:abs-stab} and check the assumptions.
    The assumptions on $m_h$, \cref{eq:mh-symmetric} and \cref{eq:mh-bounded}, follow directly since $m_h$ is equal to the $\H_h(t)$ inner-product.
  The estimates on $a_h$, \cref{eq:ah-coercive} and \cref{eq:ah-bounded} follow in the same manner as \cref{thm:coupled-exists}.
  The transport formulae and estimates for $g_h$ and $b_h$, \cref{eq:ch-formula}, 
  \cref{eq:ch-bounded} \cref{eq:bh-formula} and \cref{eq:bh-bounded}, are shown in \cref{lem:coupled-fem-transport}.
\end{proof}

\subsection{Error analysis}

We assume in this section that $\A_\Omega^h = \A_\Omega^{-\ell}$, $\B_\Omega^h = \B_\Omega^{-\ell}$, $\C_\Omega^h = \C_{\Omega}^{-\ell}$,
$\A_\Gamma^h = \A_\Gamma^{-\ell}$, $\B_\Gamma^h = \B_\Gamma^{-\ell}$, $\C_\Gamma^h = \C_{\Gamma}^{-\ell}$.
The space of lifted finite element functions $\S_h^\ell(t)$
is equipped with the follow approximation property:

\begin{lemma}
  [Approximation property]
  \label{lem:coupled-approx}
  For $\underline{\eta} = ( \eta, \xi ) \in C( \Omega(t) ) \times C( \Gamma(t) )$ the Lagrangian interpolation 
  operator $I_h \underline{\eta}$ is well defined. Furthermore, the following bounds hold for a constant $c > 0$ for all $h \in (0,h_0)$ and $t \in [0,T]$:
  \begin{align}
    \label{eq:coupled-approx}
    \norm{ \underline{\eta} - I_h \underline{\eta} }_{\H(t)}
    + h \norm{ \underline{\eta} - I_h \underline{\eta} }_{\V(t)}
    & \le c h^{k+1} \norm{ \underline{\eta} }_{\Z(t)}
    && \mbox{ for } \underline{\eta} \in \Z(t) \\
    \label{eq:coupled-approx-0}
    \norm{ \underline{\eta} - I_h \underline{\eta} }_{\H(t)}
    + h \norm{ \underline{\eta} - I_h \underline{\eta} }_{\V(t)}
    & \le c h^{2} \norm{ \underline{\eta} }_{\Z_0(t)}
    && \mbox{ for } \underline{\eta} \in \Z_0(t).
  \end{align}
\end{lemma}

\begin{proof}
  We define the interpolation operator to be $I_h( \eta, \xi ) = ( I_h \eta, I_h \xi)$ for $( \eta, \xi ) \in C( \Omega(t) \times C( \Gamma(t)) )$.
  The proof follows by combining the result of \cref{lem:surf-approx} and \cref{lem:bulk-approx}.
\end{proof}

\begin{lemma}
  \label{lem:coupled-lift-transport}
  There exists a bilinear forms $g_\ell( t; \cdot, \cdot ) \colon \H(t) \times \H(t) \to \R$ and $b_\ell( t; \cdot, \cdot ) \colon \V(t) \times \V(t) \to \R$ given by
  \begin{align*}
    g_\ell\big( t; ( \eta, \xi ), ( \zeta, \rho) \big)
    & := \alpha \int_{\Omega(t)} \eta \zeta \nabla \cdot w_h \dd x
      + \beta \int_{\Gamma(t)} \xi \rho \nabla_\Gamma \cdot w_h \dd \sigma \\
    \tilde{b_h}\big( t; ( \eta, \xi ), (\zeta, \rho) \big)
    & := \alpha \int_{\Omega(t)} \B( w_h, \A_\Omega ) \nabla \eta \cdot \nabla \zeta
      + \B_{\adv}( w_h, \b_\Omega ) \eta \cdot \nabla \zeta \\
    & \quad\qquad\qquad + ( \mdell \c_\Omega + \c_\Omega \nabla \cdot w_h ) \eta \zeta \dd x \\
    & \quad + \beta \int_{\Gamma(t)} \B( w_h, \A_\Gamma ) \nabla_\Gamma \xi \cdot \nabla_\Gamma \rho
      + \B_{\adv}( w_h, \b_\Gamma ) \xi \cdot \nabla_\Gamma \rho \\
    & \quad\qquad\qquad
      + ( \mdell c_\Gamma + \c_\Gamma \nabla_\Gamma \cdot w_h ) \xi \rho \dd \sigma \\
    & \quad + \int_{\Gamma(t)} ( \alpha \eta - \beta \xi )( \alpha \zeta - \beta \rho ) \nabla_\Gamma \cdot w_h \dd \sigma,
  \end{align*}
  such that
  \begin{align}
    \dt m( t; \underline{\eta}, \underline{\zeta} )
    & = m( t; \mdell \underline{\eta}, \underline{\zeta} )
    + m( t; \underline{\eta}, \mdell \underline{\zeta} )
    + g_\ell( t; \eta, \underline{\zeta} ) \\
    \nonumber
    & \hspace{5cm} \mbox{ for } \underline{\eta}, \underline{\zeta} \in C^1_{(\H,\phi^\ell)} \\
    \dt a( t; \underline{\eta}, \underline{\zeta} )
    & = a( t; \mdell \underline{\eta}, \underline{\zeta} )
    + a( t; \underline{\eta}, \mdell \underline{\zeta} )
    + b_\ell( t; \eta, \underline{\zeta} ) \\
    \nonumber
    & \hspace{5cm} \mbox{ for } \underline{\eta}, \underline{\zeta} \in C^1_{(\V,\phi^\ell)}.
  \end{align}
  Furthermore, there exists a constant $c > 0$ such that for all $t \in [0,T]$ and all $h \in (0,h_0)$ we have
  \begin{align}
    \abs{ g_\ell( t; \underline{\eta}, \underline{\zeta} ) }
    & \le c \norm{ \underline{\eta} }_{\H(t)} \norm{ \underline{\zeta} }_{\H(t)}
    && \mbox{ for all } \underline{\eta}, \underline{\zeta} \in \H(t) \\
    \abs{ b_\ell( t; \underline{\eta}, \underline{\zeta} ) }
    & \le c \norm{ \underline{\eta} }_{\V(t)} \norm{ \underline{\zeta} }_{\V(t)}
    && \mbox{ for all } \underline{\eta}, \underline{\zeta} \in \V(t).
  \end{align}
\end{lemma}

\begin{proof}
  We combine the results of \cref{lem:surf-lift-transport,lem:bulk-lift-transport}.
\end{proof}

The geometric perturbation results now follow directly by combining the appropriate results from \cref{sec:fem-surf-error,sec:bulk-error}.

\begin{lemma}
  \label{lem:coupled-w-md-errors}
  We have the estimates
  \begin{align}
    \abs{ {w} - {w}_h }_{L^\infty(\Omega(t))}
    + h \abs{ \nabla ({w} - {w}_h) }_{L^\infty(\Omega(t))}
    & \le ch^{k+1} \\
    \abs{ {w} - {w}_h }_{L^\infty(\Gamma(t))}
    + h \abs{ \nabla_\Gamma ({w} - {w}_h) }_{L^\infty(\Gamma(t))} & \le c h^{k+1}.
  \end{align}
  In particular, this implies
  \begin{subequations}
    \begin{align}
      \norm{ \md \underline{\eta} - \mdell \underline{\eta} }_{\H(t)}
      & \le c h^{k+1} \norm{ \underline{\eta} }_{\V(t)}
      && \mbox{ for } \underline{\eta} \in C^1_{\V} \\
      \norm{ \md \underline{\eta} - \mdell \underline{\eta} }_{\V(t)}
      & \le c h^{k} \norm{ \underline{\eta} }_{\Z_0(t)}
      && \mbox{ for } \underline{\eta} \in C^1_{\Z_0}.
    \end{align}
  \end{subequations}
\end{lemma}

\begin{proof}
  We combine the results of \cref{lem:surf-wh-estimate,lem:closed-md-error,lem:bulk-w-md-errors}.
\end{proof}

\begin{lemma}
  \label{lem:coupled-geom-pert}
  There exists a constant $c > 0$ such that for all $t \in [0,T]$ and all $h \in (0,h_0)$ the following holds for all $\underline{\eta_h}, \underline{\zeta_h} \in \V_h(t)$ with lifts $\underline{\eta_h}^\ell, \underline{\zeta_h}^\ell \in \V(t)$:
  \begin{align}
    \label{eq:m-error-coupled}
    \abs{ m( t; \underline{\eta_h}^\ell, \underline{\zeta_h}^\ell ) - m_h( t; \underline{\eta_h}, \underline{\zeta_h} ) }
    & \le c h^{k+1} \norm{ \underline{\eta_h}^\ell }_{\V(t)} \norm{ \underline{\zeta_h}^\ell }_{\V(t)} \\
    \label{eq:g-error-coupled}
    \abs{ g_\ell( t; \underline{\eta_h}^\ell, \underline{\zeta_h}^\ell ) - g_h( t; \underline{\eta_h}, \underline{\zeta_h} ) }
    & \le c h^{k+1} \norm{ \underline{\eta_h}^\ell }_{\V(t)} \norm{ \underline{\zeta_h}^\ell }_{\V(t)} \\
    \label{eq:a-error-coupled}
    \abs{ a( t; \underline{\eta_h}^\ell, \underline{\zeta_h}^\ell ) - a_{h}( t; \underline{\eta_h}, \underline{\zeta_h} ) }
    & \le c h^{k} \norm{ \underline{\eta_h}^\ell }_{\V(t)} \norm{ \underline{\zeta_h}^\ell }_{\V(t)} \\
    \label{eq:b-error-coupled}
    \abs{ b_\ell( t; \underline{\eta_h}^\ell, \underline{\zeta_h}^\ell ) - b_{h}( t; \underline{\eta_h}, \underline{\zeta_h} ) }
    & \le c h^{k} \norm{ \underline{\eta_h}^\ell }_{\V(t)} \norm{ \underline{\zeta_h}^\ell }_{\V(t)} \\
    \label{eq:gt-error-coupled}
    \abs{ g_\ell( t; \underline{\eta_h}^\ell, \underline{\zeta_h}^\ell ) - g( t; \underline{\eta_h}^\ell, \underline{\zeta_h} ) }
    & \le c h^{k} \norm{ \underline{\eta_h}^\ell }_{\V(t)} \norm{ \underline{\zeta_h}^\ell }_{\V(t)} \\
    \label{eq:bt-error-coupled}
    \abs{ b_\ell( t; \underline{\eta_h}^\ell, \underline{\zeta_h}^\ell ) - b( t; \underline{\eta_h}^\ell, \underline{\zeta_h}^\ell ) }
    & \le c h^{k} \norm{ \underline{\eta_h}^\ell }_{\V(t)} \norm{ \underline{\zeta_h}^\ell }_{\V(t)}.
  \end{align}
  For $\underline{\eta}, \underline{\zeta} \in \Z_0(t)$ with inverse lifts $\underline{\eta}^{-\ell}, \underline{\zeta}^{-\ell}$, we have
  \begin{align}
    \label{eq:a-error2-coupled}
    \abs{ a( t; \underline{\eta}, \underline{\zeta} ) - a_{h}( t; \underline{\eta}^{-\ell}, \underline{\zeta}^{-\ell} ) }
    & \le c h^{k+1} \norm{ \underline{\eta} }_{\Z_0(t)} \norm{ \underline{\zeta} }_{\Z_0(t)} \\
    \label{eq:b-error2-coupled}
    \abs{ b_\ell( t; \underline{\eta}, \underline{\zeta} ) - b_{h}( t; \underline{\eta}^{-\ell}, \underline{\zeta}^{-\ell} ) }
    & \le c h^{k+1} \norm{ \underline{\eta} }_{\Z_0(t)} \norm{ \underline{\zeta} }_{\Z_0(t)} \\
    \label{eq:amd-error-coupled}
    \abs{ a( t; \mdell \underline{\eta}, \underline{\zeta} ) - a_{h}( t; \mdh \underline{\eta}^{-\ell}, \underline{\zeta}^{-\ell} ) }
    & \le c h^{k+1} ( \norm{ \underline{\eta} }_{\Z_0(t)} + \norm{ \md \underline{\eta} }_{\Z_0(t)} ) \norm{ \underline{\zeta} }_{\Z_0(t)}
  \end{align}
\end{lemma}

\begin{proof}
  We combine the results of \cref{lem:surface-geom-pert,lem:bulk-geom-pert}.
\end{proof}

We require one final result in order to show the error bound.

\begin{lemma}
  \label{lem:coupled-b-bound2}
  For any $t \in [0,T]$, let $\underline{z}, {\pi_h} \underline{z}$ be as in \cref{eq:ritz} and $\underline{\eta} = (\eta, \xi) = \underline{z} - {\pi _h} \underline{z} \in \V(t)$, then for all $\underline{\zeta} = ( \zeta, \rho ) \in \Z_0(t)$, we have
  \begin{equation}
    \label{eq:coupled-b-bound2}
    \abs{ b\big( \underline{\eta}, \underline{\zeta} \big) }
    \le c \left( \norm{ \underline{\eta} }_{\H(t)} + h \norm{ \underline{\eta} }_{\V(t)} + h^{k+1} \norm{ \underline{z} }_{\Z_0(t)} \right)
    \norm{ \underline{\zeta} }_{\Z_0(t)}.
  \end{equation}
\end{lemma}

\begin{proof}
  The proof follows a similar path to \cref{lem:bulk-b-bound2}. We start by fixing $t \in [0,T]$.
  We recall that for any $\eta \in H^1(\Omega(t))$, there exists $\xi_\eta \in H^{1/2}(\Gamma(t))$ such that
  \[
    {}_{(H^{1/2}(\Gamma(t)))'} \langle \xi_\eta, \zeta \rangle_{H^{1/2}(\Gamma(t))}
    = ( \eta, \zeta )_{L^2( \Gamma(t)) },
  \]
  and $\xi_\eta$ satisfies
  \begin{subequations}
    \begin{align}
      \label{eq:coupled-new-dual-1}
      \norm{ \xi_\eta }_{H^{1/2}(\Gamma(t))}
      & = \norm{ \eta }_{(H^{1/2}(\Gamma(t)))'} \\
      \label{eq:coupled-new-dual-2}
      \int_{\Gamma(t)} \xi_\eta \eta \dd \sigma
      & = \norm{ \eta }_{(H^{1/2}(\Gamma(t)))'}.
    \end{align}
  \end{subequations}

  As in \cref{sec:ritz}, we introduce $\kappa > 0$ such that there exists a constant $c > 0$ such that
  \[
    a^\kappa( t; \underline{\zeta}, \underline{\zeta} ) := a( t; \underline{\zeta}, \underline{\zeta} ) + \kappa m( t; \underline{\zeta}, \underline{\zeta} ) \ge \norm{ \underline{\zeta} }_{\V(t)}^2
    \mbox{ for all } \underline{\zeta} \in \V(t).
  \]

  We wish to estimate the trace of $\eta$, the bulk component of $\underline{z} - {\pi_h} \underline{z}$, in 
  the $(H^{1/2}(\Gamma(t)))'$-norm. We consider the dual problem: Given $\xi_{\eta}$, find $\dualcoupled{\xi_\eta} = ( \dualbulk{\xi_\eta}, \dualsurface{\xi_\eta}) \in \V(t)$ such that
  \begin{equation}
    \label{eq:coupled-new-dual}
    a^\kappa\big( \underline{\zeta}, \dualcoupled{\xi_\eta} \big) = \int_{\Gamma(t)} \xi_\eta \zeta \dd \sigma
    \qquad \mbox{ for all } \underline{\zeta} = (\zeta, \rho) \in \V(t).
  \end{equation}
  This problem is the weak form of coupled bulk surface elliptic problems and has a unique 
  weak solution in $\V(t)$ (see \citealt[Thm.~3.2]{EllRan13}) which satisfies the estimate
  \[
    \norm{ \dualcoupled{\xi_\eta} }_{\V(t)} \le C \norm{ \xi_\eta }_{H^{1/2}(\Gamma(t))} = C \norm{ \eta }_{(H^{1/2}(\Gamma(t)))'}.
  \]
  The problem with $\rho = 0$ is a weak form of the problem:
  \begin{align}
    - \nabla \cdot ( \A_\Omega \nabla \dualbulk{\xi_\eta} ) + \b_\Omega \cdot \nabla \dualbulk{\xi_\eta} + ( \c_\Omega + \kappa ) \dualbulk{\xi_\eta} & = 0 \\
    \A_\Omega \nabla \dualbulk{\xi_\eta} + \alpha \dualbulk{\xi_\eta} & = \xi_\eta + \alpha \dualsurface{\xi_\eta}.
  \end{align}
  This is a Robin problem which satisfies the regularity estimate \citep{LadUra68,GilTru01}
  \[
    \norm{ \dualbulk{\xi_\eta} }_{H^2(\Omega(t))}
    \le c \norm{ \xi_\eta }_{H^{1/2}(\Gamma(t))} + c \norm{ \dualsurface{\xi_\eta} }_{H^{1/2}(\Gamma)}
    \le c \norm{ \eta }_{(H^{1/2}(\Gamma(t)))'}.
  \]
  The problem with $\zeta = 0$ is a weak form of a surface elliptic problem with right hand 
  side $\beta \dualbulk{\xi_\eta}$ which satisfies the regularity estimate \citep{Aub82}.
  \[
    \norm{ \dualsurface{\xi_\eta} }_{H^2(\Gamma(t))}
    \le c \norm{ \dualbulk{\xi_\eta} }_{L^2(\Gamma(t))}
    \le c \norm{ \eta }_{(H^{1/2}(\Gamma(t)))'}.
  \]
  Combining the two regularity estimates we see that
  \begin{equation}
    \label{eq:coupled-new-dual-reg}
    \norm{ \dualcoupled{\xi_\eta} }_{\Z_0(t)}
    \le c \norm{ \eta }_{(H^{1/2}(\Gamma(t)))'}.
  \end{equation}
  We note that the constant here is independent of time $t \in [0,T]$.

  We see using \cref{eq:coupled-new-dual-2} and $\underline{\zeta} = \underline{\eta} = ( \eta, \xi )$ in \cref{eq:coupled-new-dual} that
  \begin{equation}
    \label{eq:coupled-new-dual-3}
    \norm{ \eta }_{(H^{1/2}(\Gamma(t)))'}
    = \int_{\Gamma(t)} \xi_\eta \eta \dd \sigma
    = a^\kappa\big( \underline{\eta}, \dualcoupled{\xi_\eta} \big)
    = a^\kappa\big( \underline{\eta}, \dualcoupled{\xi_\eta} - I_h \dualcoupled{\xi_\eta} \big)
    + a^\kappa\big( \underline{\eta}, I_h \dualcoupled{\xi_\eta} \big).
  \end{equation}
  For the first term on the right hand side of \cref{eq:coupled-new-dual-3}, we apply the 
  boundedness of $a^\kappa$ and the interpolation estimate \cref{eq:coupled-approx-0} to see
  \[
    \abs{ a^\kappa\big( \underline{\eta}, \dualcoupled{\xi_\eta} - I_h \dualcoupled{\xi_\eta} \big) }
    \le c \norm{ \underline{\eta} }_{\V(t)} \norm{ \dualcoupled{\xi_\eta} - I_h \dualcoupled{\xi_\eta} }_{\V(t)}
    \le c h \norm{ \underline{\eta} }_{\V(t)} \norm{ \dualcoupled{\xi_\eta} }_{\Z_0(t)}.
  \]
  For the second term on the right hand side of \cref{eq:coupled-new-dual-3}, we apply the geometric 
  estimates \cref{eq:a-error-coupled}, \cref{eq:m-error-coupled} and \cref{eq:a-error2-coupled} and the 
  interpolation estimate \cref{eq:coupled-approx-0} to see that
  \begin{align*}
    \abs{ a^\kappa\big( \underline{\eta}, I_h \dualcoupled{\xi_\eta} \big) }
    & = \abs{ a^\kappa( \pi_h \underline{z}, I_h \dualcoupled{\xi_\eta} ) - a^\kappa_h( \Pi_h \underline{z}, \tilde{I}_h \dualcoupled{\xi_\eta} ) } \\
    & \le \abs{ a^\kappa( \pi_h \underline{z} - \underline{z}, I_h \dualcoupled{\xi_\eta} ) - a^\kappa_h( \Pi_h \underline{z} - \underline{z}^{-\ell}, \tilde{I}_h \dualcoupled{\xi_\eta} ) } \\
    & \quad + \abs{ a^\kappa( \underline{z}, I_h \dualcoupled{\xi_\eta} - \dualcoupled{\xi_\eta} ) - a^\kappa_h( \underline{z}^{-\ell}, \tilde{I}_h \dualcoupled{\xi_\eta} - \dualcoupled{\xi_\eta}^{-\ell} ) } \\
    & \quad + \abs{ a^\kappa( \underline{z}, \dualcoupled{\xi_\eta} ) - a^\kappa_h( \underline{z}^{-\ell}, \dualcoupled{\xi_\eta}^{-\ell} ) } \\
    & \le c h^{k+1} \norm{ \underline{z} }_{\Z_0(t)} \norm{ \dualcoupled{\xi_\eta} }_{\Z_0(t)}.
  \end{align*}
  Hence, combining the previous estimates and applying the dual regularity \cref{eq:coupled-new-dual-reg}, we infer that
  \begin{align*}
    \norm{ \eta }_{(H^{1/2}(\Gamma(t)))'}^2
    & \le c \left( h \norm{ \underline{\eta} }_{\V(t)} + h^{k+1} \norm{ \underline{z} }_{\Z_0(t)} \right) \norm{ \dualcoupled{\xi_\eta} }_{\Z_0(t)} \\
    & \le c \left( h \norm{ \underline{\eta} }_{\V(t)} + h^{k+1} \norm{ \underline{z} }_{\Z_0(t)} \right) \norm{ \eta }_{(H^{1/2}(\Gamma(t)))'},
  \end{align*}
  and we have shown that
  \begin{equation}
    \label{eq:coupled-eta-H12-est}
    \norm{ \eta }_{(H^{1/2}(\Gamma(t)))'}
    \le c \left( h \norm{ \underline{\eta} }_{\V(t)} + h^{k+1} \norm{ \underline{z} }_{\Z_0(t)} \right).
  \end{equation}

  Returning to \cref{eq:coupled-b-bound2}, we now see that for $(\eta, \xi) = z - \pi_h z$ and $(\zeta, \rho) \in \Z_0(t)$ that
  \begin{align*}
    \abs{ b\big( t; ( \eta, \xi ), ( \zeta, \rho ) \big) }
    & \le \alpha \abs{ \int_{\Omega(t)} \B( w, \A_\Omega ) \nabla \eta \cdot \nabla \zeta \dd x }
      + \alpha \abs{ \int_{\Omega(t)} \B_{\adv}( w, \b_\Omega )  \eta \cdot \nabla \zeta \dd x } \\
    & \quad + \alpha \abs{ \int_{\Omega(t)} ( \md \c_\Omega + \c_\Omega \nabla \cdot w ) \eta \zeta \dd x }
      + \beta \abs{ \int_{\Gamma(t)} \B( w, \A_\Gamma ) \nabla_\Gamma \xi \cdot \nabla_\Gamma \rho \dd \sigma } \\
    & \quad + \beta \abs{ \int_{\Gamma(t)} \B_{\adv}( w, \b_\Gamma ) \xi \cdot \nabla_\Gamma \rho \dd \sigma }
      + \beta \abs{ \int_{\Gamma(t)} ( \md \c_\Gamma + \c_\Gamma \nabla_\Gamma \cdot w ) \xi \rho \dd \sigma } \\
    & \quad + \abs{ \int_{\Gamma(t)} ( \alpha \eta - \beta \xi ) ( \alpha \zeta - \beta \rho ) \nabla_\Gamma \cdot w \dd \sigma } \\
    & =: \alpha I_1 + \alpha I_2 + \alpha I_3 + \beta I_4 + \beta I_5 + \beta I_6 + I_7.
  \end{align*}
  The estimates for $I_2, I_3, I_5, I_6$ and $I_7$ are clear
  \[
    \alpha I_2 + \alpha I_3 + \beta I_5 + \beta I_6 + I_7
    \le c \norm{ ( \eta, \xi ) }_{\H(t)} \norm{ ( \zeta, \rho ) }_{\Z_0(t)}.
  \]
  For $I_1$, we apply similar reasoning to \cref{lem:bulk-b-bound2} with our estimate for 
  $\eta$ \cref{eq:coupled-eta-H12-est} and for $I_4$, we apply similar reasoning to \cref{eq:surf-b-bound2}.
  Combining these estimates shows the desired bound.
\end{proof}

\begin{theorem}
  \label{thm:coupled-error}
  Let $( u, v ) \in L^2_\V$ be the solution of \cref{eq:weak-coupled} which we assume satisfies the regularity bound
  \begin{equation}
    \sup_{t \in (0,T)} \norm{ (\utext,\vtext) }_{\Z(t)}^2 + \int_0^T \norm{ \md (\utext,\vtext) }_{\Z(t)}^2 \dd t \le C_{u,v}.
  \end{equation}
  Let $(U_h, V_h) \in C^1_{\S_h}$, be the solution of the finite element scheme \cref{eq:fem-coupled} and write $(u_h, v_h) = (U_h, V_h)^\ell$.
  Then we have the following error estimate
  \begin{multline}
    \label{eq:11}
    \sup_{t \in (0,T)} \norm{ (\utext,\vtext) - (u_h,v_h) }_{\H(t)}^2 + h^2 \int_0^T \norm{ (\utext,\vtext) - (u_h,v_h) }_{\V(t)}^2 \dd t \\
    \le c \norm{ (\utext_0,\vtext_0) - (u_{h,0}, v_{h,0}) }^2_{\H(t)}
    + c h^{2k+2} C_{u,v}.
  \end{multline}
\end{theorem}

\begin{proof}
  We apply abstract \cref{thm:error} and check the assumptions.
  We know the lift is stable from \cref{lem:coupled-lift-stab}.
  The existence and boundedness of $g_\ell$ and $b_\ell^\kappa$ are dealt with in \cref{lem:coupled-lift-transport}.
  The interpolation properties \cref{eq:interp-Z0} and \cref{eq:interp-Z} are shown in \cref{lem:coupled-approx}.
  The geometric perturbation estimates 
  \cref{eq:m-error,eq:c-error,eq:ct-error,eq:a-error,eq:b-error,eq:bt-error,eq:a-error2,eq:b-error2,eq:amd-error2,eq:md-error,eq:mdV-error} are shown in \cref{lem:coupled-geom-pert,lem:coupled-w-md-errors}.
  Finally, \cref{eq:b-bound2} is shown in \cref{lem:coupled-b-bound2}.
\end{proof}

\section{Numerical results}
\label{sec:numerical-results}

\subsection{Implementation}
The  finite element methods were implemented using \textsf{DUNE}.  %
In our numerical examples, we integrate the coefficients $\A_\Omega^h, \b_\Omega^h$, $\c_\Omega^h$, $\A_\Gamma^h$, $\b_\Gamma^h$ and $c_\Gamma^h$ using a sufficiently accurate quadrature which does not 
affect the order of convergence of the schemes.%
We discretise in time using an implicit Euler time stepping scheme. The time step $\tau$ is scaled so that the optimal error scales are recovered. At each time step we solve the full system using the generalised minimal residual method,

The code produced to run these computations is available at
\begin{quote}
  \url{https://github.com/tranner/dune-evolving-domains}
\end{quote}

Let $T = 1$, and $t \in [0,T]$.
For $t \ge 0$, we define $\Omega(t)$ via a parametrisation $G \colon \Omega_0 \times \R_+ \to \R$, for $\Omega_0 = B(0,1) \subset \R^3$ the unit ball in three-dimensions.
The parametrisation $G \colon \bar\Omega \times \R_+ \to \R^3$ is given by
\begin{equation*}
  G( x, t ) = \left( a(t)^{1/2} x_1, x_2, x_3 \right), \qquad a(t) = 1 + \frac{1}{4} \sin(t),
\end{equation*}
with velocity field ${w}$ given by
\begin{equation*}
  {w}( x, t ) = \left( \frac{ \cos(t) x_1 }{ 8 ( 1 + 1/4 \sin(t) ) }, 0, 0 \right) \qquad \mbox{ for } x \in \bar\Omega(t).
\end{equation*}
The geometry is the same for each problem, which corresponds to an ellipsoidal domain growing along a single axis, but we solve in and on different parts of the domain.

For each test problem, for each iteration we complete an appropriate number of bisectional refinements (projecting boundary nodes on to the exact surface using the normal projection operator \cref{eq:cp}) from a macro triangulation in order to approximately half the mesh size $h$ and scale the time step $\tau$ to recover the optimal order of convergence -- i.e. $\tau_j = \tau_0 2^{-(k+1)j}$.
We show the error in an $L^2$-norm at the final time.
The experimental order of convergence (eoc) at level $j$ is computed by
\begin{equation*}
  (\mathrm{eoc})_j = \log( E_j / E_{j-1} ) / \log( h_j / h_{j-1} ).
\end{equation*}
Errors in an $H^1$-norm demonstrate an order of convergence less and are not listed here.

\subsection{Problem on  a bulk domain \cref{eq:bulk-eqn}}

We set the parameters in the equation as $\A = (1+x_1^2)\id$, $\b = (1,2,0)$, $\c = \cos(x_1 x_2)$ and compute additional right hand sides in \cref{eq:bulk-eqn-pde} and \cref{eq:bulk-eqn-bc} and take appropriate initial data so that the solution is given by
\begin{align*}
  \utext( x, t ) & = \sin( t ) \cos( \pi x_1 ) \cos( \pi x_2 ) && \mbox{ for } x \in \Omega(t).
\end{align*}
We compute with $k=1,2$. The results are shown in \cref{tab:bulk-1,tab:bulk-2}.

\begin{table}[h!]
    \begin{tabular}{cc|cc}
    $h$ & $\tau$ & $L^2(\Omega(T))$ error & (eoc) \\
    \hline
    $1.10017$ & $1.00000$ & $7.54412 \cdot 10^{-2}$ & --- \\
    $8.82662 \cdot 10^{-1}$ & $2.50000 \cdot 10^{-1}$ & $1.72380 \cdot 10^{-1}$ & $-3.75139$ \\
    $5.23405 \cdot 10^{-1}$ & $6.25000 \cdot 10^{-2}$ & $1.07326 \cdot 10^{-1}$ & $0.90670$ \\
    $2.79882 \cdot 10^{-1}$ & $1.56250 \cdot 10^{-2}$ & $3.17823 \cdot 10^{-2}$ & $1.94407$ \\
    $1.44128 \cdot 10^{-1}$ & $3.90625 \cdot 10^{-3}$ & $8.34529 \cdot 10^{-3}$ & $2.01489$ \\
  \end{tabular}

  \caption{$k=1$}
  \label{tab:bulk-1}
\end{table}

\begin{table}[h!]
    \begin{tabular}{cc|cc}
    $h$ & $\tau$ & $L^2(\Omega(T))$ error & (eoc) \\
    \hline
    $1.10017$ & $1.00000$ & $4.46630 \cdot 10^{-2}$ & --- \\
    $8.82662 \cdot 10^{-1}$ & $1.25000 \cdot 10^{-1}$ & $4.44526 \cdot 10^{-2}$ & $0.02144$ \\
    $5.23405 \cdot 10^{-1}$ & $1.56250 \cdot 10^{-2}$ & $7.59648 \cdot 10^{-3}$ & $3.38076$ \\
    $2.79882 \cdot 10^{-1}$ & $1.95312 \cdot 10^{-3}$ & $1.05698 \cdot 10^{-3}$ & $3.15065$ \\
    $1.44128 \cdot 10^{-1}$ & $2.44141 \cdot 10^{-4}$ & $1.38589 \cdot 10^{-4}$ & $3.06126$ \\
  \end{tabular}

  \caption{$k=2$}
  \label{tab:bulk-2}
\end{table}

\subsection{Problem on a closed surface \cref{eq:surf-eqn}}

We set the parameters in the equation as $\A = (1+x_1^2)\id$, $\b = (1,2,0)-(1,2,0)\cdot\nu \nu$, $\c = \sin(x_1 x_2)$ and compute an additional right hand side in \cref{eq:surf-eqn-pde} and take appropriate initial data so that the solution is given by
\begin{align*}
  \utext( x, t ) & = \sin( t ) x_2 x_3 && \mbox{ for } x \in \Gamma(t).
\end{align*}
We compute with $k=2,3$. The results are shown in \cref{tab:surf-2,tab:surf-3}.

\begin{table}[h!]
    \begin{tabular}{cc|cc}
    $h$ & $\tau$ & $L^2(\Gamma(T))$ error & (eoc) \\
    \hline
    $8.31246 \cdot 10^{-1}$ & $1.00000$ & $9.83996 \cdot 10^{-2}$ & --- \\
    $4.40053 \cdot 10^{-1}$ & $1.25000 \cdot 10^{-1}$ & $1.47435 \cdot 10^{-2}$ & $2.98450$ \\
    $2.22895 \cdot 10^{-1}$ & $1.56250 \cdot 10^{-2}$ & $1.99237 \cdot 10^{-3}$ & $2.94251$ \\
    $1.11969 \cdot 10^{-1}$ & $1.95312 \cdot 10^{-3}$ & $2.50039 \cdot 10^{-4}$ & $3.01456$ \\
    $5.60891 \cdot 10^{-2}$ & $2.44141 \cdot 10^{-4}$ & $3.12365 \cdot 10^{-5}$ & $3.00895$ \\
  \end{tabular}

  \caption{$k=2$}
  \label{tab:surf-2}
\end{table}

\begin{table}[h!]
    \begin{tabular}{cc|cc}
    $h$ & $\tau$ & $L^2(\Gamma(T))$ error & (eoc) \\
    \hline
    $8.31246 \cdot 10^{-1}$ & $1.00000$ & $9.88086 \cdot 10^{-2}$ & --- \\
    $4.40053 \cdot 10^{-1}$ & $6.25000 \cdot 10^{-2}$ & $7.60635 \cdot 10^{-3}$ & $4.03157$ \\
    $2.22895 \cdot 10^{-1}$ & $3.90625 \cdot 10^{-3}$ & $4.92316 \cdot 10^{-4}$ & $4.02476$ \\
    $1.11969 \cdot 10^{-1}$ & $2.44141 \cdot 10^{-4}$ & $3.08257 \cdot 10^{-5}$ & $4.02448$ \\
    $5.60891 \cdot 10^{-2}$ & $1.52588 \cdot 10^{-5}$ & $1.89574 \cdot 10^{-6}$ & $4.03416$ \\
  \end{tabular}

  \caption{$k=3$}
    \label{tab:surf-3}
\end{table}

\subsection{Problem on a coupled bulk-surface domain \cref{eq:bulksurf-eqn}}

We set the parameters in the equation as $\A_X = \id$, $\b_X = (0,0,0)$, $\c_X = 0$, for $X = \Omega$ and $\Gamma$, and $\alpha = \beta = 1$, and compute additional right hand sides in \cref{eq:bulksurf-bulk}, \cref{eq:bulksurf-coupling} and \cref{eq:bulksurf-surf} and take appropriate initial data so that the solution is given by
\begin{align*}
  \utext( x, t ) & = \sin( t ) x_1 x_2 && \mbox{ for } x \in \Omega(t) \\
  \vtext( x, t ) & = \sin( t ) x_2 x_3 && \mbox{ for } x \in \Gamma(t).
\end{align*}
We compute with $k=1,2$. The results are shown in \cref{tab:coupled-1,tab:coupled-2}.

\begin{table}[h!]
    \begin{tabular}{cc|cc|cc}
    $h$ & $\tau$ & $L^2(\Omega(T))$ error & (eoc) & $L^2(\Gamma(T))$ error & (eoc) \\
    \hline
    $1.10017$ & $1.00000$ & $1.40014 \cdot 10^{-2}$ & --- & $7.41054 \cdot 10^{-2}$ & --- \\
    $8.82662 \cdot 10^{-1}$ & $2.50000 \cdot 10^{-1}$ & $2.61297 \cdot 10^{-2}$ & $-2.83240$ & $4.53161 \cdot 10^{-2}$ & $2.23275$ \\
    $5.23405 \cdot 10^{-1}$ & $6.25000 \cdot 10^{-2}$ & $9.52446 \cdot 10^{-3}$ & $1.93118$ & $1.58725 \cdot 10^{-2}$ & $2.00746$ \\
    $2.79882 \cdot 10^{-1}$ & $1.56250 \cdot 10^{-2}$ & $2.61552 \cdot 10^{-3}$ & $2.06458$ & $4.25452 \cdot 10^{-3}$ & $2.10325$ \\
    $1.44128 \cdot 10^{-1}$ & $3.90625 \cdot 10^{-3}$ & $6.72781 \cdot 10^{-4}$ & $2.04591$ & $1.08139 \cdot 10^{-3}$ & $2.06389$ \\
  \end{tabular}

  \caption{$k=1$}
  \label{tab:coupled-1}
\end{table}

\begin{table}[h!]
    \begin{tabular}{cc|cc|cc}
    $h$ & $\tau$ & $L^2(\Omega(T))$ error & (eoc) & $L^2(\Gamma(T))$ error & (eoc) \\
    \hline
    $1.10017$ & $1.00000$ & $2.44058 \cdot 10^{-2}$ & --- & $1.22069 \cdot 10^{-1}$ & --- \\
    $8.82662 \cdot 10^{-1}$ & $1.25000 \cdot 10^{-1}$ & $2.92797 \cdot 10^{-3}$ & $9.62654$ & $9.65135 \cdot 10^{-3}$ & $11.51950$ \\
    $5.23405 \cdot 10^{-1}$ & $1.56250 \cdot 10^{-2}$ & $4.02385 \cdot 10^{-4}$ & $3.79775$ & $1.47977 \cdot 10^{-3}$ & $3.58832$ \\
    $2.79882 \cdot 10^{-1}$ & $1.95312 \cdot 10^{-3}$ & $5.08882 \cdot 10^{-5}$ & $3.30323$ & $1.87863 \cdot 10^{-4}$ & $3.29708$ \\
    $1.44128 \cdot 10^{-1}$ & $2.44141 \cdot 10^{-4}$ & $6.36219 \cdot 10^{-6}$ & $3.13299$ & $2.34864 \cdot 10^{-5}$ & $3.13304$ \\
  \end{tabular}

  \caption{$k=2$}
  \label{tab:coupled-2}
\end{table}

\clearpage
\begin{appendix}

\section{A Fa\'a di Bruno formula for parametric surfaces}
\label{sec:faa-di-bruno}

A partition of the set $\{ 1, \ldots, m \}$ is a collection of non-empty subsets $( \sigma_1, \ldots, \sigma_r )$ such that $\sigma_s \cap \sigma_{s'} = \emptyset$ if $s' \neq s$ and $\bigcup_{s=1}^r \sigma_s = \{ 1, \ldots, m \}$.
We call $r$ the order of the partition and denote by $\abs{ \sigma_s }$ the number of elements in $\sigma_s$.
We say that two sets $\sigma = \{ \sigma{(1)}, \ldots, \sigma{(\abs{\sigma})} \}$ and $\sigma' = \{ \sigma'{(1)}, \ldots, \sigma'{(\abs{\sigma'})} \}$ are ordered and write $\sigma \prec \sigma'$ if $\min_{1 \le l \le \abs{\sigma}} \sigma{(l)} < \min_{1 \le l \le \abs{\sigma'}} \sigma_{s'}{(l)}$.

We denote by $\Pp_{m,r}$ the set of ordered non-empty partitions of $\{ 1, \ldots, m \}$ of order $r$:
\begin{multline*}
  \Pp_{m,r} := \Big\{ ( \sigma_1, \ldots, \sigma_r ) :
  \abs{ \sigma_s } > 0 \mbox{ for } 1 \le s \le r,
  \sigma_s \cap \sigma_{s'} = \emptyset \mbox{ for } 1 \le s \neq s' \le r, \\
  \bigcup_{s=1}^r \sigma_s = \{ 1, \ldots, m \}, \mbox{ and }
  \sigma_s \prec \sigma_{s'} \mbox{ if } 1 \le s < s' \le r \Big\}.
\end{multline*}
We note that we have $\Pp_{m,1} = \{ ( \{ 1, 2, \ldots, m \} ) \}$ and $\Pp_{m,m} = \{ ( \{1\}, \{2\}, \ldots, \{m\} ) \}$. That is that $\Pp_{m,1}$ and $\Pp_{m,m}$ each contain one partition.

For a subset $\sigma \subset \{ 1, \ldots, m \}$, given $ j_1, \ldots, j_m \in \{ 1, \ldots, n \}$, we write for smooth $\eta \colon \R^n \to \R$
\[
  \frac{\partial^{\abs{\sigma}}}{\partial x_{j_\sigma}} \eta :=
  \frac{\partial}{\partial x_{j_{\sigma(\abs{\sigma})}}} \cdots \frac{\partial}{\partial x_{j_{\sigma(1)}}} \eta.
\]
Note that it is possible to have repeated indexes $j_r$ which are not distinguished in this formula.

\begin{theorem}
  \label{thm:faa-di-bruno}
  Let $\Gamma \subset \R^{n+1}$ be a $C^m$ surface with parametrisation $X \colon \Theta \to \Gamma$ over $\Theta \subset \R^n$.
  We denote the components of $X$ by $X(\theta) = ( X_1(\theta), \ldots, X_{n+1}(\theta) )$ for $\theta \in \Theta$ and the components of $\nabla_\Gamma$ by $( (\nabla_\Gamma)_1, \ldots (\nabla_\Gamma)_{n+1} )$.
  Let $f \colon \Gamma \to \R$ be in $C^m(\Gamma)$ and write $F \colon \Theta \to \R$ for the function defined by $F( \theta ) = f( X( \theta ) )$ for $\theta \in \Theta$.
  Then $F \in C^m( \Theta )$ and
  \begin{equation}
    \label{eq:faa-di-bruno}
    \frac{\partial^m}{\partial \theta_{j_m} \ldots \partial \theta_{j_1}} F(\theta)
    = \sum_{r=1}^m \sum_{\lambda_1, \ldots, \lambda_r=1}^{n+1}
    (\nabla_\Gamma)_{\lambda_r} \cdots (\nabla_\Gamma)_{\lambda_1} f( X(\theta) )
    \sum_{\Pp_{m,r}} \prod_{s=1}^r \frac{\partial^{\abs{\sigma_s}}}{\partial \theta_{j_{\sigma_s}}} X_{\lambda_s}(\theta).
  \end{equation}
\end{theorem}

Before we show the result we will include some examples to show how this result can be interpreted and related to previous results.

\begin{example}
  \begin{enumerate}
  \item Consider the case $n=1$ and $\Gamma$ is a flat hypersurface.
    In this case $j_1 = \ldots j_m = 1$. Then the \cref{eq:faa-di-bruno} translates to
    \begin{align*}
      \frac{\mathrm{d}^m}{\mathrm{d} \theta^m} F( \theta )
      = \sum_{r=1}^m \sum_{\lambda_1, \ldots, \lambda_r=1}^{2}
      ( \nabla_\Gamma )_{\lambda_r} \ldots ( \nabla_\Gamma )_{\lambda_1} f( X(\theta) )
      \sum_{\Pp_{m,r}} \prod_{s=1}^r
      \frac{\mathrm{d}^{\abs{\sigma_s}}}{\mathrm{d} \theta^{\abs{\sigma_s}}} X_{\lambda_s}(\theta).
    \end{align*}
    We note that $(\nabla_\Gamma)_\lambda f = \partial f / \partial x$ if $\lambda = 1$ and $0$ otherwise and that $X(\theta)_2 = 0$.
    Finally we note that the terms involving $X(\theta)$ can be reordered as
    \begin{align*}
      \sum_{\Pp_{m,r}} \prod_{s=1}^r
      \frac{\mathrm{d}^{\abs{\sigma_s}}}{\mathrm{d} \theta^{\abs{\sigma_s}}} (X(\theta))_{\lambda_s}
      = \sum_{E(m,r)} \prod_{q=1}^m c_{i} \left( \frac{\mathrm{d}^q}{\mathrm{d}\theta^q} X(\theta) \right)^{i_q},
    \end{align*}
    where $E(m,r)$ is given by
    \[
      E(m,r) = \{ \vec{i} \in \Nbb^m_0 : \sum_{q=1}^m i_q = r \mbox{ and } \sum_{q=1}^m q i_q = m \}.
    \]
    The set $E(m,r)$ rather than being a set of partitions maps each partitions $(\sigma_1, \ldots, \sigma_r)$ to a vector $\vec{i} \in \Nbb^m_0$ 
    such that $i_q$ is the number of sets in the partition such that $\abs{ \sigma_s } = q$ for $q = 1, \ldots, m$.
    We can make the simplification in this case since, here, the order of derivatives is not important, whereas we wish to track this in \cref{eq:faa-di-bruno}.
    Furthermore, the mapping from partitions to vectors $\vec{i}$ is not one-to-one (in particular because $\vec{i}$ does not care about 
    ordering) which results in the constant $c_{\vec{i}}$ which counts how many partitions map to the same $\vec{i}$.

    We recover the result of \citep[Eq. 2.9]{Ber89} in the scalar case:
    \begin{align*}
      \frac{\mathrm{d}^m}{\mathrm{d} \theta^m} F( \theta )
      = \sum_{r=1}^m
      \frac{\partial^r f}{\partial x^r}
      ( X(\theta) )
      \sum_{E(m,r)} \prod_{q=1}^m c_{i} \left( \frac{\mathrm{d}^q}{\mathrm{d}\theta^q} X(\theta) \right)^{i_q}.
    \end{align*}

  \item We return to the general case, but to low numbers of derivatives.
    It is clear that
    \[
      \frac{\partial}{\partial \theta_{j_1}} F( \theta )
      = \sum_{{\lambda_1}=1}^{n+1} (\nabla_\Gamma)_{\lambda_1} f(X(\theta)) \frac{\partial}{\partial \theta_{j_1}} X_{\lambda_1}(\theta).
    \]

    To start to understand the inductive procedure we will apply in the proof of \cref{thm:faa-di-bruno}, we construct a second derivative.
    We apply the product rule to see
    \begin{align*}
      & \frac{\partial^2}{\partial \theta_{j_2} \partial \theta_{j_1}} F( \theta ) \\
      & = \frac{\partial}{\partial \theta_{j_2}} \sum_{{\lambda_1}=1}^{n+1}
        ( \nabla_\Gamma )_{\lambda_1} f(X(\theta)) \frac{\partial}{\partial \theta_{j_1}} X_{\lambda_1}(\theta) \\
      & = \sum_{\lambda_1=1}^{n+1} \left( \frac{\partial}{\partial \theta_{j_2}}
        (\nabla_\Gamma)_{\lambda_1} f(X(\theta)) \right) \frac{\partial}{\partial \theta_{j_1}} X_{\lambda_1}(\theta) \\
      & \quad + (\nabla_\Gamma)_{\lambda_1} f(X(\theta)) \left(
        \frac{\partial}{\partial \theta_{j_2}}\frac{\partial}{\partial \theta_{j_1}} X_{\lambda_1}(\theta)
        \right) \\
      & = \sum_{\lambda_1=1}^{n+1} \sum_{\lambda_2=1}^{n+1}
        ( \nabla_\Gamma )_{\lambda_2} ( \nabla_\Gamma )_{\lambda_1} f( X( \theta ) )
        \frac{\partial}{\partial \theta_{j_2}} X_{\lambda_2}(\theta)
        \frac{\partial}{\partial \theta_{j_1}} X_{\lambda_1}(\theta) \\
      & \quad + \sum_{\lambda_1=1}^{n+1}
        ( \nabla_\Gamma )_{\lambda_1} f( X( \theta ) )
        \frac{\partial^2}{\partial \theta_{j_2} \partial \theta_{j_1} } X_{\lambda_1}(\theta).
    \end{align*}
    We see that the extra derivative either applies to the terms from the lower order derivative involving $X$ or 
    else apply to the terms involving $f$, in which case an extra first derivative of $X$ is included. The result is a sum of these two terms.

    To understand further how the terms involving $X$ arise, we compute a third order derivative.
    Applying the same procedure as above we see
    \begin{align*}
      & \frac{\partial^3}{\partial \theta_{j_3} \partial \theta_{j_2} \partial \theta_{j_3}}
        F(\theta) \\
      & = \frac{\partial}{\partial \theta_{j_3}}
        \left(
        \sum_{\lambda_1=1}^{n+1} \sum_{\lambda_2=1}^{n+1}
        ( \nabla_\Gamma )_{\lambda_2} ( \nabla_\Gamma )_{\lambda_1} f( X( \theta ) )
        \frac{\partial}{\partial \theta_{j_2}} X_{\lambda_2}(\theta)
        \frac{\partial}{\partial \theta_{j_1}} X_{\lambda_1}(\theta)
        \right. \\
      & \quad
        \left.
        + \sum_{\lambda_1=1}^{n+1}
        ( \nabla_\Gamma )_{\lambda_1} f( X( \theta ) )
        \frac{\partial^2}{\partial \theta_{j_2} \partial \theta_{j_1} } X_{\lambda_1}(\theta)
        \right) \\
      & = \sum_{\lambda_1=1}^{n+1} \sum_{\lambda_2=1}^{n+1}
        \left( \frac{\partial}{\partial \theta_{j_3}} ( \nabla_\Gamma )_{\lambda_2} ( \nabla_\Gamma )_{\lambda_1} f( X( \theta ) ) \right)
        \frac{\partial}{\partial \theta_{j_2}} X_{\lambda_2}(\theta)
        \frac{\partial}{\partial \theta_{j_1}} X_{\lambda_1}(\theta) \\
       & \quad +
        \sum_{\lambda_1=1}^{n+1} \sum_{\lambda_2=1}^{n+1}
        ( \nabla_\Gamma )_{\lambda_2} ( \nabla_\Gamma )_{\lambda_1} f( X( \theta ) )
        \frac{\partial}{\partial \theta_{j_3}} \left(
        \frac{\partial}{\partial \theta_{j_2}} X_{\lambda_2}(\theta)
         \frac{\partial}{\partial \theta_{j_1}} X_{\lambda_1}(\theta) \right) \\
      & \quad
        + \sum_{\lambda_1=1}^{n+1}
        \left(
        \frac{\partial}{\partial \theta_{j_3}} ( \nabla_\Gamma )_{\lambda_1} f(X(\theta))
        \right)
        \frac{\partial^2}{\partial \theta_{j_2} \partial \theta_{j_1} } X_{\lambda_1}(\theta)
         \\
      & \quad
        + \sum_{\lambda_1=1}^{n+1}
        ( \nabla_\Gamma )_{\lambda_1} f( X( \theta ) )
        \frac{\partial}{\partial \theta_{j_3}} \frac{\partial^2}{\partial \theta_{j_2} \partial \theta_{j_1} } X_{\lambda_1}(\theta).
    \end{align*}
    Again, we see that the $\theta_{j_3}$ derivative can either be applied to the terms involving $f$ or the terms involving $X$ and the result is a sum of these terms.
    Computing further we see
    \begin{align*}
      & \frac{\partial^3}{\partial \theta_{j_3} \partial \theta_{j_2} \partial \theta_{j_3}}
        F(\theta) \\
      & = \sum_{\lambda_1=1}^{n+1} \sum_{\lambda_2=1}^{n+1} \sum_{\lambda_3=1}^{n+1}
        ( \nabla_\Gamma )_{\lambda_3} ( \nabla_\Gamma )_{\lambda_2} ( \nabla_\Gamma )_{\lambda_1} f( X( \theta ) )
        \frac{\partial}{\partial \theta_{j_3}} X_{\lambda_3}(\theta)
        \frac{\partial}{\partial \theta_{j_2}} X_{\lambda_2}(\theta)
        \frac{\partial}{\partial \theta_{j_1}} X_{\lambda_1}(\theta) \\
      & \quad +
        \sum_{\lambda_1=1}^{n+1} \sum_{\lambda_2=1}^{n+1}
        ( \nabla_\Gamma )_{\lambda_2} ( \nabla_\Gamma )_{\lambda_1} f( X( \theta ) ) \\
      & \qquad\quad
        \cdot \left(
        \frac{\partial^2}{\partial \theta_{j_3} \partial \theta_{j_2}} X_{\lambda_2}(\theta)
        \frac{\partial}{\partial \theta_{j_1}} X_{\lambda_1}(\theta)
        + \frac{\partial}{\partial \theta_{j_2}} X_{\lambda_2}(\theta)
        \frac{\partial^2}{\partial \theta_{j_3} \partial \theta_{j_1}} X_{\lambda_1}(\theta) \right. \\
      & \qquad\qquad\quad
        \left. + \frac{\partial}{\partial \theta_{j_3}} X_{\lambda_2}(\theta)
        \frac{\partial^2}{\partial \theta_{j_2} \partial \theta_{j_1} } X_{\lambda_1}(\theta)
        \right)
         \\
      & \quad
        + \sum_{\lambda_1=1}^{n+1}
        ( \nabla_\Gamma )_{\lambda_1} f( X( \theta ) )
        \frac{\partial^3}{\partial \theta_{j_3} \partial \theta_{j_2} \partial \theta_{j_1} } X_{\lambda_1}(\theta).
    \end{align*}
    We end up with terms involving one, two and three derivatives of $f$ times sums of products of one, two and three terms involving derivatives of $X$ respectively.
    We notice that the derivatives on $X$ are not all arbitrary combinations of $j$'s and $\lambda$'s.
    The $\theta_{j_1}$ derivative always is against $X_{\lambda_1}$, the $\theta_{j_2}$ derivative is always 
    against $X_{\lambda_1}$ or $X_{\lambda_2}$ and the $\theta_{j_3}$ derivative is against $X_{\lambda_1}$, $X_{\lambda_2}$ or $X_{\lambda_3}$.
    Moreover, we see that the minimum index against $X_{\lambda_s}$ is increasing in $s$.
    This property is written more formally in the definition of $\Pp_{m,r}$.
    Making this association rigorous is the route we take to proving the theorem.
  \end{enumerate}
\end{example}

We will show the result through the following two lemmas.
 
\begin{lemma}
  The derivatives of $F$ can be written as
  \begin{equation}
    \label{eq:faa-di-bruno-lem1}
    \frac{\partial^m}{\partial \theta_{j_m} \ldots \partial \theta_{j_1}} F(\theta)
    = \sum_{r=1}^m \sum_{\lambda_1, \ldots, \lambda_r =1}^{n+1}
    (\nabla_\Gamma)_{\lambda_r} \cdots (\nabla_\Gamma)_{\lambda_1} f( X(\theta) )
    \alpha( j_1, \ldots, j_m; \lambda_1, \ldots, \lambda_r )( \theta ) \end{equation}
 where $\alpha$ satisfies the recurrence relationships
 \begin{subequations}
    \begin{align}
      \label{eq:base}
      \alpha(j; \lambda)(\theta)
      & = \frac{\partial}{\partial \theta_j} X_\lambda(\theta) \\
      \label{eq:case-1}
      \alpha(j_1,\ldots, j_{m+1}; \lambda )( \theta )
      & = \frac{\partial}{\partial \theta_{j_{m+1}}} \alpha( j_1, \ldots, j_m; \lambda )( \theta ) \\
      \label{eq:case-r}
      \alpha(j_1,\ldots, j_{m+1}; \lambda_1, \ldots, \lambda_r )( \theta )
      & = \frac{\partial}{\partial \theta_{j_{m+1}}} \alpha( j_1, \ldots, j_m; \lambda_1, \ldots, \lambda_r )( \theta ) \\
      \nonumber
      & \quad + \left( \frac{\partial }{\partial \theta_{j_{m+1}}}X_{\lambda_r}( \theta ) \right)
        \alpha(j_1, \ldots, j_m, \lambda_1, \ldots, \lambda_{r-1} )( \theta )
      && \mbox{ for } 2 \le r \le m \\
      \label{eq:case-m}
      \alpha(j_1, \ldots, j_{m+1}; \lambda_1, \ldots, \lambda_{m+1} )( \theta )
      & = \left( \frac{\partial }{\partial \theta_{j_{m+1}}}X_{\lambda_{m+1}}(\theta) \right) \alpha(j_1, \ldots, j_m, \lambda_1, \ldots, \lambda_{m} )( \theta ).
    \end{align}
  \end{subequations}
\end{lemma}

\begin{proof}
  For the base case \cref{eq:base}, we see that
  \[
    \frac{\partial}{\partial \theta_j} F(\theta)
    = \sum_{\lambda=1}^{n+1} (\nabla_\Gamma)_\lambda f( X(\theta) ) \frac{\partial }{\partial \theta_j}X_\lambda(\theta)
    = \sum_{\lambda=1}^{n+1} (\nabla_\Gamma)_\lambda f( X(\theta) ) \alpha( j ; \lambda )( \theta ).
  \]

  Then, given \cref{eq:faa-di-bruno-lem1} holds up to $m$ derivatives, we see that
  \begin{align*}
    & \frac{\partial}{\partial \theta_j} \frac{\partial^m}{\partial \theta_{j_m} \cdots \partial \theta_{j_1}} F(\theta) \\
    & = \frac{\partial}{\partial \theta_j} \sum_{r=1}^m \sum_{\lambda_1=1}^{n+1} \cdots \sum_{\lambda_r=1}^{n+1} (\nabla_\Gamma)_{\lambda_r} \cdots (\nabla_\Gamma)_{\lambda_1} f( X(\theta) ) \alpha( j_1, \ldots, j_m; \lambda_1, \ldots, \lambda_r )(\theta) \\
    & = \sum_{r=1}^m \sum_{\lambda_1=1}^{n+1} \cdots \sum_{\lambda_r=1}^{n+1} \left \{ \left( \frac{\partial}{\partial \theta_j} (\nabla_\Gamma)_{\lambda_r} \cdots (\nabla_\Gamma)_{\lambda_1} f( X(\theta) ) \right) \alpha( j_1, \ldots, j_m; \lambda_1, \ldots, \lambda_r )(\theta) \right. \\
    & \qquad + \left. (\nabla_\Gamma)_{\lambda_r} \cdots (\nabla_\Gamma)_{\lambda_1} f( X(\theta) ) \frac{\partial}{\partial \theta_j} \alpha( j_1, \ldots, j_m; \lambda_1, \ldots, \lambda_r )(\theta) \right\} \\
    & = \sum_{r=1}^m \sum_{\lambda_1=1}^{n+1} \cdots \sum_{\lambda_r=1}^{n+1} \left \{ \sum_{\lambda_{r+1}=1}^{n+1} (\nabla_\Gamma)_{\lambda_{r+1}} (\nabla_\Gamma)_{\lambda_r} \cdots (\nabla_\Gamma)_{\lambda_1} f( X(\theta) )
      \left( \frac{\partial }{\partial \theta_{j}}X_{\lambda_{r+1}}(\theta) \right) \alpha( j_1, \ldots, j_m; \lambda_1, \ldots, \lambda_r )(\theta) \right. \\
    & \qquad + \left. (\nabla_\Gamma)_{\lambda_r} \cdots (\nabla_\Gamma)_{\lambda_1} f( X(\theta) ) \frac{\partial}{\partial \theta_j} \alpha( j_1, \ldots, j_m; \lambda_1, \ldots, \lambda_r )( \theta ) \right\} \\
    & = \sum_{r=2}^{m+1} \sum_{\lambda_1=1}^{n+1} \cdots \sum_{\lambda_r=1}^{n+1} (\nabla_\Gamma)_{\lambda_r} \cdots (\nabla_\Gamma)_{\lambda_1} f( X( \theta) ) \left( \frac{\partial }{\partial \theta_{j}}X_{\lambda_{r}}(\theta) \right) \alpha( j_1, \ldots, j_m; \lambda_1, \ldots, \lambda_{r-1} )(\theta) \\
    & \quad + \sum_{r=1}^{m} \sum_{\lambda_1=1}^{n+1} \cdots \sum_{\lambda_r=1}^{n+1} (\nabla_\Gamma)_{\lambda_r} \cdots (\nabla_\Gamma)_{\lambda_1} f(X(\theta)) \frac{\partial}{\partial \theta_j} \alpha( j_1, \ldots, j_m; \lambda_1, \ldots, \lambda_r )(\theta) \\
    & = \sum_{\lambda_1=1}^{n+1} ( \nabla_\Gamma )_{\lambda_1} f(X(\theta)) \frac{\partial}{\partial \theta_j} \alpha( j_1, \ldots, j_m; \lambda_1 )( \theta ) \\
    & \quad + \sum_{r=2}^{m} \sum_{\lambda_1=1}^{n+1} \cdots \sum_{\lambda_r=1}^{n+1}
      ( \nabla_\Gamma )_{\lambda_r} \cdots ( \nabla_\Gamma )_{\lambda_1} f( X(\theta) )
      \left\{
      \left( \frac{\partial }{\partial \theta_{j}}X_{\lambda_{r}}(\theta)\right) \alpha( j_1, \ldots, j_m; \lambda_1, \ldots, \lambda_{r-1} )( \theta )
      \right. \\
    & \qquad\qquad
      \left.
      + \frac{\partial}{\partial \theta_j} \alpha( j_1, \ldots, j_m ; \lambda_1, \ldots, \lambda_r ) ( \theta )
      \right\} \\
    & \quad + \sum_{\lambda_{1}=1}^{n+1} \cdots \sum_{\lambda_{m+1}=1}^{n+1}
      ( \nabla_\Gamma )_{\lambda_{m+1}} \cdots ( \nabla_\Gamma )_{\lambda_{1}} f( X( \theta ) )
      \left( \frac{ \partial }{ \partial \theta_ j } X_{\lambda_{m+1}}( \theta ) \right) \alpha( j_1, \ldots, j_m; \lambda_1, \ldots, \lambda_{m} )( \theta ).
  \end{align*}
  Reading off coefficients gives the result.
\end{proof}

\begin{lemma}
  It holds that
 \[
   \alpha(j_1,\ldots, j_m; \lambda_1, \ldots, \lambda_r)
   = \sum_{\Pp_{m,r}} \prod_{s=1}^r \frac{\partial^{\abs{\sigma_s}}}{\partial \theta_{j_{\sigma_s}}} X_{\lambda_s}(\theta).
\]
\end{lemma}

\begin{proof}
  For the base case, $m=1$, we see that $\alpha( j; \lambda )(\theta) = {\partial}/{\partial \theta_j}\, X_\lambda(\theta)$ and clearly the set of ordered partitions is $\Pp_{1,1} = \{ (1) \}$.

  Suppose the identity holds for derivatives up to order $m$.

  First, we consider $r=1$, then
  \begin{align*}
    \alpha( j_1, \ldots, j_{m+1} ; \lambda )(\theta)
    = \frac{\partial}{\partial \theta_{j_{m+1}}} \alpha( j_1, \ldots, j_m; \lambda )(\theta)
    = \cdots
    = \frac{\partial^{m+1}}{\partial \theta_{j_{m+1}} \ldots \partial \theta_{j_1}} X_{\lambda}(\theta).
  \end{align*}
  We also have that there is only one partition of $1, \ldots, m+1$ into $1$ subset: $\Pp_{m+1,1} = \{ ( 1, \ldots, m+1 ) \}$.

  Next, we consider $r=m+1$, then
  \begin{align*}
    \alpha( j_1, \ldots, j_{m+1}; \lambda_1, \ldots, \lambda_{m+1} )(\theta)
    & = \left( \frac{\partial}{\partial \theta_{j_{m+1}}}X_{\lambda_{m+1}}(\theta) \right) \alpha( j_1, \ldots, j_m; \lambda_1, \ldots, \lambda_{m} )(\theta)
    \\
    & = \cdots
    = \prod_{s=1}^{m+1} \frac{\partial}{\partial \theta_{j_s}}X_{\lambda_{s}}(\theta).
  \end{align*}
  We also have that there is only partition of $1, \ldots, m+1$ into $m+1$ ordered sets: $\Pp_{m+1,m+1} = \{ ( ( 1 ), ( 2 ), \ldots, ( m+1 ) ) \}$ so this case is complete.

  Next, we consider $2 \le r \le m-1$. We see
  \begin{align*}
    & \alpha( j_1, \ldots, j_{m+1}; \lambda_1, \ldots, \lambda_r )( \theta ) \\
    & = \frac{\partial}{\partial \theta_{j_{m+1}}} \alpha( j_1, \ldots, j_m; \lambda_1, \ldots, \lambda_r )( \theta )
      + \left( \frac{\partial}{\partial \theta_{j_{m+1}}} X_{\lambda_r}(\theta) \right)
      \alpha(j_1, \ldots, j_m, \lambda_1, \ldots, \lambda_{r-1} )(\theta) \\
    & = \sum_{\Pp_{m,r}} \sum_{\alpha=1}^r \prod_{s=1}^r \frac{\partial^{\abs{\tilde{\sigma}_{s,\alpha}}}}{\partial \theta_{J_{\tilde{\sigma}_{s,\alpha}}}} X_{\lambda_s}(\theta)
      + \sum_{\Pp_{m,r-1}} \left( \frac{\partial}{\partial \theta_{j_{m+1}}} X_{\lambda_r}(\theta) \right)
      \prod_{s=1}^r \frac{\partial^{\abs{\sigma_s}}}{\partial \theta_{j_{\sigma_s}}} X_{\lambda_s}( \theta ) \\
    & = \sum_{Q} \prod_{s=1}^r \frac{\partial^{\abs{\sigma_s}}}{\partial \theta_{{j}_{\sigma_s}}} X_{\lambda_s}( \theta ),
  \end{align*}
  where $\tilde{\sigma}_{s,\alpha} = \sigma_s \cup \{ m+1 \}$ if $s = \alpha$ and $\sigma_s$ otherwise and $Q$ is given by
  \begin{align*}
    Q & := \bigcup_{\alpha=1}^r \{ ( \tilde{\sigma}_{1,\alpha}, \ldots, \tilde{\sigma}_{r,\alpha} ) : (\sigma_1,\ldots, \sigma_r) \in \Pp_{m,r} \} \\
      & \quad \cup \{ ( \sigma_1, \ldots, \sigma_{r-1}, ( m+1 ) ) : ( \sigma_1, \ldots \sigma_{r-1} ) \in \Pp_{m,r-1} \} \\
      & =: \bigcup_{\alpha=1}^r Q_\alpha \cup Q_0.
  \end{align*}
  The proof will be complete if we show $Q = \Pp_{m+1,r}$.

  First, let $ (\sigma_1, \ldots, \sigma_r) \in Q$. It is clear that $(\sigma_1, \ldots, \sigma_r)$ is a partition. There are two 
  cases to check the ordering $(\sigma_1,\ldots, \sigma_r) \in Q_0$ (i) or $Q_\alpha$ for $\alpha = 1, \ldots, r$ (ii).
  For case (i), we have that $\sigma_r = \{ m+1 \}$ and $(\sigma_1,\ldots, \sigma_{r-1} ) \in \Pp_{m,r-1}$. The first assertion shows 
  that $\sigma_s \prec \sigma_r$ for $1 \le s \le r-1$ and the second shows that $\sigma_s \prec \sigma_{s'}$ for $1 \le s < s' \le r-1$. So we see that $( \sigma_1, \ldots, \sigma_r ) \in \Pp_{m,r}$.
  For case (ii), we have that ${m+1} \in \sigma_{\alpha}$ and $(\sigma_1, \ldots, \sigma_r) = ( \tilde{\sigma'}_{1,\alpha}, \ldots, \tilde{\sigma'}_{r,\alpha} )$ 
  for some $(\sigma'_1, \ldots, \sigma'_r) \in \Pp_{m,r}$. Since $(\sigma'_1,\ldots, \sigma'_r)$ is an ordered partition, we see 
  that $\sigma'_{\alpha} = \sigma_{\alpha} \setminus \{ {m+1} \}$ is non empty and the smallest index in $\sigma'_{\alpha}$ is the same as 
  the smallest index in $\sigma_\alpha$ so that the ordering property is preserved.

  Second, let $(\sigma_1, \ldots, \sigma_r) \in \Pp_{m+1,r}$.
  Let $1 \le s \le r$ be such that ${m+1} \in \sigma_{s}$.
  First, suppose that removing ${m+1}$ from $\sigma_{s}$ results in a non-empty set.
  Then $( \sigma_1, \ldots, \sigma_s \setminus \{ m+1 \}, \ldots, \sigma_r )$ is a partition in $\Pp_{m,r}$. In this case, $(\sigma_1, \ldots, \sigma_r) \in Q_{s}$.
  Otherwise suppose $\sigma_s \setminus \{ {m+1} \}$ is empty.
  Then, by the ordering of partitions in $\Pp_{m,r}$, we must have that $s = r$ and $(\sigma_1, \ldots, \sigma_{r-1})$ is an ordered partition of $(1, \ldots, {m})$.
  In this case $(\sigma_1, \ldots, \sigma_r) \in Q_0$.

  Thus we have shown the desired form of $\alpha$.
\end{proof}

\begin{proof}[Proof of \cref{thm:faa-di-bruno}]
  Simply combine the two previous lemmas.
\end{proof}

We could also apply a similar result using an inverse parametrisation $X^{-1}$ and recover coefficients involving derivatives of $X^{-1}$. 
In the applications we consider, higher derivatives of $X^{-1}$ are hard to estimate. As an alternative we give the following result which is based on rearranging terms in \cref{eq:faa-di-bruno}.

\begin{corollary}
  \label{cor:faa-di-bruno-inverse}
  Under the same assumptions as \cref{thm:faa-di-bruno} we have:
  \begin{equation}
    \label{eq:faa-bi-bruno-inverse}
    \begin{aligned}
      & ( \nabla_\Gamma )_{i_m} \cdots ( \nabla_\Gamma )_{i_1} f( X( \theta ) ) \\
      & = \sum_{j_m, \ldots, j_1=1}^n
      \frac{\partial^m}{\partial \theta_{j_m} \ldots \partial \theta_{j_1}} F(\theta)
      \left( \prod_{s=1}^m ( \nabla X( \theta ) )^\dagger_{j_s,i_s} \right) \\
      & \qquad -
      \sum_{r=1}^{m-1} \sum_{\lambda_1, \ldots, \lambda_r=1}^{n+1}
      (\nabla_\Gamma)_{\lambda_r} \cdots (\nabla_\Gamma)_{\lambda_1} f( X(\theta) )
      \sum_{j_m, \ldots, j_1=1}^n
      \sum_{\Pp_{m,r}} \prod_{s=1}^r \frac{\partial^{\abs{\sigma_s}}}{\partial \theta_{j_{\sigma_s}}} X_{\lambda_s}(\theta)
      \prod_{s=1}^m ( \nabla X )^\dagger_{j_s,i_s}
    \end{aligned}
  \end{equation}
  where $P_{ij} = P_{ij}(x) = \delta_{ij} - \nu_i(x) \nu_j(x)$ is projection onto the tangent space of $\Gamma$ at $x$.
\end{corollary}

\begin{proof}
  We start by noting that we have two equivalent ways of computing surface derivatives: either using the projection onto the 
  tangent space of $\Gamma$ or using the parametrisation of $\Gamma$. We apply each of these formulae to the function $x \mapsto x_i$ and take the $(\nabla_\Gamma)_j$ derivative at $x = X(\theta)$ to see
  \begin{equation}
    \label{eq:right-inverse1}
    P_{ij}( x ) = \delta_{ij} - \nu_i(x) \nu_j(x)
    = ( \nabla_\Gamma )_j ( x_i )
    = \sum_{\lambda=1}^n \left( \frac{\partial }{\partial \theta_\lambda} X_i(\theta) \right) ( \nabla X(\theta) )_{\lambda, j}^\dagger.
  \end{equation}
  In particular, we see that
  \begin{equation}
    \label{eq:right-inverse}
    \sum_{j_m=1}^{n} \cdots \sum_{j_1=1}^n
    \left( \prod_{s=1}^m \frac{\partial}{\partial \theta_{j_s} } X_{\lambda_s}(\theta) \right)
    \left( \prod_{s=1}^m ( \nabla X(\theta) )^\dagger_{j_s i _s} \right)
    = \prod_{s=1}^m P_{\lambda_s, i_s}( x ).
  \end{equation}

  The next step of the proof is to split the right hand side of \cref{eq:faa-di-bruno} into terms involving $m$th order derivatives of $X$ and the rest.
  We multiply by $\prod_{s=1}^m ( \nabla X )^\dagger_{j_s,i_s}$ and sum of each $j_s$ in turn to see that
 
  \begin{align*}
    & \sum_{j_m, \ldots, j_1=1}^n \frac{\partial^m}{\partial \theta_{j_m} \ldots \partial \theta_{j_1}} F(\theta) \prod_{s=1}^m ( \nabla X )^\dagger_{j_s,i_s} \\
    & = \sum_{\lambda_1, \ldots, \lambda_m=1}^{n+1}
    (\nabla_\Gamma)_{\lambda_m} \cdots (\nabla_\Gamma)_{\lambda_1} f( X(\theta) )
      \sum_{j_m, \ldots, j_1=1}^n
      \prod_{s=1}^m \frac{\partial}{\partial \theta_{j_s}} X_{\lambda_s}( \theta )
      \prod_{s=1}^m ( \nabla X )^\dagger_{j_s,i_s}
    \\
    & \qquad +
      \sum_{r=1}^{m-1} \sum_{\lambda_1, \ldots, \lambda_r=1}^{n+1}
      (\nabla_\Gamma)_{\lambda_r} \cdots (\nabla_\Gamma)_{\lambda_1} f( X(\theta) )
      \sum_{j_m, \ldots, j_1=1}^n
      \sum_{\Pp_{m,r}} \prod_{s=1}^r \frac{\partial^{\abs{\sigma_s}}}{\partial \theta_{j_{\sigma_s}}} X_{\lambda_s}(\theta)
      \prod_{s=1}^m ( \nabla X )^\dagger_{j_s,i_s}.
  \end{align*}
  For the first term on the right hand side we apply \cref{eq:right-inverse} and that the tangential gradient is already in the tangent space to see
  \begin{align*}
    & \sum_{\lambda_1, \ldots, \lambda_m=1}^{n+1}
    (\nabla_\Gamma)_{\lambda_m} \cdots (\nabla_\Gamma)_{\lambda_1} f( X(\theta) )
      \sum_{j_m, \ldots, j_1=1}^n
      \prod_{s=1}^m \frac{\partial}{\partial \theta_{j_s}} X_{\lambda_s}( \theta )
      \prod_{s=1}^m ( \nabla X )^\dagger_{j_s,i_s} \\
    & = \sum_{\lambda_1, \ldots, \lambda_m=1}^{n+1}
    (\nabla_\Gamma)_{\lambda_m} \cdots (\nabla_\Gamma)_{\lambda_1} f( X(\theta) )
      \prod_{s=1}^m P_{\lambda_s, i_s}( x ) \\
    & = ( \nabla_\Gamma )_{i_m} \cdots ( \nabla_\Gamma )_{i_1} f( X( \theta ) ).
  \end{align*}
  The result then follows by simply rearranging the terms.
\end{proof}

\end{appendix}

\end{document}